\documentclass{amsart}
 \usepackage[fontsize=10pt]{fontsize} 

\usepackage[a4paper, left=1.5cm, right=1.5cm, top=0.5cm,bottom=0.7cm]{geometry}  

\usepackage{amsmath,amsthm} 
\usepackage{latexsym,amssymb}  
\usepackage{amsfonts,mathrsfs} 
\usepackage{stmaryrd} 

\usepackage{color}
\usepackage{graphicx}
\usepackage{epstopdf} 
\usepackage{indentfirst} 
\usepackage{url}
\usepackage{color}
\usepackage[linktocpage=true, hidelinks, colorlinks, linkcolor=blue, citecolor=magenta, bookmarksopen=true, urlcolor=brown]{hyperref}

\usepackage{enumitem}
 \usepackage[normalem]{ulem} 
 \usepackage{verbatim, comment} 
 \usepackage{svn, mathtools, mathscinet} 

\usepackage[all]{xy} 
\usepackage{tikz-cd, amscd}
 \usepackage{array}
 \usepackage{tabu} 
 
 \label{Theorems environ}
 \newtheorem{thm}{Theorem}[section]
\newtheorem{theorem}[thm]{Theorem}
\newtheorem{prop}[thm]{Proposition}

\newtheorem{lem}[thm]{Lemma}
\newtheorem{lemma}[thm]{Lemma}

\newtheorem{cor}[thm]{Corollary}
\newtheorem{lem-def}[thm]{Lemma-Definition}

\theoremstyle{definition}
\newtheorem{defn}[thm]{Definition}
\newtheorem{dfn}[thm]{Definition}

\newtheorem{rmk}[thm]{Remark}
\newtheorem{rem}[thm]{Remark}
\newtheorem{remark}[thm]{Remark}

 \newtheorem{construction}[thm]{Construction}

\newtheorem{exam}[thm]{Example}

\newtheorem{notation}[thm]{Notation}
\newtheorem{Notation}[thm]{Notation}

\newtheorem{convention}[thm]{Convention}

\numberwithin{equation}{section}
\label{••••••••••3}
\label{From Min-Wang}

\newcommand{\Zp}{{\bZ_p}}

\newcommand{\Ainf}{{\mathrm{A_{inf}}}}

\newcommand{\AAinf}{{\bA_{\mathrm{inf}}}}            
\newcommand{\BBdRp}{{\bB_{\mathrm{dR}}^+}}           
\newcommand{\BBdR}{{\bB_{\mathrm{dR}}}}              

\newcommand{\colim}{{\mathrm{colim}}}       

\newcommand{\dlog}{{\mathrm{dlog}}}         
\newcommand{\End}{{\mathrm{End}}}           
\newcommand{\Fil}{{\mathrm{Fil}}}           
\newcommand{\Gr}{{\mathrm{Gr}}}             
\newcommand{\Hom}{{\mathrm{Hom}}}           
\newcommand{\id}{{\mathrm{id}}}             
\newcommand{\Ker}{{\mathrm{Ker}}}           
\newcommand{\pr}{{\mathrm{pr}}}             
\newcommand{\Rep}{{\mathrm{Rep}}}           
\newcommand{\Sh}{{\mathrm{Sh}}}             
\newcommand{\Spa}{{\mathrm{Spa}}}           
\newcommand{\Spf}{{\mathrm{Spf}}}           
\newcommand{\Vect}{{\mathrm{Vect}}}          

\newcommand{\GL}{{\mathrm{GL}}}             

\newcommand{\an}{{\mathrm{an}}}             
\newcommand{\dR}{{\mathrm{dR}}}             
\newcommand{\et}{{\mathrm{\text{\'e}t}}}    
\newcommand{\geo}{{\mathrm{geo}}}           
\newcommand{\pd} {{\mathrm{pd}}}                           
\newcommand{\perf}{\mathrm{perf}}           
\newcommand{\proet}{{\mathrm{pro\text{\'e}t}}}

\label{••••••••••5}

\label{site: prismatic}

\usepackage{relsize}
\usepackage[bbgreekl]{mathbbol}
\DeclareSymbolFontAlphabet{\mathbb}{AMSb} 
\DeclareSymbolFontAlphabet{\mathbbl}{bbold}
\newcommand{\Prism}{{\mathlarger{\mathbbl{\Delta}}}} 
\newcommand{\prism}{{\mathlarger{\mathbbl{\Delta}}}} 
 \newcommand{\pris}{{\mathlarger{\mathbbl{\Delta}}}} 

\newcommand{\okpris}{{(\mathcal{O}_K)_\prism}}

\newcommand{\okprisast}{{(\mathcal{O}_K)_{\prism, \ast}}}

\newcommand{\okprisperf}{(\calO_K)^{\perf}_{\Prism}}

 \newcommand{\Strat}{\mathrm{Strat}}
 \newcommand{\strat}{\mathrm{Strat}}

\newcommand{\ya}{{\rangle}}
\newcommand{\za}{{\langle}}

\label{to move around}

\label{Topology, THH, TC, TP}

 \label{colors}

 \newcommand{\ghblue}[1]{{\color{blue}#1}}

\label{arrows}

\newcommand \into {\hookrightarrow }
\renewcommand \to {\rightarrow}
\newcommand \onto {\twoheadrightarrow}
\renewcommand{\projlim}{\varprojlim}

\label{•••••••••4}
 
 \label{matrix, groups}

\newcommand{\vect}{\mathrm{Vect}}

\def\Mat{\mathrm{Mat}}

\newcommand{\nil}{\mathrm{nil}}

\label{module, homological alg}
 
\newcommand{\rep}{{\mathrm{Rep}}}

\def\cont{\mathrm{cont}}
\def\inf{{\mathrm{inf}}}

\newcommand{\gal}{{\mathrm{Gal}}}

\newcommand{\fil}{{\mathrm{Fil}}}

\newcommand{\spa}{{\mathrm{Spa}}}

\newcommand{\spf}{{\mathrm{Spf}}}

\def\an{\mathrm{an}}
\def\perf{\mathrm{perf}}

\def\et{{\mathrm{\acute{e}t}}}

\def\proet{{\mathrm{pro\acute{e}t}}}

\def\geo{{\mathrm{geo}}}

\newcommand{\Lie}{\mathrm{Lie}}

\newcommand{\VB}{\mathrm{VB}}

\newcommand{\hatotimes}{\widehat{\otimes}}

\label{condesend, solid math}

\label{!•••••••••}
\label{ALL: p-adic Hodge}

\label{LAV}

\newcommand{\la}{{\mathrm{la}}}
\newcommand{\dan}{\text{$\mbox{-}\mathrm{an}$}}

\newcommand{\dla}{\text{$\mbox{-}\mathrm{la}$}}

\renewcommand{\log}{\mathrm{log}}

 \newcommand{\hatl}{{\widehat{L}}}

\newcommand{\hatg}{{\hat{G}}}
\newcommand{\bdrplusm}{{\bfb_{\dR, m}^+}}
\newcommand{\bdrpluslm}{{\bfb_{\dR, L, m}^+}}

\newcommand{\kpinfty}{{K_{p^\infty}}}

\newcommand{\kinfty}{{K_{\infty}}}
\newcommand{\hatkinfty}{{\widehat{K_{\infty}}}}

\label{LAV: Lie gp Lie alg}

\newcommand{\gammak}{{\Gamma_K}}

\newcommand{\gk}{{G_K}}

\label{period ring: Fontaine}

\newcommand\dr{{\mathrm{dR}}}
\newcommand\HT{{\mathrm{HT}}}
\newcommand\Sen{{\mathrm{Sen}}}

\newcommand\sen{{\mathrm{Sen}}}

\newcommand{\ainf}{{\mathbf{A}_{\mathrm{inf}}}}

\newcommand{\bdrplus}{{\mathbf{B}^+_{\mathrm{dR}}}}
\newcommand{\bdr}{{\mathbf{B}_{\mathrm{dR}}}}

\label{period ring: relative}

\newcommand{\bnht}{{\mathbf{B}_{\mathrm{nHT}}}}

\label{period ring: Gao-Min-Wang}
 
\newcommand{\nht}{{\mathrm{nHT}}}
\newcommand{\ndR}{{\mathrm{ndR}}}

\newcommand{\bdrplusl}{{\mathbf{B}^+_{\dR,  L}}}

\newcommand{\mic}{\mathrm{MIC}} 
\newcommand{\bm}{\mathbb{M}}

\label{period ring: bb font}

\newcommand{\bbdrplus}{{\mathbb{B}_{\mathrm{dR}}^{+}}}

\newcommand*{\wt}[1]{\widetilde{#1}}

\label{period ring: AB ring}

\label{period ring: integral Hodge}

\def \ok {{\mathcal{O}_K}}

\def \oc {{\mathcal{O}_C}}

\newcommand{\ocflat}{{\mathcal{O}_C^\flat}}

\label{peirod module: D}

\label{imperf residue}

\label{Cats: Fontaine}

\label{Cats: Rep and Higgs}

\label{••••••••1}
\label{ALL: site, sheaves}

\newcommand{\rg}{\mathrm{R}\Gamma}

 \newcommand{\bbl}{{\mathbb{L}}}

\label{site: etale proetale}

\newcommand{\oxet}{{\mathcal{O}_{X_\et}}}
\newcommand{\xet}{{X_\et}}
\newcommand{\xproet}{{X_\proet}}

\label{syntomic coho}

\label{font: roman rm}

\newcommand{\rC}{{\mathrm C}}
\newcommand{\rD}{{\mathrm D}}
\newcommand{\rd}{{\mathrm d}}

\newcommand{\rH}{{\mathrm H}}

\newcommand{\rL}{{\mathrm L}}

\newcommand{\rR}{{\mathrm R}}

\newcommand{\rW}{{\mathrm W}}

\label{font: mathbb}

\newcommand{\bA}{{\mathbb A}}
\newcommand{\bB}{{\mathbb B}}
\newcommand{\bC}{{\mathbb C}}

\newcommand{\bM}{{\mathbb M}}
\newcommand{\bN}{{\mathbb N}}

\newcommand{\bQ}{{\mathbb Q}}

\newcommand{\bW}{{\mathbb W}}

\newcommand{\bZ}{{\mathbb Z}}

  \newcommand{\bbz}{{\mathbb{Z}}}

\newcommand{\zp}{{\mathbb{Z}_p}}
\newcommand{\qp}{{\mathbb{Q}_p}}

\label{font: mathcal}

\newcommand{\calA}{{\mathcal A}}
\newcommand{\calB}{{\mathcal B}}
\newcommand{\calC}{{\mathcal C}}

\newcommand{\calE}{{\mathcal E}}
\newcommand{\calF}{{\mathcal F}}

\newcommand{\calJ}{{\mathcal J}}
\newcommand{\calK}{{\mathcal K}}

\newcommand{\calM}{{\mathcal M}}

\newcommand{\calO}{{\mathcal O}}

\newcommand{\calS}{{\mathcal S}}

\newcommand{\calX}{{\mathcal X}}

 \renewcommand{\o}{{{\mathcal{O}}}}

\newcommand{\cale}{{\mathcal{E}}}

\newcommand{\cm}{{\mathcal{M}} }

\newcommand{\calf}{\mathcal{F}}

\newcommand{\calm}{\mathcal{M}}

\newcommand{\calx}{{\mathcal X}}
\newcommand{\cx}{{\mathcal X}}

\label{font: mathfrak}

\newcommand{\frakS}{{\mathfrak S}}

\newcommand{\frakX}{{\mathfrak X}}

\newcommand{\gs}{{\mathfrak{S}}}

\newcommand{\fkt}{{\mathfrak{t}}}
\newcommand{\fkx}{{\mathfrak{X}}}

 \label{font: hat}

\label{font: bar}

\newcommand{\barK}{{\overline{K}}}

\label{font: bold}

 \label{font: mathbf}
 \newcommand{\bfa}{\mathbf{A}}
\newcommand{\bfb}{\mathbf{B}}

\label{••••••••2}

\label{stack proj}

\label{words}

\label{Gao added to dR paper}

\newcommand{\bbdrplusm}{{\mathbb{B}_{\mathrm{dR}, m}^{+}}}

 \newcommand{\MIC}{{\mathrm{MIC}}} 
\newcommand{\bfB}{\mathbf{B}}

\newcommand{\en}{\mathrm{en}}

\newcommand{\fkxastpris}{\mathfrak{X}_{\pris, \ast}}

\newcommand{\fkxastprisperf}{\mathfrak{X}_{\pris, \ast}^\perf}

\newcommand{\rella}{{\mathrm{rel}\dla}}

\newcommand{\rla}{{\mathrm{rla}}}

\newcommand{\aastninf}{{\mathbf{A}_{\ast\mbox{-}\mathrm{ninf}}}}

\newcommand{\bastnht}{\mathbf{B}_{\ast\mbox{-}\nht}}

\newcommand{\bastndrm}{\mathbf{B}_{\ast\mbox{-}\ndR, m}}

\newcommand{\kmun}{{K(\mu_n)}}

\newcommand{\kpin}{{K(\pi_n)}}

\label{address}
\author[]{Hui Gao}   \address{Department of Mathematics and Shenzhen International Center for Mathematics, Southern University of Science and Technology, Shenzhen 518055, China}   \email{gaoh@sustech.edu.cn}

\author[]{Yu Min}
\address{Department of Mathematics, Imperial College London, London SW7 2RH}
\email{y.min@imperial.ac.uk}

\author[]{Yupeng Wang}
\address{Beijing International Center of Mathematics research, Peking University, Yiheyuan 5, Beijing, 100190, China.}
\email{2306393435@pku.edu.cn}  

\begin{document}
\title[]{Prismatic crystals and $p$-adic Riemann--Hilbert correspondence}
 \subjclass[2010]{Primary  14F30,11S25}
 
\begin{abstract}  \normalsize{ 
We systematically study relative and absolute $\mathbbl{\Delta}_{\mathrm{dR}}^+$-crystals on the (log-) prismatic site of a smooth (resp.~ semi-stable) formal scheme.
Using explicit computation of stratifications, we classify  (local) relative  crystals by certain nilpotent connections, and classify (local)  absolute  crystals by certain enhanced connections.
By using a $p$-adic Riemann--Hilbert functor and   an infinite dimensional Sen theory over the Kummer tower, we globalize the results  on absolute crystals and further classify them by certain small (global)  $\mathbb{B}_{\mathrm{dR}}^+$-local systems.
}
\end{abstract}

\date{\today}
\maketitle
\setcounter{tocdepth}{1}
\tableofcontents

\newcommand{\bw}{{\mathbb{W}}}
\newcommand{\bbw}{{\mathbb{W}}}
 \newcommand{\DR}{{\mathrm{DR}}}
 
 \newcommand{\bndr}{{\mathbf{B}_{\mathrm{ndR}}}}
 \newcommand{\ndr}{{\mathrm{ndR}}}
  \newcommand{\gsdr}{{\gs^+_{\dR}}}
 \newcommand{\gsdrm}{{\gs^+_{\dR, m}}}
 \newcommand{\barx}{{\bar{x}}}
 
\newcommand{\bastndrlm}{\bfb_{\ast\mathrm{-ndR},L,m}}
\newcommand{\kmus}{{K(\mu_s)}}
 \newcommand{\kpis}{{K(\pi_s)}}
\newcommand{\prisdrm}{{\pris_{\dR, m}^+}}
\newcommand{\prisdr}{{\pris_{\dR}^+}}


\newpage
\section{Introduction}

 \subsection{Main theorem on absolute crystals}
 In this paper, we study certain vector bundles on the prismatic site, and relate them with  differential objects on the \'etale site and  representation-theoretic objects on the pro-\'etale site. To state the main Theorem \ref{thm: intro main thm}, we quickly set up some notations, particularly the many ``de Rham" rings/sheaves. 
Let $\ok$ be a mixed characteristic complete discrete valuation ring with perfect residue field $k$. Let $\pi$ be a fixed uniformizer, and let $E=E(u) \in W(k)[u]$ be its Eisenstein polynomial over $W(k)$.
 Let $\mathfrak X$ be a quasi-compact smooth (resp. semi-stable) formal scheme over $\ok$, set  $\ast =\emptyset$ (resp. $\ast= \log$) accordingly; let $X$ be its rigid generic fiber.  Let $\fkxastpris$ be the absolute (log-) prismatic  site; let $\fkxastprisperf$ be the subsite consisting of perfect (log-) prisms.  
Let $1\leq m <\infty$. Consider the following ``de Rham" rings/sheaves.
\begin{itemize}
\item  Define the prismatic de Rham  period sheaf of level $m$  on $\fkxastpris$ and $\fkxastprisperf$  by
\[\pris_{\dR, m}^+: =\o_\pris[1/p]/\mathcal{I}_\pris^m.\]
When $m=1$, this becomes $\prism_\HT$, the rational Hodge--Tate period sheaf.

\item Following \cite{Sch13}, define the pro-\'etale  de Rham period sheaf of level $m$ on $\xproet$  by
\[ \bbdrplusm: = W(\hat{\o}_{X^\flat}^+)[1/p]/(\ker\theta)^m.\]
 When $m=1$, this is the pro-\'etale structure sheaf $\hat{\o}_X$.
 
 \item  Define Fontaine's de Rham period ring of level $m$ by 
$$\bdrplusm:  =W(\o_C^\flat)[1/p]/(\ker\theta)^m. $$
Define the (arithmetic) de Rham period ring of level $m$ by 
 $\gsdrm: =K[[E]]/E^m$.
 When $m=1$, they are  $C=\widehat{\barK}$ and $K$.
 
\item  Define the \'etale  de Rham (base change) sheaf  of level $m$  on $\xet$ by
\[\o_{X, \gs^+_{\dR}, m}:=\o_{X, \et}\otimes_K \gs^+_{\dR, m}. \] 
When $m=1$, this is precisely $\o_{X, \et}$.

\item Using a certain   subcategory of  $X_{C, \et}$, Liu--Zhu \cite{LZ17}   defines  an \'etale  de Rham ``base change" sheaf of level $m$ on $X_C$ by
\[ \o_{X_C, \bdrplus, m}: = ``\o_{X_C, \et} \hatotimes \bdrplusm". \]
Here the right hand side is only meant to be suggestive as there is no canonical map $C\to \bdrplus$; cf. Notation \ref{nota: ringed space} for details. When $m=1$, this is precisely $\o_{X_C, \et}$.
\end{itemize}
We need to define many  module categories to state our main theorem. In general, for $\mathbb F$   a sheaf of rings on a site $\calS$, define
   \begin{equation*}\label{Dfn-A crystal}  \Vect(\calS,\mathbb F):= \projlim_{\calA\in \calS} \vect(\mathbb F(\calA))
   \end{equation*}
   where $\vect(\mathbb F(\calA))$ is the category of finite projective $\mathbb F(\calA)$-modules.
We also introduce a constant used throughout the paper. Let $E'(\pi)=\frac{d}{du}(E)|_\pi$. 
Recall $*\in\{\emptyset,\log\}$. 
Let \begin{equation*}
a=
\begin{cases}
  -E'(\pi), &  \text{if } \ast=\emptyset \\
 -\pi E'(\pi), &  \text{if } \ast=\log.
\end{cases}
\end{equation*}  

\begin{defn}[Pro-\'etale objects] 
 An object $\bbw \in \vect(\xproet, \bbdrplusm)$ is called a \emph{$\bbdrplus$-local system} of level $m$, cf. \cite[\S 7]{Sch13}; when $m=1$, it is often called a \emph{generalized representation}.  For $x:\spa   L \to X$ a classical point where $L/K$ is a finite extension, the specialization $\bbw_x$ gives rise to a $\bdrplusm$-semi-linear representation of $G_L=\gal(\barK/L)$. Say $\bbw_x$ is $a$-small if the Sen weights of the $C$-representation $\bbw_x/t\bw_x$ are in the range
\[ \bbz + a^{-1}\mathfrak{m}_{\o_C} \] 
(Caution: the constant $a$ depends only on $K$, and not on $L$). Say $\bw$ is $a$-small if $\bbw_x$ is $a$-small for all classical points $x$.
 Let 
\[ \vect^{a}(\xproet, \bbdrplusm)\]
be the subcategory  of $a$-small objects, cf. \S \ref{sec: a small bdrplus loc sys}.  
\end{defn}

\begin{defn}[$\gk$-equivariant \'etale objects] Let $\MIC_{\gk}(X_{C, \et}, \o_{X_C, \bdrplus, m})$ denote the category where an object consists of  a $\calm \in \vect(X_{C, \et},   \o_{X_C, \bdrplus, m})$ equipped with a (semi-linear) $\gk$-action, together with a  $\gk$-equivariant  integrable connection
\[ \nabla: \calm \to \calm \otimes_{\o_{X_C}} \Omega^1_{X_C}(-1).\]
\end{defn}

\begin{defn}[Enhanced \'etale objects]
\begin{enumerate}
\item An (arithmetic) $E$-connection on $(X_\et,\o_{X, \gs^+_{\dR}, m})$ consists of 
 a $M \in \vect(X_\et,\o_{X, \gs^+_{\dR}, m})$ equipped with   an $\o_X$-linear   map $\phi: M\to M$  satisfying Leibniz rule with respect to $E\frac{d}{dE}:\frakS_{\dR}^+\to\frakS_{\dR}^+$. Say an $E$-connection $(M,\phi)$ is \emph{$a$-small} if  
     \[\lim_{n\to+\infty}a^n\prod_{i=0}^{n-1}(\phi-i) = 0.\]

\item Let $\mathrm{MIC}^a_{\en}(X_\et,\o_{X, \gs^+_{\dR}, m})$ denote the category of \emph{$a$-small enhanced connections}, where an object is an $M \in \vect(X_\et,\o_{X, \gs^+_{\dR}, m})$ equipped with a 
topologically nilpotent integrable connection  
      \[\nabla_M:M\to M\otimes_{\o_{X, \gsdr}}\Omega^1_{X, \gsdr}\{-1\},\]
         and an $a$-small  $E$-connection (as in Item (1)) $\phi:M\to M$ such that
        the following diagram is commutative:      
        \[
    \begin{tikzcd}
M \arrow[d, "\phi"] \arrow[r, "\nabla"] & M\otimes_{\o_{X, \gsdr}}\Omega^1_{X, \gsdr}\{-1\} \arrow[d, " \phi_M\otimes 1 +1\otimes \phi_{(E)^{-1}}  "] \\
M \arrow[r, "\nabla"]                   & M\otimes_{\o_{X, \gsdr}}\Omega^1_{X, \gsdr}\{-1\}.                              
\end{tikzcd}
\] 
 Here: $\Omega^1_{X, \gsdr}$ is the ``base change" of $\Omega_X^1$ to $\gsdr$ with $\{-1\}$ signifying the Breuil--Kisin twist, over which there is an $\o_X$-linear endomorphism $\phi_{(E)^{-1}}$, cf. Def \ref{defn: abs enhanced conn} for full details, cf. also Notation \ref{nota: enhanced conn local} for the explicit local case. When $m=1$, the right vertical arrow becomes $\phi_M-1$.
 (For now, the unfamiliar  readers could simply assume the definition says that there is some ``(twisted)-commutativity relation" between $\nabla$ and $\phi$).
\end{enumerate}
\end{defn}
 

The following is our main theorem, which studies relations between modules over various ``de Rham  sheaves" defined above. 
(Throughout the paper, we use $\into$ resp. $\simeq$ in commutative diagrams to signify a functor is fully faithful  resp. an equivalence.)

\begin{theorem} \label{thm: intro main thm}
We have a commutative diagram of fully faithful tensor functors, where all horizontal arrows are equivalences:
\begin{equation} \label{diag: intro main thm}
\begin{tikzcd}
{ \mathrm{MIC}_{\en}^{a}(X_\et,\o_{X, \gs^+_{\dR}, m})} \arrow[d, hook] &  & {\vect(\fkxastpris, \pris_{\dR,m}^+) } \arrow[ll, "\simeq"'] \arrow[rr, "\simeq"] \arrow[d, hook] &  & {\vect^{a}(\xproet, \bbdrplusm)} \arrow[d, hook] \\
{\MIC_{\gk}(X_{C, \et}, \o_{X_C, \bdrplus, m})}                                     &  & {\vect(\fkxastprisperf, \pris_{\dR,m}^+) } \arrow[rr, "\simeq"] \arrow[ll, "\simeq"']        &  & { \vect(\xproet, \bbdrplusm)}.                         
\end{tikzcd}
\end{equation}
Let $\bm \in \vect(\fkxastpris, \pris_{\dR,m}^+)$, and denote the corresponding objects  by
\[
\begin{tikzcd}
M \arrow[d, mapsto]  &  & \bm \arrow[ll, mapsto] \arrow[rr, mapsto] \arrow[d, mapsto] &  & \mathbb W \arrow[d, mapsto] \\
{\calm} &  & \bm^\perf \arrow[ll, mapsto] \arrow[rr, mapsto]     &  & \mathbb W          
\end{tikzcd}
\]
then we have functorial quasi-isomorphisms of cohomologies
\[ 
\begin{tikzcd}
{\rg(X_\et, \rg(\phi, \DR(M)))} \arrow[d, "\simeq"] &  & {\rg(\fkxastpris, \bm)} \arrow[ll, "\simeq"'] \arrow[rr, "\simeq"] \arrow[d, "\simeq"] &  & {\rg(\xproet, \bW)} \arrow[d, "="] \\
{\rg(X_{C, \et},\rg(\gk, \DR(\calm)))}           &  & {\rg(\fkxastprisperf, \bm^\perf)} \arrow[ll, "\simeq"'] \arrow[rr, "\simeq"]           &  & {\rg(\xproet, \bW)}.               
\end{tikzcd} \]
Here, on the left column,  $\rg(\phi, \DR(M))$ is the ``$\phi$-enhanced" de Rham complex   associated to $(M, \nabla)$, and  $\rg(\gk, \DR(\calm))$ is the $\gk$-cohomology of the de Rham complex associated to $\cm$.
\end{theorem}

\begin{remark}
One can define similar categories when $m=\infty$ (by taking limits), then the above theorem still holds. In the main  text, unless specifically mentioned, we always assume $m <\infty$ for brevity: for example, then all relevant representations are \emph{Banach} spaces (instead of Fr\'echet spaces which will require some small extra argument).
\end{remark}

\begin{rem}[Known cases] \label{rem: known cases}
    We   summarize and discuss previously known cases of Theorem \ref{thm: intro main thm}. First, consider the \emph{bottom} row of diagram \eqref{diag: intro main thm}, which is regarded as the (relatively) easier part of the main theorem. These equivalences were known by \cite{MW22} when $\fkx$ is smooth, $\ast=\emptyset$ and $m=1$.  
The general case follows similar argument, cf.  \S \ref{sec: crystal perf pris} (where the log-prismatic case builds on discussions of perfect log prisms in \S \ref{sec log perf prism}). 
We now thus focus on the top row equivalences
 \begin{equation} \label{diag: intro top}
 \mathrm{MIC}_{\en}^{a}(X_\et,\o_{X, \gs^+_{\dR}, m})  \xleftarrow{F_\et} \vect(\fkxastpris, \pris_{\dR,m}^+)  \xrightarrow{F_\proet} \vect^{a}(\xproet, \bbdrplusm).
\end{equation}  
\begin{enumerate}
   \item 
Consider the special case $\fkx=\spf \ok$,  $m=1$; that is, the case of Hodge--Tate   crystals on $\ok$. The top row becomes
\[ \End^a(K) \simeq \vect(\okprisast, \pris_\HT) \simeq \rep_\gk^a(C). \]
When $\ast=\emptyset$ (i.e., the prismatic case), the equivalence $ \End^a(K) \simeq \vect(\okpris, \pris_\HT)$ (indeed the derived version) is proved by  Bhatt--Lurie  \cite{BL-a, BL-b} using the Hodge--Tate stack. Building on this stacky approach, and using a certain period ring ``$B_\en$", \cite{AHLB1} further  proves  $\vect(\okpris, \pris_\HT) \simeq \rep_\gk^a(C)$. Independently, these results are also proved in \cite{GMWHT}  using a site-theoretic approach and a locally analytic Sen theory over the Kummer tower, which works uniformly for $\ast \in \{\emptyset, \log\}$. We refer to the introduction of  \cite{GMWHT} for some extensive comparisons about these approaches.

\item Again for $\fkx=\spf \ok$, but with $m$ general; that is, the case of $\prisdrm$-crystals. These results are proved in \cite{GMWdR}, again following the site-theoretic approach of \cite{GMWHT}. Some partial (full faithful) results are also independently  obtained in \cite{Liu23}; recently, Liu further develops a stacky approach \cite{Liustack} (for $\ast=\emptyset$), which reproves these results and generalizes to the derived version.

\item Consider  the general case where $\fkx$ is a (quasi-compact) smooth formal scheme over $\ok$, but still with $\ast=\emptyset$ and $m=1$. 
Using a stacky approach, the  equivalence   $F_\et$ in \eqref{diag: intro top} is proved in \cite{BL-b}, and   equivalence  $F_\proet$ is proved in \cite{AHLB3} (building on geometric results in \cite{AHLB2}). Independently, \cite{MW22} proves full faithfulness of these functors, using a site-theoretic approach. 
\end{enumerate}
\end{rem}

We also briefly discuss results on crystals on the \emph{relative} (log-) prismatic site.

\begin{theorem}[{cf. \S   \ref{sec: rel pris smooth}, \S \ref{SubSec-Log-Rel}}] \label{thm:intro rel crystal} 
Let $\fkx=\spf R$ where $R$ is a small smooth (resp. semi-stable) $\ok$-algebra, and consider the \emph{relative} prismatic (resp. log-prismatic) site $(R/(\gs, (E), \ast))_{\Prism, \ast}$ where $\ast=\emptyset$ (resp. $\log$) and where $(\gs, (E), \ast)$ is the (arithmetic) Breuil--Kisin (log-) prism as in Notation \ref{nota: BK prism}. 
Let $(\gs(R), (E), \ast)$ be the  relative Breuil--Kisin (log-) prism. Then there is an equivalence of categories 
 \[\Vect((R/(\gs, (E), \ast))_{\Prism, \ast},\prism^+_{\dR, m})\xrightarrow{\simeq}{\rm MIC}_{E}^{\nil}(\gs(R)_{\dR,m}^+).\]
Here the LHS is the category of $\prism^+_{\dR, m}$-crystals on the relative site, and the RHS is the category of nilpotent $E$-connections on $\gs(R)_{\dR,m}^+ =\prism^+_{\dR, m}((\gs(R), (E), \ast))$. For corresponding objects $\bm \mapsto (M, \nabla_M)$, there is a cohomology comparison
\[ \rg((R/(\gs, (E), \ast))_{\Prism, \ast}, \bm) \simeq \DR(M, \nabla_M).\]
\end{theorem}

\begin{remark} \label{rem known relative case}
Theorem \ref{thm:intro rel crystal} is known when $\fkx$ is small smooth (thus $\ast=\emptyset$) and $m=1$, by \cite{Tia23}.
\end{remark}

 \begin{remark}[What is new] \label{rem:what is new}
 In conclusion to  Remark  \ref{rem: known cases} resp.   Remark \ref{rem known relative case}, Theorem  \ref{thm: intro main thm} resp. Theorem   \ref{thm:intro rel crystal} is new  only when the relative dimension of $\fkx/\spf \ok$ is nonzero. More precisely:
\begin{itemize}
\item  when $\fkx$ only has semi-stable reduction, it is new for any $m \geq 1$;
\item when $\fkx$ is smooth, it is new for $m \geq 2$.
\end{itemize}
As we shall see from the main text, the semi-stable case follows from similar ideas as  the smooth case, once the following results (of independent interests) are established:
\begin{itemize}
\item We compute explicitly the (semi-stable) co-simplicial Breuil--Kisin prisms and their (de Rham) variants: it turns out their behaviour parallels the smooth case, cf. \S \ref{SubSec-Log-Rel} for the relative case and \S \ref{sec: loc abs pris} for the absolute case;
\item We show (with appropriate set-up) that the category of perfect log-prisms  is equivalent to the category of perfect prisms; thus the discussions on the perfect prismatic site are the same for smooth and semi-stable case, cf.  \S \ref{sec log perf prism}.
\end{itemize} 
 \end{remark}

 With Remark \ref{rem:what is new} at hand, we focus on discussing new ideas when $m \geq 2$ in the following.
 Indeed, we shall only discuss the \emph{absolute} case as this is the most interesting case and is related with $p$-adic Riemann--Hilbert correspondence and Sen theory.
We now sketch the core steps   for Theorem \ref{thm: intro main thm}.
 
\begin{construction}[The main roadmap for Theorem \ref{thm: intro main thm}]
\label{cons: roadmap}
Recall again the functors
 \begin{equation*} \label{diag: intro top another}
 \mathrm{MIC}_{\en}^{a}(X_\et,\o_{X, \gs^+_{\dR}, m})  \xleftarrow{F_\et} \vect(\fkxastpris, \pris_{\dR,m}^+)  \xrightarrow{F_\proet} \vect^{a}(\xproet, \bbdrplusm).
\end{equation*}  
\begin{enumerate}
\item The \emph{global} functor $F_\proet$ is precisely the composition:
\[ F_\proet: \vect(\fkxastpris, \pris_{\dR,m}^+)  \to \vect(\fkxastprisperf, \pris_{\dR,m}^+) \simeq \vect(\xproet, \bbdrplusm) \]
where the first functor is induced by restriction to the perfect subsite, and the second equivalence is the relatively easy fact mentioned in beginning of Rem \ref{rem: known cases}. However it is difficult  to directly prove full faithfulness of $F_\proet$ (equivalently, of the restriction functor).
 
 \item Using Breuil--Kisin cosimplicial prisms, we  can explicitly construct the \emph{local} version of $F_\et$ (this is done for $m=1, \ast=\emptyset$ in \cite{MW22} where all operators are \emph{linear}; the computation becomes more complicated here essentially because we now work on the ``deformation" $\gs_\dR$ of $K$); we then can prove the \emph{local} $F_\et$ is an equivalence.  However, it is difficult to \emph{globalize} the equivalence $F_\et$.
 
\item In some sense, the difficulty in above two items comes from difficulty of \emph{prismatic cohomology}.
Two further ingredients are needed to save the situation.
Firstly, using a \emph{geometric}  Riemann--Hilbert functor  (generalizing work of Liu--Zhu \cite{LZ17}), we can prove the \emph{global} equivalence:
\[ \vect(\xproet, \bbdrplusm) \simeq \MIC_{\gk}(X_{C, \et}, \o_{X_C, \bdrplus, m})\]
which then formally induces an equivalence of subcategories consisting of   $a$-small objects:
\[ \vect^{a}(\xproet, \bbdrplusm) \simeq \MIC^{a}_{\gk}(X_{C, \et}, \o_{X_C, \bdrplus, m}).\]
The next ingredient, which is a key innovation in this paper, is that we can prove a \emph{global} equivalence of categories of \emph{\'etale/coherent} objects:
\[\mathrm{MIC}_{\en}^{a}(X_\et,\o_{X, \gs^+_{\dR}, m})  \simeq \MIC^a_{\gk}(X_{C, \et}, \o_{X_C, \bdrplus, m}). \]
 One can easily write out an \emph{explicit} functor from left to right (roughly, base change); the difficult direction is from right to left (i.e., \emph{descent} or \emph{decomplete}), whose construction hinges on an \emph{arithmetic} infinite dimensional \emph{(locally) analytic}  Sen theory over the Kummer tower. This Sen theory will be discussed in more detail in Rem \ref{rem: new Sen}. 
 
 \item The global functor $F_\proet$ in Item (1) composed with global equivalences in Item (3) leads to a \emph{global} functor:
\[ \vect(\fkxastpris, \pris_{\dR,m}^+) \to \mathrm{MIC}_{\en}^{a}(X_\et,\o_{X, \gs^+_{\dR}, m}).\]
To prove equivalence, it suffices to work locally as both sides are Zariski stacks: but the local case is already proved in Item (2), concluding the roadmap.
\end{enumerate}
\end{construction}

\begin{construction}[Searching for a Sen operator] \label{rem: new Sen}
As a short summary, this paper is a continuation of the series \cite{GMWHT, GMWdR, MW22}, but additionally incorporating an idea inspired by the series \cite{AHLB1, AHLB3} (with a technical tool supplied by \cite{RC22}). 
In the following, we first discuss why either approach \emph{alone} would not suffice for our purpose; then we discuss how to combine the power  of the two approaches.
\begin{enumerate}
\item For  $\fkx=\spf \ok$, the work \cite{GMWdR} generalizes \cite{GMWHT} (where $m=1$) to $m \geq 2$ case.
A key argument there is a d\'evissage argument (cf. \cite[\S 12]{GMWdR})  inspired by \cite{Fon04}; unfortunately, this method breaks for a general  $\fkx$, essentially because  the naive  d\'evissage argument can only work with one matrix (corresponding to  one Sen operator):   for general $\fkx$, there are $(\dim \fkx/\ok +1)$ operators (one Sen operator plus $\dim \fkx/\ok$-many geometric connection operators).
\item One could also try to generalize the methods of \cite{AHLB3} to $m\geq 2$ case, but will quickly meet a serious  trouble.
\begin{enumerate}
\item  When $m=1$, Sen theory can be \emph{extended $C$-linearly} without  losing  too much \emph{linear-algebra theoretic information}; for example in the classical paper \cite{Sen81}, one could define a \emph{Sen operator} over the cyclotomic field $\kpinfty$ (cf. Notation \ref{notafields}), which could then be \emph{linearly extended} to a $C$-linear endomorphism;
\item When $m \geq 2$, there is no such ``linear extension". For example, in \cite{Fon04}, Fontaine generalizes \cite{Sen81} to representations over $\bdrplusm$ and defines a \emph{$t$-connection} operator on a module over the ring $\kpinfty[[t]]/t^m$; this connection can \emph{not}  extend over $\bdrplusm$ (unless $m=1$) as that operator does not even act on $\bdrplusm$.
\item Indeed, a key observation for this paper comes from the ring $B_\en$ used in \cite{AHLB1}; it is equipped with a \emph{$C$-linear} ``Sen operator" $\phi_{B_\en}$ and a $\gk$-action. 
This ring is used as a  \emph{period ring} in  \cite{AHLB1} to prove
\[\End^a(K)  \simeq \rep_\gk^a(C) \]
with functors being
\[ M \mapsto W=(M\otimes_K B_\en)^{\phi_M\otimes 1+1\otimes \phi_{B_\en}},\]
\[W \mapsto M=(W\otimes_C B_\en)^{\gk=1}.\]
We point out one can easily construct an analogue of $B_\en$ when $m \geq 2$: but it is  a $\bdrplusm$-algebra, and hence   no longer admits ``Sen operator"!
 \end{enumerate}

\item A key observation in this paper is a \emph{(locally) analytic} interpretation of the ring $B_\en$ (denoted as $\bnht$ in this paper):
\begin{itemize}
\item  we prove $\bnht$ is a \emph{relatively locally analytic representation} (over $C$) in the sense of Rodr\'{\i}guez Camargo \cite{RC22}.  This makes it legitimate to consider \emph{infinite dimensional} Sen theory on this ring;
\item We then observe the ``correct" Sen theory to study  $\bnht$ is the Sen theory over the \emph{Kummer tower} $K_\infty$ (in contrast to  cyclotomic tower $\kpinfty$, cf. Notation \ref{notafields}) first developed in \cite{GMWHT, GMWdR} and now generalized to infinite dimensional case;
\item Now both items above could be generalized to study the $m \geq 2$ case, i.e., the ring   $\bfb_{\ndr, m}$  which is  a \emph{relatively locally analytic representation} (over $\bdrplusm$); its corresponding Sen module (over Kummer tower) now admits a Sen operator, making it possible to function again as a \emph{period ring}! That is: with this Sen theory, we can again use techniques of differential equations.

 \item Above bullets discuss a  (relatively) \emph{locally analytic} Sen theory; it turns out for our purpose, we shall need a further refined \emph{analytic} (not just locally analytic) Sen theory. This is a key new discovery in this paper, which unlike the \emph{locally analytic} version, (in general) does \emph{not} work for the classical cyclotomic tower. That is: we \emph{have to} use the Kummer tower here. Magically, the $\tau$-operator (unlike the $\gamma$-operator in classical cyclotomic tower case) behave more like a ``geometric" operator. \footnote{We expect possible connection between this analytic-geometric operator and  the theory of  analytic prismatization (work in progress by  by Ansch\"{u}tz, Le Bras, Rodr\'{\i}guez Camargo and Scholze).}
\end{itemize} 
\end{enumerate}
\end{construction}

 \subsection{Structure of the paper}
 We divide the paper into several parts.
 \begin{enumerate}
 \item Part One:   \emph{relative} local prismatic crystals. 
 \begin{itemize}
  \item  From \S \ref{sec: rel pris smooth} to \S \ref{SubSec-Log-Rel}, we study \emph{relative} local prismatic crystals. The smooth case generalizes results of \cite{Tia23} from $m=1$ case to general case.
\S \ref{SubSec-Log-Rel} separately treats semi-stable case to contain the length of \S \ref{sec: rel pris smooth}.
   Some computations here are useful for subsequent computations in the  local \emph{absolute} case. 
\end{itemize}

\item Part Two: \emph{absolute} local prismatic crystals.
\begin{itemize}
\item  \S \ref{sec: arith infinite dim} quickly generalizes   previous work \cite{GMWdR} (where $\fkx=\spf \ok$) to the \emph{infinite} rank case: that is, we need to deal with infinite rank prismatic crystals as they will naturally appear as certain period rings.

\item \S \ref{sec: loc abs pris} forms a first major technical core of this paper, which computes explicitly stratifications related with absolute prismatic crystals on a local chart. Although we are guided by our previous works, this case still poses several challenging technical difficulties: most notably, unlike \cite{MW22} where all operators are \emph{linear}, now all operators are connections. 

\item \S \ref{sec: coho abs local} computes cohomology of stratifications from last section; it is separated from previous section to contain the length.
\end{itemize}

\item Part Three: geometric Riemann--Hilbert correspondence.
\begin{itemize}
\item  \S \ref{sec: bdr loc sys} collects basic facts on $\bbdrplus$-local systems. We then in \S \ref{sec: global RH} establish Riemann--Hilbert correspondence for them. The main results mostly follow d\'evissage argument using results on Simpson correspondence in \cite{MW22}.
We also collect  some local explicit formula in \S \ref{subsec local RH} for subsequent use.
\end{itemize}

\item Part Four: relating primatic crystals with RH correspondence.
\begin{itemize}
\item In \S \ref{sec: crystal perf pris}, we prove  $\pris_\dR^+$-crystals  on the \emph{perfect} prismatic site are   classified by    $\bB_{\dR}^+$-local systems on the pro-\'etale site. 

\item Thus, in particular, from a   $\pris_\dR^+$-crystals on the (whole) prismatic site, we can restrict to the perfect subsite, and thus pass to a $\bB_{\dR}^+$-local system; the explicit local formulae for these relations will be presented in \S \ref{sec: m to cm}.
This section serves as an interlude, concludes the \emph{geometric} discussions, and points out remaining questions. In particular, the explicit formula invites, and will be used in the discussion of \emph{arithmetic} Sen theory in the following part.
\end{itemize}

\item Part Five: (analytic) arithmetic  Sen theory.
\begin{itemize}
\item In \S \ref{sec: infinite Sen}, we develop Sen theory for \emph{infinite dimensional relatively locally analytic} $\bdrplus$-representations, building on the work of Rodriguez-Camargo \cite{RC22}, and then pass to a Sen theory over the Kummer tower, using ideas of \cite{GMWHT, GMWdR}.

\item  In \S \ref{sec: decomp RH}, we decomplete the RH correspondence in \S \ref{sec: global RH}, using the infinite dimensional Sen theory from  \S \ref{sec: infinite Sen}.  This makes it possible to define $a$-smallness.

\item \S \ref{sec: analytic sen kummer} is a key innovation in this paper. We develop an \emph{analytic} Sen theory over the Kummer tower, for certain $a$-small Galois representations. This is inspired by some explicit computations in \cite{AHLB1}, but now will be developed using locally analytic tool in \cite{RC22} and the Kummer tower construction in  \cite{GMWHT, GMWdR}. This gives us a \emph{correct} analytic tool to study prismatic crystals.

\item In \S \ref{sec: M to W}, we use analytic Sen theory from \S \ref{sec: analytic sen kummer} to build the \emph{global} equivalence from enhanced connections to $a$-small $\gk$-equivariant connections, mentioned in Cons \ref{cons: roadmap}(3). 
\end{itemize}

\item Part Six: global theorems.
\begin{itemize}
\item Previously, we can prove \emph{local} equivalence between prismatic crystals with enhanced connections; but it is hard to globalize. We also have \emph{global} functors from prismatic crystals to  $\gk$-equivariant connections, but it is  difficult to even prove  full-faithfulness. With the equivalence in \S \ref{sec: M to W} (built using analytic Sen theory), we can finally \emph{combine} the benefit of both constructions: that is, the local equivalences are now globalized using   passage to  $\gk$-equivariant connections.  
\end{itemize}
 
 \end{enumerate}

\subsection*{Acknowledgment} 
We thank 
Johannes Ansch\"utz, 
Ben Heuer, 
Arthur-C\'esar Le Bras, 
Ruochuan Liu, 
Lue Pan, 
and Juan Esteban Rodr\'{\i}guez Camargo 
 for valuable discussions and feedback. 
 Hui Gao is partially
supported by the National Natural Science Foundation of China under agreements   NSFC-12071201, NSFC-12471011.  
Yu Min has received funding from the European Research Council (ERC) under the European Union's Horizon 2020 research and innovation programme (grant agreement No. 884596).
Yupeng Wang is partially supported by CAS Project for Young Scientists in Basic Research, Grant No. YSBR-032.

\newpage
\section{Notations, conventions and tool-kits}

 In this section, we record some notations/conventions used throughout the paper. In particular, in Convention \ref{conv: consistency notation}, we  list our choice of notations for sheaves on different geometric settings; we hope this could help the readers to distinguish them throughout the paper. Other notations are mostly standard; the readers could skip for now and come back when needed.

\subsection{Notation styles}

\begin{convention}[Notations of sheaves] \label{conv: consistency notation}
 In this paper, we study many sheaves  on prismatic, \'etale and pro-\'etale sites. We try to be \emph{consistent} with our choice of notations.
\begin{enumerate}
\item Prismatic notations $\bm$.
\begin{itemize}
\item We use blackboard font   $\mathbb M$ (possibly with decorations, e.g., $\bm^\perf$; similar for following) to denote  sheaves on (absolute/relative) prismatic sites. 

\end{itemize}

\item Pro-\'etale notations $\bbw, W, V, D$.
\begin{itemize}
\item We use $\mathbb W$ to denote sheaves on pro-\'etale sites.

\item We use $W$ to denote representation-theoretic (local)  objects related with   pro-\'etale sheaves. These representations can be ``decompleted", which will be denoted using notations such as $V$ or $D$.
\end{itemize}

 \item \'{E}tale/coherent notations $M, \cm$.
 \begin{itemize}
\item We use $M$ to denote   coherent objects (e.g., connections with extra structures) related with  \emph{prismatic sheaves}. 
 \item We use $\calM$ to denote   coherent  objects (sheaves or modules) related with  \emph{pro-\'etale sheaves}.
 \end{itemize}

\item  Sen operators $\nabla_i, \phi$ (on coherent objects). 
\begin{itemize}
\item We use $\nabla_i$ to denote \emph{geometric} Sen/connection operators; 
\item We use $\phi$ to denote \emph{arithmetic} Sen/connection operators.
\end{itemize}
\end{enumerate}
\end{convention}

\subsection{Notations: fields, groups, locally analytic vectors} \label{subsec lav nota}

\begin{notation} \label{notafields}
 We   introduce some field notations. 
\begin{itemize}
\item Let $\mu_1$ be a primitive $p$-root of unity, and inductively fix $\mu_n$ so that $\mu_n^p=\mu_{n-1}$. Let $K_{p^\infty}=  \cup _{n=1}^\infty
K(\mu_{n})$.

\item Let $\pi_0=\pi$ be a fixed uniformizer of $\ok$, and inductively fix some $\pi_n$ so that $\pi_n^p=\pi_{n-1}$. 
Let $K_{\infty}   = \cup _{n = 1} ^{\infty} K(\pi_n)$.
When $p\geq 3$, \cite[Lem. 5.1.2]{Liu08} implies $\kpinfty \cap \kinfty=K$; when $p=2$, by \cite[Lem. 2.1]{Wangxiyuan}, we can and do choose some $\pi_n$ so that $\kpinfty \cap \kinfty=K$.
\end{itemize}  \
Let $L=\kpinfty \kinfty$. 
Let $$G_{\kinfty}:= \gal (\overline K / K_{\infty}), \quad G_{\kpinfty}:= \gal (\overline K / K_{p^\infty}), \quad G_L: =\gal(\overline K/L).$$
Further define $\Gamma_K, \hat{G}$ as in the following diagram:
\[
\begin{tikzcd}
                                       & L                                                                                             &                             \\
\kpinfty \arrow[ru, "\langle \tau\rangle", no head] &                                                                                               & \kinfty \arrow[lu, no head]. \\
                                       & K \arrow[lu, "\Gamma_K", no head] \arrow[ru, no head] \arrow[uu, "\hat{G}"', no head, dashed] &
\end{tikzcd}
\]
Here we let $\tau \in \gal(L/K_{p^\infty})$ be \emph{the} topological generator such that $\tau(\pi_i)=\pi_i\mu_i$ for each $i$.
\end{notation}

We shall freely use the notion of locally analytic vectors in this paper; their relevance in $p$-adic Hodge theory is first discussed in \cite{BC16}. 

\begin{Notation}  \label{nota tau la}
\begin{enumerate}
\item Given a $\hat{G}$-representation $W$, we use
$$W^{\tau=1}, \quad W^{\gamma=1}$$
to mean $$ W^{\gal(L/K_{p^\infty})=1}, \quad
W^{\gal(L/K_{\infty})=1}.$$
And we use
$$
W^{\tau\dla}, \quad W^{\tau_n\dan}, \quad  W^{\gamma\dla} $$
to mean
$$
W^{\gal(L/K_{p^\infty})\dla}, \quad
W^{\gal(L/(K_{p^\infty}(\pi_n)))\dan}, \quad
W^{\gal(L/K_{\infty})\dla}.  $$


\item  For an element $g$ (in a group), we define following (formal) expression:
\[ \log g: =(-1)\cdot \sum_{k \geq 1} (1-g)^k/k.\]
  Given a $\hat{G}$-locally analytic representation $W$, the following two Lie-algebra operators (acting on $W$) are well-defined:
\begin{itemize}
\item  for $g\in \gal(L/\kinfty)$ enough close to 1, one can define $\nabla_\gamma := \frac{\log g}{\log(\chi_p(g))}$;
\item for $n \gg 0$ hence $\tau^{p^n}$ enough close to 1, one can define $\nabla_\tau :=\frac{\log(\tau^{p^n})}{p^n}$.
\end{itemize}
These two Lie-algebra operators form a $\qp$-basis of $\Lie(\hat{G}$).
\end{enumerate}
\end{Notation}


\subsection{Notations: prisms and prismatic sheaves}
\begin{notation}[Breuil--Kisin log prism] \label{nota: BK prism}
  Let $E(u)\in \frakS = \rW(k)[[u]]$ be the Eisenstein polynomial of $\pi$. 
 \begin{enumerate}
     \item  This gives rise to the Breuil--Kisin prism (with respect to choice of $\pi$)
 \[ (\gs, (E)) \in \okpris. \]
 
 \item     The map $\bN\xrightarrow{1\mapsto \pi}\calO_K$ induces a log structure, which we use to define the log-prismatic site $(\ok)_{\pris, \log}$ in the sense of \cite{Kos21}.  
 The map $\bN\xrightarrow{1\mapsto u}\frakS$ induces a log structure  $M_{\frakS}$; by requiring $\delta(u) = \delta_{\log}(1) = 0$, we obtain a log prism
 \[ (\frakS,(E),M_{\frakS}) \in (\ok)_{\pris, \log}.\]
 
 \item Recall in Notation \ref{notafields}, we defined a compatible sequence of $p^n$-th root of unity $\mu_n$; the sequence $(1, \mu_1, \cdots, \mu_n, \cdots)$ defines an element $\epsilon \in \ocflat$. Let $[\epsilon] \in \ainf$ be its Teichm\"uller lift. Define $\xi := \frac{[\epsilon]-1}{[\epsilon]^{\frac{1}{p}}-1}$, which is a generator of the kernel of $\theta: \ainf \to \oc$. 
Then we have the Fontaine prism $(\ainf, (\xi))$.
 There is a natural morphism of   prisms
  $$\iota: (\frakS,(E(u)) )\xrightarrow{u\mapsto [\pi^{\flat}]}(\Ainf,(\xi) ).$$ 
 
 \item  
  The map $\bN\xrightarrow{1\mapsto[\pi^{\flat}]}\Ainf$ induces a log structure, leading to the  Fontaine log prism  $(\Ainf,(\xi),M_{\Ainf})$.
  There is a natural morphism of log prisms
  $$\iota: (\frakS,(E(u)),M_{\frakS})\xrightarrow{u\mapsto [\pi^{\flat}]}(\Ainf,(\xi),M_{\Ainf}).$$ 
 \end{enumerate} 
 \end{notation}

\begin{notation}[Prismatic sheaf $\prism^+_{\dR, m}$] \label{nota: de Rham sheaf}
 Let $1 \leq m \leq \infty$. We define the  prismatic sheaf $\prism^+_{\dR, m}$ of level $m$, (uniformly) for all the relative/absolute (log)-prismatic sites as follows.
\begin{enumerate}
\item 
For a  prism $(P, Q)$ or a log-prism $(P, Q, M_P)$, and for $m <\infty$ define
\[ P^+_{\dR, m}: = P[1/p]/Q^m. \]
The $p$-adic topology induces a Banach $\qp$-algebra structure on $ P^+_{\dR, m}$.
 When $m=\infty$,   write
\[ P^+_{\dR}: =P^+_{\dR, \infty} :=\projlim_{1\leq m <\infty} P^+_{\dR, m} \]
and equip it with Fr\'echet topology. 
For example, for Breuil--Kisin (log-) prism in Notation \ref{nota: BK prism}, we can write 
\[\frakS_{\dR}^+ := \frakS[{1}/{p}]^{\wedge}_E =K[[E]], \quad \frakS_{\dR,m}^+ = \frakS_{\dR}^+/E^m\frakS_{\dR}^+.\]
 
 \item  Let $\ast \in \{\emptyset, \log\}$.   Consider the relative (log-) prismatic site $(\mathfrak{X}/(A, I, \ast))_{\prism, \ast}$ (where $(A, I, \ast)$ is a $\ast$-prism, and $\mathfrak{X}\to \spf(A/I)$ a formal scheme), or  the absolute (log-) prismatic site $\mathfrak{Y}_{\prism, \ast}$ (where $\mathfrak{Y}$ is a formal scheme).  
 Define  the sheaf $\prism^+_{\dR, m}$ on each of these sites, by specifying for $(B, J, \ast)$   an object in one of the (four) sites,
\[ \prism^+_{\dR, m}((B, J, \ast)) :=B^+_{\dR, m}. \] 

 \item It is   customary to write \[ \prism^+_{\dR}: =\prism^+_{\dR, \infty} \]
and
\[ \prism_{\HT}: =\prism^+_{\dR, 1}. \]
\end{enumerate}
\end{notation}

   \begin{notation}[Breuil--Kisin twist] \label{nota: BK twist rel pris}
Let $(A, I)$ be a prism.
    Let $I^{-n}:=\Hom_A(I^n, A)$, which admits a canonical morphism $A \to I^{-n}$. 
    As $(A, I)$ is a prism, $A$ is $I$-torsion free and $I^{-n}$ is a finite projective $A$-module of rank 1.
    For any $A^+_{\dR, m}$-module $P$, denote $P\{-n\}: =P\otimes_A I^{-n}$.  This notation comes from the fact that $P\{-n\}$ coincides with the Breuil--Kisin twist  by $-n$  with respect to the prism $(A,I)$ (cf. \cite{BL-a}), because we are inverting $p$ and quotient out $I^m$.  
\end{notation}

\subsection{Conventions: crystals and stratifications}
 
 \begin{convention} \label{conv: def of cosimplicial rings}
    Let $A^{\bullet}$ be a cosimplicial ring.
     For any $0\leq i\leq n+1$, let $p_i:A^n\to A^{n+1}$
     be the $i$-th face map   induced by the order-preserving map $[n]\to[n+1]$ whose image does not contain $i$. For any $0\leq i\leq n$, let $\sigma_i:A^{n+1}\to A^n$ be the $i$-th degeneracy map induced by the order-preserving map $[n+1]\to[n]$ such that the preimage of $i$ is $\{i,i+1\}$. For any $0\leq i\leq n$, let $q_i:A^0\to A^n$ be the structure map induced by the map $[0]\to [n]$ sending $0$ to $i$. 
\end{convention}

\begin{dfn} \label{def: stratifications}
For a cosimplicial ring $A^{\bullet}$, a \emph{stratification with respect to $A^{\bullet}$} is a pair $(M,\varepsilon)$ consisting of a finite projective $A^0$-module $M$ and an $A^1$-linear isomorphism
 \[\varepsilon: M\otimes_{A^0,p_0}A^1\to M\otimes_{A^0,p_1}A^1,\]
 such that the following are   satisfied:
 \begin{enumerate}
 \item  $p_2^*(\varepsilon)\circ p_0^*(\varepsilon) = p_1^*(\varepsilon): M\otimes_{A^0,q_2}A^2\to M\otimes_{A^0,q_0}A^2$, which is called the \emph{cocycle condition};
 \item   $\sigma_0^*(\varepsilon) = \id_M$.
 \end{enumerate}
Write $\mathrm{Strat}(A^\bullet)$ for the category of stratifications (satisfying cocycle conditions) with respect to $A^\bullet$.
\end{dfn}

\subsection{Some other notations}
\begin{notation}[Some ringed spaces]
\label{nota: ringed space}
 We introduce several ringed spaces. Let $X$ be a smooth rigid space over $K$.
\begin{enumerate}
\item  Consider a valuation ring \[Q \in \{ \frakS_{\dR}^+=K[[E]], \kpinfty[[t]], \kinfty[[E]] \}. \]
Let $I \subset Q$ be the maximal ideal, and let $Q_m:=Q/I^m$.
Define the \'etale sheaves     
\[   \o_{X, Q, m}:= \calO_X  \otimes_{K} Q_m \]
and  
\[ \o_{X, Q, \infty}:= \varprojlim_{1\leq m<\infty} \calO_X  \otimes_{K} Q_m. \] 
Then for any $1\leq m\leq \infty$, we   can consider ringed spaces
\[ (X, \o_{X, Q, m}). \]

\item See \cite[\S 3.1]{LZ17} for full details. Let $\calB$ be the full subcategory of  $X_{C,\et}$ whose objects consist of those \'etale maps to $X_{C}$ that are the base changes of standard étale morphisms $Y\to X_{K^{\prime}}$ defined over some finite extension  $K\subset K' \subset C$ in $L$ where $Y$ is also
  required to admit a toric chart after some finite extension of $K^{\prime}$.
   For   $m<\infty$, the presheaf
  \[(Y = \Spa(A,A^+)\to X_{K^{\prime}})\in \calB\mapsto A\widehat \otimes_{K^{\prime}} \bdrplus/t^m\]
  is actually a sheaf of $\bdrplus/t^m$-algebras on $X_{C,\et}$; denote the sheaf by $\calO_{X_{C},\bdrplus, m}$. Define the sheaf 
  \[\calO_{X_{C}, \bdrplus, \infty}:=\varprojlim_m \calO_{X_{C}, \bdrplus, m}.\]
  Then for   $1\leq m\leq \infty$, we have ringed spaces
\[\calx_m:= (X_C, \o_{X_C, \bdrplus, m}). \]
Note $\calx_\infty$ is exactly $\calx^+$ in \cite[Def 3.5]{LZ17}, whence our choice of notation.
 \end{enumerate}
\end{notation}

\begin{defn}[$b$-Leibniz rule] \label{def: b-conn}
Let $B$ be a ring equipped with a differential operator $\partial: B \to B$; let $b \in B$. Let $M$ be a $B$-module. Say a map $\nabla: M\to M$ is a \emph{$b$-connection} (or it satisfies ``$b$-Leibniz rule") with respect to $\partial$ if for $f\in B, x\in M$, we have
 $$\nabla(fx)=f\nabla(x)+b\cdot \partial(f) x.$$ 
\end{defn}

  \begin{defn}[Semi-linear representations]\label{defsemilinrep}
 Suppose $\mathcal G$ is a topological group that acts continuously on a topological ring $R$. We use $\rep_{\mathcal G}(R)$ to denote the category where an object is a finite projective $R$-module $M$ (topologized via the topology on $R$) with a continuous and \emph{semi-linear} $\mathcal G$-action in the usual sense that
$$g(rx)=g(r)g(x), \forall g\in \mathcal G, r \in R, x\in M.$$ 
 \end{defn}

 \subsection{Formal identities} \label{subsec:formal identity}
 
 We record some formal identities that will be  repeatedly used.
 \begin{lemma} \label{lem expansion identity}
Consider the ring $R=\bQ[[P_1, P_2,Q_1,Q_2, Z^{\pm 1}]]$.  
For 
$$Q\in \{Q_1, Q_2, Q_1+Q_2-ZQ_1Q_2\},  \quad \text{and } P \in \{P_1, P_2, P_1+P_2\}, $$
define an element
\begin{equation}\label{eq: expansion formal}
     (1-ZQ)^{-\frac{P}{Z}}: =\sum_{n\geq 0}\left(\prod_{i=0}^{n-1}(P+iZ)\right) \cdot Q^{[n]}. 
\end{equation}
(caution: the case for $Q=Q_1+Q_2-ZQ_1Q_2$ still lands inside $R$ and is well-defined.) 
Then we have
\begin{enumerate}
  \item $ (1-ZQ_1)^{-\frac{P}{Z}} \cdot (1-ZQ_2)^{-\frac{P}{Z}}=
(1- Z(Q_1+Q_2-ZQ_1Q_2))^{-\frac{P}{Z}}$,
  \item $ (1-ZQ)^{-\frac{P_1}{Z}}  \cdot (1-ZQ)^{-\frac{P_2}{Z}}= (1-ZQ)^{-\frac{P_1+P_2}{Z}}$.  
\end{enumerate}
     \end{lemma}
 \begin{proof}
The proof is completely formal, by noticing the \emph{definition} of Eqn. \eqref{eq: expansion formal} comes from the following ``formal expansion":
\[ (1-ZQ)^{-\frac{P}{Z}}=\sum_{n\geq 0}\binom{-P/Z}{n}(-ZQ)^n=\sum_{n\geq 0}(-Z)^n\prod_{i=0}^{n-1}(-\frac{P}{Z}-i)Q^{[n]}=\sum_{n\geq 0}\prod_{i=0}^{n-1}(P+iZ)Q^{[n]}.\]
\end{proof}

\newpage
\addtocontents{toc}{\ghblue{Relative prismatic crystals}}
 
\section{Local relative $\mathbbl{\Delta}_\dR^+$-crystals: smooth case} \label{sec: rel pris smooth}
 
  Throughout this section, we fix an orientable base prism $(A,I)$ and denote $\overline A:=A/I$. Fix a generator $\beta$ of $I$.  
  For an $\overline{A}$-algebra $R$, let $(R/(A,I))_{\Prism}$ be the relative prismatic site. Recall the  sheaf $\Prism_{\dR,m}^+$ is defined  in Notation \ref{nota: de Rham sheaf}.
  The goal of this section is Thm \ref{Thm-dRCrystalasXiConnection-Rel}, where we show for $R$ small smooth, there is an equivalence of categories
     \[\Vect((R/(A,I))_{\Prism},\prism^+_{\dR, m})\xrightarrow{\simeq}{\rm MIC}_{\beta}^{\nil}(A(R)_{\dR,m}^+),\]
 together with   cohomology comparison results.  The $m=1$ case is known by \cite{Tia23}.
For general $m$ treated in this section, we need to set up the correct notion of connections; in addition, we shall use the formal identity in Lem \ref{lem expansion identity} to streamline some computations.

The structure of this section is as follows.
We review relevant (explicit) prismatic rings in \S \ref{SubSec-COmmAlg-Rel}; these are used to explicitly compute stratifications in \S \ref{subsec: ana strat rel pris}. We introduce the category of  connections in \S \ref{subsec: cat beta conn rel}, which we also translate as more concrete $\beta$-connections; these are related to prismatic crystals in \S \ref{subsec: from rel cry to conn}. Finally, we prove cohomology comparisons in \S \ref{subsec: rel pris coho compa}.

\subsection{Structure of prismatic rings} \label{SubSec-COmmAlg-Rel}
\begin{defn} \label{defn: chart rel pris}
    For a formally smooth $\overline A$-algebra $R$ of relative dimension $d$, we say it is \emph{small smooth} if there is a $p$-completely \'etale morphism 
    $$\Box:\overline A\za T_1^{\pm 1},\cdots, T_d^{\pm 1}\ya\to R.$$
     We call such a morphism $\Box$ a \emph{chart} on $R$. We will not emphasize the chart when the context  is clear.
\end{defn}

\begin{construction}[Cover of final object]  \label{constr: cover rel pris}
Let $R$ be a small smooth $\overline A$-algebra of relative dimension $d$ with a chart $\Box$. 
  Denote
 \[ A\za\underline T^{\pm 1}\ya := A\za T_1^{\pm 1}, \cdots, T_d^{\pm 1}\ya, \]
 and equip it with a $\delta$-structure by setting $\delta(T_i) = 0$ for all $i$; 
 thus $(A\za \underline T^{\pm 1}\ya,IA\za \underline T^{\pm 1}\ya)$ is a prism over $(A,I)$.
 By \cite[Lem. 2.18]{BS22}, $R$ admits a lifting $A(R)$ over $A$ and there is a unique lifting of $\Box$ as in top row of the following diagram
 \[
 \begin{tikzcd}
A\za \underline T^{\pm 1}\ya \arrow[rr, dashed] \arrow[d]  &  & A(R) \arrow[d] \\
\overline{A}\za \underline T^{\pm 1}\ya \arrow[rr, "\Box"] &  & R .            
\end{tikzcd}
\]
This makes $(A(R),IA(R))$ a prism over $(A\za \underline T^{\pm 1}\ya,IA\za \underline T^{\pm 1}\ya)$. It is well-known that $(A(R),IA(R))$ is a cover of the final object of $\Sh((R/(A,I))_{\Prism})$ (cf. \cite[Lem. 4.2]{Tia23}).   Let $(A(R)^{\bullet},IA(R)^{\bullet})$ be the \v Cech nerve of this cover.
   \end{construction}

 
 \begin{construction}[Cosimplicial structure of $A(R)^{\bullet}$] \label{const: cosimp BK rel pris} 
 For any $n\geq 0$, we have
 \[A(R)^n = A(R)^{\widehat \otimes_A(n+1)}\{\frac{\underline T_0-\underline T_1}{\beta},\cdots,\frac{\underline T_0-\underline T_n}{\beta}\}^{\wedge}_{\delta},\]
 in the sense of \cite[Prop. 3.13]{BS22}.  
 Here, $A(R)^{\widehat \otimes_A(n+1)}$ denotes the $(p,I)$-complete tensor products of $(n+1)$ copies of $A(R)$ over $A$, $\underline T_i = (\underline T)_i=(T_{1,i},\dots, T_{d,i})$ denotes the coordinate on the $(i+1)$-th factor of $A(R)^{\widehat \otimes_A(n+1)}$  and $$\frac{\underline T_0-\underline T_i}{\beta}: = \frac{ T_{1,0}-  T_{1,i}}{\beta},\dots,\frac{T_{d,0}- T_{d,i}}{\beta}.$$
 Let $p_i:A(R)^n\to A(R)^{n+1}$ and $\sigma_j: A(R)^{n+1}\to A(R)^n$ denote the   face morphism and degeneracy morphism as in \ref{conv: def of cosimplicial rings}, respectively.  
They are determined by 
 \begin{equation}\label{Equ-SimplicialStructure-Rel}
     \begin{split}
         & p_i(\underline T_j) = \left\{
           \begin{array}{rcl}
               \underline T_{j+1}-\underline T_1, & 0=i<j; \\
               \underline T_{j+1}, & 0\neq i\leq j \\
               \underline T_j, & i>j;
           \end{array}\right.\\
         & \sigma_i(\underline T_j) = \left\{
           \begin{array}{rcl}
               \underline T_{j-1}, & i<j \\
               \underline T_j, & i\geq j.
           \end{array}\right.
     \end{split}
 \end{equation} 
 Here, the expression such as ``$p_i(\underline T_j) = \underline T_{j+1}-\underline T_1$" means  $p_i( T_{s,j}) =  T_{s,j+1}-  T_{s,1}$ for each $1\leq s \leq d$; similar for other occurrences in this paper.
\end{construction}

 
 The structure of mod $I$ reduction of  $A(R)^{\bullet}$ turns out to be very explicit.
 
 \begin{lem}[\emph{\cite[Prop. 5.7 \& (5.7.1)]{Tia23}}]\label{Lem-Tian-Structure}~  
  Use notations in Construction \ref{const: cosimp BK rel pris}. For any $n\geq 0$, there is an isomorphism 
   \[R\{\underline Y_1,\dots,\underline Y_n\}^{\wedge}_{\pd}\xrightarrow{\cong}A(R)^n/IA(R)^n\]
   induced by sending $\underline Y_i$ to $\frac{\underline T_0-\underline T_i}{\beta}$,
   where $R\{Y_1,\dots,Y_n\}^{\wedge}_{\pd}$ denotes the $p$-completion of the free pd-polynomial ring generated by free variables $\underline Y_1,\dots,\underline Y_n$ over $R$ and $\underline Y_i$ denotes $Y_{1,i},\dots,Y_{d,i}$.\footnote{Caution: our variables are slightly different from those in \cite[(5.7.1)]{Tia23}. Indeed, the $\frac{\beta_{i,j}}{d}$ appearing in loc.cit. is exactly $\frac{Y_{i,j-1}-Y_{i,j}}{\beta }$ above.}
 \end{lem}

 
 Denote   $A_{\dR,m}^+=\Prism_{\dR,m}^+((A,I))$.
Define the cosimplicial $A_{\dR,m}^+$-algebra
\[  A(R)^{\bullet,+}_{\dR,m}: =\prism^+_{\dR, m}(A(R)^{\bullet},IA(R)^{\bullet}).\]
 When $m = \infty$, we also put $A(R)_{\dR}^{\bullet,+}:=A(R)_{\dR,\infty}^{\bullet,+}$. 
 Clearly, to describe the cosimplicial structure of these rings, it suffices to treat the case $m=\infty$.
 


 
 \begin{prop}\label{Prop-Structure-Rel}
   For any $0\leq i\leq n$, put $\underline Y_i = \frac{\underline T_0-\underline T_i}{\beta\underline T_0} \in A(R)^n$. 
   Then identifying $A(R)$ with the first component of $A(R)^n$ induces an isomorphism of $A_{\dR}^+$-algebras 
   \[\varprojlim_m((A(R)/I^mA(R))\{\underline Y_1,\dots,\underline Y_n\}^{\wedge}_{\pd}[\frac{1}{p}])\cong A(R)_{\dR}^{n,+}\]
   such that for any $m\geq 1$, the ring $(A(R)/I^mA(R))\{\underline Y_1,\dots,\underline Y_d\}^{\wedge}_{\pd}$ is the $p$-adic completion of the free pd-polynomial ring with free variables $\underline Y_1,\dots,\underline Y_n$ over $A(R)/I^mA(R)$.
   Moreover, via these isomorphisms for all $n$, the face morphisms $p_i:A(R)_{\dR}^{n,+}\to A(R)_{\dR}^{n+1,+}$ and the degeneracy morphisms $\sigma_i:A(R)_{\dR}^{n+1,+}\to A(R)_{\dR}^{n,+}$ are determined by (\ref{Equ-Face-R}) and (\ref{Equ-Degeneracy-R}) below, respectively.
   \begin{equation}\label{Equ-Face-R}
        \begin{split}
            &p_i(\underline Y_{j}) = \left\{\begin{array}{rcl}  
                (\underline Y_{j+1}-\underline Y_{1})(1-\beta \underline Y_{1})^{-1}, & i=0 \\
                \underline Y_{j}, & j<i\\
                \underline Y_{j+1}, & 0<i\leq j;
            \end{array}\right.\\
            &p_i(\underline T) = \left\{\begin{array}{rcl}  
                \underline T_1 = \underline T(1-\beta \underline Y_1), & i=0 \\
                \underline T, &i > 0;
            \end{array}\right.
        \end{split}
    \end{equation}
    \begin{equation}\label{Equ-Degeneracy-R}
        \begin{split}
            &\sigma_i(\underline Y_{j})=\left\{\begin{array}{rcl}
               0, &  i=0,j=1\\
               \underline Y_{j}, & j\leq i\\
               \underline Y_{j-1}, & j>i;
           \end{array}\right.\\
           &\sigma_i(\underline T_0) = \underline T_0.
        \end{split}
    \end{equation} 
     \end{prop}

 \begin{proof}
  By derived Nakayama  lemma, we are reduced to show  that
   \[A(R)/IA(R)\{\underline Y_1,\dots,\underline Y_n\}^{\wedge}_{\pd}[\frac{1}{p}]\cong A(R)_{\dR,1}^{n,+},\]
   which follows from Lemma \ref{Lem-Tian-Structure} by noting that $T_{i,0}$'s are invertible.
   The equations (\ref{Equ-Face-R}) and (\ref{Equ-Degeneracy-R}) can be deduced from (\ref{Equ-SimplicialStructure-Rel}) combined with the definition of $\underline Y_i$'s.
 \end{proof}

 
\subsection{Analysis of stratifications} \label{subsec: ana strat rel pris}

By  \cite[Prop. 2.7]{BS23},  evaluation on  $(A(R)^{\bullet},IA(R)^{\bullet})$ induces an equivalence of categories: 
\[\Vect((R/(A,I))_{\Prism}, \prism^+_{\dR, m}) \xrightarrow{\simeq} {\rm Strat}(A(R)_{\dR,m}^{\bullet,+}) \]
where the right hand side is the category of stratifications, cf. Def. \ref{def: stratifications}.
In this subsection, we study the structure of these stratifications.

 \begin{construction}[Analysis of stratification data]  \label{constr: Analysis of rel pris strat} 
   Let $(M,\varepsilon) \in {\rm Strat}(A(R)_{\dR,m}^{\bullet,+})$.   
   Identify $M$ with a sub-$A(R)_{\dR,m}^+$-module of $M\otimes_{A(R)_{\dR,m}^+,p_0}A(R)_{\dR,m}^{1,+}$ via the map $x\mapsto x\otimes 1$. The stratification data is determined by   the restriction of $\varepsilon$ to $M$.
 According to Proposition \ref{Prop-Structure-Rel}, for any $x\in M$, we can write
 \begin{equation}\label{Equ-Strat-Rel}
     \varepsilon(x) = \sum_{\underline n\in\bN^d}\nabla_{\underline n}(x)\underline Y_1^{[\underline n]},
 \end{equation}
 where $\nabla_{\underline n}$'s are $A$-linear \footnote{$\nabla_{\underline n}$ is not necessarily  $A(R)_{\dR,m}^+$-linear since $p_0\neq p_1$ on $A(R)_{\dR,m}^+$.} endomorphisms of $M$ for all $\underline n\in\bN^d$ such that
 \begin{equation}\label{Equ-StratNil-Rel}
     \lim_{|\underline n|\to+\infty}\nabla_{\underline n} = 0
 \end{equation}
 with respect to the topology on $M$. 
  By (\ref{Equ-Face-R}) and (\ref{Equ-Degeneracy-R}), it is easy to carry out the following three computations:
 \begin{equation}\label{Equ-P2P0-Rel}
     \begin{split}
        p_2^*(\varepsilon)\circ p_0^*(\varepsilon)(x) & = \sum_{\underline l,\underline m,\underline n\in\bN^d}\nabla_{\underline l}(\nabla_{\underline m+\underline n}(x))(1-\beta\underline Y_1)^{-\underline m-\underline n}(-1)^{|\underline m|}\underline Y_1^{[\underline l]}\underline Y_1^{[\underline m]}\underline Y_2^{[\underline n]}\\
        & = \sum_{\underline l,\underline m,\underline n\in\bN^d}\nabla_{\underline l}(\nabla_{\underline m+\underline n}(x))(1-\beta\underline Y_1)^{-\underline m-\underline n}(-1)^{|\underline m|}\binom{\underline l+\underline m}{\underline m}\underline Y_1^{[\underline l+\underline m]}\underline Y_2^{[\underline n]}, 
     \end{split}
 \end{equation}
 and
 \begin{equation}\label{Equ-P1-Rel}
     p_1^*(\varepsilon)(x) = \sum_{\underline n\in\bN^d}\nabla_{\underline n}(x)Y_2^{[\underline n]},
 \end{equation}
 and   
 \begin{equation}\label{Equ-Sigma-Rel}
     \sigma_0^*(\varepsilon)(x) = \nabla_{\underline 0}(x).
 \end{equation}
 Here, for any $\underline l = (l_1,\dots,l_d),\underline m =(m_1,\dots,m_d)\in\bN^d$, $\binom{\underline l+\underline m}{\underline m} = \prod_{i=1}^d\binom{l_i+m_i}{m_i}$ and $|\underline m| = m_1+\cdots+m_d$.
 We deduce that $(M,\varepsilon)$ satisfies the cocycle condition in Def. \ref{def: stratifications} if and only if:
\begin{itemize}
\item $\nabla_{\underline 0} = \id_M$, and
\item for any $\underline n\in\bN^r$, 
 \begin{equation}\label{Equ-StratCondition-Rel}
     \nabla_{\underline n}(x) = \sum_{\underline l,\underline m\in\bN^d}\nabla_{\underline l}(\nabla_{\underline m+\underline n}(x))(1-\beta\underline Y)^{-\underline m-\underline n}(-1)^{|\underline m|}\binom{\underline l+\underline m}{\underline m}\underline Y^{[\underline l+\underline m]}.
 \end{equation}  
 and $ \nabla_{\underline n}(x) \to 0$  as $|\underline n| \to \infty$.
\end{itemize}    
 \end{construction}

 \begin{notation}\label{nota: any stratification}  
 Let $M$ be an $A(R)_{\dR,m}^+$-module. For each $\underline n\in\bN^d$, let $\nabla_{\underline n}$  be an $A$-linear endomorphism of $M$. This induces an $A(R)_{\dR,m}^{1,+}$-linear map:
 \[ \varepsilon:  M\otimes_{A(R)_{\dR,m}^+,p_0}A(R)_{\dR,m}^{1,+} \to M\otimes_{A(R)_{\dR,m}^+,p_1}A(R)_{\dR,m}^{1,+} \]
via 
\[  \varepsilon(x) = \sum_{\underline n\in\bN^d}\nabla_{\underline n}(x)\underline Y_1^{[\underline n]}. \]
 For simplicity,  given $1\leq i\leq d$, denote $\nabla_i:= \nabla_{\underline 1_i}$, where $\underline 1_i = (0,\dots, 1,\dots, 0)\in\bN^d$ with $1$ appearing exactly at the $i$-th component. 
 \end{notation}

 \begin{lem}\label{Lem-Technique-Rel}
Use Notation \ref{nota: any stratification}.  
   Then the following conditions are equivalent.
   \begin{enumerate}
       \item    
    $(M,\varepsilon)$ satisfies the cocycle condition; equivalently, $\nabla_{\underline 0} = \id_M$ and \eqref{Equ-StratCondition-Rel}  is true for all $\underline n$.
   
        \item  The maps $\nabla_i$'s commute with each other, and for any $\underline n=(n_1,\dots,n_d)\in\bN^d$,
   \[\nabla_{\underline n} = \prod_{i=1}^d\prod_{j=1}^{n_i-1}(\nabla_i+\beta j),\]
   which tends to zero as $|\underline n| = n_1+\cdots+n_d\to+\infty$; here as convention this includes the condition $\nabla_{\underline 0} = \id_M$.
   \end{enumerate}
   Moreover, if the above equivalent conditions hold true, then use notation as in Lemma \ref{lem expansion identity}, we have
   \begin{equation}\label{Equ-Stratification-Rel-I}
       \varepsilon = (1-\beta\underline Y_1)^{-\beta^{-1}\underline \nabla} = \prod_{i=1}^d(1-\beta Y_{i,1})^{-\beta\nabla_i} = \sum_{\underline n=(n_1,\dots,n_d)\in\bN^d} \left(\prod_{i=1}^d\prod_{j=1}^{n_i-1}(\nabla_i+\beta j) \right)\underline Y_1^{[\underline n]}.
   \end{equation}
 \end{lem}
 \begin{proof}
   The idea is similar to  proof of \cite[Lem. 4.16]{GMWdR} and hence we will proceed as there. 

   For ``Item (1) $\Rightarrow$ Item (2)'': 
   Comparing the coefficients of $Y_{i}$ on both sides of (\ref{Equ-StratCondition-Rel}), we see that for any $\underline n = (n_1,\dots,n_d)$ and any $1\leq i\leq d$,
   \begin{equation}\label{Equ-StratIter-Rel}
       \nabla_{\underline n+\underline 1_i} = (\nabla_i+n_i\beta)\nabla_{\underline n}
   \end{equation}
   yielding that
   \[\nabla_i\nabla_j = \nabla_{\underline 1_i+\underline 1_j} = \nabla_j\nabla_i,\]
   as desired.
   Using \eqref{Equ-StratIter-Rel}, by iteration, we get 
   \[\nabla_{\underline n}  = \prod_{i=1}^r\prod_{j=1}^{n_i-1}(\nabla_i+\beta j) \]
   which tends to zero as desired by (\ref{Equ-StratNil-Rel}).

   For ``Item (2) $\Rightarrow$ Item (1)'': In this case, the formula of $\varepsilon$ defined by \eqref{Equ-Stratification-Rel-I} is well-defined.   
   It suffices to verify the cocycle condition, which we do in a \emph{formal}  way (using the formal   Lemma \ref{lem expansion identity}, cf. Rem \ref{rem why formal OK} below for explanation).
   Indeed, by Proposition \ref{Prop-Structure-Rel}, we have
\begin{equation} \label{eqpivar29} p_1^*(\varepsilon) = (1-\beta p_1(\underline Y_1))^{-\beta^{-1}\underline \nabla} = (1-\beta \underline Y_2)^{-\beta^{-1}\underline \nabla}.
\end{equation} 
Thus, we can formally compute
\begin{equation*}
       \begin{split}
           p_2^*(\varepsilon)\circ p_0^*(\varepsilon)  =& (1-\beta p_2(\underline Y_1))^{-\beta^{-1}\underline \nabla}\cdot (1-\beta p_0(\underline Y_1))^{-\beta^{-1}\underline \nabla}\\
           =& (1-\beta \underline Y_1)^{-\beta^{-1}\underline \nabla}\cdot (1-\beta (1-\beta\underline Y_1)^{-1}(\underline Y_2-\underline Y_1))^{-\beta^{-1}\underline \nabla}\\
=& (1-\beta \underline Y_2)^{-\beta^{-1}\underline \nabla}=p_1^*(\varepsilon) 
 \end{split}
   \end{equation*} 
 \end{proof}

 \begin{rem} \label{rem why formal OK}
    The formula Eqn \eqref{eqpivar29} \emph{concretely} means $\prod_{i=1}^d(1-\beta Y_{i,2})^{-\beta\nabla_i}$ where each $(1-\beta Y_{i,2})^{-\beta\nabla_i}$ is \emph{concretely} defined by  Lemma \ref{lem expansion identity}.
    Similar for the formulae in the computation of $p_2^*(\varepsilon)\circ p_0^*(\varepsilon)$. The fact that  these \emph{formal} computations are actually \emph{valid} follows from  Lemma \ref{lem expansion identity}. 
 \end{rem}

 \subsection{Category of $\beta$-connections} \label{subsec: cat beta conn rel}
 ~
 We introduce the category of $\beta$-connections.

 \begin{defn}[$\beta$-connections]
 \label{Dfn-BetaConnection} 
 \begin{enumerate}
 \item 
 For notation simplicity, let $\widehat \Omega^1$   denote the module of $(p,I)$-continuous differentials of $A(R)_{\dR,m}^+$ over $A_{\dR,m}^+$, and let
 $\widehat \Omega^1\{-1\}$ denote the   twist by $-1$ as in Notation \ref{nota: BK twist rel pris}.
  We abuse notation, and use $\mathrm{d}$ to denote the composite
   $$\mathrm{d}: A(R)_{\dR,m}^+   \to  \widehat \Omega^1 \to \widehat \Omega^1\{-1\}$$
 where the first map is  the usual differential, and the second map is induced by the map $A \to I^{-1}$.
 

 \item 
A connection over $A(R)_{\dR,m}^+$ \emph{with respect to $\mathrm{d}$} is   a finite projective $A(R)_{\dR,m}^+$-module $M$ equipped with an $A_{\dR,m}^+$-linear additive map
 \[\nabla_M:M\to M\otimes_{A(R)_{\dR,m}^+}  \widehat \Omega^1\{-1\} \]
satisfying Leibniz rule with respect to $\rm d$. Say $(M,\nabla_M)$ is \emph{integrable} (or \emph{flat}), if $\nabla_M\wedge\nabla_M = 0$.


 \item 
The elements $\frac{\rm d \log T_i}{\beta}$ form an $A(R)_{\dR,m}^+$-basis of $\widehat \Omega^1\{-1\}$.
Using this basis, a connection $(M,\nabla_M)$ can be written as
\[ \nabla_M = \sum_{i=1}^d \nabla_i \otimes \frac{\rm d \log T_i}{\beta}\]
where for each $i$, $\nabla_i: M \to M$ 
is an $A^+_\dR$-linear map satisfying ``$\beta$-Leibniz rule" with respect to $\frac{d}{d\log T_i}=T_i\frac{d}{dT_i}$ (cf. Def \ref{def: b-conn}), in the sense that
\[ \nabla_i(fx) =f\nabla_i(x) + \beta T_i\frac{d}{dT_i}(f) x, \quad \forall f \in A(R)_{\dR,m}^+ \text{ and } x \in M. \]
The connection $(M,\nabla_M)$ is integrable if and only if
$[\nabla_i,\nabla_j] = 0$ for all $i,j$.

 \item 
Because of above equivalent translations, we use ${\rm MIC}_{\beta}(A(R)_{\dR,m}^+)$ to denote the category of integrable connections with respect to $\mathrm{d}$, and simply  call these objects \emph{integrable $\beta$-connections}.

  \item An integrable $\beta$-connection $(M,\nabla_M)$ is called \emph{topologically nilpotent}, if for $\nabla_M= \sum_{i=1}^d\nabla_i\otimes\frac{\dlog T_i}{\beta}$, each $\nabla_i$ is topologically nilpotent; this condition is independent of choices of charts for $R$.
   Let ${\rm MIC}_{\beta}^{\nil}(A(R)_{\dR,m}^+)$ be the full subcategory of such objects.  
  
 
 \end{enumerate}
 \end{defn}

  \begin{defn}
   Let $(M,\nabla_M) \in {\rm MIC}_{\beta}^{\nil}(A(R)_{\dR,m}^+)$. We denote by $\rD\rR(M)=\rD\rR(M,\nabla_M)$ the induced de Rham complex
  \[M\xrightarrow{\nabla_M} M\otimes_{A(R)_{\dR,m}^+}  \widehat \Omega^1\{-1\} \xrightarrow{\nabla_M} 
 \cdots   \xrightarrow{\nabla_M}  M\otimes_{A(R)_{\dR,m}^+} \wedge^d(  \widehat \Omega^d\{-d\}) ,\]
 where for any $n\geq 0$, $\widehat \Omega^n\{-n\}:=\wedge^n_{A(R)_{\dR,m}^+}\widehat \Omega^1\{-1\}$.
  
 \end{defn}

 
 

 \subsection{From crystals to connections} \label{subsec: from rel cry to conn}
 
\begin{lemma} \label{Cor-pdStructure-Rel}
Let $\calK$ be the kernel of the degeneracy morphism $\sigma_0:A(R)_{\dR,m}^{1,+}\to A(R)_{\dR,m}^+$.  
For any $j\geq 1$,  let $\calK^{[j]}$ be the $(p,I)$-complete $j$-th pd-power of $\calK$.
\begin{enumerate}
\item Then $\calK$ is the closed pd-ideal generated by $\{Y_{1,1}^{[n_1]}\cdots Y_{d,1}^{[n_d]}\mid n_1+\cdots+n_d\geq 1, n_1,\dots,n_d\geq 0\}$ via the isomorphisms in Proposition \ref{Prop-Structure-Rel}. 
\item  The map sending $Y_{i, 1}$ to $\frac{\rm d \log T_i}{\beta}$ induces an isomorphism of $A(R)_{\dR,m}^+$-modules 
\[\calK/\calK^{[2]} \simeq \widehat \Omega^1\{-1\}.\]
\end{enumerate} 
\end{lemma} 
\begin{proof}
 Let $\calJ$ denote  the closed pd-ideal of $A(R)_{\dR,m}^{1,+}$ generated by 
   \[\{Y_{1,1}^{[n_1]}\cdots Y_{d,1}^{[n_d]}\mid n_1+\cdots+n_d\geq 1, n_1,\dots,n_d\geq 0\}.\] 
   Then by (\ref{Equ-Degeneracy-R}), $\calJ\subset\calK$. 
   We claim $\calJ = \calK$. When $m=1$,   it is easy to see $\calJ = \calK$, since by (\ref{Equ-Degeneracy-R}), $\sigma_0$ is nothing but the $R$-linear map
   \[R\{Y_{1,1},\dots,Y_{d,1}\}^{\wedge}_{\pd}[\frac{1}{p}]\to R_K\]
   sending $Y_{i,1}$'s to $0$. 
      In general, by modulo $I$, we see that $\calK\subset \calJ+I\calK$. 
   Then the desired equality follows from the (classical) $I$-completeness of $A(R)_{\dR,m}^+$.
   
    Now note $\calK^{[j]}$ is the closed ideal generated by $\{Y_{1,1}^{[n_1]}\cdots Y_{d,1}^{[n_d]}\mid n_1+\cdots+n_d\geq j, n_1,\dots,n_d\geq 0\}$. 
   This directly implies Item (2).
\end{proof}

\begin{construction}\label{construction:p_1-p_0=d}
Consider $p_0, p_1: A(R)_{\dR,m}^+ \to A(R)_{\dR,m}^{1,+}$. Then the image of $p_1-p_0$ falls inside $\mathcal K$. Thus we can define a map
\[ \tilde{d}: A(R)_{\dR,m}^+ \xrightarrow{p_1-p_0} \calK \longrightarrow \calK/\mathcal{K}^{[1]} \]
One readily checks that the following diagram is commutative:  
\[
\begin{tikzcd}
{A(R)_{\dR,m}^+} \arrow[r, "\tilde d"] \arrow[d, "\mathrm{id}"] & {\calK/\mathcal{K}^{[1]}} \arrow[d, "\simeq"] \\
{A(R)_{\dR,m}^+} \arrow[r, "\mathrm{d}"]                                 & \widehat \Omega^1\{-1\}                    
\end{tikzcd}
\]
where the right vertical map is the  isomorphism in Lem. \ref{Cor-pdStructure-Rel}. Thus, in the following we will also use $\mathrm{d}$ to denote $\tilde{d}$.

\end{construction}

 \begin{defn}[From stratifications to $\beta$-connections] \label{def: from strat to conn}    Let $(M,\varepsilon) \in {\rm Strat}(A(R)_{\dR,m}^{\bullet,+})$. 
 Define 
   \[\nabla_M:M\to M\otimes_{A(R)_{\dR,m}^+}\widehat \Omega^1_{A(R)_{\dR,m}^+/A_{\dR,m}^+}\{-1\}\]
   by setting 
   \[\nabla_M: = \sum_{i=1}^d-\nabla_i \otimes\frac{\dlog T_i}{\beta};\]
    here $\nabla_i:= \nabla_{\underline 1_i}$ as in Notation \ref{nota: any stratification}.
Because of Lem \ref{Lem-Technique-Rel}, this induces a functor
\[ {\rm Strat}(A(R)_{\dR,m}^{\bullet,+}) \to {\rm MIC}_{\beta}^{\nil}(A(R)_{\dR,m}^+).\]
 \end{defn}
 
   \begin{remark}[Intrinsic construction of connections]   \label{Construction-Intrinsic-Rel}
   One can also intrinsically define the functor in Def \ref{def: from strat to conn}, using the ``intrinsic" definition in Def \ref{Dfn-BetaConnection}(2). 
  Let $(M,\varepsilon) \in {\rm Strat}(A(R)_{\dR,m}^{\bullet,+})$. 
  \begin{enumerate}
      \item   Since $\sigma_0^*(\varepsilon) = \id_M$, we  have
   \[(\id_M-\varepsilon)(M) \subset  M\otimes_{A(R)_{\dR,m}^+,p_1}A(R)_{\dR,m}^{1,+}\cdot\calK\]
   which induces a map 
   \[\overline{\id_M-\varepsilon}:M\to  M\otimes_{A(R)_{\dR,m}^+}\calK/\calK^{[2]}. \]

   \item If one identifies $\calK/\calK^{[2]}$ with $\Omega_{A(R)_{\dR,m}^+/A_{\dR,m}^+}^1\{-1\}$ via Lem  \ref{Cor-pdStructure-Rel}, one quickly checks that this gives rise to the (explicit) connection in Def \ref{def: from strat to conn}.
  \end{enumerate} 
 \end{remark}

   The main theorem in this section is the following.

 \begin{thm}\label{Thm-dRCrystalasXiConnection-Rel}
   Assume $R$ is small smooth as in Def \ref{defn: chart rel pris}.
      \begin{enumerate}
       \item 
      There are equivalences of categories 
   \[\Vect((R/(A,I))_{\Prism}, \prism^+_{\dR, m}) \xrightarrow{\simeq} {\rm Strat}(A(R)_{\dR,m}^{\bullet,+}) \xrightarrow{\simeq} {\rm MIC}_{\beta}^{\nil}(A(R)_{\dR,m}^+),\]
   where the first functor is via evaluation on  $(A(R)^{\bullet},IA(R)^{\bullet})$,  and the second functor is   in Def \ref{def: from strat to conn} (or Rem \ref{Construction-Intrinsic-Rel}).   
    The equivalences preserve  ranks, tensor products and dualities. 
   
       \item Let $\bM \in \Vect((R/(A,I))_{\Prism},\prism^+_{\dR, m})$, and let $(M,\nabla_M) \in {\rm MIC}_{\beta}^{\nil}(A(R)_{\dR,m}^+)$ be the associated connection. There exists a quasi-isomorphism
   \[\rD\rR(M,\nabla_M)\simeq \rR\Gamma((R/(A,I))_{\Prism},\bM)\]
   which is functorial in $\bM$.
   \end{enumerate}
 \end{thm}
\begin{proof} 
Consider Item (1). The first   equivalence follows from  \cite[Prop. 2.7]{BS23}.
The second   equivalence is a translation of Lemma \ref{Lem-Technique-Rel}.  Cohomology comparison will be proved in \S \ref{subsec: rel pris coho compa}.
 \end{proof}

 \begin{exam}[Short exact sequence]~ 
   Assume $1\leq m<n\leq \infty$. Let $\bM$ be a truncated $\Prism_{\dR}^+$-crystal of level $n$ on $(R/(A,I))_{\Prism}$ and $(M,\nabla_M)$ be its associated $\beta$-connection over $A(R)_{\dR,n}^+$. Then $I^m\bM$ and $\bM/I^m\bM$ are truncated $\Prism_{\dR}^+$-crystal of levels $n-m$ and $m$, whose associated $\beta$-connections are $(\beta^mM,\nabla_{M\mid_{\beta^mM}})$ and $(M/\beta^mM,\overline \nabla_M)$, respectively. Here, $\nabla_{M\mid_{\beta^mM}}$ and $\overline \nabla_M$ denotes the restriction to $\beta^mM$ and reduction modulo $\beta^mM$ of $\nabla_M$, respectively. For the sake of simplicity, we denote these $\beta$-connections by $(\beta^mM,\nabla_{M})$ and $(M/\beta^mM,\nabla_M)$. Then we have an isomorphism of truncated $\Prism_{\dR}^+$-crystals 
   \[\bM/I^{n-m}\bM \xrightarrow{\times\beta^m}I^m\bM\]
   which identifies $(M/\beta^{n-m}M,\nabla_M)$ with $(\beta^mM,\nabla_M)$ via the multiplication by $\beta^m$. Then the short exact sequence of sheaves
   \[0\to \bM/I^{n-m}\bM\xrightarrow{\times\beta^m}\bM\to\bM/I^m\bM\to 0\]
   induces an exact sequence of $\beta$-connections
   \[0\to (M/\beta^{n-m}M,\nabla_M)\xrightarrow{\times\beta^m}(M,\nabla_M)\to(M/\beta^{m}M,\nabla_M)\to 0\]
   and a fortiori an exact sequence of de Rham complexes
   \[0\to \rD\rR(M/\beta^{n-m}M,\nabla_M)\xrightarrow{\times\beta^m}\rD\rR(M,\nabla_M)\to\rD\rR(M/\beta^{m}M,\nabla_M)\to 0.\]
 \end{exam}

  \subsection{Local cohomology comparison:  relative crystal vs.    connections} \label{subsec: rel pris coho compa}
 We focus on the proof of Theorem \ref{Thm-dRCrystalasXiConnection-Rel}(2) in this subsection. The strategy is as follows. 
Using \v Cech-Alexander complex, we shall   construct an explicit (natural) morphism 
 \[ \DR(M,\nabla_M) \to \rR\Gamma((R/(A,I))_{\Prism},\bM).\]
 Once constructed, it must be a quasi-isomorphism: indeed, by  derived Nakayama lemma, one can mod $I$ and reduces to the case $m=1$, which is the case already proved by Tian \cite[Thm. 5.14]{Tia23}.

 \begin{lem}\label{Lem-CechAlex-Rel}
  Let $\bM\in\Vect((R/(A,I))_{\Prism},\prism^+_{\dR, m})$. There exists a quasi-isomorphism 
   \[\rR\Gamma((R/(A,I))_{\Prism},\bM) \simeq \bM(A(R)^{\bullet},IA(R)^{\bullet}),\]
   which is functorial in $\bM$.
 \end{lem}
 \begin{proof}
   This follows from \v Cech-to-derived spectral sequence together with the following lemma.
 \end{proof}
  \begin{lem}\label{Lem-CechAlex-Rel-I}
      For any $\calB = (B,IB)\in (R/(A,I))_{\Prism}$ and any $i\geq 1$, $\rH^i(\calB,\bM) = 0$.
 \end{lem}
 \begin{proof}
      By derived Nakayama Lemma, we are reduced to the case for $m=1$. We need to show that for any covering $\calB\to\calC = (C,IC)$ with induced \v Cech nerve $\calC^{\bullet} = (C^{\bullet},IC^{\bullet})$, the canonical map
      \[\bM(\calB)\to\bM(\calC^{\bullet})\cong\bM(\calB)\otimes_{B/IB[\frac{1}{p}]}C^{\bullet}/IC^{\bullet}[\frac{1}{p}] \cong \bM(\calB)\otimes_BC^{\bullet}\]
      is a homotopy equivalence, where the first isomorphism above follows from that $\bM$ is a crystal.
      Then we can conclude as $B\to C^{\bullet}$ is a homotopy equivalence.
 \end{proof}

 The following construction is inspired by Tian \cite[\S 5.8]{Tia23}; in particular, when $m=1$, everything is exactly the same with the construction in \cite[\S 5.8]{Tia23}, especially the paragraph below \cite[Lem. 5.15]{Tia23}.

 \begin{construction}[de Rham complex of $\bM(A(R)^{\bullet},IA(R)^{\bullet})$]\label{Construction-Bicosimplicial-Rel}
 \begin{enumerate}
 \item      For  any $n\geq 0$, define 
   \[\widehat \Omega^1_{A(R)^{n,+}_{\dR,m}}:= \left( A(R)^{n,+}_{\dR,m}\otimes_{A(R)^+_{\dR,m}}\widehat \Omega^1_{A(R)_{\dR,m}^+/A_{\dR,m}^+}\cdot\beta^{-1} \right) \oplus\widehat \Omega^1_{A(R)^{n,+}_{\dR,m}/A(R)_{\dR,m}^+}.\]
    By Proposition \ref{Prop-Structure-Rel}, we get an isomorphism
   \[\widehat \Omega^1_{A(R)^{n,+}_{\dR,m}}\cong (\bigoplus_{i=1}^dA(R)^{n,+}_{\dR,m}\cdot\frac{\dlog T_i}{\beta})\oplus(\bigoplus_{j=1}^nA(R)^{n,+}_{\dR,m}\cdot\rd\underline Y_j)\]
   where $A(R)^{n,+}_{\dR,m}\cdot\rd\underline Y_j:=\bigoplus_{i=1}^dA(R)^{n,+}_{\dR,m}\cdot\rd Y_{i,j}$ and we identify $\underline T$ with $\underline T_0$.
   Then $\widehat \Omega^1_{A(R)^{\bullet,+}_{\dR,m}}$ carries a structure of cosimplicial $A(R)_{\dR,m}^{\bullet,+}$-module induced by (\ref{Equ-Face-R}) and (\ref{Equ-Degeneracy-R}).

 \item    For any $r\geq 0$, define 
   \[\widehat \Omega^r_{A(R)_{\dR,m}^{n,+}}: = \wedge^r\widehat \Omega^1_{A(R)^{n,+}_{\dR,m}}.\]
   Then $\widehat \Omega^r_{A(R)_{\dR,m}^{\bullet,+}}$ is a cosimplicial $A(R)_{\dR,m}^{\bullet,+}$-module as well and taking derivation:
   \[\rd: A(R)_{\dR,m}^{\bullet,+} \to \widehat \Omega^1_{A(R)_{\dR,m}^{\bullet,+}}\]
   induces a de Rham complex of cosimplicial $A(R)_{\dR,m}^{\bullet,+}$-modules $\rD\rR(A(R)_{\dR,m}^{\bullet,+},\rd)$:
   \[A(R)_{\dR,m}^{\bullet,+}\xrightarrow{\rd}\widehat \Omega^1_{A(R)_{\dR,m}^{\bullet,+}}\xrightarrow{\rd}\widehat \Omega^2_{A(R)_{\dR,m}^{\bullet,+}}\to\cdots.\]
   
  \item   In general, for any $\bM\in\Vect((R/(A,I))_{\Prism},\prism^+_{\dR, m})$ with induced $\beta$-connection $(M,\nabla_M)$, we get cosimplicial $A(R)_{\dR,m}^{\bullet,+}$-modules
   \begin{equation}\label{Equ-Q0Isom-Rel}
       \bM(A(R)^{\bullet},IA(R)^{\bullet})\otimes_{A(R)_{\dR,m}^{\bullet,+}}\widehat \Omega^r_{A(R)_{\dR,m}^{\bullet,+}}\cong M\otimes_{A(R)_{\dR,m}^+,q_0}\widehat \Omega^r_{A(R)_{\dR,m}^{\bullet,+}},
   \end{equation}
   where the isomorphism comes from the canonical isomorphism \[M\otimes_{A(R)_{\dR,m}^+,q_0}A(R)_{\dR,m}^{\bullet,+}\cong\bM(A(R)^{\bullet},IA(R)^{\bullet})\]
   induced by the injection $\{0\} \xrightarrow{\subset }\{0,1,\dots,\bullet\}$. Then $\nabla_M$ induces a morphism 
   \[\bM(A(R)^{\bullet},IA(R)^{\bullet})\otimes_{A(R)_{\dR,m}^{\bullet,+}}\widehat \Omega^r_{A(R)_{\dR,m}^{\bullet,+}}\to\bM(A(R)^{\bullet},IA(R)^{\bullet})\otimes_{A(R)_{\dR,m}^{\bullet,+}}\widehat \Omega^{r+1}_{A(R)_{\dR,m}^{\bullet,+}}\]
   via the isomorphism (\ref{Equ-Q0Isom-Rel}) such that for any $x\in M$ and any $\omega\in \widehat \Omega^r_{A(R)_{\dR,m}^{\bullet,+}}$, it is determined by 
   \begin{equation}\label{Equ-CosimplicialdR-Rel}
       x\otimes\omega\mapsto \nabla_M(x)\wedge\omega+x\otimes \rd(\omega).
   \end{equation}
   Denote this morphism by $\rd_M$.
   It will be shown by Lemma \ref{Lem-PreserveCosimplicial-Rel} that $\rd_M$ preserves cosimplicial structures and hence induces a de Rham complex of cosimplicial $A(R)_{\dR,m}^{\bullet,+}$-modules 
   $$\rD\rR(\bM(A(R)^{\bullet},IA(R)^{\bullet})\otimes_{A(R)_{\dR,m}^{\bullet,+}}\widehat \Omega^*_{A(R)_{\dR,m}^{\bullet,+}},\rd_M)$$ 
   which is defined as
   \begin{equation}\label{Equ-CosimplicialdRComplex-Rel}
       \bM(A(R)^{\bullet},IA(R)^{\bullet})\xrightarrow{\rd_M}\bM(A(R)^{\bullet},IA(R)^{\bullet})\otimes_{A(R)_{\dR,m}^{\bullet,+}}\widehat \Omega^1_{A(R)_{\dR,m}^{\bullet,+}}\xrightarrow{\rd_M}\cdots.
   \end{equation}
   Note   the above constructions are functorial in $\bM$.
 \end{enumerate}
 \end{construction}

 Now, we are able to finish the proof of Theorem \ref{Thm-dRCrystalasXiConnection-Rel}.

 \begin{proof}[Proof of Theorem \ref{Thm-dRCrystalasXiConnection-Rel}(2):] By (\ref{Equ-CosimplicialdRComplex-Rel}), we have a natural morphism
   \begin{equation}\label{Equ-CompareCoho-Rel-I}
       \bM(A(R)^{\bullet},IA(R)^{\bullet}) \to \rD\rR(\bM(A(R)^{\bullet},IA(R)^{\bullet})\otimes_{A(R)_{\dR,m}^{\bullet,+}}\widehat \Omega^r_{A(R)_{\dR,m}^{\bullet,+}},\rd_M).
   \end{equation}
   On the other hand, the de Rham complex $\rD\rR(\bM(A(R)^{\bullet},IA(R)^{\bullet})\otimes_{A(R)_{\dR,m}^{\bullet,+}}\widehat \Omega^r_{A(R)_{\dR,m}^{\bullet,+}},\rd_M)$ contains $\rD\rR(M,\nabla_M)$ as a sub-complex at the degree $0$ (i.e. the term with $\bullet = 0$). This induces a natural morphism 
   \begin{equation}\label{Equ-CompareCoho-Rel-II}
       \rD\rR(M,\nabla_M)\to \rD\rR(\bM(A(R)^{\bullet},IA(R)^{\bullet})\otimes_{A(R)_{\dR,m}^{\bullet,+}}\widehat \Omega^r_{A(R)_{\dR,m}^{\bullet,+}},\rd_M).
   \end{equation}
   By Lemma \ref{Lem-CechAlex-Rel}, it is enough to show both (\ref{Equ-CompareCoho-Rel-I}) and (\ref{Equ-CompareCoho-Rel-II}) are quasi-isomorphisms.    
   By derived Nakayama Lemma, we are reduced to the case for $m = 1$; that is, $\bM$ is a Hodge--Tate crystal. Then the result follows from the same argument for the proof of \cite[Thm. 5.14]{Tia23}. 
 \end{proof}


The following Lem \ref{Lem-PreserveCosimplicial-Rel} is claimed in Construction \ref{Construction-Bicosimplicial-Rel}(3).
  
 \begin{lem}\label{Lem-PreserveCosimplicial-Rel}
    Keep notations as above. Then the $\rd_M$ determined by (\ref{Equ-CosimplicialdR-Rel}) preserves cosimplicial structures.
 \end{lem}
 \begin{proof}
   We first show that $\rd_M$ commutes with the degeneracy morphisms. 
   For any $0\leq i\leq n$, fix an $x\in M$ and an $\omega\in \widehat \Omega^r_{A(R)_{\dR,m}^{n+1,+}}$, and denote by $\sigma_i$ the $i$-th degeneracy map.
   Then we have
   \begin{equation*}
       \begin{split}
           \rd_M(\sigma_i(x\otimes\omega)) & = \rd_M(x\otimes\sigma_i(\omega)), ~\text{as}~\sigma_i\circ q_0 = q_0\\
           & = \rd_M(x)\wedge\sigma_i(\omega)+x\otimes \rd\sigma_i(\omega),~\text{by (\ref{Equ-CosimplicialdR-Rel})}\\
           & = \nabla_M(x)\wedge\sigma_i(\omega)+x\otimes \sigma_i(\rd\omega),~\text{as}~\rd~\text{preserves cosimplicial structure}\\
           & = \sigma_i(\nabla_M(x)\wedge\omega+x\otimes\rd\omega),~\text{as}~\sigma_i\circ q_0=q_0\\
           & = \sigma_i(\rd_M(x\otimes\omega)),~\text{by (\ref{Equ-CosimplicialdR-Rel})},
       \end{split}
   \end{equation*}
   which is exactly what we want.
   
   It remains to show $\rd_M$ commutes with the face morphisms. For any $0\leq i\leq n+1$, fix an $x\in M$ and an $\omega\in \widehat \Omega^r_{A(R)_{\dR,m}^{n,+}}$, and denote by $p_i$ the $i$-th face map. 
   Assume $i\geq 1$. 
   As $p_i\circ q_0 = q_0$, the above calculation still works for replacing $\sigma_i$ by $p_i$. This implies $\nabla_M$ commutes with $p_i$ for $i\geq 1$. 
   Now assume $i = 0$ and then $p_0\circ q_0 = q_1$, which is induced by the map $\{0\}\xrightarrow{0\mapsto 1}\{0,1,\dots,\bullet\}$. 
   In this case, we have to use the stratification $(M,\varepsilon)$ induced by $\bM$ to identify 
   \[\varepsilon=(1-\beta\underline Y_1)^{-\frac{\underline \nabla}{\beta}}: M\otimes_{A(R)_{\dR,m}^+,q_1}\widehat \Omega^r_{A(R)_{\dR,m}^{\bullet,+}}\xrightarrow{\cong}M\otimes_{A(R)_{\dR,m}^+,q_0}\widehat \Omega^r_{A(R)_{\dR,m}^{\bullet,+}}.\]
   Then we have
   \begin{equation*}
       \begin{split}
           \rd_M(p_0(x\otimes\omega)) = & \rd_M(\varepsilon(x)\otimes p_0(\omega)) = \rd_M((1-\beta \underline Y_1)^{-\frac{\underline \nabla}{\beta}}(x)\otimes p_0(\omega))\\
           = & -\sum_{i=1}^d\nabla_i(\varepsilon(x))\otimes\frac{\dlog T_i}{\beta}\wedge p_0(\omega) + \sum_{i=1}^d\frac{\partial}{\partial Y_{i,1}}((1-\beta \underline Y_1)^{-\frac{\underline \nabla}{\beta}}(x))\otimes\rd Y_{i,1}\wedge p_0(\omega) + \varepsilon(x)\otimes\rd p_0(\omega)\\
           = & -\sum_{i=1}^d\nabla_i(\varepsilon(x))\otimes\frac{\dlog T_i}{\beta}\wedge p_0(\omega) + \sum_{i=1}^d\nabla_i(\varepsilon(x))\otimes(1-\beta Y_{i,1})^{-1}\rd Y_{i,1}\wedge p_0(\omega) + \varepsilon(x)\otimes\rd p_0(\omega)\\
           = & -\sum_{i=1}^d\varepsilon(\nabla_i(x))\otimes(\frac{\dlog T_i}{\beta}-(1-\beta Y_{i,1})^{-1}\rd Y_{i,1})\wedge p_0(\omega) + \varepsilon(x)\otimes\rd p_0(\omega)
       \end{split}
   \end{equation*}
   By Proposition \ref{Prop-Structure-Rel}, for any $1\leq l\leq d$, we have that
   \begin{equation*}
       \begin{split}
           \dlog (p_0(T_l)) & = \dlog T_{l,1},\text{ as we identify }\underline T\text{ with }\underline T_0\\
           & = \dlog((1-\beta Y_{l,1})T_l),\text{ by \eqref{Equ-Face-R}}\\
           & = -\beta(1-\beta Y_{l,1})^{-1}\rd Y_{l,1}+\dlog T_l.
       \end{split}
   \end{equation*}
   Therefore, we see that
   \begin{equation*}
       \begin{split}
           \rd_M(p_0(x\otimes\omega)) = & -\sum_{i=1}^d\varepsilon(\nabla_i(x))\otimes\frac{\dlog p_0(T_i)}{\beta}\wedge p_0(\omega) + \varepsilon(x)\otimes\rd p_0(\omega)\\
           = & p_0(-\sum_{i=1}^d\nabla_i(x)\otimes\frac{\dlog T_i}{\beta}\wedge\omega + x\otimes\rd\omega)\\
           = & p_0(\rd_M(x\otimes \omega)).
       \end{split}
   \end{equation*}
   In other words, we have that $\rd_M$ commutes with $p_0$ as desired.
 \end{proof}

 \newpage 
 \section{Local relative $\mathbbl{\Delta}_\dR^+$-crystals: semi-stable case} \label{SubSec-Log-Rel}

 In this section, we study local $\prism^+_{\dR}$-crystals on the relative log-prismatic site.  We only consider the case where the ``base log-prism" is the Breuil--Kisin log prism $(\frakS,(E),M_{\frakS})$; this is harmless for applications since we only consider semi-stable formal schemes over $\ok=\gs/E$.
 The main Theorem \ref{Thm-dRCrystalasXiConnection-RelLog}  classifies these $\prism^+_{\dR}$-crystals by   connections. 
  Indeed, the results are parallel to the relative prismatic case in previous  \S \ref{sec: rel pris smooth}, once we explicitly compute the relative Breuil--Kisin cosimplicial log-prisms (which are of independent interest). We have chosen to separate the semi-stable case in this section to contain the length of  \S \ref{sec: rel pris smooth}.

\subsection{Structure of prismatic rings}

 \begin{defn}\label{Convention-Rel-Log}
   We say an $\calO_K$-algebra $R$ of relative dimension $d$ is \emph{small semi-stable} if for some $0\leq r\leq d$, there is a $p$-completely \'etale morphism 
   \[\Box:\calO_K\za T_0,\dots,T_r,T_{r+1}^{\pm 1},\cdots, T_d^{\pm 1}\ya/(T_0\cdots T_r-\pi)\to R.\] 
   We call such a morphism   a \emph{chart} on $R$. We equip $R$ with the log structure $M_R$ induced by $$P_r:=\bN^{r+1}\cong \bigoplus_{i=0}^r\bN\cdot e_i\xrightarrow{e_i\mapsto T_i,~\forall~i}R.$$
   Then $(R,M_R)$ is formally log smooth over $(\calO_K,M_{\calO_K})$. 
   We denote by $(R/(\frakS,(E)))_{\Prism,\log}$ the relative logarithmic prismatic site of $(R,M_R)$ over $(\frakS,(E),M_{\frakS})$ in the sense of \cite{Kos21}.
 \end{defn}

 \begin{construction} \label{constr: cover of final obj rel log pris}
   Put $\frakS^{\Box}:=\frakS\za T_0,\dots,T_r,T^{\pm 1}_{r+1}, T_d^{\pm}\ya/(T_0\cdots T_r-u)$ and let $M_{\frakS^{\Box}}$ be the log structure on $\frakS^{\Box}$ induced by $$\bN^{r+1}\cong P_r:= \bigoplus_{i=0}^r\bN e_i\xrightarrow{e_i\mapsto T_i}\frakS^{\Box}.$$ Then by letting $\delta(T_j) = \delta_{\log}(e_i) = 0$ for all $0\leq i\leq r$ and $0\leq j\leq d$, we get a log prism $(\frakS^{\Box},(E),M_{\frakS^{\Box}})$ over $(\frakS,(E),M_{\frakS})$. 
   Using the \'etaleness of $\Box$, by \cite[Lem. 2.18]{BS22} and Lemma \ref{lem:rigidity of log-structrue}, $R$ admits a unique lifting $\frakS(R)$ over $\frakS^{\Box}$ which is compatible with the chart $\Box$ modulo $E$ such that the $\delta_{\log}$ structure on $\frakS^{\Box}$ extends to $\frakS(R)$ uniquely. 
That is, we have a diagram
\[
\begin{tikzcd}
\frakS^{\Box} \arrow[d] \arrow[r]                                                                       & \frakS(R) \arrow[d] \\
{\calO_K\za T_0,\dots,T_r,T_{r+1}^{\pm 1},\cdots, T_d^{\pm 1}\ya/(T_0\cdots T_r-\pi)} \arrow[r, "\Box"] & R                  
\end{tikzcd}
\]   
   So we get a log prism 
   $$(\frakS(R),(E),M_{\frakS(R)})\in (R/(\frakS,(E)))_{\Prism,\log}.$$
 \end{construction}

 \begin{lem}\label{lem:log covering for BK}
     The log prism $(\frakS(R),(E),M_{\frakS(R)})$ is a cover of the final object of $\Sh((R)_{\Prism,\log})$ resp. $\Sh((R/(\frakS,(E)))_{\Prism,\log})$.
 \end{lem}
 \begin{proof}
     It suffices to show $(\frakS(R),(E),M_{\frakS(R)})$ is a cover on the absolute site $(R)_{\Prism,\log}$. Equivalently, we have to show for any $(A,I,M)\in (R)_{\Prism,\log}$, it admits a cover  $(C,IC,N)$ which is the target of a map $(\frakS(R),(E),M_{\frakS(R)})\to (C,IC,N)$. 

     As the chart $\Box$ is \'etale, by \cite[Lem. 2.18]{BS22}, we are reduced to the case for
     \[R = \calO_K\za T_0,\dots,T_r,T_{r+1}^{\pm 1},\cdots, T_d^{\pm 1}\ya/(T_0\cdots T_r-\pi).\]
     By Lemma \ref{lem:existence of log-structure}, there are $t_i$'s in $A$ lifting $T_i$'s in $A/I$ such that the log-structure $M\to A$ is induced by the pre-log-structure 
     \[P_r = \bigoplus_{i=0}^r\bN\cdot e_i\xrightarrow{e_i\mapsto t_i,~\forall~i}A.\]
     Put
     \[\begin{split}
         B & = \left(A[u,T_0,\dots,T_r,T_{r+1}^{\pm 1},\dots,T_d^{\pm 1}]/(T_0\dots T_r-u)\right)[(\frac{t_0}{T_0})^{\pm 1},\dots,(\frac{t_r}{T_r})^{\pm 1}]\\
           & = A[u,T_0,\dots,T_r,T_{r+1}^{\pm 1},\dots,T_d^{\pm 1},Z_0^{\pm 1},\dots,Z_r^{\pm 1}]/(T_0\dots T_r-u, T_0Z_0-t_0,\dots,T_rZ_r-t_r)\\
           & = A[T_{r+1}^{\pm 1},\dots,T_d^{\pm 1},Z_0^{\pm 1},\dots,Z_r^{\pm 1}].
     \end{split}\]
     Clearly, $B$ is $(p,I)$-completely faithfully flat over $A$ and that 
     \[(\frac{t_0}{T_0}-1 = Z_0-1,\dots,\frac{t_r}{T_r}-1 = Z_r-1, T_{r+1}-t_{r+1},\dots,T_d-t_d)\]
     forms a $(p,I)$-completely regular sequence relative to $A$. Endow $B$ with the $\delta$-structure compatible with $A$ such that $\delta(T_i) = 0$ for all $i$.
     Let $J\subset B$ be the ideal generated by $IB$ and $(\frac{t_0}{T_0}-1,\dots,\frac{t_d}{T_d}-1)$. By \cite[Prop. 3.19]{BS22}, the $(C:=B\{\frac{J}{I}\}^{\wedge}_{\delta},IC)$ defines a covering of the prism $(A,I)$ in $(R)_{\Prism}$ which is the target of the obvious map $(\frakS(R),(E))\to (C,IC)$. By construction, the log-structures on $C$ associated to the composites
     \[(P_r\to \frakS(R)\to C) \text{ and }P_r\to A\to C\]
     coincide. Denote it by $N$ and then we get a log-prism $(C,IC,N)\in (R)_{\Prism,\log}$ satisfying all conditions as desired.
 \end{proof}

 \begin{construction} \label{cons: cech nerve rel pris log}
 Denote by $(\frakS(R)^{\bullet}_{\geo,\log},(E),M_{\frakS(R)^{\bullet}_{\geo,\log}})$ the \v Cech nerve of the cover of the final object of $\Sh((R/(\frakS,(E)))_{\Prism,\log})$; here the subscript ``geo" is to emphasize that we are working on the \emph{relative} site \footnote{The absolute version is constructed in \S \ref{sec: loc abs pris}, cf. Notation \ref{nota: many dr rings abs and rel} for a summary.}. 
Let $\frakS(R)_{\geo,\log,\dR,m}^{\bullet}$ be the cosimplicial ring induced by evaluating $\prism^+_{\dR, m}$ at $(\frakS(R)^{\bullet}_{\geo,\log},(E),M_{\frakS(R)^{\bullet}_{\geo,\log}})$. 
 In particular, we have 
 \[\frakS(R)_{\geo,\log,\dR,m}^{0} = \prism^+_{\dR, m}(\frakS(R),(E),M_{\frakS(R)}) = \frakS(R)_{\dR,m}^+.\]
  \end{construction}

  \begin{lem}[Explicit description of $\frakS(R)^{\bullet}_{\geo,\log}$]\label{lem:explicit description of cech nerve-rel}
      For any $n\geq 0$, we have
    \[\frakS(R)^{n}_{\geo,\log} = \frakS(R)^{\otimes_{\frakS}(n+1)}[[1-\frac{\underline T_1}{\underline T_0},\cdots,1-\frac{\underline T_n}{\underline T_0}]]\{\frac{1-\frac{\underline T_1}{\underline T_0}}{E},\dots,\frac{1-\frac{\underline T_n}{\underline T_0}}{E}\}^{\wedge}_{\delta},\]
    where $\frakS(R)^{\otimes_{\frakS}(n+1)}$ denotes the tensor products of $(n+1)$ copies of $\frakS(R)$ over $\frakS$, $\underline T_i$ denotes $T_{1,i},\dots,T_{d,i}$ which is induced by the chart on $R$ on the $(i+1)$-th factor of $\frakS(R)^{\otimes_{\frakS}(n+1)}$ and $1-\frac{\underline T_i}{\underline T_0}$ denotes $1-\frac{T_{1,i}}{T_{1,0}},\dots, 1-\frac{T_{d,i}}{T_{d,0}}$. The log-structure $M_{\frakS(R)^{\bullet}_{\geo,\log}}$ on $\frakS(R)^{n}_{\geo,\log}$ is induced by the composite
    \[M_{\frakS(R)}\to\frakS(R)\xrightarrow{q_i}\frakS(R)^{\otimes_{\frakS}(n+1)}[[1-\frac{\underline T_1}{\underline T_0},\cdots,1-\frac{\underline T_n}{\underline T_0}]]\{\frac{1-\frac{\underline T_1}{\underline T_0}}{E},\dots,\frac{1-\frac{\underline T_n}{\underline T_0}}{E}\}^{\wedge}_{\delta},\]
    which is independent of $q_i$.
  \end{lem}
  \begin{proof}
      By the proof of Lemma \ref{lem:log covering for BK}, we see that $\frakS(R)^{\otimes_{\frakS}(\bullet+1)}[[1-\frac{\underline T_1}{\underline T_0},\cdots,1-\frac{\underline T_{\bullet}}{\underline T_0}]]\{\frac{1-\frac{\underline T_1}{\underline T_0}}{E},\dots,\frac{1-\frac{\underline T_{\bullet}}{\underline T_0}}{E}\}^{\wedge}_{\delta}$ with the log-structure $M_{\bullet}$ induced from
      \[M_{\frakS(R)}\to\frakS(R)\to \frakS(R)^{\otimes_{\frakS}(\bullet+1)}[[1-\frac{\underline T_1}{\underline T_0},\cdots,1-\frac{\underline T_{\bullet}}{\underline T_0}]]\{\frac{1-\frac{\underline T_1}{\underline T_0}}{E},\dots,\frac{1-\frac{\underline T_{\bullet}}{\underline T_0}}{E}\}^{\wedge}_{\delta}\]
      is a well-defined cosimplicial log-prism in $(R)_{\Prism,\log}$. It suffices to show this is exactly $(\frakS(R)^{\bullet}_{\geo,\log},(E),M_{\frakS(R)^{\bullet}_{\geo,\log}})$.
      We only deal with the $n=1$ case while the general case follows from the same (but more tedious) argument.

      Consider the monoid $(\bigoplus_{i=0}^r\bN\cdot e_{i,0})\oplus_{\bN}(\bigoplus_{i=0}^r\bN\cdot e_{i,1})$, which is the push-out of the diagram 
      \[\bigoplus_{i=0}^r\bN\cdot e_{i,0}\xleftarrow{1\mapsto e_{0,0}+\cdots+e_{r,0}}\bN\xrightarrow{1\mapsto e_{0,1}+\cdots+e_{r,1}}\bigoplus_{i=0}^r\bN\cdot e_{i,1}.\]
      There is an obvious way to make 
      \[(\frakS(R)^{\otimes_{\frakS}2},(E),(\bigoplus_{i=0}^r\bN\cdot e_{i,0})\oplus_{\bN}(\bigoplus_{i=0}^r\bN\cdot e_{i,1}))\]
      a pre-log prism over $(\frakS,(E),\bN)$. There exists a ``multiplication'' map of prelog rings
      \[\pr:((\bigoplus_{i=0}^r\bN\cdot e_{i,0})\oplus_{\bN}(\bigoplus_{i=0}^r\bN\cdot e_{i,1})\to\frakS(R)^{\otimes_{\frakS}2})\to (\bigoplus_{i=0}^r\bN\cdot e_i\to R)\]
      whose induced map on monoids
      \[(\bigoplus_{i=0}^r\bN\cdot e_{i,0})\oplus_{\bN}(\bigoplus_{i=0}^r\bN\cdot e_{i,1})\to \bigoplus_{i=0}^r\bN\cdot e_i\]
      is given by carrying each $e_{i,j}$ with $j\in \{0,1\}$ to $e_i$. By \cite[Construction 2.18]{Kos21}, the $(p,E)$-complete exactification of $\pr$ above is given by 
      \[(\bigoplus_{i=0}^r\bN\cdot e_i\xrightarrow{e_i\mapsto T_{i,0},~\forall~i}\left(\frakS(R)^{\otimes_{\frakS}2}[(\frac{T_{0,1}}{T_{0,0}})^{\pm 1},\dots,(\frac{T_{r,1}}{T_{r,0}})^{\pm 1}]\right)^{\wedge_{(p,E)}})\to (\bN^{r+1}\to R)\]
      which extends $\pr$ above and sends each $\frac{T_{i,1}}{T_{i,0}}$ to $1$. This gives a pre-log prism 
      \[(\left(\frakS(R)^{\otimes_{\frakS}2}[(\frac{T_{0,1}}{T_{0,0}})^{\pm 1},\dots,(\frac{T_{r,1}}{T_{r,0}})^{\pm 1}]\right)^{\wedge_{(p,E)}},(E),\bN^{r+1})\]
      over $(\frakS,\bN,M_{\frakS})$, in which $E, \frac{T_{1,1}}{T_{1,0}}-1,\dots,\frac{T_{d,1}}{T_{d,0}}-1$ form a $(p,E)$-completely regular sequence relative to $\frakS(R)$. By \cite[Prop. 3.19]{BS22}, the 
      \[B:=(\left(\frakS(R)^{\otimes_{\frakS}2}[(\frac{T_{0,1}}{T_{0,0}})^{\pm 1},\dots,(\frac{T_{r,1}}{T_{r,0}})^{\pm 1}]\right)^{\wedge_{(p,E)}}\{\frac{\frac{\underline T_1}{\underline T_0}-1}{E(u)}\}^{\wedge_{(p,E)}}_{\delta}\]
      exists and $(p,E)$-completely faithfully flat over $\frakS(R)$. Endow $B$ with the log-strutcure $M_B$ induced by 
      \[\bigoplus_{i=0}^r\bN\cdot e_i\xrightarrow{e_i\mapsto T_i,\forall~i}B\]
      and then we get a log-prism $(B,(E),M_B)\in (R)_{\Prism,\log}$.

      We claim that $(B,(E),M_B) = (\frakS(R)_{\geo,\log}^1,(E),M_{\frakS(R)^1_{\geo,\log}})$. Indeed, for any $(A,I,M_A)\in (R)_{\Prism,\log}$ together with two morphisms 
      \[f_0,f_1:(\frakS(R),(E),M_{\frakS(R)})\to (A,I,M_A).\]
      By Lemma \ref{lem:rigidity of log-structrue}, the composites
      \[M_{\frakS(R)}\to\frakS(R)\xrightarrow{f_i}A\]
      induce the same log-structure on $A$ which is exactly $M_A$. In particular, for any $0\leq i\leq r$, the $f_0(T_i)$ is a unit-multiple of $f_1(T_i)$, yielding a unique morphism
      \[F:(\left(\frakS(R)^{\otimes_{\frakS}2}[(\frac{T_{0,1}}{T_{0,0}})^{\pm 1},\dots,(\frac{T_{r,1}}{T_{r,0}})^{\pm 1}]\right)^{\wedge_{(p,E)}},(E),\bN^{r+1})\to(A,I,M_A)\]
      making the following diagram
      \[
      \xymatrix@C=0.5cm{
        (\left(\frakS(R)^{\otimes_{\frakS}2}[(\frac{T_{0,1}}{T_{0,0}})^{\pm 1},\dots,(\frac{T_{r,1}}{T_{r,0}})^{\pm 1}]\right)^{\wedge_{(p,E)}},\bN^{r+1})\ar[rr]\ar[d]&&(A,M_A)\ar[d]\\
        (R,\bN^{r+1})\ar[rr]&& (A/I,M_A)
      }
      \]
      commute. In particular, we have $F(\frac{T_{i,1}}{T_{i,0}}-1)\in I$ for any $0\leq i\leq d$, yielding a unique morphism
      \[(B,(E),M_B)\to (A,I,M_A)\]
      as desired.

      To conclude, it remains to show that
      \[(B,(E),M_B) = (\frakS(R)^{\otimes_{\frakS}2}[[1-\frac{\underline T_1}{\underline T_0}]]\{\frac{1-\frac{\underline T_1}{\underline T_0}}{E}\}^{\wedge}_{\delta},(E),M_1).\]
      By universal property of $(B,(E),M_B)$, there is a unique morphism
      \[i:(B,(E),M_B) \to (\frakS(R)^{\otimes_{\frakS}2}[[1-\frac{\underline T_1}{\underline T_0}]]\{\frac{1-\frac{\underline T_1}{\underline T_0}}{E}\}^{\wedge}_{\delta},(E),M_1).\]
      It remains to construct a morphism over $(\frakS,(E),\bN)$
      \[j:(\frakS(R)^{\otimes_{\frakS}2}[[1-\frac{\underline T_1}{\underline T_0}]]\{\frac{1-\frac{\underline T_1}{\underline T_0}}{E}\}^{\wedge}_{\delta},(E),M_1)\to(B,(E),M_B)\]
      which is the inverse of $i$. As $1-\frac{T_{i,1}}{T_{i,0}}$ is topologically nilpotent in $B$ for each $i$, the natural map $\frakS(R)^{\otimes_{\frakS}2}\to B$ extends uniquely to a map $\frakS(R)^{\otimes_{\frakS}2}[[1-\frac{\underline T_1}{\underline T_0}]]\to B$ and thus induces a unique map
      \[\frakS(R)^{\otimes_{\frakS}2}[[1-\frac{\underline T_1}{\underline T_0}]]\{\frac{1-\frac{\underline T_1}{\underline T_0}}{E}\}^{\wedge}_{\delta}\to B\]
      which is compatible with $\delta_{\log}$-structures as desired. By construction, one can check $i$ and $j$ are the inverses of each other, and then complete the proof.
  \end{proof}

 \begin{prop}\label{Prop-Structure-Rel-Log}
   For any $1\leq i\leq n$ and $1\leq j\leq d$, put $Y_{j,i} = \frac{1-\frac{T_{j,i}}{T_{j,0}}}{-E}$. Then identifying $\frakS(R)$ with the first component of $\frakS(R)_{\geo,\log}^n$ induces an isomorphisms of $\frakS_{\dR}^+$-algebras 
   \[\varprojlim_m((\frakS(R)/E^m\frakS(R))\{\underline Y_1,\dots,\underline Y_n\}^{\wedge}_{\pd}[\frac{1}{p}])\cong \frakS(R)_{\geo,\log,\dR}^{n,+}.\]
   Via these isomorphisms for all $n$, the face and degeneracy morphisms are determined by \emph{exactly the same formula} as in \eqref{Equ-Face-R} and \eqref{Equ-Degeneracy-R} for $\beta = E$ (which we omit here).
 \end{prop}
 \begin{proof}
   Note that $\delta(\frac{T_{j,i}}{T_{j,0}}) = 0$ for any $1\leq j\leq d$. By the same proof of \cite[Lem. 2.11 and Prop. 2.12]{GMWHT}, we get an isomorphism of cosimplicial rings
   \[R\{\underline Y_1,\dots,\underline Y_{\bullet}\}^{\wedge}_{\pd}\cong \frakS(R)_{\geo,\log}^{\bullet,+}/(E)\]
   where the structure morphisms of the former is given by \eqref{Equ-SimplicialStructure-Rel}.
   This gives an isomorphism
   \[R\{\underline Y_1,\dots,\underline Y_{\bullet}\}^{\wedge}_{\pd}[1/p]\cong \frakS(R)_{\geo,\log,\dR}^{\bullet,+}/(E)\]
   by inverting $p$. The same argument for the proof of \cite[Lem. 3.11(1)]{GMWdR} shows that the morphism
   \[\varprojlim_m((\frakS(R)/E^m\frakS(R))\{\underline Y_1,\dots,\underline Y_n\}^{\wedge}_{\pd}[\frac{1}{p}])\to\frakS(R)_{\geo,\log,\dR}^{n,+}\]
   carrying each $Y_{j,i}$ to $\frac{1-\frac{T_{j,i}}{T_{j,0}}}{-E}$ is well-defined. This is an isomorphism, as can be checked by modulo $E$ and apply derived Nakayama Lemma. The face and degeneracy formulae can be checked readily. 
    \end{proof}

\subsection{From crystals to connections} 

\begin{construction}
  Since $u$ is a unit in $\frakS_{\dR,m}^+$, we see the log structure on $\frakS(R)_{\dR,m}^+$ induced by $M_{\frakS(R)}\to \frakS(R)\to\frakS(R)_{\dR,m}^+$ is trivial. So the module $\widehat \Omega^1_{\frakS(R)_{\dR,m}^+/\frakS_{\dR,m}^+}$ of continuous differentials coincides with its logarithmic analogue and admits an isomorphism
 \[\widehat \Omega^1_{\frakS(R)_{\dR,m}^+/\frakS_{\dR,m}^+}\cong \bigoplus_{i=1}^d\frakS(R)_{\dR,m}^+\dlog T_i.\]
 Using this, one can   define the category of topologically nilpotent  $E$-connections over $\frakS(R)_{\dR,m}^+$ exactly as we did in Definition \ref{Dfn-BetaConnection};     denote the category by ${\rm MIC}_{E}^{\nil}(\frakS(R)_{\dR,m}^+)$.
 \end{construction}
 
The following main theorem of this section is the semi-stable analogue of Theorem \ref{Thm-dRCrystalasXiConnection-Rel}.  
 
 \begin{thm}\label{Thm-dRCrystalasXiConnection-RelLog}
   Let $R$ be a small semi-stable  $\calO_K$-algebra.   
   \begin{enumerate}
       \item There are equivalences of categories:
   \[\Vect((R/(\frakS,(E)))_{\Prism,\log},\prism^+_{\dR, m})\xrightarrow{\simeq}
 {\rm Strat}(\frakS(R)_{\geo,\log,\dR,m}^{\bullet,+})   \xrightarrow{\simeq}
   {\rm MIC}_{E}^{\nil}(\frakS(R)_{\dR,m}^+).\]

       \item Let $\bM$ be a $\Prism_{\dR}^+$-crystal on $(R/(\frakS,(E)))_{\Prism,\log}$ and $(M,\nabla_M)$ be the associated $E$-connection over $\frakS(R)_{\dR,m}^+$. Then there exists a quasi-isomorphism
   \[\rD\rR(M,\nabla_M)\simeq \rR\Gamma((R/(\frakS(R),(E)))_{\Prism,\log},\bM)\]
   which is functorial in $\bM$.
   \end{enumerate}
 \end{thm}
  \begin{proof}
   The idea is the same as in Theorem \ref{Thm-dRCrystalasXiConnection-Rel}, thus we shall be brief.
   
Item (1).  
 The first equivalence follows from \cite[Prop. 2.7]{BS23} together with Lemma \ref{lem:log covering for BK}. 
  We now prove 
 \[{\rm Strat}(\frakS(R)_{\geo,\log,\dR,m}^{\bullet,+})\simeq {\rm MIC}_{E}^{\nil}(\frakS(R)_{\dR,m}^+).\]
Let $(M,\varepsilon)$ be a stratification with respect to $\frakS(R)_{\geo,\log,\dR,m}^{\bullet,+}$ such that for any $x\in M$, we have 
 \[\varepsilon(x) = \sum_{\underline n\in\bN^d}\nabla_{\underline n}(x)\underline Y_1^{[\underline n]}\]
 as (\ref{Equ-Strat-Rel}).
 Then it is easy to see that $\varepsilon$ satisfies the cocycle condition if and only if $\nabla_M^{\underline 0} = \id_M$ and 
 \eqref{Equ-StratCondition-Rel} holds true for any $\underline n\in\bN^d$. Applying Lemma \ref{Lem-Technique-Rel}, we conclude that $(M,\varepsilon)$ satisfies the cocycle condition if and only if for any $\underline n = (n_1,\dots,n_d)\in \bN^d$,  
 \[\nabla_{\underline n} = \prod_{i=1}^d\prod_{j=1}^{n_i-1}(\nabla_i-E j),\]
where $\nabla_i: = \nabla_{\underline 1_i}$ is nilpotent for each $i$. 
Similar to Construction \ref{Construction-Intrinsic-Rel}, we see that 
 \[\nabla_M: = \sum_{i=1}^d\nabla_i\otimes\dlog T_i\cdot E^{-1}\]
 is a topologically nilpotent $E$-connection over $\frakS(R)_{\dR,m}^+$. Therefore, we deduce that desired equivalence above and complete the proof of Item (1). Note that the stratification $(M,\varepsilon)$ corresponding to the topologically nilpotent $E$-connection $(M,\nabla_M)$ (equivalently, a $\Prism_{\dR}^+$-crystal $\bM$), is determined such that for any $x\in M$,
 \begin{equation}\label{Equ-Stratification-Rel-Log}
     \varepsilon(x) = (1+E\underline Y_1)^{E^{-1}\nabla_M}(x): = \sum_{\underline n\in\bN^d}\prod_{i=1}^d\prod_{j=1}^{n_i-1}(\nabla_i-jE)(x)\underline Y_1^{[\underline n]}.
 \end{equation}  
 
 Consider Item (2). The argument for the proof of Lemma \ref{Lem-CechAlex-Rel-I} together with Corollary \ref{cor:rigidity of log-structrue} implies that for any   $\Prism_{\dR,m}^+$-crystal $\bM$, any $\calB\in(R/(\frakS,(E)))_{\Prism,\log}$ and any $i\geq 1$, $\rH^i(\calB,\bM) = 0$. So it follows from the \v Cech-to-derived spectral sequence that there exists a quasi-isomorphism, which is functorial in $\bM$,
 \[\rR\Gamma((R/(\frakS,(E)))_{\Prism,\log},\bM)\cong \bM(\frakS(R)_{\geo,\log}^{\bullet},(E),M_{\frakS(R)_{\geo,\log}^{\bullet}}).\]
 Proceeding as in Construction \ref{Construction-Bicosimplicial-Rel} for replacing $\widehat \Omega^1_{A(R)_{\dR,m}^+/A_{\dR,m}^+}$ by $\widehat \Omega^1_{\frakS(R)_{\dR,m}^+/\frakS_{\dR,m}^+}$, We get a de Rham complex of cosimplicial $\frakS(R)_{\geo,\log,\dR,m}^{\bullet,+}$-modules:
   \begin{equation}\label{Equ-CosimplicialdRComplex-Rel-Log}
      \bM(\frakS(R)_{\geo,\log}^{\bullet},(E),M_{\frakS(R)_{\geo,\log}^{\bullet}})\xrightarrow{\nabla_M}\bM(\frakS(R)_{\geo,\log}^{\bullet},(E),M_{\frakS(R)_{\geo,\log}^{\bullet}})\otimes_{\frakS(R)_{\geo,\log,\dR,m}^{\bullet,+}}\widehat \Omega^{1}_{\frakS(R)_{\geo,\log,\dR,m}^{\bullet,+}}\xrightarrow{\nabla_M}\cdots,
   \end{equation}
   which is denoted by  $\rD\rR(\bM(\frakS(R)_{\geo,\log}^{\bullet},(E),M_{\frakS(R)_{\geo,\log}^{\bullet}})\otimes_{\frakS(R)_{\geo,\log,\dR,m}^{\bullet,+}}\widehat \Omega^{*}_{\frakS(R)_{\geo,\log,\dR,m}^{\bullet,+}},\nabla_M)$, such that it reduces to $\rD\rR(M,\nabla_M)$, the de Rham complex induced by $(M,\nabla_M)$, and that there exists a natural map
   \[\bM(\frakS(R)_{\geo,\log}^{\bullet},(E),M_{\frakS(R)_{\geo,\log}^{\bullet}})\to\rD\rR(\bM(\frakS(R)_{\geo,\log}^{\bullet},(E),M_{\frakS(R)_{\geo,\log}^{\bullet}})\otimes_{\frakS(R)_{\geo,\log,\dR,m}^{\bullet,+}}\widehat \Omega^{*}_{\frakS(R)_{\geo,\log,\dR,m}^{\bullet,+}},\nabla_M).\]
   Then by derived Nakayama Lemma, we only need to show Item (2) for $m = 1$.

   Now, we focus on the case for $m=1$. As $E = 0$ in this setting, the stratification \eqref{Equ-Stratification-Rel-Log} reduces to
   \begin{equation*}
     \varepsilon(x) = \exp(\sum_{i=1}^dY_{i,1}\nabla_i)(x): = \sum_{\underline n\in\bN^d}\underline \nabla^{\underline n}(x)\underline Y_1^{[\underline n]},
   \end{equation*} 
   and the connection
   \[\nabla_M = \sum_{i=1}^d\nabla_i\otimes\frac{\dlog T_i}{E}:M\to M\otimes_{R}\widehat \Omega^{1,\log}_{R}\{-1\} = \oplus_{i=1}^dM\cdot\frac{\dlog T_i}{E}\]
   is exactly a Higgs field on $M$. So we are in the setting of \cite{Tia23}, and thus can conclude by the same argument for the proof of \cite[Th. 5.14]{Tia23}.
 \end{proof}

 \begin{rmk}
     Similar to Remark \ref{Construction-Intrinsic-Rel}, one can also give an intrinsic construction of connections in the semi-stable case by showing that Lemma \ref{Cor-pdStructure-Rel} also holds true in this case. We omit the details here because everything is the same as the smooth case and we will not use this intrinsic construction.
 \end{rmk}

 Note that in Convention \ref{Convention-Rel-Log}, if we put $r = 0$, then $R$ is small smooth. In other words, we may regard a small smooth $\calO_K$-algebra $R$ as a small semi-stable $\calO_K$-algebra such that the log structure $M_R$ is induced by the composition $\bN\xrightarrow{1\mapsto \pi}\calO_K\to R$. Comparing Proposition \ref{Prop-Structure-Rel} with Proposition \ref{Prop-Structure-Rel-Log}, we get an obvious isomorphism of cosimplicial rings 
 \[A(R)_{\dR,m}^{\bullet,+}\cong \frakS(R)_{\geo,\log,\dR}^{\bullet,+}\]
 for $A = \frakS$ by identifying $\underline T$ and $\underline Y_i$'s. Then the following corollary is obvious:

 \begin{cor}\label{Prop-UsualVSLog}
  Let $R$ be a small smooth $\calO_K$-algebra. 
  We have a commutative diagram of equivalences of categories where the horizontal functor is induced by the   forgetful functor 
   $(R/(\frakS,(E)))_{\Prism,\log}\to (R/(\frakS,(E)))_{\Prism}$ sending each log prism $(B,EB,M_B)$ to $(B,IB)$ 
  \[
  \begin{tikzcd}
{ \Vect((R/(\frakS,(E)))_{\Prism},\prism^+_{\dR, m})} \arrow[rr, "\simeq"] \arrow[rd, "\simeq"] &                                      & {\Vect((R/(\frakS,(E)))_{\Prism,\log},\prism^+_{\dR, m})} \arrow[ld, "\simeq"] \\
                                                                                                & {\mic_E^{\nil}(\frakS(R)_{\dR,m}^+)}. &                                                                               
\end{tikzcd}
\]
 For corresponding objects
 \[\begin{tikzcd}
\bm \arrow[rd,mapsto] \arrow[rr,mapsto] &                 & \bm_\log \arrow[ld,mapsto] \\
                          & {(M, \nabla_M)} &                    
\end{tikzcd}
\]
   we have 
   \[\rR\Gamma((R/(\frakS,(E)))_{\Prism},\bM)\simeq \rR\Gamma((R/(\frakS,(E)))_{\Prism,\log},\bM_\log)\simeq \rD\rR(M,\nabla_M).\]
 \end{cor}

 \newpage
  \addtocontents{toc}{\ghblue{Absolute prismatic crystals}}

\section{Arithmetic stratification: an axiomatized review} \label{sec: arith infinite dim}

We now move on to the (local) study of \emph{absolute} $\prism_{\dR}^+$-crystals. In the point case, i.e. $R=\ok$, the $\prism_{\dR}^+$-crystals  are   classified in \cite{GMWdR}, by $a$-small $E$-connections (Def \ref{def: a small arith}). That is, we have
\begin{equation} \label{eq: sec arith inf k case}
\vect(\okprisast, \prism_{\dR,m}^+) \simeq \strat(\gs^{\bullet, +}_{\ast,\dR, m}) \simeq \mic^a(\gsdrm).
\end{equation} 
Let $F/K$ be a finite extension, the goal of this section is to consider objects  possibly of \emph{infinite} rank. The main Prop \ref{prop:point abs crystal inf rank} proves the equivalences:
 \begin{equation} \label{eq: sec arith inf}
  \strat^\wedge(\gs^{\bullet, +}_{\ast,\dR, m}\otimes_K F) \simeq \mic^{\wedge, a}(\gsdrm \otimes_K F)
\end{equation} 
If we remove all the superscripts $\wedge$, then the equivalences of \eqref{eq: sec arith inf} follow from \eqref{eq: sec arith inf k case}, since the relevant objects in \eqref{eq: sec arith inf} are just those in \eqref{eq: sec arith inf k case}  with an ``$F$-structure". It turns out the cases with $\wedge$ follow from exactly the same proof (once   topological issues are taken care of).
We need the infinite rank case since we need to deal with a \emph{specific} example, cf.~$S^1$ in Lem \ref{coho of Sone}. Note that we do not discuss \emph{infinite rank crystals} here: the descent result of Drinfeld--Mathew \cite[Theorem 5.8]{mathew2022faithfully} is only concerned with coherent objects, thus it is   not clear if infinite rank crystals can be described via stratifications.

The results of this section are not needed until \S \ref{sec: analytic sen kummer}; the readers are advised to skip it for now and come back later. 
Nonetheless, we put this section here as the discussions come from the point case \cite{GMWdR}.



Throughout this section, we always assume $m<\infty$, and thus all relevant rings are Banach spaces (instead of Fr\'echet one).
 
    \begin{notation} \label{nota: a sec 6}
   For $*\in\{\emptyset,\log\}$, let \begin{equation*}
a=
\begin{cases}
  -E'(\pi), &  \text{if } \ast=\emptyset \\
 -\pi E'(\pi), &  \text{if } \ast=\log
\end{cases}
\end{equation*}  
\end{notation}
 
 \begin{defn} \label{def: a small arith}
 Let  $F/K$ be a finite extension. 
Let $\mic^{\wedge, a}(F[[E]]/E^m)$ be the category where an object is an ON (orthonormal) Banach $F[[E]]/E^m$-module $M$, equipped with  a continuous additive map $\phi: M\to M$  satisfying Leibniz rule with respect to $E\frac{d}{dE}: F[[E]]/E^m \to F[[E]]/E^m$, such that $\phi$ is $a$-small in the sense that
 \begin{equation}
 \label{eq axiom phi nilp} \lim_{n\to+\infty}a^n\prod_{i=0}^{n-1}(\phi-i) = 0
 \end{equation} 
 Note above condition is satisfied for $M$ if and only if it is so for the reduction
$\phi: M/E \to M/E.$
 \end{defn}

 \begin{defn}     
We define a notion of infinite rank (ON) stratifications, compare Def \ref{def: stratifications} for the usual version. 
Recall we always equip  $\prisdrm(A, I, \ast)$ with a topology making it a Banach $K$-algebra, cf. Notation \ref{nota: de Rham sheaf}.
  Let $(\gs^\bullet_\ast, (E), \ast)$ be the \v Cech nerve of $(\gs, (E), \ast) \in \okprisast$. We then obtain a co-simplicial  (topological) ring
 $$\frakS_{*,\dR,m}^{\bullet,+} =\prism^+_{\dR, m}((\gs^\bullet_\ast, (E), \ast)).$$
For simplicity, denote $S^n= \gs_{\ast,\dR,m}^{n,+}$.
Let $\strat^\wedge(\gs^{\bullet, +}_{\ast,\dR, m}\otimes_K F)$ be the category where an object is an ON Banach $F[[E]]/E^m=S^0\otimes F$-module $M$  and an $S^1$-linear isomorphism
 \[\varepsilon: M\hatotimes_{S^0,p_0}S^1 \simeq  M\hatotimes_{S^0,p_1}S^1,\]
such that the cocycle condition $p_2^*(\varepsilon)\circ p_0^*(\varepsilon) = p_1^*(\varepsilon)$ is verified, and $\sigma_0^*(\varepsilon) = \id_M$.
 \end{defn}
  
 


  \begin{notation}
  Recall the structure of $\gs_{\ast,\dr, m}^{\bullet,+}$ as in \cite{GMWdR}.  
 Use Notation \ref{nota: a sec 6}. 
Let \[X_i=\frac{E(u_i)-E(u_0)}{aE(u_0)}.\]  
For each $n \geq 0$, denote $S^n= \gs_{\ast,\dR,m}^{n,+}$, then 
\[S^n:= \left(\o_K[[E]]/E^m\{X_1, \cdots, X_n\}_{\mathrm{pd}}^\wedge\right)[1/p]   \]
The structure maps in the cosimplicial ring are   induced by
\begin{equation} \label{newEqu-Face}
        \begin{split}
           & p_i(X_j) = \left\{\begin{array}{rcl}
                (X_{j+1}-X_1)(1+aX_1)^{-1}, & i=0 \\
                X_j, & j<i\\
                X_{j+1}, & 0<i\leq j;
            \end{array}\right.\\
            & p_i(E(u_0))=\left\{\begin{array}{rcl}
               E(u_1) = E(u_0)(1+aX_1),  & i=0\\
               E(u_0),  & i>0;
            \end{array}\right.
        \end{split}
    \end{equation}
    \begin{equation} \label{newEqu-StratificationD}
        \begin{split}
           & \sigma_i(X_j)=\left\{\begin{array}{rcl}
               0, &  i=0,j=1\\
               X_j, & j\leq i\\
               X_{j-1}, & j>i;
           \end{array}\right.\\
           & \sigma_i( E(u_{0})) = E(u_{ 0 }), \qquad \forall i.
        \end{split}
    \end{equation}
 \end{notation}

 \begin{construction}\label{cons: ari strat}
 Let  $(M, \phi) \in \mic^{\wedge, a}(F[[E]]/E^m)$, the formula 
  \[ \varepsilon(x)=(1+aX_1)^\phi(x)=  \sum_{n \geq 0} (a^n\prod_{i=0}^{n-1}(\phi -i)(x))\cdot  X_1^{[n]}\]
  converges because of $a$-smallness of $\phi$, and  induces an $S^1$-linear isomorphism
   \[ \varepsilon: M \hatotimes_{p_0, S^0} S^1 \simeq M \hatotimes_{p_1, S^0} S^1.\]
   We claim that $\varepsilon$ satisfies the cocycle condition. 
 \end{construction}
 \begin{proof}
 To verify the cocycle condition, we use similar ``formal" argument similar to Lem \ref{Lem-Technique-Rel}, cf. also Rem. \ref{rem why formal OK}. Indeed, note we have the formal identity
\[\varepsilon= (1+aX_1)^{\phi}=(\frac{E_1}{E_0})^\phi \]
One can easily compute
\[ p_1^\ast(\varepsilon) =(\frac{E_2}{E_0})^\phi \]
\[ p_2^\ast(\varepsilon) =(\frac{E_1}{E_0})^\phi \]
\[ p_0^\ast(\varepsilon) =(\frac{E_2}{E_1})^\phi \]
Thus we have $  p_2^*(\varepsilon)\circ p_0^*(\varepsilon) =  p_1^\ast(\varepsilon)$ formally by Lem \ref{lem expansion identity}.
 \end{proof}
 
 \begin{prop} \label{prop:point abs crystal inf rank}
There are equivalences of categories
\begin{equation} \label{eq: sec arith inf new}
   \strat^\wedge(\gs^{\bullet, +}_{\ast,\dR, m}\otimes_K F) \simeq \mic^{\wedge, a}(\gsdrm \otimes_K F).
\end{equation} 
 \end{prop}
 \begin{proof}  
For finite rank case, the equivalence of categories is  proved in \cite[Prop 4.5]{GMWdR}. Note in \emph{loc. cit.}, $F$ is $K$; but this is indeed irrelevant for results here. One also readily examines that all these results still hold true in the Banach case: so long $\phi$ is $a$-small, all relevant  infinite summations in \emph{loc. cit.} still converge. 
 \end{proof}

 The reason we introduce infinite rank connections in Def \ref{def: a small arith} is so that we can treat the following \emph{special example}.

  \begin{lemma} \label{coho of Sone}
  Consider the ring 
\[ S^1=   F[[E]]/E^m \{ X_1\}_\pd . \]
Consider a differential operator on $S^1$,
\[ \phi:= \frac{d}{d(\log E(u_0)) }:= E(u_0)\cdot \frac{d}{E'(u_0)du_0}.\]
Then it induces a map
\[\phi: S^1 \to S^1 \] 
which defines an object in $\mic^{\wedge, a}(F[[E]]/E^m)$.
In addition, we have a short exact sequence:
\begin{equation} \label{eq short exact Sone}
0 \to S^0 \xrightarrow{p_1} S^1 \xrightarrow{\phi} S^1 \to 0
\end{equation}  
 \end{lemma}
\begin{proof}
The operator satisfies $\phi(E)=E$ and $\phi(X_1)=- X_1-\frac{1}{a}$, and thus induces $\phi: S^1 \to S^1$. To check condition  \eqref{eq axiom phi nilp}, it suffices to do so after modulo $E$.
As $X^{[n]}$'s   form an ON basis for $S^1/E$,  it suffices to check \eqref{eq axiom phi nilp} on  $X^{[n]}$: this can be readily verified.
To prove the sequence \eqref{eq short exact Sone} is short exact, again one can modulo $E$; it then reduces to prove the following is short exact:
\[ 0 \to F \to F\{X_1\}_\pd \xrightarrow{\phi}
 F\{X_1\}_\pd  \to 0\] 
This is the computation already done in   \cite[Lem 2.6(3)]{AHLB1}, which we shall quickly repeat here for convenience; further relation with \cite{AHLB1} will be discussed in \S \ref{sec: analytic sen kummer}.
  Indeed, given $f=\sum_{n\geq 0} a_nX^{[n]} \in  F\{X\}_\pd$ (where $X=X_1$ for brevity), we have
\[ \phi(f) = \sum_{n\geq 0} (-na_n-\frac{a_{n+1}}{a}) X^{[n]}  \]
It is easy to see $\phi(f)=0$ if only and only $f\in F$.
Given $b=\sum_{n\geq 0} b_nX^{[n]} \in  F\{X\}_\pd$, in order for $\phi(f)=b$, it can be inductively solved that $a_{n+1}=a(-b_n-na_n)$ (which converges to zero).
\end{proof}

\begin{lemma}\label{lem a small solution}
Let $(M, \phi) \in  \mic^{\wedge, a}(F[[E]]/E^m)$;  let $$\varepsilon: M\hatotimes_{p_0, S^0} S^1 \simeq M\hatotimes_{p_1, S^0} S^1$$
be the corresponding stratification from Prop \ref{prop:point abs crystal inf rank}.
Also consider $(S^1, \phi_{S^1})$ from above.
\begin{enumerate}
\item  
We have the following commutative diagram
\[
\begin{tikzcd}
{M\hatotimes_{p_0, S^0} S^1} \arrow[d, "\phi\otimes 1+1\otimes \phi"] \arrow[rr, "\varepsilon"] &  & {M\hatotimes_{p_1, S^0} S^1} \arrow[d, "-1\otimes \phi"] \\
{M\hatotimes_{p_0, S^0} S^1} \arrow[rr, "\varepsilon"]                                          &  & {M\hatotimes_{p_1, S^0} S^1}.                           
\end{tikzcd}
\]

\item Consider the map 
$ M\hatotimes_{p_0, S^0} S^1 \onto M$ 
induced by the degeneracy map $S^1 \to S^0$ (that is, modulo $X_1$ map). It induces a bijection
\[ (M\hatotimes_{p_0, S^0} S^1 )^{\phi\otimes 1+1\otimes \phi=0} \simeq M.\] 
The inverse map is given by
\begin{equation} \label{eq var inver axiom}
 x\in M \mapsto \varepsilon^{-1}(x)=(1+aX)^{-\phi}(x).
\end{equation} 

\item There is a short exact sequence:
\[ 0 \to M \xrightarrow{\varepsilon^{-1}} M\hatotimes_{p_0, S^0} S^1 \xrightarrow{\phi\otimes 1+1\otimes \phi} M \hatotimes_{p_0, S^0} S^1 \to 0. \]
In other words, the $\phi$-cohomology of  $M\hatotimes_{p_0, S^0} S^1$ is quasi-isomorphic to $M[0]$ via $\varepsilon^{-1}$.
\end{enumerate} 
\end{lemma}
\begin{proof} 
(1). Note $\varepsilon$ is $S^1$-linear, so it suffices to prove for $x \in M$ 
\[ -(1\otimes \phi) (1+aX_1)^\phi(x) =(1+aX_1)^\phi(\phi(x)).\]
\footnote{One can also ``formally" see this: note 
$1\otimes \phi$ only derives on $X_1$-part, and note $(-1\otimes \phi)(1+aX_1)=1+aX_1$; thus by chain rule, one sees
 \[ -(1\otimes \phi) (1+aX_1)^{\phi \otimes 1}= (1+aX_1)^{\phi-1} \cdot (\phi \otimes 1) \cdot (1+aX_1)=(1+aX_1)^{\phi} \cdot (\phi\otimes 1).  \]
} That is, we need to check
\[  -\sum_{n \geq 0} \prod_{i=0}^{n-1}(a\phi-ai)(x) X^{[n-1]}(- X_1-\frac{1}{a}) = \sum_{n \geq 0} \prod_{i=0}^{n-1}(a\phi-ai)(\phi(x))X^{[n]}.  \]
Consider the coefficients of $X^{[n]}$, they are
\[  n\prod_{i=0}^{n-1}(a\phi-ai)(x)  +\prod_{i=0}^{n}(a\phi-ai)(x)\cdot \frac{1}{a} =  \prod_{i=0}^{n-1}(a\phi-ai)(\phi(x))\]
whence we conclude. 

Items (2) and (3) easily follow since the right hand side of the commutative diagram induces a short exact sequence
\[ 0 \to M  \xrightarrow{\mathrm{id}\otimes 1} M\hatotimes_{p_1, S^0} S^1  \xrightarrow{-1\otimes \phi }   M\hatotimes_{p_1, S^0} S^1  \to 0\]
which comes from by completed tensoring  \eqref{eq short exact Sone} with $M$.
\end{proof}

  \newpage 
\section{Local absolute $\mathbbl{\Delta}_\dR^+$-crystals} \label{sec: loc abs pris}

Let $R$ be a small smooth (resp. semi-stable) $\ok$-algebra. 
The goal of this section is to classify   absolute $\prism_{\mathrm{dR}}^+$-crystals by certain small enhanced connections: that is, we have an equivalence of categories:
\[\Vect((R)_{\Prism,\ast},\Prism_{\dR,m}^+) \simeq    \MIC_\en^a(X_\et, \o_{X, \gs_\dR^+, m}).\]
Cohomology comparisons will be proved in  \S \ref{sec: coho abs local} (to contain the length of this section).

   For   convenience of comparison with the relative case treated in \S \ref{sec: rel pris smooth}, the writing style in this section is exactly the same as that in \S \ref{sec: rel pris smooth}. That is, we first explicitly study relevant prismatic rings, which are used to explicitly compute stratifications. We then introduce the notion of enhanced connections, which are then related with stratifications hence prismatic crystals.
 
\subsection{Structure of prismatic rings}

\begin{notation} \label{nota: bk prism abs pris}
    Let $\ast \in \{\emptyset, \log \}$. Let $R$ be an $\ok$-algebra.
    \begin{enumerate}
        \item When $\ast=\emptyset$, suppose $R$ is a small smooth $\ok=\gs/E$-algebra. Let 
        $$(\frakS(R),(E)) \in (R)_\prism$$
        be the prism constructed in Construction \ref{constr: cover rel pris} (with $(A, I)=(\gs, (E))$); call it the relative Breuil--Kisin prism.

        \item When $\ast=\log$, suppose $R$ is small semi-stable $\ok$-algebra. Let 
        \[ (\gs(R)_\log, (E), \log) \in (R)_{\pris, \log} \] be the relative log Breuil--Kisin prism in Construction  \ref{constr: cover of final obj rel log pris}.
            \end{enumerate}
            Now we can uniformly denote the above notations as
            \[(\gs(R)_\ast, (E), \ast) \in (R)_{\pris, \ast}. \]
            Note $(\gs(R)_\ast, (E), \ast)$ is a cover of the final object of $\Sh((R)_{\Prism, \ast})$, by \cite[Lem. 2.1]{MW22} resp. Lemma \ref{lem:log covering for BK}.   Denote by $(\frakS(R)_{\ast}^{\bullet},(E),\ast)$ the \v Cech nerve associated to this covering.
\end{notation}

The structure of $(\frakS(R)^{\bullet},(E))$ is already computed in \cite{MW22}: caution here the cosimplicial object is constructed in the \emph{absolute} prismatic site, unlike relative case in \S \ref{sec: rel pris smooth}.
We now compute the log version.

  \begin{lem}\label{lem:explicit description of cech nerve-abs}
    For any $n\geq 0$, we have
    \[\frakS(R)^{n}_{\log} = \frakS(R)^{\otimes(n+1)}[[1-\frac{u_i}{u_0},1-\frac{\underline T_i}{\underline T_0}\mid 1\leq i\leq n]]\{\frac{1-\frac{u_i}{u_0}}{E(u_0)},\frac{1-\frac{\underline T_i}{\underline T_0}}{E(u_0)}\mid 1\leq i\leq n\}^{\wedge}_{\delta},\]
    where $\frakS(R)^{\otimes(n+1)}$ denotes the tensor products of $(n+1)$ copies of $\frakS(R)$ over $\rW(k)$, $\underline T_i$ denotes $T_{1,i},\dots,T_{d,i}$ which is induced by the chart on $R$ while $u_i$ denotes the analogue of $u$ on the $(i+1)$-th factor of $\frakS(R)^{\otimes(n+1)}$.
 The log-structure $M_{\frakS(R)^{\bullet}_{\log}}$ on $\frakS(R)^{n}_{\log}$ is induced by the composite
    \[M_{\frakS(R)}\to\frakS(R)\xrightarrow{q_i}\frakS(R)^{\otimes(n+1)}[[1-\frac{u_i}{u_0},1-\frac{\underline T_i}{\underline T_0}\mid 1\leq i\leq n]]\{\frac{1-\frac{u_i}{u_0}}{E(u_0)},\frac{1-\frac{\underline T_i}{\underline T_0}}{E(u_0)}\mid 1\leq i\leq n\}^{\wedge}_{\delta},\]
    which is independent of $q_i$.
  \end{lem}
  \begin{proof}
      By the proof of Lemma \ref{lem:log covering for BK}, we see that $\frakS(R)^{\otimes(\bullet+1)}[[1-\frac{u_i}{u_0},1-\frac{\underline T_i}{\underline T_0}\mid 1\leq i\leq \bullet]]\{\frac{1-\frac{u_i}{u_0}}{E(u_0)},\frac{1-\frac{\underline T_i}{\underline T_0}}{E(u_0)}\mid 1\leq i\leq \bullet\}^{\wedge}_{\delta}$ with the log-structure $M_{\bullet}$ induced from
      \[M_{\frakS(R)}\to\frakS(R)\to\frakS(R)^{\otimes(\bullet+1)}[[1-\frac{u_i}{u_0},1-\frac{\underline T_i}{\underline T_0}\mid 1\leq i\leq \bullet]]\{\frac{1-\frac{u_i}{u_0}}{E(u_0)},\frac{1-\frac{\underline T_i}{\underline T_0}}{E(u_0)}\mid 1\leq i\leq \bullet\}^{\wedge}_{\delta}\]
      is a well-defined cosimplicial log-prism in $(R)_{\Prism,\log}$. It suffices to show this is exactly $(\frakS(R)^{\bullet}_{\log},(E),M_{\frakS(R)^{\bullet}_{\log}})$.
      We only deal with the $n=1$ case while the general case follows from the same (but more tedious) argument.

      Now, our aim is to show that
      \[(\frakS(R)^{\otimes2}[[1-\frac{u_1}{u_0},1-\frac{\underline T_1}{\underline T_0}]]\{\frac{1-\frac{u_1}{u_0}}{E(u_0)},\frac{1-\frac{\underline T_1}{\underline T_0}}{E(u_0)}\}^{\wedge}_{\delta},(E),M_1)\]
      is the self-coproduct of $(\frakS,(E),M_{\frakS(R)})$ in $(R)_{\Prism,\log}$. Noting that
      \[\frac{u_1}{u_0} = \frac{T_{0,1}}{T_{0,0}}\cdots \frac{T_{r,1}}{T_{r,0}},\]
      we have the identifications
      \[\frakS(R)^{\otimes2}[[1-\frac{u_1}{u_0},1-\frac{\underline T_1}{\underline T_0}]] = \frakS(R)^{\otimes2}[[1-\frac{T_{j,1}}{T_{j,0}}\mid 0\leq j\leq d]]\]
      and
      \[\frakS(R)^{\otimes2}[[1-\frac{u_1}{u_0},1-\frac{\underline T_1}{\underline T_0}]]\{\frac{1-\frac{u_1}{u_0}}{E(u_0)},\frac{1-\frac{\underline T_1}{\underline T_0}}{E(u_0)}\}^{\wedge}_{\delta} = \frakS(R)^{\otimes2}[[1-\frac{T_{j,1}}{T_{j,0}}\mid 0\leq j\leq d]]\{\frac{1-\frac{T_{j,1}}{T_{j,0}}}{E(u_0)}\mid 0\leq j\leq d\}^{\wedge}_{\delta}.\]
      Thus, we only need to show that
      \[(\frakS(R)^{\otimes2}[[1-\frac{T_{j,1}}{T_{j,0}}\mid 0\leq j\leq d]]\{\frac{1-\frac{T_{j,1}}{T_{j,0}}}{E(u_0)}\mid 0\leq j\leq d\}^{\wedge}_{\delta},(E),M_1) = (\frakS(R)^1_{\log},(E),M_{\frakS(R)_{\log}^1}).\]

      Note that there is an obvious pre-log prism
      \[(\frakS(R)^{\otimes 2},(E),(\bigoplus_{i=0}^r\bN\cdot e_{i,0})\oplus(\bigoplus_{i=0}^r\bN\cdot e_{i,1}))\]
      together with a ``multiplication'' map of prelog rings
      \[\pr:((\bigoplus_{i=0}^r\bN\cdot e_{i,0})\oplus(\bigoplus_{i=0}^r\bN\cdot e_{i,1})\to\frakS(R)^{\otimes 2})\to (\bigoplus_{i=0}^r\bN\cdot e_i\to R)\]
      whose induced map on monoids
      \[(\bigoplus_{i=0}^r\bN\cdot e_{i,0})\oplus (\bigoplus_{i=0}^r\bN\cdot e_{i,1})\to \bigoplus_{i=0}^r\bN\cdot e_i\]
      is given by carrying each $e_{i,j}$ with $j\in \{0,1\}$ to $e_i$. By \cite[Construction 2.18]{Kos21}, the $(p,E)$-complete exactification of $\pr$ above is given by 
      \[(\bigoplus_{i=0}^r\bN\cdot e_i\xrightarrow{e_i\mapsto T_{i,0},~\forall~i}\left(\frakS(R)^{\otimes 2}[(\frac{T_{0,1}}{T_{0,0}})^{\pm 1},\dots,(\frac{T_{r,1}}{T_{r,0}})^{\pm 1}]\right)^{\wedge_{(p,E)}})\to (\bN^{r+1}\to R)\]
      which extends $\pr$ above and sends each $\frac{T_{i,1}}{T_{i,0}}$ to $1$. This gives a pre-log prism 
      \[(\left(\frakS(R)^{\otimes 2}[(\frac{T_{0,1}}{T_{0,0}})^{\pm 1},\dots,(\frac{T_{r,1}}{T_{r,0}})^{\pm 1}]\right)^{\wedge_{(p,E)}},(E),\bN^{r+1})\]
      in which $\frac{T_{0,1}}{T_{0,0}}-1,\dots,\frac{T_{d,1}}{T_{d,0}}-1$ form a $(p,E)$-completely regular suquence relative to $\frakS(R)$. By \cite[Prop. 3.19]{BS22}, the 
      \[B:=(\left(\frakS(R)^{\otimes_{\frakS}2}[(\frac{T_{0,1}}{T_{0,0}})^{\pm 1},\dots,(\frac{T_{r,1}}{T_{r,0}})^{\pm 1}]\right)^{\wedge_{(p,E)}}\{\frac{1-\frac{T_{j,1}}{T_{j,0}}}{E(u_0)}\mid 0\leq j\leq d\}^{\wedge_{(p,E)}}_{\delta}\]
      exists and $(p,E)$-completely faithfully flat over $\frakS(R)$. Endow $B$ with the log-strutcure $M_B$ induced by 
      \[\bigoplus_{i=0}^r\bN\cdot e_i\xrightarrow{e_i\mapsto T_i,\forall~i}B\]
      and then we get a log-prism $(B,(E),M_B)\in (R)_{\Prism,\log}$.

      The same argument in the proof of Lemma \ref{lem:explicit description of cech nerve-rel} implies that 
      \[(B,(E),M_B) = (\frakS(R)_{\log}^1,(E),M_{\frakS(R)^1_{\log}}).\]
      To conclude, we are reduced to show that
      \[\left(\frakS(R)^{\otimes_{\frakS}2}[(\frac{T_{0,1}}{T_{0,0}})^{\pm 1},\dots,(\frac{T_{r,1}}{T_{r,0}})^{\pm 1}]\right)^{\wedge_{(p,E)}}\{\frac{1-\frac{T_{j,1}}{T_{j,0}}}{E(u_0)}\mid 0\leq j\leq d\}^{\wedge_{(p,E)}}_{\delta} \cong (\frakS(R)^{\otimes2}[[1-\frac{T_{j,1}}{T_{j,0}}\mid 0\leq j\leq d]]\{\frac{1-\frac{T_{j,1}}{T_{j,0}}}{E(u_0)}\mid 0\leq j\leq d\}^{\wedge}_{\delta}.\]
      However, this can be deduced from a similar argument used in the last paragraph of the proof of Lemma \ref{lem:explicit description of cech nerve-rel}.
  \end{proof}

  \begin{notation} \label{nota: many dr rings abs and rel}
  We summarize some of the prisms and ``de Rham"   rings used in this section.
  \begin{itemize}
  \item   Let $(\gs^\bullet_\ast, (E), \ast)$ be the \v Cech nerve of $(\gs, (E), \ast) \in \okprisast$.
  \item In Notation \ref{nota: bk prism abs pris}, we have \v Cech nerve $(\frakS(R)_\ast^{\bullet},(E), \ast)$ on the \emph{absolute} (log-) prismatic site $(R)_{\pris, \ast}$.
  \item Since $(\frakS(R)_\ast,(E), \ast)$ is also a cover on the \emph{relative} site $(R/(\frakS,(E)))_{\Prism,\ast}$, we denote the \v Cech nerve, constructed in Construction \ref{const: cosimp BK rel pris}  resp. \ref{cons: cech nerve rel pris log}, by
$(\frakS(R)^{\bullet}_{\ast,\geo},(E),\ast)$.
  \end{itemize}
  We have morphism of co-simplicial (log-) prisms:
  \[ (\gs^\bullet_\ast, (E), \ast) \to (\gs(R)^\bullet_\ast, (E), \ast) \to (\frakS(R)^{\bullet}_{\ast,\geo},(E),\ast).\]
  Evaluation  of $\prism^+_{\dR, m}$ induces morphism of cosimplicial rings:
    \[\frakS_{*,\dR,m}^{\bullet,+} \to \frakS(R)_{\ast,\dR,m}^{\bullet,+} \to  \frakS(R)^{\bullet,+}_{\ast,\geo,\dR,m}. \]
By \cite[Cor. 3.6]{GMWdR}, they are $K$-cosimplicial algebras. Furthermore, in degree zero, we have 
\begin{itemize}
\item  
$\frakS_{\dR, m}^+= \frakS_{\log,\dR,m}^+$, cf. \cite[Construction 3.1]{GMWdR};
\item  $ \frakS(R)^+_{\ast,\dR,m}=\frakS(R)^+_{\ast,\geo,\dR,m} =R_\gsdrm =R\hatotimes_{\calO_K} \gsdrm$.
Indeed, to verify this, it suffices to consider the case where $\Box$ in Cons \ref{constr: cover of final obj rel log pris} is identity; it is then easy to verify using the explicit $\frakS^{\Box}$.
\end{itemize} 
  \end{notation}

 Similar to Proposition \ref{Prop-Structure-Rel}, we have the following result. In below, recall as Notation \ref{nota: a sec 6}, \[ a=
\begin{cases}
  -E'(\pi), &  \text{if } \ast=\emptyset \\
 -\pi E'(\pi), &  \text{if } \ast=\log.
\end{cases} \]

 \begin{prop}\label{Prop-Structure-Abs}
  For any $1\leq i\leq n$, put $X_i: = \frac{E(u_0)-E(u_i)}{-aE(u_0)}$ and for any $1\leq j\leq d$, put $Y_{s,j} = \frac{T_{0,j}-T_{s,j}}{E(u_0)T_{0,j}}$. Then identifying $\frakS(R)$ with the first component of $\frakS(R)^n_{*}$ (i.e., identifying $u,T_1,\dots,T_d\in\frakS(R)$ with $u_0,T_{1,0},\dots,T_{d,0}\in\frakS(R)^n_{*}$) induces an isomorphism 
   \[(\frakS(R)/(E(u)^m))\{X_1,\dots,X_n,\underline Y_1,\dots,\underline Y_n\}^{\wedge}_{\pd}[\frac{1}{p}]\xrightarrow{\cong}\frakS(R)^{n,+}_{*,\dR,m}.\]
   Via these isomorphisms for all $n$, the face and degeneracy maps of $\frakS(R)_{*,\dR,m}^{\bullet,+}$ are given by 
   \begin{equation}\label{Equ-Face-A}
        \begin{split}
            &p_i(X_j) = \left\{\begin{array}{rcl}  
                (X_{j+1}-X_{1})(1+aX_{1})^{-1}, & i=0 \\
                X_{j}, & j<i\\
                X_{j+1}, & 0<i\leq j;
            \end{array}\right.\\
            &p_i(Y_{s,j}) = \left\{\begin{array}{rcl}  
                (Y_{s,j+1}-Y_{s,1})(1+aX_1)^{-1}(1-E(u) Y_{s,1})^{-1}, & i=0 \\
                Y_{s,j}, & j<i\\
                Y_{s,j+1}, & 0<i\leq j;
            \end{array}\right.\\
            &p_i(\underline T) = \left\{\begin{array}{rcl}  
                \underline T_1 = \underline T(1-E(u) \underline Y_1), & i=0 \\
                \underline T, &i > 0;
            \end{array}\right.\\
            &p_i(E(u)) = \left\{\begin{array}{rcl}  
                E(u_1) = E(u)(1+a X_1), & i=0 \\
                E(u), &i > 0;
            \end{array}\right.
        \end{split}
    \end{equation}
    and 
    \begin{equation}\label{Equ-Degeneracy-A}
        \begin{split}
            &\sigma_i(X_j)=\left\{\begin{array}{rcl}
               0, &  i=0,j=1\\
               X_j, & j\leq i\\
               X_{j-1}, & j>i;
            \end{array}\right.\\
            &\sigma_i(Y_{s,j})=\left\{\begin{array}{rcl}
               0, &  i=0,j=1\\
               Y_{s,j}, & j\leq i\\
               Y_{s,j-1}, & j>i;
            \end{array}\right.\\
            &\sigma_i(\underline T) = \underline T,\\
            &\sigma_i(E(u)) = E(u),
        \end{split}
    \end{equation}
    for any $1\leq s\leq d$, respectively.
 \end{prop}
 \begin{proof}
   The well-definedness of the desired map follows from \cite[Lem. 3.11(1)]{GMWdR} and the proof therein.
   Similar to the proof for Proposition \ref{Prop-Structure-Rel}, by derived Nakayama Lemma, we are reduced to the case for $m = 1$. Then the desired isomorphism follows from the same argument for \cite[Lem. 4.13(2)]{MW22} for both $* = \emptyset$ and $* = \log$. Finally, the formulae (\ref{Equ-Face-A}) and (\ref{Equ-Degeneracy-A}) follow from the constructions of $X_i$'s and $Y_{j,k}$'s immediately.
 \end{proof}
 \begin{cor}\label{Cor-Structure-A-SmoothVSemistable}
   Assume $R$ is small smooth over $\calO_K$ and view it as small semi-stable by consider the log structure induced by $\bN\to\calO_K\to R$. Then for any $n\geq 1$, the natural morphism $\frakS(R)_{\dR,m}^{n,+}\to \frakS(R)_{\log,\dR,m}^{n,+}$ is determined by identifying $u,\underline T$ and $\underline Y_{i}$'s, and sending $X_j$'s to $\pi X_j$'s via isomorphisms given in Proposition \ref{Prop-Structure-Abs}.
 \end{cor}
 \begin{proof}
   This is easy as $X_j$ represents $\frac{E(u_0)-E(u_j)}{E'(\pi)E(u_0)}$ in $\frakS(R)_{\dR,m}^n$ and represents $\frac{E(u_0)-E(u_j)}{\pi E'(\pi)E(u_0)}$ in $\frakS(R)^n_{\dR,\log,m}$.
 \end{proof}
\subsection{Analysis of stratifications}
 By \cite[Prop. 2.7]{BS23}, the evaluation at $(\frakS(R)_\ast^{\bullet},(E), \ast)$  induces an equivalence 
   \[\Vect((R)_{\Prism,*},\prism^+_{\dR, m})\to{\rm Strat}(\frakS(R)_{*,\dR,m}^{\bullet,+}).\]
  In this subsection, we study stratifications satisfying the cocycle condition. 
 Let $(M,\varepsilon)$ be a stratification with respect to $\frakS(R)_{*,\dR,m}^{\bullet,+}$.
 We have to determine when $(M,\varepsilon)$ satisfies the cocycle condition. 
 By virtues of Proposition \ref{Prop-Structure-Abs}, for any $x\in M$, we may write
 \begin{equation}\label{Equ-Strat-Abs}
     \varepsilon(x) = \sum_{(i,\underline n)\in\bN\times\bN^d}\nabla_{i,\underline n}(x) X_1^{[i]}\underline Y_1^{[\underline n]}
 \end{equation}
 with $\rW(k)$-linear endomorphisms $\nabla_{i,\underline n}$ of $M$ satisfying 
 \begin{equation}\label{Equ-StratNil-Abs}
     \lim_{i+|\underline n|\to +\infty}\nabla_{i,\underline n} = 0.
 \end{equation} 
 By (\ref{Equ-Face-A}) and (\ref{Equ-Degeneracy-A}), it is easy to see that
 \begin{equation}\label{Equ-P2P0-Abs}
     \begin{split}
         & p_2^*(\varepsilon)\circ p_0^*(\varepsilon)(x)\\
       = & p_2^*(\varepsilon)(\sum_{(j,\underline n)\in\bN\times\bN^d}\nabla_{j,\underline n}(x)(1+aX_1)^{-j-|\underline n|}(1-E(u)\underline Y_1)^{-\underline n}(X_2-X_1)^{[j]}(\underline Y_2-\underline Y_1)^{[\underline n]})\\
       = & \sum_{(i,\underline l),(j,\underline n)\in\bN\times\bN^d}\nabla_{i,\underline l}(\nabla_{j,\underline n}(x))(1+aX_1)^{-j-|\underline n|}(1-E(u)\underline Y_1)^{-\underline n}X_1^{[i]}(X_2-X_1)^{[j]}\underline Y_1^{[\underline l]}(\underline Y_2-\underline Y_1)^{[\underline n]}\\
       = & \sum_{(i,\underline l),(j,\underline m),(k,\underline n)\in\bN\times\bN^d}\nabla_{i,\underline l}(\nabla_{j+k,\underline m+\underline n}(x))(1+aX_1)^{-j-k-|\underline m|-|\underline n|}(1-E(u)\underline Y_1)^{-\underline m-\underline n}\\
       & \cdot (-1)^{j+|\underline m|}X_1^{[i]}X_1^{[j]}X_2^{[k]}\underline Y_1^{[\underline l]}\underline Y_1^{[\underline m]}\underline Y_2^{[\underline n]}\\
       = & \sum_{(i,\underline l),(j,\underline m),(k,\underline n)\in\bN\times\bN^d}\nabla_{i,\underline l}(\nabla_{j+k,\underline m+\underline n}(x))(1+aX_1)^{-j-k-|\underline m|-|\underline n|}(1-E(u)\underline Y_1)^{-\underline m-\underline n}\\
       & \cdot (-1)^{j+|\underline m|}\binom{i+j}{i}\binom{\underline l+\underline m}{\underline l}X_1^{[i+j]}\underline Y_1^{[\underline l+\underline m]}X_2^{[k]}\underline Y_2^{[\underline n]},
     \end{split}
 \end{equation}
 that
 \begin{equation}\label{Equ-P1-Abs}
     \begin{split}
         p_1^*(\varepsilon)(x)=\sum_{(k,\underline n)\in\bN\times\bN^d}\nabla_{k,\underline n}(x)X_2^{[k]}\underline Y_2^{[\underline n]},
     \end{split}
 \end{equation}
 and that 
 \begin{equation}\label{Equ-Sigma0-Abs}
     \sigma_0^*(\varepsilon)(x) = \nabla^{0,\underline 0}_M(x).
 \end{equation}
 Therefore, we deduce that $(M,\varepsilon)$ satisfies the cocycle condition if and only if $\nabla_{0,\underline 0} = \id_M$ and for any $(k,\underline n)\in \bN\times\bN^d$, 
 \begin{equation}\label{Equ-StratCondition-Abs}
    \begin{split}
        \nabla_{k,\underline n}(x) = &\sum_{(i,\underline l),(j,\underline m)\in\bN\times\bN^d}\nabla_{i,\underline l}(\nabla_{j+k,\underline m+\underline n}(x))\\
        &\cdot(1+aX)^{-j-k-|\underline m|-|\underline n|}(1-E(u)\underline Y)^{-\underline m-\underline n} (-1)^{j+|\underline m|}\binom{i+j}{i}\binom{\underline l+\underline m}{\underline l}X^{[i+j]}\underline Y^{[\underline l+\underline m]}.
    \end{split}
 \end{equation}
 
 \begin{notation}\label{Convention-E-connection}
     Let $\phi_{M}:=a^{-1}\nabla_{1,\underline 0}$ and for any $1\leq i\leq d$, let $\nabla_i:=\nabla_{0,\underline 1_i}$. Define 
     \[\nabla_{M}:=\sum_{i=1}^d\nabla_i\otimes\frac{\dlog T_i}{E(u)}:M\to M\otimes_{A(R)_{\dR,m}^+}\widehat \Omega^1_{A(R)_{\dR,m}^+/A_{\dR,m}^+}\{-1\}.\]
 \end{notation}

 \begin{lem}\label{Lem-Technique-Abs}
     Keep notations as above and assume $\nabla_{0,\underline 0} = \id_M$. Then the following are equivalent.
     \begin{enumerate}
         \item The formula \eqref{Equ-StratCondition-Abs} holds true for all $(i,\underline n)$. Equivalently, $(M,\varepsilon)$ satisfies the cocycle condition.
         
         \item The maps $\nabla_i$'s commute with each other and $[\phi_M,\nabla_i] = \nabla_i$ 
                 such that for any $(i,\underline n)\in\bN\times\bN^d$,
         \begin{equation}\label{Equ-Iteration-Abs}
         \begin{split}
             \nabla_{i,\underline n}(x) & = \prod_{j=0}^{i-1}(a\phi_M-a(j+|\underline n|))\prod_{k=1}^d\prod_{l=0}^{n_k-1}(\nabla_k+lE(u))(x)\\
             & = \prod_{k=1}^d\prod_{l=0}^{n_k-1}(\nabla_k+lE(u))\prod_{j=0}^{i-1}(a\phi_M-aj)(x),
         \end{split}
         \end{equation}
         which tends to zero as $i+|\underline n|\to+\infty$.
     \end{enumerate}
     Moreover, if the above equivalent conditions are satisfied, then for any $1\leq i\leq d$, $\nabla_i$ is nilpotent and the stratification $(M,\varepsilon)$ satisfies that for any $x\in M$, 
    \begin{equation}\label{Equ-Stratification-Abs-I}
       \begin{split}
           \varepsilon(x) & = (1-E(u)\underline Y_1)^{-E(u)^{-1}\nabla_{M}}(1+aX_1)^{\phi_M}(x)\\
           &:= \sum_{i\geq 0,\underline n\in \bN^d} \prod_{s=1}^d\prod_{k=0}^{n_s-1}(\nabla_s+kE(u))\underline Y_1^{[\underline n]}\cdot\prod_{r=0}^{i-1}(a\phi_M-ar)(x)X_1^{[i]}\\
           & = (1+aX_1)^{\phi_M}(1-(1+aX_1)^{-1}E(u)\underline Y_1)^{-E(u)^{-1}\nabla_{M}}(x)\\
           &:= \sum_{i\geq 0,\underline n\in\bN^d}\prod_{r=0}^{i-1}(a\phi_M-a(i+|\underline n|))X_1^{[i]}\prod_{s=1}^d\prod_{k=0}^{n_s-1}(\nabla_s+kE(u))(x)\underline Y_1^{[\underline n]}.
       \end{split}
   \end{equation}
 \end{lem}
 \begin{proof}
  
     The proof is similar to that of Proposition \ref{Prop-Structure-Rel}. Applying $\partial_{X}$ to the right hand side of (\ref{Equ-StratCondition-Abs}), we get
     \begin{equation}\label{Equ-PartialX-I}
         \begin{split}
             (\dagger):= & \partial_X(\sum_{(i,\underline l),(j,\underline m)\in\bN\times\bN^d}\nabla_{i,\underline l}(\nabla_{j+k,\underline m+\underline n}(x))\\
             & \quad \cdot (1+aX)^{-j-k-|\underline m|-|\underline n|}(1-E(u)\underline Y)^{-\underline m-\underline n} (-1)^{j+|\underline m|}\binom{i+j}{i}\binom{\underline l+\underline m}{\underline l}X^{[i+j]}\underline Y^{[\underline l+\underline m]}\\
           = & \sum_{(i,\underline l),(j,\underline m)\in\bN\times\bN^d}-a(j+k+|\underline m|+|\underline n|)\nabla_{i,\underline l}(\nabla_{j+k,\underline m+\underline n}(x))\\
             &\quad  \cdot (1+aX)^{-j-k-|\underline m|-|\underline n|-1}(1-E(u)\underline Y)^{-\underline m-\underline n} (-1)^{j+|\underline m|}\binom{i+j}{i}\binom{\underline l+\underline m}{\underline l}X^{[i+j]}\underline Y^{[\underline l+\underline m]}\\
             & + \sum_{(i,\underline l),(j,\underline m)\in\bN\times\bN^d}\nabla_{i,\underline l}(\nabla_{j+k,\underline m+\underline n}(x))\\
             & \quad \cdot (1+aX)^{-j-k-|\underline m|-|\underline n|}(1-E(u)\underline Y)^{-\underline m-\underline n} (-1)^{j+|\underline m|}\binom{i+j}{i}\binom{\underline l+\underline m}{\underline l}X^{[i+j-1]}\underline Y^{[\underline l+\underline m]}.
         \end{split}
     \end{equation}
     
     Now, we assume (1) is true. Then we have $(\dagger) = 0$. By letting $X = 0$ in (\ref{Equ-PartialX-I}), we get that 
     \begin{equation*}
         \begin{split}
             0 =& \sum_{\underline l,\underline m\in\bN^d}-a(k+|\underline m|+|\underline n|)\nabla_{0,\underline l}(\nabla_{k,\underline m+\underline n}(x))(1-E(u)\underline Y)^{-\underline m-\underline n} (-1)^{|\underline m|}\binom{\underline l+\underline m}{\underline l}\underline Y^{[\underline l+\underline m]}\\
             & + \sum_{\underline l,\underline m\in\bN^d}\nabla_{1,\underline l}(\nabla_{k,\underline m+\underline n}(x))(1-E(u)\underline Y)^{-\underline m-\underline n} (-1)^{|\underline m|}\binom{\underline l+\underline m}{\underline l}\underline Y^{[\underline l+\underline m]}\\
             & + \sum_{\underline l,\underline m\in\bN^d}\nabla_{0,\underline l}(\nabla_{1+k,\underline m+\underline n}(x))(1-E(u)\underline Y)^{-\underline m-\underline n} (-1)^{1+|\underline m|}\binom{\underline l+\underline m}{\underline l}\underline Y^{[\underline l+\underline m]}.
         \end{split}
     \end{equation*}
     By letting $\underline Y = 0$ in above formula, we deduce that
     \begin{equation*}
         -a(k+|n|)(\nabla_{k,\underline n}(x))+\nabla_{1,\underline 0}(\nabla_{k,\underline n}(x))+\nabla_{1+k,\underline n}(x) = 0.
     \end{equation*}
     In other words, for any $(k,\underline n)\in\bN\times\bN^d$, we get 
     \begin{equation}\label{Equ-Iteration-Abs-I}
         \nabla_{1+k,\underline n}(x) = (a\phi_M-a(k+|n|))(\nabla_{k,\underline n}(x)).
     \end{equation}
     By iteration, we conclude that for any $(k,\underline n)\in\bN\times\bN^d$,
     \begin{equation}\label{Equ-Iteration-Abs-II}
         \nabla_{k,\underline n}(x) = \prod_{i=0}^{k-1}(a\phi_M-a(i+|n|))(\nabla_{0,\underline n}(x)).
     \end{equation}
     Note that by letting $X = 0$ in (\ref{Equ-StratCondition-Abs}) for $k = 0$, we obtain that
     \[\nabla_{0,\underline n}(x) = \sum_{\underline l,\underline m\bN^d}\nabla_{0,\underline l}(\nabla_{0,\underline m+\underline n}(x))(1-E(u)\underline Y)^{-\underline m-\underline n}\binom{\underline l+\underline m}{\underline l}\underline Y^{[\underline l+\underline m]}.\]
     So by Lemma \ref{Lem-Technique-Rel} (for $\beta = -E(u)$), we deduce that $\nabla_i$'s commutes with each other and that for any $\underline n\in\bN^d$, 
     \begin{equation}\label{Equ-Iteration-Abs-III}
         \nabla_{0,\underline n}(x) = \prod_{s=0}^{d}\prod_{j=0}^{n_s-1}(\nabla_s+jE(u))(x).
     \end{equation}
     Combining (\ref{Equ-Iteration-Abs-II}) with (\ref{Equ-Iteration-Abs-III}) together, we conclude that for any $(k,\underline n)\in\bN\times\bN^d$,
     \begin{equation}\label{Equ-Iteration-Abs-IV}
         \nabla_{k,\underline n}(x) = \prod_{i=0}^{k-1}(a\phi_M-a(i+|n|))\prod_{s=0}^{d}\prod_{j=0}^{n_s-1}(\nabla_s+jE(u))(x).
     \end{equation}
     So the first part of (\ref{Equ-Iteration-Abs}) holds true. In particular, by letting $\underline n = \underline 0$ in (\ref{Equ-Iteration-Abs-I}) and (\ref{Equ-Iteration-Abs-IV}), we get that for any $k\geq 0$,
     \begin{equation}\label{Equ-Iteration-Abs-V}
         \nabla_{k,\underline 0}(x) = (a\phi_M-ka)(\nabla_{k-1,\underline 0}(x)) = \prod_{i=0}^{k-1}(a\phi_M-ia)(x).
     \end{equation}
     On the other hand, by applying $\partial_{Y_s}$ to the right hand side of (\ref{Equ-StratCondition-Abs}) and then letting $\underline Y = 0$, we get that 
     \begin{equation*}
     \begin{split}
        0 = &\sum_{i,j\geq 0}n_sE(u)\nabla_{i,\underline 0}(\nabla_M^{j+k,\underline n}(x))(1+aX)^{-j-k-|\underline n|}(-1)^{j}\binom{i+j}{i}X^{[i+j]}\\
        & +\sum_{i,j\geq 0}\nabla_{i,\underline 1_s}(\nabla_{j+k,\underline n}(x))(1+aX)^{-j-k-|\underline n|}(-1)^{j}\binom{i+j}{i}X^{[i+j]}\\
        & +\sum_{i,j\geq 0}\nabla_{i,\underline 0}(\nabla_{j+k,\underline 1_s+\underline n}(x))\cdot(1+aX)^{-j-k-1-|\underline n|}(-1)^{j+1}\binom{i+j}{i}X^{[i+j]}.
     \end{split}
     \end{equation*}
     By letting $X = 0$ in above formula, we obtain that
     \begin{equation*}
         n_sE(u)\nabla_{k,\underline n}(x)+\nabla_{0,\underline 1_s}(\nabla_{k,\underline n}(x)) - \nabla_{k,\underline n+\underline 1_s}(x) = 0.
     \end{equation*}
     In other words, for any $(k,\underline n)\in\bN\times\bN^d$ and any $1\leq s\leq d$, we conclude that
     \begin{equation}\label{Iteration-Abs-VI}
         \nabla_{k,\underline n+\underline 1_s}(x) = (\nabla_s+n_sE(u))\nabla_{k,\underline n}(x).
     \end{equation}
     By iteration and using (\ref{Equ-Iteration-Abs-III}) together with (\ref{Equ-Iteration-Abs-V}), we deduce that for any $(k,\underline n)\in\bN\times\bN^d$,
     \begin{equation}\label{Equ-Iteration-Abs-VII}
     \begin{split}
         \nabla_{k,\underline n}(x) = \prod_{s=1}^d\prod_{j=0}^{n_s-1}(\nabla_s+jE(u))\prod_{i=0}^{k-1}(a\phi_M-ia)(x),
     \end{split}
     \end{equation}
    which implies the second part of (\ref{Equ-Iteration-Abs}).
    Applying (\ref{Equ-Iteration-Abs-IV}) and (\ref{Equ-Iteration-Abs-VII}) to $(k,\underline n) = (1,\underline 1_s)$, we see that 
    \[(a\phi_M-a)\nabla_s = a\nabla_s\phi_M.\]
    Equivalently, we conclude that $[\phi_M,\nabla_i] = \nabla_i$ for any $1\leq i\leq d$ as desired. So Item (2) follows.
    
        Conversely, we assume Item (2) is true. We claim that to conclude Item (1), it is enough to show that $(\dagger) = 0$. Indeed, if this is true, then the right hand side of (\ref{Equ-StratCondition-Abs}) is independent of $X$ and hence equal to its evaluation at $X = 0$; that is, it is equal to
    \[\sum_{\underline l,\underline m\in\bN^d}\nabla_{0,\underline l}(\nabla_{k,\underline m+\underline n}(x))(1-E(u)\underline Y)^{-\underline m-\underline n}(-1)^{|\underline m|}\binom{\underline l+\underline m}{\underline l}\underline Y^{[\underline l+\underline m]}.\] 
    Since $\nabla_{k,\underline m+\underline n}(x) = \nabla_{0,\underline m+\underline n}(\nabla_{k,\underline 0})$ by (\ref{Equ-Iteration-Abs-VII}), we can conclude Item (1) by showing 
    \[\nabla_{0,\underline n}(\nabla_{k,\underline 0}(x)) = \sum_{\underline l,\underline m\in\bN^d}\nabla_{0,\underline l}(\nabla_{0,\underline m+\underline n}(\nabla_{k,\underline 0}(x)))(1-E(u)\underline Y)^{-\underline m-\underline n}(-1)^{|\underline m|}\binom{\underline l+\underline m}{\underline l}\underline Y^{[\underline l+\underline m]}.\]
    But this can be deduced from Lemma \ref{Lem-Technique-Rel} immediately.
    
    Now, we are going to prove that $(\dagger) = 0$. Using (\ref{Equ-Iteration-Abs-I}), we may replace $-a(j+k+|\underline m|+|\underline n|)\nabla_{j+k,\underline m+\underline n}(x)$ by $\nabla_{1+j+k,\underline m+\underline n}(x) - \phi_M(\nabla_{j+k,\underline m+\underline n}(x))$ in (\ref{Equ-PartialX-I}) and then deduce that
    \begin{equation*}
        \begin{split}
            (\dagger) = & \sum_{(i,\underline l),(j,\underline m)\in\bN\times\bN^d}\nabla_{i,\underline l}(\nabla_{1+j+k,\underline m+\underline n}(x))\\
            &\quad  \cdot (1+aX)^{-j-k-|\underline m|-|\underline n|-1}(1-E(u)\underline Y)^{-\underline m-\underline n} (-1)^{j+|\underline m|}\binom{i+j}{i}\binom{\underline l+\underline m}{\underline l}X^{[i+j]}\underline Y^{[\underline l+\underline m]}\\
            & - \sum_{(i,\underline l),(j,\underline m)\in\bN\times\bN^d}\nabla_{i,\underline l}(\phi_M(\nabla_{j+k,\underline m+\underline n}(x)))\\
            &\quad  \cdot (1+aX)^{-j-k-|\underline m|-|\underline n|-1}(1-E(u)\underline Y)^{-\underline m-\underline n} (-1)^{j+|\underline m|}\binom{i+j}{i}\binom{\underline l+\underline m}{\underline l}X^{[i+j]}\underline Y^{[\underline l+\underline m]}\\
            & + \sum_{(i,\underline l),(j,\underline m)\in\bN\times\bN^d}\nabla_{i,\underline l}(\nabla_{j+k,\underline m+\underline n}(x))\\
            & \quad \cdot (1+aX)^{-j-k-|\underline m|-|\underline n|}(1-E(u)\underline Y)^{-\underline m-\underline n} (-1)^{j+|\underline m|}\binom{i+j}{i}\binom{\underline l+\underline m}{\underline l}X^{[i+j-1]}\underline Y^{[\underline l+\underline m]}.
        \end{split}
    \end{equation*}
    By (\ref{Equ-Iteration-Abs-V}) and (\ref{Equ-Iteration-Abs-VII}), we see that
    \[\nabla_{i,\underline l}\phi_M = \nabla_{0,\underline l}\nabla_{i,\underline 0}\nabla_{1,\underline 0} =  \nabla_{0,\underline l}\nabla_{i+1,\underline 0}+ia\nabla_{0,\underline l}\nabla_{i,\underline l} = \nabla_{i+1,\underline l}+ia\nabla_{i,\underline l}.\]
    So we deduce that 
    \begin{equation*}
        \begin{split}
            (\dagger) = & \sum_{(i,\underline l),(j,\underline m)\in\bN\times\bN^d}\nabla_{i,\underline l}(\nabla_{1+j+k,\underline m+\underline n}(x))\\
            &\quad  \cdot (1+aX)^{-j-k-|\underline m|-|\underline n|-1}(1-E(u)\underline Y)^{-\underline m-\underline n} (-1)^{j+|\underline m|}\binom{i+j}{i}\binom{\underline l+\underline m}{\underline l}X^{[i+j]}\underline Y^{[\underline l+\underline m]}\\
            & - \sum_{(i,\underline l),(j,\underline m)\in\bN\times\bN^d}\nabla_{i+1,\underline l}(\nabla_{j+k,\underline m+\underline n}(x))\\
            &\quad  \cdot (1+aX)^{-j-k-|\underline m|-|\underline n|-1}(1-E(u)\underline Y)^{-\underline m-\underline n} (-1)^{j+|\underline m|}\binom{i+j}{i}\binom{\underline l+\underline m}{\underline l}X^{[i+j]}\underline Y^{[\underline l+\underline m]}\\
            & - \sum_{(i,\underline l),(j,\underline m)\in\bN\times\bN^d}ia\nabla_{i,\underline l}(\nabla_{j+k,\underline m+\underline n}(x))\\
            &\quad  \cdot (1+aX)^{-j-k-|\underline m|-|\underline n|-1}(1-E(u)\underline Y)^{-\underline m-\underline n} (-1)^{j+|\underline m|}\binom{i+j}{i}\binom{\underline l+\underline m}{\underline l}X^{[i+j]}\underline Y^{[\underline l+\underline m]}\\
            & + \sum_{(i,\underline l),(j,\underline m)\in\bN\times\bN^d}\nabla_{i,\underline l}(\nabla_{j+k,\underline m+\underline n}(x))\\
            & \quad \cdot (1+aX)^{-j-k-|\underline m|-|\underline n|}(1-E(u)\underline Y)^{-\underline m-\underline n} (-1)^{j+|\underline m|}\binom{i+j}{i}\binom{\underline l+\underline m}{\underline l}X^{[i+j-1]}\underline Y^{[\underline l+\underline m]}.
        \end{split}
    \end{equation*}
    Using $i\binom{i+j}{i}X^{[i+j]} = X\binom{i+j-1}{i-1}X^{[i+j-1]}$, we obtain that 
    \begin{equation*}
        \begin{split}
            (\dagger) = & \sum_{(i,\underline l),(j,\underline m)\in\bN\times\bN^d}\nabla_{i,\underline l}(\nabla_{j+k,\underline m+\underline n}(x))\\
            &\quad  \cdot (1+aX)^{-j-k-|\underline m|-|\underline n|}(1-E(u)\underline Y)^{-\underline m-\underline n} (-1)^{j-1+|\underline m|}\binom{i+j-1}{i}\binom{\underline l+\underline m}{\underline l}X^{[i+j-1]}\underline Y^{[\underline l+\underline m]}\\
            & - (1+aX)^{-1}\sum_{(i,\underline l),(j,\underline m)\in\bN\times\bN^d}\nabla_{i,\underline l}(\nabla_{j+k,\underline m+\underline n}(x))\\
            &\quad  \cdot (1+aX)^{-j-k-|\underline m|-|\underline n|}(1-E(u)\underline Y)^{-\underline m-\underline n} (-1)^{j+|\underline m|}\binom{i-1+j}{i-1}\binom{\underline l+\underline m}{\underline l}X^{[i-1+j]}\underline Y^{[\underline l+\underline m]}\\
            & - (1+aX)^{-1}aX\sum_{(i,\underline l),(j,\underline m)\in\bN\times\bN^d}\nabla_{i,\underline l}(\nabla_{j+k,\underline m+\underline n}(x))\\
            &\quad  \cdot (1+aX)^{-j-k-|\underline m|-|\underline n|}(1-E(u)\underline Y)^{-\underline m-\underline n} (-1)^{j+|\underline m|}\binom{i+j-1}{i-1}\binom{\underline l+\underline m}{\underline l}X^{[i+j-1]}\underline Y^{[\underline l+\underline m]}\\
            & + \sum_{(i,\underline l),(j,\underline m)\in\bN\times\bN^d}\nabla_{i,\underline l}(\nabla_{j+k,\underline m+\underline n}(x))\\
            & \quad \cdot (1+aX)^{-j-k-|\underline m|-|\underline n|}(1-E(u)\underline Y)^{-\underline m-\underline n} (-1)^{j+|\underline m|}\binom{i+j}{i}\binom{\underline l+\underline m}{\underline l}X^{[i+j-1]}\underline Y^{[\underline l+\underline m]}\\
            = & \sum_{(i,\underline l),(j,\underline m)\in\bN\times\bN^d}\nabla_{i,\underline l}(\nabla_{j+k,\underline m+\underline n}(x))\\
            &\quad  \cdot (1+aX)^{-j-k-|\underline m|-|\underline n|}(1-E(u)\underline Y)^{-\underline m-\underline n} (-1)^{j+|\underline m|}\binom{i+j-1}{i-1}\binom{\underline l+\underline m}{\underline l}X^{[i+j-1]}\underline Y^{[\underline l+\underline m]}\\
            & - (1+aX)^{-1}\sum_{(i,\underline l),(j,\underline m)\in\bN\times\bN^d}\nabla_{i,\underline l}(\nabla_{j+k,\underline m+\underline n}(x))\\
            &\quad  \cdot (1+aX)^{-j-k-|\underline m|-|\underline n|}(1-E(u)\underline Y)^{-\underline m-\underline n} (-1)^{j+|\underline m|}\binom{i-1+j}{i-1}\binom{\underline l+\underline m}{\underline l}X^{[i-1+j]}\underline Y^{[\underline l+\underline m]}\\
            & - (1+aX)^{-1}aX\sum_{(i,\underline l),(j,\underline m)\in\bN\times\bN^d}\nabla_{i,\underline l}(\nabla_{j+k,\underline m+\underline n}(x))\\
            &\quad  \cdot (1+aX)^{-j-k-|\underline m|-|\underline n|}(1-E(u)\underline Y)^{-\underline m-\underline n} (-1)^{j+|\underline m|}\binom{i+j-1}{i-1}\binom{\underline l+\underline m}{\underline l}X^{[i+j-1]}\underline Y^{[\underline l+\underline m]}\\
            =& 0,
        \end{split}
    \end{equation*}
    as desired, where for the second equality, we use $\binom{i+j}{i}=\binom{i+j-1}{i}+\binom{i+j-1}{i-1}$. So we see that Item (1) is true.
    
    To complete the proof, it remains to prove the ``moreover'' part. The formula (\ref{Equ-Stratification-Abs-I}) follows from (\ref{Equ-Strat-Abs}) together with (\ref{Equ-Iteration-Abs}). By standard Lie theory (cf. the proof of \cite[Prop. 2.6]{MW22}), $\nabla_i$'s are all nilpotent. These complete the proof. 
    
 \end{proof}
 
 \begin{rmk}
   Thanks to the nilpotency of $\nabla_i$'s, the limit \[\lim_{i+|\underline n|\to+\infty}\nabla_{i,\underline n}=0\]
   holds true if and only if
   \begin{equation}\label{Equ-NilpotencyForarithmeticPart}
       \lim_{i\to+\infty}\nabla_{i,\underline 0} = \lim_{i\to+\infty}\prod_{j=0}^{i-1}(a\phi_M-ja) = 0.
   \end{equation}
 \end{rmk}
 
\subsection{Enhanced connections}
 \begin{defn}  \label{defn: ka k kg}
Use Notation \ref{nota: many dr rings abs and rel}.   Let 
   $$\sigma_{0,a}:\frakS_{*,\dR,m}^{1,+}\to\frakS_{\dR,m}^{+}$$ 
   $$\sigma_0:\frakS(R)_{*,\dR,m}^{1,+}\to\frakS(R)_{\dR,m}^{+}$$
   $$\sigma_{0,g}: \frakS(R)_{*,\geo,\dR,m}^{1,+}\to \frakS(R)_{\geo,\dR,m}^{+}$$ be the corresponding degeneracy morphism. 
   Denote by $\calK_a$ (resp. $\calK$, $\calK_g$) the kernel of $\sigma_{0,a}$ (resp. $\sigma_0$, $\sigma_{0,g}$). Then we have natural maps 
   \begin{equation}\label{Equ-ExactSeqforIdeal}
       \calK_a\to\calK\to\calK_g.
   \end{equation}
(The ideal $\calK_g$   was denoted as ``$\calK$" in Corollary \ref{Cor-pdStructure-Rel}.)
 \end{defn}

 \begin{cor}\label{Cor-pdStructure-Abs} 
   \begin{enumerate}
       \item The ideal $\calK$ is the closed pd-ideal generated by $\{X_1^{[n_0]}Y_{1,1}^{[n_1]}\cdots Y_{d,1}^{[n_d]}\mid n_0+n_1+\cdots+n_d\geq 1, n_0,n_1,\dots,n_d\geq 0\}$ via the isomorphisms in Proposition \ref{Prop-Structure-Abs}.
       
       \item For any $j\geq 1$, let $\calK^{[j]}$ be the closed $j$-th pd-power of $\calK$. Then (\ref{Equ-ExactSeqforIdeal}) induces an exact sequence of $\frakS(R)_{\dR,m}^+$-modules
       
       \begin{equation}\label{Equ-ExactSeqforModule}
           \begin{tikzcd}
0 \arrow[r] & {\calK_a/\calK_a^{[2]}\otimes_{\frakS_{\dR,m}^+}\frakS(R)_{\dR,m}^+} \arrow[d, "\simeq"] \arrow[r] & {\calK/\calK^{[2]}} \arrow[d, "="] \arrow[r] & {\calK_g/\calK_g^{[2]}} \arrow[d, "\simeq"] \arrow[r]                      & 0 \\
0 \arrow[r] & {\widehat \Omega^1_{\frakS_{\dR,m}^+}\otimes_{\frakS_{\dR,m}^+}\frakS(R)_{\dR,m}^+} \arrow[r]      & {\calK/\calK^{[2]}} \arrow[r]                & {\widehat \Omega^1_{\frakS(R)_{\dR,m}^+/\frakS_{\dR,m}^+}\{-1\}} \arrow[r] & 0
\end{tikzcd}
                \end{equation}

   \end{enumerate}
    which induces a non-canonical isomorphism
    \begin{equation}\label{Equ-NonCanonicalSplitting}
        \begin{split}
            \calK/\calK^{[2]}&\cong (\calK_a/\calK_a^{[2]}\otimes_{\frakS_{\dR,m}^+}\frakS(R)_{\dR,m}^+)\oplus\calK_g/\calK_g^{[2]}\\
            & \cong \frakS(R)_{\dR,m}^+a^{-1}\dlog E(u)\oplus\widehat \Omega^1_{\frakS(R)_{\dR,m}^+/\frakS_{\dR,m}^+}\{-1\}\\
            & \cong \frakS(R)_{\dR,m}^+a^{-1}\dlog E(u)\oplus(\bigoplus_{i=1}^d\frakS(R)_{\dR,m}^+\dlog T_i\cdot E(u)^{-1}).
        \end{split}
    \end{equation}
 \end{cor}
 \begin{proof}
   The Item (1) follows from a similar argument used in the proof of Proposition \ref{Prop-Structure-Rel}. Moreover, it can be easily concluded that 
   \[\calK/\calK^{[2]} \cong \frakS(R)_{\dR,m}^+ X_1\oplus(\bigoplus_{i=1}^d\frakS(R)_{\dR,m}^+ Y_{i,1}).\]
   By \cite[Lem. 4.13]{GMWdR}, we have 
   \[\calK_a/\calK_a^{[2]} \cong \frakS_{\dR,m}^+X_1 \cong \frakS_{\dR,m}^+a^{-1}\dlog E(u).\]
   By Corollary \ref{Cor-pdStructure-Rel} and its logarithmic analogue, we have 
   \[\calK_g/\calK_g^{[2]} \cong \bigoplus_{i=1}^d\frakS(R)_{\dR,m}^+ Y_{i,1} \cong \widehat \Omega^1_{\frakS(R)_{\dR,m}^+/\frakS_{\dR,m}^+}\cdot E(u)^{-1}.\]
   Now Item (2) follows directly.
 \end{proof}

 \begin{defn} \label{defn: conn arithmetic case}
     Let $X$ be a smooth rigid space over $K$. Consider the ringed space $(X_\et, \o_{X, \gs_\dR^+, m})$ in Notation \ref{nota: ringed space}.
     \begin{enumerate}
 \item  
For notational simplicity,  denote 
\[\Omega^1_{X, \gs_{\dR }^+}:=\Omega^1_X \otimes_{\calO_X} \o_{X, \gs_\dR^+ }. \]
As this sheaf admits $\gs_\dR^+$-structure, one can twist it by $(E)^{-1}:=\Hom((E), \gsdr)$ (similar to Notation \ref{nota: BK twist rel pris}), which we denote as
\[\Omega^1_{X, \gs_{\dR }^+}\{-1\}:=\Omega^1_{X, \gs_{\dR }^+}\otimes_{\gsdr} (E)^{-1}.\]
Similar to Def \ref{Dfn-BetaConnection}, we abuse notation, and let $\mathrm d$ be the composite:
\[ \mathrm{d}: \o_{X, \gs_\dR^+ }   \to \Omega^1_{X, \gs_{\dR }^+} \to \Omega^1_{X, \gs_{\dR }^+}\{-1\}\]
where the first map is $\gsdr$-linear extension of  the usual differentiation on $\o_X$.

\item An integrable (or flat) connection on $(X, \o_{X, \gs_\dR^+, m})$ with respect to $\mathrm{d}$ consists of $M \in \vect(X_\et, \o_{X, \gs_\dR^+, m})$, together with an $\frakS_{\dR,m}^+$-linear additive map
\[ \nabla: M \to M\otimes_{\o_{X, \gs_\dR^+ }} \Omega^1_{X, \gs_\dR^+ }\{-1\} \]
 satisfying   Leibniz rule with respect to $\mathrm{d}$, such that $\nabla_M\wedge\nabla_M = 0$.  
     Note that if we let $T_1,\dots,T_d$ be the local coordinate on $X$ and write $\nabla_M = \sum_{i=1}^d\nabla_i\otimes\frac{\dlog T_i}{E}$, then the $\nabla_i:M\to M$ satisfies the ``$E$-Leibniz rule'' with respect to $T_i\frac{d}{d T_i}$ (cf. Definition \ref{def: b-conn}). Say $(M,\nabla_M)$ is \emph{topologically nilpotent} if it is integrable and for any local expansion $\nabla_M= \sum_{i=1}^d\nabla_i\otimes\frac{\dlog T_i}{E}$ (with respect to the given local coordinate on $X$) as above
     \[\lim_{n_1+\dots+n_d\to+\infty}\prod_{i=0}^d\prod_{j=0}^{n_i}(\nabla_i-jE) = 0\]
     with respect to the topology on $M$. This is equivalent to requiring that each $\nabla_i$ is topologically nilpotent.
     

     \item 
     By an \emph{(arithmetic) $E$-connection} on $M \in (X, \o_{X, \gs_\dR^+, m})$ we mean an $\o_X$-linear morphism $\phi_M:M\to M$ satisfying the ``$E$-Leibniz rule'' with respect to $E\frac{d}{dE}:\frakS_{\dR}^+\to\frakS_{\dR}^+$ (Def \ref{def: b-conn}); that is, for any local section $x\in M$ and any $f\in \frakS_{\dR}^+$, 
     \[\phi_M(fx) = f\phi_M(x)+E\frac{d}{dE}(f)x.\] 
     Let $a\in K$. We say an $E$-connection $(M,\phi_M)$ is \emph{$a$-small} (with $a$ in Notation \ref{nota: a sec 6}) if 
     \[\lim_{n\to+\infty}a^n\prod_{i=0}^{n-1}(\phi_M-i) = 0\]
     with respect to the topology on $M$. 
     \end{enumerate}
 \end{defn}
 
 \begin{notation}
 Consider the $\gsdr$-module $(E)^{-n}:=\Hom((E^n), \gsdr)$ of rank 1; the differential operator $E\frac{d}{dE}$ on $\gsdr$ stablizes $(E^n)$, hence induces an operator on $(E)^{-n}$. That is, there is an operator
\[ \phi_{(E)^{-n}}: (E)^{-n} \to (E)^{-n} \]
such that for $a \in (E^n)$, 
\[ (\phi_{(E)^{-n}}(f))(a)=f(E\frac{d}{dE}(a)) \]
This can be $\o_X$-linearly extended to $\o_{X, \gsdr}$.
Recall $\Omega^1_{X, \gs_{\dR }^+}\{-n\}:=\Omega^1_{X, \gs_{\dR }^+}\otimes_{\gsdr} (E)^{-n}$.
Abuse notation, use the same notation to denote 
\[ \phi_{(E)^{-n}}: =0\otimes 1+ 1\otimes \phi_{(E)^{-n}}:  \Omega^1_{X, \gs_{\dR }^+}\{-n\} \to \Omega^1_{X, \gs_{\dR }^+}\{-n\}. \]
When $m=1$, $\phi_{(E)^{-n}}$ is nothing but multiplication by $-n$.
 \end{notation}
     
 \begin{defn}[Enhanced connection] \label{defn: abs enhanced conn}
An \emph{enhanced connection} is an object $M \in \vect(X_\et, \o_{X, \gs_\dR^+, m})$ equipped with:
\begin{itemize}
    \item a topologically nilpotent integrable connection
      \[\nabla_M:M\to M\otimes_{\o_{X, \gsdr}}\Omega^1_{X, \gsdr}\{-1\},\]
      \item  an (arithmetic) $E$-connection $\phi_M:M\to M$, such that
      
\item the following diagram is commutative:      
        \[
    \begin{tikzcd}
M \arrow[d, "\phi"] \arrow[r, "\nabla"] & M\otimes_{\o_{X, \gsdr}}\Omega^1_{X, \gsdr}\{-1\} \arrow[d, " \phi_M\otimes 1 +1\otimes \phi_{(E)^{-1}}  "] \\
M \arrow[r, "\nabla"]                   & M\otimes_{\o_{X, \gsdr}}\Omega^1_{X, \gsdr}\{-1\} .                             
\end{tikzcd}
\] 
\end{itemize}    
    Say an {enhanced connection} $(M,\nabla_M,\phi_M)$ is $a$-small (with $a$ in Notation \ref{nota: a sec 6}) if the  $E$-connection $\phi_M$ is $a$-small. Denote  the category of $a$-small enhanced connections by 
    \[  \MIC_\en^a(X_\et, \o_{X, \gs_\dR^+, m}).\]
Consider the ``enhanced de Rham double complex":
\[
\begin{tikzcd}
M \arrow[d, "\phi"] \arrow[r, "\nabla"] & {M\otimes_{\o_{X, \gsdr}}\Omega^1_{X, \gsdr}\{-1\}} \arrow[d, "\phi_M\otimes 1 +1\otimes \phi_{(E)^{-1}}"] \arrow[r, "\nabla^2"] & {M\otimes_{\o_{X, \gsdr}}\Omega^2_{X, \gsdr}\{-2\}} \arrow[r] \arrow[d, "\phi_M\otimes 1 +1\otimes \phi_{(E)^{-2}}"] & \cdots \\
M \arrow[r, "\nabla"]                   & {M\otimes_{\o_{X, \gsdr}}\Omega^1_{X, \gsdr}\{-1\}} \arrow[r, "\nabla^2"]                               & {M\otimes_{\o_{X, \gsdr}}\Omega^2_{X, \gsdr}\{-2\}} \arrow[r]                               & \cdots
\end{tikzcd}
\] 
Denote (each) row complex by $\DR(M)$, and denote the totalization of the double complex by $\rg(\phi, \DR(M))$.   
 \end{defn}

\begin{notation}\label{nota: enhanced conn local}
Let $X$ be the generic fiber of  $\spf R$ where $R$ is as in Notation \ref{nota: bk prism abs pris}. In this case, $X$ has a chart. 
Denote $\MIC_\en^a(X_\et, \o_{X, \gs_\dR^+, m})$ in this case by $ \mic_\en^a(R_\gsdrm) $ where $R_\gsdrm=R_K \hatotimes_K \gsdrm$. For $M \in  \mic_\en^a(R_\gsdrm) $,  one can write      $\nabla_M = \sum_{i=1}^d\nabla_i\otimes \frac{\dlog T_i}{E}$ where  each $\nabla_i:M\to M$ satisfies  $E$-Leibniz rule (Def \ref{def: b-conn}). Since we have
\[\phi_{(E)^{-1}} (\frac{\dlog T_i}{E})=-\frac{\dlog T_i}{E},\]
the commutative diagram in Def \ref{defn: abs enhanced conn} translates to the commutative diagram that for each $i$,
\[
\begin{tikzcd}
M \arrow[d, "\phi"] \arrow[r, "\nabla_i"] & M \arrow[d, "\phi-1"] \\
M \arrow[r, "\nabla_i"]                   & M                    
\end{tikzcd}
\]
that is: for each $i$,  $[\phi_M,\nabla_i]=\nabla_i$.
\end{notation}

\subsection{Crystals and enhanced connections}

  \begin{construction}\label{Construction-Intrinsic}
     Let $(M,\varepsilon)$ be a stratification with respect to $\frakS(R)_{*,\dR,m}^{\bullet}$  satisfying the cocycle condition. Since $\sigma_0^*(\varepsilon) = \id_M$, we see that for any $x\in\bM$, 
    \[(\id_M-\varepsilon)(x)\in M\otimes_{\frakS(R)_{\dR,m}^+,p_1}\frakS(R)_{*,\dR,m}^{1,+}\cdot \calK,\]
    which induces a map 
    \[\overline{\id_M-\varepsilon}: M\to M\otimes_{\frakS(R)_{\dR,m}^+}\calK/\calK^{[2]}.\]
    Using the decomposition (\ref{Equ-NonCanonicalSplitting}), we see that
    \begin{equation*}
        \begin{split}
            \overline{\id_M-\varepsilon} = (\phi_M',\nabla_M'): M&\to M\otimes_{\frakS_{\dR,m}^+}\calK_a/\calK_a^{[2]}\oplus M\otimes\calK_g/\calK_g^{[2]}\\
            &\cong M\frac{\dlog E(u)}{a}\oplus \bigoplus_{i=1}^dM\cdot\frac{\dlog T_i}{E(u)}.
        \end{split}
    \end{equation*}
    By (\ref{Equ-Stratification-Abs-I}), we see that for any $x\in M$
   \[(\overline{\id_M-\varepsilon})(x) = a\phi_M(x)X_1+\nabla_1(x)Y_{1,1}+\dots+\nabla_d(x)Y_{d,1}.\]
As we identify $X$ resp. $Y_{s,1}$'s with $-\frac{\log E(u)}{a}$ resp. $\frac{\dlog T_i}{E(u)}$'s, it is easy to see 
 \[(\phi_M',\nabla_M') = (a\phi_M,\nabla_M).\]
Thus, the above construction gives an (intrinsic) definition of a functor
\[\Vect((R)_{\Prism,\ast},\Prism_{\dR,m}^+)\simeq     \mic_\en^a(R_\gsdrm). \]
  \end{construction}

\begin{theorem} \label{thm: abs local equiv}
   Let $R$ be a small smooth (resp. semi-stable) $\calO_K$-algebra as in Notation \ref{nota: bk prism abs pris}.  The functor in Construction \ref{Construction-Intrinsic} induces an equivalence of categories
   \[\Vect((R)_{\Prism,\ast},\Prism_{\dR,m}^+)\xrightarrow{\simeq}    \mic_\en^a(R_\gsdrm) \]
   which preserves ranks, tensor products and duals.
\end{theorem}
\begin{proof}
This is a translation of  Lemma \ref{Lem-Technique-Abs}.


\end{proof}

\newpage
\section{Local cohomology comparison: absolute crystals vs.  enhanced connections} \label{sec: coho abs local}

Let $\ast\in\{\emptyset, \log\}$.
Consider the equivalence in Thm \ref{thm: abs local equiv}:
 \[\Vect((R)_{\Prism,\ast},\Prism_{\dR,m}^+)\xrightarrow{\simeq}    \mic_\en^a(R_\gsdrm). \]
  Consider a pair of corresponding objects
\[ \bm \mapsto (M,  \nabla, \phi). \]
In this section, the main result is Theorem \ref{thm:pris coho vs enhanced coho}, where we prove cohomology comparison
\[ \rR\Gamma((R)_{\Prism, \ast},\bM)  \simeq \rD\rR(M,\nabla,\phi); \]
 here the RHS is defined in Def \ref{defn: abs enhanced conn}.

\begin{remark}
In the main text of this section, we only write out the \emph{prismatic} case in full. As the many previous sections show, the log-prismatic case is always \emph{parallel} to the prismatic case, and only requires minimal modifications, cf. Rem \ref{rem:log is same coho} near the end. Indeed, as we shall see in Construction \ref{construction:bi-cosimplicial prism}, we construct many bi-co-simplicial objects  in this section, which then involves decorations such as $\bullet$ (arithmetic index), $\clubsuit$ (geometric index) and a further   $\star$ (differential index in Notation \ref{nota: absolute diff}) ; we do not want to keep the extra decoration $\ast \in \{\emptyset, \log\}$ which will burden the notation.
\end{remark}

\begin{construction}[Strategy] \label{cons:strategy abs coho}
Before going to details, we first explain the idea of our strategy. 
We are partly guided by a stacky approach in Item (1); we choose a site-theoretic approach in Item (2) for technical convenience (at this point).

\begin{enumerate}
\item To get ourselves motivated, we first look at the case of integral Hodge--Tate crystals. By \cite{BL-b}, one can regard Hodge--Tate crystals (i.e. $\overline \calO_{\Prism}$-crystals) $\bM$ on $R_{\Prism}$ as vector bundles on ${\rm WCart}_{R}^{\HT}$, the Hodge--Tate stack associated to $\Spf(R)$, such that 
\[\rR\Gamma((R)_{\Prism},\bM)\simeq \rR\Gamma({\rm WCart}_{R}^{\HT},\bM).\]

Note that the Breuil--Kisin prism $(\frakS,(E))$ induces a flat cover $\Spf(\calO_K)\to {\rm WCart}_{\calO_K}^{\HT}$. In particular, we have a pullback diagram
\begin{equation*}
    \xymatrix{
    {\rm WCart}_{R/\frakS}^{\HT}\ar[r]\ar[d]&\Spf(\calO_K)\ar[d]\\
    {\rm WCart}_{R}^{\HT}\ar[r]& {\rm WCart}_{\calO_K}^{\HT}.
    }
\end{equation*}
Then ${\rm WCart}_{R/\frakS}^{\HT}\to {\rm WCart}_{R}^{\HT}$ is a flat cover. Considering its \v Cech nerve, we then have the following
\[
\rR\Gamma((R)_{\Prism},\bM)\simeq \rR\Gamma({\rm WCart}_{R}^{\HT},\bM)\simeq \rR\lim\rR\Gamma((R_{\frakS^{\bullet}/E}/(\frakS^{\bullet}))_{\Prism},\bM).
\]


So we are reduced to computing $\rR\Gamma((R_{\frakS^{\bullet}/E}/(\frakS^{\bullet}))_{\Prism},\bM)$, which has been treated in \S \ref{subsec: rel pris coho compa}. Namely, we need to consider the prismatic lifting $(\frakS^{\bullet}(R),(E))$ and the associated \v Cech nerve $(\frakS^{\bullet}(R)^{\clubsuit},(E))$, and study the de Rham complex $(\bM(\frakS^{\bullet}(R)^{\clubsuit},(E)),\rd_{\bM(\frakS^{\bullet}(R),(E))})$ in Construction \ref{Construction-Bicosimplicial-Rel}. By Theorem \ref{Thm-dRCrystalasXiConnection-Rel}(2), one can check there are quasi-isomorphisms
\[\rR\Gamma((R_{\frakS^{\bullet}/E}/(\frakS^{\bullet}))_{\Prism},\bM)\simeq (\bM(\frakS^{\bullet}(R)^{\clubsuit},(E)),\rd_{\bM(\frakS^{\bullet}(R),(E))})\simeq \rD\rR(M\widehat \otimes_{\calO_K}\frakS^{\bullet}/(E),\nabla\otimes\id).\]
So we have
\[\rR\Gamma((R)_{\Prism},\bM)\simeq\rR\lim\rD\rR(M\widehat \otimes_{\calO_K}\frakS^{\bullet}/(E),\nabla\otimes\id)\simeq \rD\rR(M,\nabla,\phi),\]
where the first quasi-isomorphism follows by combining all above quasi-isomorphisms, and the second quasi-isomorphism follows from  \cite[Th. 6.7]{GMWdR}.

\item We do not yet have a stacky interpretation of $\Prism_{\dR,m}^+$-crystals so far\footnote{With current progress on (log-) prismatic cohomology, together with explicit computations in this paper, a stacky approach should be reachable; we hope to come back to this in future investigations.}. Even when $m=1$, i.e. the case of rational Hodge--Tate crystals, we have to use analytic stack and the flat descent of coherent cohomology of analytic stack. Thus we have to construct $(\bM(\frakS^{\bullet}(R)^{\clubsuit},(E)),\rd_{\bM(\frakS^{\bullet}(R),(E))})$ and relate it with $\rR\Gamma((R)_{\Prism},\bM)$ explicitly. 
A key observation in our approach is Lem \ref{lem:diagonal}: the diagonal of the bi-cosimplicial prism $(\frakS^{\bullet}(R)^{\clubsuit},(E))$ is exactly the \v Cech nerve $(\frakS(R)^{\bullet},(E))$ associated to $(\frakS(R),(E))\in (R)_{\Prism}$ (cf. Notation \ref{nota: many dr rings abs and rel}). Thus, one can compare $\bM(\frakS^{\bullet}(R)^{\clubsuit},(E))$ with $\bM(\frakS(R)^{\bullet},(E))$ by (cosimplicial) Eilenberg--Zilber theorem; in addition, a standard  Alexander--\v{C}ech method shows that $\bM(\frakS(R)^{\bullet},(E))$ is quasi-isomorphic to $\rR\Gamma((R)_{\Prism},\bM)$.
\end{enumerate}
 \end{construction}


\begin{construction}\label{construction:bi-cosimplicial prism}
We construct a bi-cosimplicial ring $\frakS^{\bullet}(R)^{\clubsuit}$ as follows:
\begin{enumerate}
\item Let $(\frakS^n,(E))$ be the self coproduct of $(n+1)$-copies of $(\frakS,(E))$ in $(\calO_K)_{\Prism}$. So $n$ is the \emph{arithmetic/absolute index}.

\item Let $R_{\frakS^n/E}$ be the base-change of the smooth $\calO_K$-algebra $R$ along $\calO_K\to\frakS^n/E$.  Let $(\frakS^n(R),(E))$ be the prismatic lifting of $R_{\frakS^n/E}$ over $\frakS^n$ induced by the chart on $R$. 

\item Let $(\frakS^n(R)^m,(E))$ be the self coproduct of $(m+1)$-copies of $(\frakS^n(R),(E))$ in $(R_{\frakS^n/E}/(\frakS^n,(E)))_{\Prism}$. So $m$ is the \emph{geometric/relative index}.

\item 
Thus, we now have a bi-cosimplicial prism $(\frakS^{\bullet}(R)^{\clubsuit},(E))$ in $(R)_{\Prism}$. Use Convention \ref{conv: def of cosimplicial rings}:
\begin{itemize}
\item Let $p_i^a, \sigma_i^a, q_i^a$ be the cosimplicial maps for $(\frakS^{\bullet},(E))$ (with $\bullet$ varying);
\item For a fixed $\bullet$, let $p_i^g, \sigma_i^g, q_i^g$ be the cosimplicial maps for $(\frakS^{\bullet}(R)^{\clubsuit},(E))$ (with $\clubsuit$ varying).
\end{itemize}

\item Finally, consider the diagonal of the bi-cosimplicial prism $(\frakS^{\bullet}(R)^{\clubsuit},(E))$, i.e., $$(\frakS^{\bullet}(R)^{\bullet},(E)).$$ For this cosimplicial ring (with only one index $\bullet$), it has face maps $p_i = p_i^g\circ p_i^a$ and degeneracy maps $\sigma_i = \sigma_i^g\circ\sigma_i^a$.
\end{enumerate}
\end{construction}

\begin{lem}\label{lem:diagonal}
Use notations in Construction \ref{construction:bi-cosimplicial prism}.
We have an isomorphism of cosimplicial prisms
\[ (\frakS^{\bullet}(R)^{\bullet},(E)) \simeq (\frakS(R)^{\bullet},(E)).\]  
    
\end{lem}
\begin{proof}
   It suffices to check the universal property on $((\frakS^{\bullet}(R)^{\bullet}),(E))$:  any $(n+1)$ morphisms 
    $$f_i:(\frakS(R),(E))\to(B,J) \text{    in } (R)_{\Prism}$$
     give rise to a \emph{unique} morphism $f:(\frakS^{n}(R)^n,(E))\to (B,J)$.

    Consider the composite  
    \[(\frakS,(E))\hookrightarrow(\frakS(R),(E))\xrightarrow{f_i}(B,J) \text{ in $(\calO_K)_{\Prism}$;}\]
    The universal property for $(\frakS^n,(E))$ induces a unique morphism 
    \[h:(\frakS^n,(E))\to (B,J) \text{ in $(\calO_K)_{\Prism}$.}\]
    Consider the composite  
    \[(\frakS\za\underline T^{\pm 1}\ya,(E))\to(\frakS(R),(E))\xrightarrow{f_i}(B,J),\]
    which we shall still denote by $f_i$. Then one can uniquely extend $h$ to a morphism of prisms
    \[h_i:(\frakS^n\za\underline T^{\pm 1}\ya,(E))\to (B,J) \]
    such that $h_i(T_r) = f_i(T_r)$ for any $1\leq r\leq d$ by considering the restriction of $f_i$'s to $\frakS\subset \frakS\za\underline T^{\pm 1}\ya$ and noting that $(\frakS^n,(E))$ is a self-coproduct of $(\frakS,(E))$'s. As $\frakS^n(R)$ is $(p,E)$-completelt \'etale over $\frakS^n\za\underline T^{\pm 1}\ya$, the morphism $h_i$ above extends uniquely to a morphism, which we shall denote the same notation $h_i$, of prisms
    \[h_i:(\frakS^n(R),(E))\to(B,J).\]
    By the above construction and the $(p,E)$-complete \'etaleness of $\frakS\za\underline T^{\pm 1}\ya\to \frakS(R)$, for any $0\leq j\leq n$ the following diagram commutes:
    \[
    \begin{tikzcd}
{(\frakS,(E))} \arrow[r] \arrow[d, "q_i^a"] & {(\frakS(R),(E)} \arrow[r, "f_i"] \arrow[d, "q_i^a"] & {(B,J)} \arrow[d, "="] \\
{(\frakS^n,(E))} \arrow[r]                  & {(\frakS^n(R),(E))} \arrow[r, "h_i"]                 & {(B,J)}               
\end{tikzcd}
\] 
    Now, the $h_i$'s induce a unique morphism 
    \[f:(\frakS^n(R)^n,(E))\to (B,J)\]
    in $(R_{\frakS^n/E}/(\frakS^n,(E)))_{\Prism}$ such that 
    $h_i = f\circ q_i^g$
    for any $0\leq i\leq n$, where $q_i^g:\frakS^n(R)\to\frakS^n(R)^n$ is the canonical map induced by $[0]\xrightarrow{0\mapsto i}[n]$. Now, one can conclude by noting that
    $f_i = h_i\circ q_i^a = f\circ q_i^g\circ q_i^a$.
\end{proof}

Now, we define the bi-cosimplicial ring $\frakS^{\bullet}(R)^{\clubsuit,+}_{\dR}$ to be
\[\frakS^{\bullet}(R)^{\clubsuit,+}_{\dR}:=\Prism_{\dR}^+(\frakS^{\bullet}(R)^{\clubsuit},(E)).\]
Let $q_0:\frakS(R)\to\frakS^{\bullet}(R)^{\clubsuit}$ be the composite of morphisms
\[\frakS(R)\xrightarrow{q_0^a}\frakS^{\bullet}(R)\xrightarrow{q_0^g}\frakS^{\bullet}(R)^{\clubsuit}.\]
Then one can regard $\frakS^{\bullet}(R)^{\clubsuit,+}_{\dR}$ as an $\frakS(R)_{\dR}^+$-algebra via $q_0$.

\begin{lem}\label{lem:structure of bi-cosimplicial ring}
    Put $a = -E'(\pi)$. For any $1\leq i\leq \bullet$ and $1\leq j\leq \clubsuit$, let $X_i = \frac{E(u_0)-E(u_i)}{aE(u_0)}$ and $\underline Y_{j} = \frac{\underline T_0-\underline T_j}{-E(u)\underline T_0}$. Regard $\frakS^{\bullet}(R)^{\clubsuit,+}_{\dR}$ as an $\frakS(R)_{\dR}^+$-algebra via $q_0$ and then there exists an isomorphism of $\frakS(R)_{\dR}^+$-algebra
    \[\frakS(R)_{\dR}^+[X_1,\dots,X_{\bullet},\underline Y_1,\dots,\underline Y_{\clubsuit}]^{\wedge}_{\pd}\cong \frakS^{\bullet}(R)^{\clubsuit,+}_{\dR}.\]
    Via this isomorphism, the induced cosimplicial structure on $\frakS(R)_{\dR}^+[X_1,\dots,X_{\bullet},\underline Y_1,\dots,\underline Y_{\clubsuit}]^{\wedge}_{\pd}$ is determined as follows:

    The face maps $p^a_i$ and $p_i^g$ act on $X$'s and $\underline Y$'s via the formula

    \begin{equation*}
        \begin{split}
            &p^a_i(X_j) = \left\{
            \begin{array}{rcl}
                (X_{j+1}-X_1)(1+aX_1)^{-1}, & i=0 \\
                X_{j+1}, 0<i\leq j\\
                X_j, & i>j;
            \end{array}
            \right.\\
            &p^a_i(\underline Y_j) = \left\{
            \begin{array}{rcl}
                (1+aX_1)^{-1}\underline Y_j, & i=0 \\
                \underline Y_j, & i\neq 0;
            \end{array}
            \right.\\
            &p^g_i(X_j) = X_j;\\
            &p^g_i(\underline Y_j) = \left\{
            \begin{array}{rcl}
                (\underline Y_{j+1}-\underline Y_1)(1+E(u)\underline Y_1)^{-1}, & i=0 \\
                \underline Y_{j+1}, & 0<i\leq j\\
                \underline Y_j, & i>j.
            \end{array}
            \right.
        \end{split}
    \end{equation*}

    while the degeneracy maps $\sigma^a_i$ and $\sigma^g_i$ act on $X$'s and $\underline Y$'s via the formula
\begin{equation*}
    \begin{split}
        &\sigma^a_i(X_j) = \left\{
        \begin{array}{rcl}
            0, & i=0~\&~j=1 \\
            X_{j-1}, & i<j~\&~j\neq 1\\
            X_j, & i\geq j;
        \end{array}
        \right.\\
        &\sigma^a_i(\underline Y_j) = \underline Y_j;\\
        &\sigma^g_i(X_j) = X_j;\\
        &\sigma^g_i(\underline Y_j) = \left\{
        \begin{array}{rcl}
            0, & i=0~\&~j=1 \\
            \underline Y_{j-1}, & i<j~\&~j\neq 1\\
            \underline Y_j, & i\geq j.
        \end{array}
        \right.
    \end{split}
\end{equation*}
\end{lem}
\begin{proof}
    It follows from a similar argument for the proof of Proposition \ref{Prop-Structure-Abs}.
\end{proof}

\begin{notation} \label{nota: absolute diff}
Let $\Omega^1_{\frakS^{\bullet}(R)^{\clubsuit,+}_{\dR}/\frakS^{\bullet,+}_{\dR}}$ be the module of continuous differential forms. 
More precisely, we let 
\[\Omega^1_{\frakS^{\bullet}(R)^{\clubsuit,+}_{\dR}/\frakS^{\bullet,+}_{\dR}} = \Omega^1_{\frakS^{\bullet}(R)^{+}_{\dR}/\frakS^{\bullet,+}_{\dR}}\{-1\}\oplus \Omega^1_{\frakS^{\bullet}(R)^{\clubsuit,+}_{\dR}/\frakS^{\bullet}(R)^+_{\dR}}\]
so that 
\[\frac{\dlog \underline T}{E(u)} = \frac{\dlog \underline T_0}{E(u_0)},\rd\underline Y_1 = \frac{\rd(\underline T_0-\underline T_1)}{-E(u_0)\underline T_0},\dots,\rd \underline Y_{\clubsuit}= \frac{\rd(\underline T_0-\underline T_{\clubsuit})}{-E(u_0)\underline T_0}\]
form a basis of $\Omega^1_{\frakS^{\bullet}(R)^{\clubsuit,+}_{\dR}/\frakS^{\bullet,+}_{\dR}}$.
For any $\star\geq 0$, define $\Omega^{\star}_{\frakS^{\bullet}(R)^{\clubsuit,+}_{\dR}/\frakS^{\bullet,+}_{\dR}} = \wedge^{\star}\Omega^1_{\frakS^{\bullet}(R)^{\clubsuit,+}_{\dR}/\frakS^{\bullet,+}_{\dR}}$. Let 
\[\rd^{\bullet:\clubsuit}:\frakS^{\bullet}(R)^{\clubsuit,+}_{\dR}\to \Omega^1_{\frakS^{\bullet}(R)^{\clubsuit,+}_{\dR}/\frakS^{\bullet,+}_{\dR}}\]
be the canonical differential, and then it induces a complex 
\[\frakS^{\bullet}(R)^{\clubsuit,+}_{\dR}\xrightarrow{\rd^{\bullet,\clubsuit}}\Omega^1_{\frakS^{\bullet}(R)^{\clubsuit,+}_{\dR}/\frakS^{\bullet,+}_{\dR}}\xrightarrow{\rd^{\bullet,\clubsuit}}\Omega^2_{\frakS^{\bullet}(R)^{\clubsuit,+}_{\dR}/\frakS^{\bullet,+}_{\dR}}\to \cdots\]
preserving bicosimplicial structure.
\end{notation} 

 Now, let $\bM$ be a $\Prism^+_{\dR}$-crystal on $(R)_{\Prism}$ with   induced enhanced  connection $(M,\nabla=\sum_{i=1}^d\nabla_i\otimes\frac{\dlog T_i}{E(u)},\phi)$ over $\frakS(R)_{\dR}^+$. Then the stratification corresponding to $\bM$ is given by 
 \[ \varepsilon_M = (1+E(u)\underline Y_1)^{E(u)^{-1}\underline \nabla}(1+aX_1)^{\phi} = (1+aX_1)^{\phi}(1+(1+aX_1)^{-1}A(u)\underline Y_1)^{E(u)^{-1}\underline \nabla}).\]
 Also, the morphism $q_0:\frakS(R)\to\frakS^{\bullet}(R)^{\clubsuit}$ induces a natural isomorphism
 \[M\otimes_{\frakS(R),q_0}\frakS^{\bullet}(R)^{\clubsuit}\cong \bM(\frakS^{\bullet}(R)^{\clubsuit}).\]
 So   $M\otimes_{\frakS(R),q_0}\frakS^{\bullet}(R)^{\clubsuit}$ carries a bi-cosimplicial structure of $\frakS^{\bullet}(R)^{\clubsuit}$ inheriting from that on $\bM(\frakS^{\bullet}(R)^{\clubsuit})$. By abuse of notation, we still denote  by $p_i^a,\sigma_i^a,p_i^g,\sigma_i^g$ the corresponding face maps and degeneracy maps. The next lemma describes the actions of $p_i^a,\sigma_i^a,p_i^g,\sigma_i^g$ on $M\otimes_{\frakS(R),q_0}\frakS^{\bullet}(R)^{\clubsuit}$.
 \begin{lem}\label{lem:inherited bi-cosimplicial structure}
     Keep notations as above. Then the actions of $p_i^a,\sigma_i^a,p_i^g,\sigma_i^g$ on $M\otimes_{\frakS(R),q_0}\frakS^{\bullet}(R)^{\clubsuit}$ are determined such that for any $x\in M$, we have
     \begin{equation*}
         \begin{split}
             &p_i^a(x) = \left\{
             \begin{array}{rcl}
                 (1+aX_1)^{\phi}(x), & i=0 \\
                 x, & i\geq 1;
             \end{array}
             \right.\\
             &p_i^g(x) = \left\{
             \begin{array}{rcl}
                 (1+E(u)\underline Y_1)^{E(u)^{-1}\phi}, & i=0 \\
                 x, & i\geq 1;
             \end{array}
             \right.\\
             &\sigma_i^a(x) = \begin{array}{rcl}
                 x, & \forall i\geq 0;
             \end{array}\\
             &\sigma_i^g(x) = \begin{array}{rcl}
                 x, & \forall i\geq 0.
             \end{array}
         \end{split}
     \end{equation*}
 \end{lem}
 \begin{proof}
     Recall that the $q_0 = q_0^g\circ q_0^a$.
     Then the formula for degeneracy map are clear: For any $i\geq 0$, the result follows from that
     \[\sigma_i^a\circ q_0 = \sigma_i^a\circ q_0^g\circ q_0^a = q_0^g\circ \sigma_i^a\circ q_0^a = q_0^g\circ q_0^a = q_0\]
     and that
     \[\sigma_i^g\circ q_0 = \sigma_i^g\circ q_0^g\circ q_0^a = q_0^g\circ q_0^a = q_0.\]
     The same argument also works for face maps $p_i^a$ and $p_i^g$ as long as $i\geq 1$. So we only need to check the formula also holds true for $p_0^a$ and $p_0^g$.
     
     For $p_0^a$, we have
     \[p_0^a\circ q_0 = p_0^a\circ q_0^g\circ q_0^a = q_0^g\circ p_0^a\circ q_0^a = q_0^g\circ q_1^a.\]
     We claim that for any fixed $\clubsuit$, we have 
     \[(1+aX_1)^{\phi}\circ q_1^a = q_0^a.\]
     Granting this, the result holds true for $p_0^a$ as $q_0$ acts on $X_1$ trivially.

     Fix $\clubsuit$. To see the claim, it suffices to show that the stratification with respect to $\frakS^{\bullet}(R)^{\clubsuit}$ corresponding to the cosimplicial $\frakS^{\bullet}(R)^{\clubsuit}$-module $\bM(\frakS^{\bullet}(R)^{\clubsuit},(E))$ is given by $(1+aX_1)^{a^{-1}\phi}$. 

     Note that in Lemma \ref{lem:structure of bi-cosimplicial ring}, when we fix $\clubsuit$, the cosimplicial structure on \[\frakS^{\bullet}(R)^{\clubsuit,+}_{\dR} = \frakS(R)_{\dR}^+[X_1,\dots,X_{\bullet},\underline Y_1,\dots,\underline Y_{\clubsuit}]^{\wedge}_{\pd}\]
     is the same as that in \cite[Prop. 3.9]{GMWdR}. So one can conclude by applying \cite[Prop. 4.5]{GMWdR}.

     For $p_0^g$, we have 
     \[p_0^g\circ q_0 = p_0^g\circ q_0^g\circ q_0^a = q_1^g\circ q_0^a.\]
     For fixed $\bullet$, by Lemma \ref{Lem-Technique-Rel}, we have 
     \[(1+E(u)\underline Y_1)^{E(u)^{-1}\underline \nabla})\circ q_1^g = q_0^g.\]
     This implies the result for $p_0^g$ immediately.
 \end{proof}

 Now, we are going to construct a ``de Rham complex'' for $\bM(\frakS^{\bullet}(R)^{\clubsuit})$.
 \begin{construction}\label{construction:tri-complex}
     Fix $\bullet$ and $\clubsuit$. We define 
     \[\nabla^{\bullet,\clubsuit}:M\otimes_{\frakS(R)_{\dR}^+,q_0} \Omega^{\star}_{\frakS^{\bullet}(R)^{\clubsuit,+}_{\dR}/\frakS^{\bullet,+}_{\dR}}\to M\otimes_{\frakS(R),q_0} \Omega^{\star+1}_{\frakS^{\bullet}(R)^{\clubsuit,+}_{\dR}/\frakS^{\bullet,+}_{\dR}}\]
     as follows: For any $x\in M$ and any $\omega\in \otimes_{\frakS(R),q_0} \Omega^{\star}_{\frakS^{\bullet}(R)^{\clubsuit,+}_{\dR}/\frakS^{\bullet,+}_{\dR}}$, we have
     \[\nabla_{\star}^{\bullet,\clubsuit}(x\otimes\omega) = q_0^*(\nabla(x))\wedge\omega+x\otimes \rd^{\bullet,\clubsuit}(\omega),\]
     where $q_0^*(\nabla(x))\in M\otimes_{\frakS(R),q_0}\otimes_{\frakS(R),q_0} \Omega^1_{\frakS^{\bullet}(R)^{\clubsuit,+}_{\dR}/\frakS^{\bullet}_{\dR}}$ denotes the pull-back of 
     \[\nabla(x)\in M\otimes_{\frakS(R)} \Omega^1_{\frakS(R)_{\dR}^+/\frakS_{\dR}^+}\]
     via the map $q_0:\frakS(R)^+_{\dR}\to \frakS^{\bullet}(R)^{\clubsuit,+}_{\dR}$.
 \end{construction}
  We remark that when $m=1$, the above construction coincides with \cite[Construction 4.30]{MW22} after expressing $\frakS^{\bullet}(R)^{\clubsuit,+}_{\dR}$ as $\frakS(R)_{\dR}^+[X_1,\dots,X_{\bullet},\underline Y_1,\dots,\underline Y_{\clubsuit}]^{\wedge}_{\pd}$ via Lemma \ref{lem:structure of bi-cosimplicial ring}. Moreover, in what follows, when $m=1$, everything has been proved and is compatible with that in \cite[\S 4.3]{MW22}.

  \begin{lem}\label{lem: nablaast perserve bicosimp}
     Keep notations in Construction \ref{construction:tri-complex}. Then we have that $\nabla^{\bullet,\clubsuit}_{\star+1}\circ\nabla^{\bullet,\clubsuit}_{\star} = 0$, and that $\nabla_{\star}^{\bullet,\clubsuit}$ preserves the bi-cosimplicial structure on $M\otimes_{\frakS(R)^+_{\dR},q_0}\frakS^{\bullet}(R)^{\clubsuit,+}_{\dR}$ introduced in Lemma \ref{lem:inherited bi-cosimplicial structure}.
 \end{lem}
 \begin{proof}
 The proof is standard (and tedious), and is postponed to the end of this section.
 \end{proof}

Thanks to the above Lemma \ref{lem: nablaast perserve bicosimp}, now we do have a ``de Rham complex for the bi-cosimplicial module'' $$(M\otimes_{\frakS(R)_{\dR}^+,q_0} \Omega^{\star}_{\frakS^{\bullet}(R)^{\clubsuit,+}_{\dR}/\frakS^{\bullet,+}_{\dR}},\nabla_{\star}^{\bullet,\clubsuit}).$$
Now, we construct some morphisms of complexes.
\begin{construction}\label{construction:several rhos}
    \begin{enumerate}
        \item[(1)] Let $\clubsuit = 0$ in $(M\otimes_{\frakS(R)_{\dR}^+,q_0} \Omega^{\star}_{\frakS^{\bullet}(R)^{\clubsuit,+}_{\dR}/\frakS^{\bullet,+}_{\dR}},\nabla_{\star}^{\bullet,\clubsuit})$, then we get
        \[(M\otimes_{\frakS(R)_{\dR}^+,q_0} \Omega^{\star}_{\frakS^{\bullet}(R)^{0,+}_{\dR}/\frakS^{\bullet,+}_{\dR}},\nabla_{\star}^{\bullet,0}) = (M\otimes_{\frakS(R)_{\dR}^+,q^a_0} \Omega^{\star}_{\frakS_{\dR}^{\bullet,+}(R)/\frakS^{\bullet,+}_{\dR}}\{-\star\},\nabla_{\star}^{\bullet,0})\]
        by the definition of $\Omega^{\star}_{\frakS^{\bullet}(R)^{\clubsuit,+}_{\dR}/\frakS^{\bullet,+}_{\dR}}$.
        Let $\rho^{\clubsuit = 0}$ be the natural inclusion
        \[\rho^{\clubsuit = 0}:(M\otimes_{\frakS(R)_{\dR}^+,q^a_0} \Omega^{\star}_{\frakS^{\bullet,+}_{\dR}(R)/\frakS^{\bullet,+}_{\dR}}\{-\star\},\nabla_{\star}^{\bullet,0})\hookrightarrow(M\otimes_{\frakS(R)_{\dR}^+,q_0} \Omega^{\star}_{\frakS^{\bullet}(R)^{\clubsuit,+}_{\dR}/\frakS^{\bullet,+}_{\dR}},\nabla_{\star}^{\bullet,\clubsuit}).\]

        \item[(2)]  Consider the bi-complex $\rD\rR(M,\nabla,\phi)=:\rD\rR^{\bullet,\star}$
        
        \[\begin{tikzcd}
          {\rD\rR^{0,\star}:} \arrow[d, "{\phi-\star\cdot\id}"'] & M \arrow[d, "\phi"'] \arrow[r, "\nabla"] & M\otimes_{\frakS_{\dR}^+(R)}\Omega^1_{\frakS_{\dR}^+(R)/\frakS^{+}_{\dR}}\{-1\} \arrow[d, "\phi-\cdot\id"'] \arrow[r, "\nabla"] & M\otimes_{\frakS_{\dR}^+(R)}\Omega^2_{\frakS_{\dR}^+(R)/\frakS^{+}_{\dR}}\{-2\} \arrow[d, "\phi-2\cdot\id"'] \arrow[r] & \cdots \\
          {\rD\rR^{1,\star}:}                                             & M \arrow[r, "\nabla"]                    & M\otimes_{\frakS_{\dR}^+(R)}\Omega^1_{\frakS_{\dR}^+(R)/\frakS^{+}_{\dR}}\{-1\} \arrow[r, "\nabla"]                              & M\otimes_{\frakS_{\dR}^+(R)}\Omega^2_{\frakS_{\dR}^+(R)/\frakS^{+}_{\dR}}\{-2\} \arrow[r]                               & \cdots
        \end{tikzcd}\]
        Let 
        \[H(X,Y) = \frac{(1+aX)^{Y}}{X} = \sum_{n\geq 1}\big(a^n\prod_{i=1}^{n-1}(Y-i)\big)X^{[n]}.\]
        Then we define 
        \[\rho_{\rm GMW}:\rD\rR^{\bullet,\star}\to (M\otimes_{\frakS(R)_{\dR}^+,q^a_0} \Omega^{\star}_{\frakS^{\bullet,+}_{\ast,\dR}(R)/\frakS^{\bullet,+}_{\ast,\dR}}\{-\star\},\nabla_{\star}^{\bullet,0})\]
        such that for each fixed $\star\geq 0$, it reduces to 
        \begin{equation*}
            \begin{tikzcd}
              M\otimes_{\frakS_{\dR}^+(R)}\Omega^{\star}_{\frakS_{\dR}^+(R)/\frakS^{+}_{\dR}}\{-\star\} \arrow[r, "a(\phi-\star\cdot\id)"] \arrow[d, "\id"'] & M\otimes_{\frakS_{\dR}^+(R)}\Omega^{\star}_{\frakS_{\dR}^+(R)/\frakS^{+}_{\dR}}\{-\star\} \arrow[d, "{H(X_1,\phi-\star\cdot\id)}"']                               &  &        \\
              {M\otimes_{\frakS(R)_{\dR}^+,q^a_0} \Omega^{\star}_{\frakS^{0,+}_{\dR}(R)/\frakS^{0,+}_{\dR}}\{-\star\}} \arrow[r, "\rd^0"]          & {M\otimes_{\frakS(R)_{\dR}^+,q^a_0} \Omega^{\star}_{\frakS^{1,+}_{\dR}(R)/\frakS^{1,+}_{\dR}}\{-\star\}} \arrow[r, "\rd^1"] & {M\otimes_{\frakS(R)_{\dR}^+,q^a_0} \Omega^{\star}_{\frakS^{2,+}_{\dR}(R)/\frakS^{2,+}_{\dR}}\{-\star\}} \arrow[r, "\rd^2"] & \cdots
           \end{tikzcd}
        \end{equation*}
        where $\rd^i$ on the second line above denotes the induced differential from the cosimplicial structure on $M\otimes_{\frakS(R)_{\dR}^+,q^a_0} \Omega^{\star}_{\frakS^{\bullet,+}_{\dR}(R)/\frakS^{\bullet,+}_{\dR}}\{-\star\}$. 
        The same argument in \cite[Const. 6.5]{GMWdR} shows that $\rho_{\rm GMW}$ is well-defined; that is, it does be a morphism of bicomplexes.

        \item[(3)] Let $\star = 0$ in $(M\otimes_{\frakS(R)_{\dR}^+,q_0} \Omega^{\star}_{\frakS^{\bullet}(R)^{\clubsuit,+}_{\dR}/\frakS^{\bullet,+}_{\dR}},\nabla_{\star}^{\bullet,\clubsuit})$, then we get the bi-cosimplicial $\frakS^{\bullet}(R)^{\clubsuit,+}_{\dR}$-module 
        $M\otimes_{\frakS(R)_{\dR}^+,q_0}\frakS^{\bullet}(R)^{\clubsuit,+}_{\dR}$ studied in Lemma \ref{lem:inherited bi-cosimplicial structure}, which is exactly isomorphic to $\bM(\frakS^{\bullet}(R)^{\clubsuit},(E))$.
        So the natural inclusion 
        \[M\otimes_{\frakS(R)_{\dR}^+,q_0}\frakS^{\bullet}(R)^{\clubsuit,+}_{\dR}\hookrightarrow M\otimes_{\frakS(R)_{\dR}^+,q_0} \Omega^{\star}_{\frakS^{\bullet}(R)^{\clubsuit,+}_{\dR}/\frakS^{\bullet,+}_{\dR}}\]
        induces a morphism
        \[\rho^{\star = 0}:\bM(\frakS^{\bullet}(R)^{\clubsuit},(E))\to M\otimes_{\frakS(R)_{\dR}^+,q_0} \Omega^{\star}_{\frakS^{\bullet}(R)^{\clubsuit,+}_{\dR}/\frakS^{\bullet,+}_{\dR}}.\]

        \item[(4)] As $(\frakS(R)^{\bullet},(E))$ is the diagonal of $(\frakS^{\bullet}(R)^{\clubsuit},(E))$ by Lemma \ref{lem:diagonal}, we obtain a natural inclusion
        \[\rho_{\rm diag}:\bM(\frakS(R)^{\bullet},(E))\hookrightarrow \bM(\frakS^{\bullet}(R)^{\clubsuit},(E)).\]
    \end{enumerate}
\end{construction}

\begin{rmk} \label{rem:log is same coho}
    We point that the above arguments also work when $R$ is small semi-stable after replacing everything by its logarithmic analogue. For example, in Construction \ref{construction:bi-cosimplicial prism}, the only difference is that we replace 
\[ \text{    
    $(\frakS^{\bullet},(E))\in(\calO_K)_{\Prism}$, \quad  $R_{\frakS^{\bullet}/E}$, \quad and $(\frakS^{\bullet}(R)^{\clubsuit},(E))\in(R_{\frakS^{\bullet}/E}/(\frakS^{\bullet},(E)))_{\Prism}$ etc.} \]
    by
\[ \text{     
     $(\frakS^{\bullet}_{\log},(E))\in(\calO_K)_{\Prism,\log}$,\quad  $R_{\frakS_{\log}^{\bullet}/E}$, \quad  and $(\frakS_{\log}^{\bullet}(R)^{\clubsuit},(E))\in(R_{\frakS_{\log}^{\bullet}/E}/(\frakS_{\log}^{\bullet},(E)))_{\Prism,\log}$ respectively.}\]  Similar to Lemma \ref{lem:diagonal}, by the same proof, we still have that the diagonal of $(\frakS^{\bullet}(R)^{\clubsuit},(E))$ coincides with the \v Cech nerve $(\frakS(R)^{\bullet}_{\log},(E))\in (R)_{\Prism,\log}$ associated to $(\frakS(R),(E))\in (R)_{\Prism,\log}$ such that Lemma \ref{lem:structure of bi-cosimplicial ring} again holds true for $a = -\pi E'(\pi)$. Then everything above still holds true if one replaces 
    \[\frakS^{\bullet}(R)^{\clubsuit,+}_{\dR}, \quad \frakS^{\bullet,+}_{\dR} \text{ and } \quad \Omega^{\star}_{\frakS^{\bullet}(R)^{\clubsuit,+}_{\dR}/ \frakS^{\bullet,+}_{\dR}}\] by
    \[\frakS_{\log}^{\bullet}(R)^{\clubsuit,+}_{\dR} = \Prism_{\dR}^+(\frakS^{\bullet}_{\log}(R)^{\clubsuit},(E)), \quad \frakS^{\bullet,+}_{\log,\dR}= \Prism_{\dR}^+(\frakS^{\bullet}_{\log},(E)) \text{ and  }\quad \Omega^{\star,\log}_{\frakS_{\log}^{\bullet}(R)^{\clubsuit,+}_{\dR}/ \frakS^{\bullet,+}_{\log,\dR}}, \text{ respectively}.\]
\end{rmk}

The following is our main theorem in this section. 
\begin{thm}\label{thm:pris coho vs enhanced coho}
Let $\ast\in\{\emptyset, \log\}$. 
Let $\bm \in \Vect((R)_{\Prism,\ast},\Prism_{\dR,m}^+)$, with corresponding 
 $(M,\nabla,\phi) \in \mic_\en^a(R_\gsdrm)$. Then there exists an explicit quasi-isomorphism
    \[\rD\rR(M,\nabla,\phi)\simeq \rR\Gamma((R)_{\Prism,\ast},\bM)\]
    which is functorial in $\bM$.
\end{thm}
\begin{proof}
As noted in Rem \ref{rem:log is same coho}, we only treat the prismatic case. 
   The morphisms in Construction \ref{construction:several rhos} induce the following morphisms of complexes:
    \[\begin{split}
        \rD\rR(M,\nabla,\phi)&\xrightarrow{\rho_{\rm GMW}}{\rm Tot}(M\otimes_{\frakS(R)_{\dR}^+,q^a_0} \Omega^{\star}_{\frakS^{\bullet,+}_{\ast,\dR}(R)/\frakS^{\bullet,+}_{\ast,\dR}}\{-\star\},\nabla_{\star}^{\bullet,0})\\
        &\xrightarrow{\rho^{\clubsuit = 0}} {\rm Tot}(M\otimes_{\frakS(R)_{\dR}^+,q_0} \Omega^{\star}_{\frakS_{\ast}^{\bullet}(R)^{\clubsuit,+}_{\dR}/\frakS^{\bullet,+}_{\ast,\dR}},\nabla_{\star}^{\bullet,\clubsuit})\\
        &\xleftarrow{\rho^{\star = 0}}\bM(\frakS_{\ast}^{\bullet}(R)^{\clubsuit},(E))\\
        &\xleftarrow{\rho_{\rm diag}}\bM(\frakS(R)_{\ast}^{\bullet},(E)),
    \end{split}\]
    where ${\rm Tot}(X)$ denotes the total complex of $X$. 
    As we have a quasi-isomorphism 
    \[\bM(\frakS(R)^{\bullet},(E))\simeq\rR\Gamma((R)_{\Prism,\ast},\bM),\]
    it suffices to show that $\rho_{\rm GMW}$, $\rho^{\clubsuit = 0}$, $\rho^{\star = 0}$ and $\rho_{\rm diag}$ are all quasi-isomorphisms.

    For $\rho_{\rm GMW}$: By a standard spectral sequence argument, it suffices to show that it induces a quasi-isomorphism for each fixed $\star$. But this follows from the same argument in the proof of \cite[Prop. 6.6]{GMWdR}.

    For $\rho^{\clubsuit = 0}$ and $\rho^{\star = 0}$: As above, it suffices to show that they induce quasi-isomorphisms for each fixed $\bullet$. But this has been explained in the proof of Theorem \ref{Thm-dRCrystalasXiConnection-Rel}.

    It remains to show that $\rho_{\rm diag}$ is again a quasi-isomorphism. But this is well-known as the (cosimplicial) Eilenberg-Zilber Theorem.
\end{proof}

 \begin{proof}[Proof of Lem. \ref{lem: nablaast perserve bicosimp}]
     Note that for any fixed $\bullet$, Construction \ref{construction:tri-complex} reduces to Construction \ref{Construction-Bicosimplicial-Rel} with $(A,I) = (\frakS^{\bullet},(E))$. So we see that $\nabla^{\bullet,\clubsuit}_{\star+1}\circ\nabla^{\bullet,\clubsuit}_{\star} = 0$ and that $\nabla_{\star}^{\bullet,\clubsuit}$ commutes with $p_i^g$ and $\sigma_i^g$, by Lemma \ref{Lem-PreserveCosimplicial-Rel}.
     
     It remains to prove that $\nabla_{\star}^{\bullet,\clubsuit}$ commutes with $p_i^a$ and $\sigma_i^a$. Before this, let us recall that for any $x\in M$, we have
     \[\nabla(x) = \sum_{r=1}^d\nabla_r(x)\otimes \frac{\dlog T_i}{E(u)}.\]
     Then for any $x\in M$ and for any $i\geq 0$, we have that
     \[\sigma_i^a(\nabla(x))  = \sum_{r=1}^d\sigma_i^a(\nabla_r(x))\otimes \frac{\dlog T_i}{\sigma^a_i(E(u))} = \sum_{r=1}^d\nabla_r(x)\otimes \frac{\dlog T_i}{E(u)} = \nabla(x).\]
     For the face maps, when $i\geq 1$, we have
     \[p_i^a(\nabla(x)) = \sum_{r=1}^dp_i^a(\nabla_r(x))\otimes \frac{\dlog T_i}{p_i^a(E(u))}
     = \sum_{r=1}^d\nabla_r(x)\otimes \frac{\dlog T_i}{E(u)} = \nabla(x)\]
     and when $ i = 0$, we have
     \[\begin{split}p_0^a(\nabla(x))
     =& \sum_{r=1}^dp_0^a(\nabla_r(x))\otimes \frac{\dlog T_i}{p_0^a(E(u))}\\
     =&\sum_{r=1}^d(1+aX_1)^{\phi}\nabla_r(x)\otimes \frac{\dlog T_i}{E(u_1)}\\
     =&(1+aX_1)^{\phi-1}\sum_{r=1}^d\nabla_r(x)\otimes \frac{\dlog T_i}{E(u)}\\
     =&\sum_{r=1}^d\nabla_r((1+aX_1)^{\phi}(x))\otimes \frac{\dlog T_i}{E(u)}\\
     =& \sum_{r=1}^d\nabla_r(p_0^a(x))\otimes \frac{\dlog T_i}{E(u)}\\
     =& \nabla(p_0^a(x)),\end{split}\]
     where the last second equality follows from that $[\phi,\nabla_r] = \nabla_r$ for any $1\leq r\leq d$.

     We first check that $\nabla_{\star}^{\bullet,\clubsuit}$ commutes with $\sigma_i^a$: For any $x\in M$ and $\omega\in \Omega^{\star}_{\frakS_{\ast}^{\bullet}(R)^{\clubsuit,+}_{\dR}/\frakS^{\bullet,+}_{\ast,\dR}}$, we have
     \begin{equation*}
         \begin{split}
             \nabla_{\star}^{\bullet,\clubsuit}(\sigma_i^a(x\otimes\omega)) & = \nabla_{\star}^{\bullet,\clubsuit}(x\otimes\sigma_i^a(\omega))\\
             & = \nabla(x)\wedge\sigma_i^a(\omega)+x\otimes\rd^{\bullet,\clubsuit}(\sigma_i^a(\omega))\\
             &= \sigma_i^a(\nabla(x))\wedge\sigma_i^a(\omega)+\sigma_i^a(x)\otimes\sigma_i^a(\rd^{\bullet,\clubsuit}(\omega))\\
             & = \sigma_i^a(\nabla_{\star}^{\bullet,\clubsuit}(x\otimes\omega)).
         \end{split}
     \end{equation*}
     For the face maps, the same calculation as above shows that $\nabla_{\star}^{\bullet,\clubsuit}$ commutes with $p_i^a$ when $i\geq 1$. It remains to show it also commutes with $p_i^0$. Indeed, in this case, we have
     \begin{equation*}
         \begin{split}
             \nabla_{\star}^{\bullet,\clubsuit}(p_0^a(x\otimes\omega)) & = \nabla_{\star}^{\bullet,\clubsuit}(p^a_0(x)\otimes p_0^a(\omega))\\
             &= \nabla(p_0^a(x))\wedge p_0^a(\omega)+p_0^a(x)\otimes\rd^{\bullet,\clubsuit}(p_0^a(\omega))\\
             &= p_0^a(\nabla(x))\wedge p_0^a(\omega)+p_0^a(x)\otimes p_0^a(\rd^{\bullet,\clubsuit}(\omega))\\
             & = p_0^a(\nabla_{\star}^{\bullet,\clubsuit}(x\otimes\omega))
         \end{split}
     \end{equation*}
     as desired. The proof is completed.
 \end{proof}

 \newpage
\addtocontents{toc}{\ghblue{Geometric RH correspondence  } } 
\section{Generalities on $\bbdrplus$-local systems} \label{sec: bdr loc sys}
 Let $X$ be a smooth rigid analytic space over $K$.
In this section, we prove some generality  results (trivializations, decompletions, etc.~) on  the category $\vect(X_\proet, \bbdrplusm)$ of  \emph{$\bbdrplusm$-local systems} on the pro-\'etale site of $X$.  The results are used in next \S \ref{sec: global RH}  to construct  a $p$-adic  Riemann--Hilbert functor.

\subsection{$v$-descent of $\bbdrplus$-local systems}

  We begin with a well-known lemma.
    \begin{lem}\label{lem:projectivity}
      Let $B$ be a ring with a non-zero divisor $t$ and $B_m:=B/t^m$ for any $m\geq 1$. 
      
      \begin{enumerate}
          \item Let $M$ be a $B_m$-module. If $M/tM$ is a finite projective $B_1$-module of rank $r$ and for any $0\leq n\leq m-1$, the multiplication by $t^n$ induces an isomorphism of $B_1$-modules 
          \[M/tM\xrightarrow{\cong} t^nM/t^{n+1}M,\]
          then $M$ is a finite projective $B_m$-module of rank $r$.

          \item Assume moreover $B$ is $t$-adically complete. Let $M$ be a $B$-module such that $M/tM$ is a finite projective $B_1$-module of rank $r$ and for any $n\geq 0$, the multiplication by $t^n$ induces an isomorphism of $B_1$-modules 
          \[M/tM\xrightarrow{\cong} t^nM/t^{n+1}M,\]
          then $M$ is a finite projective $B$-module of rank $r$.
      \end{enumerate}
  \end{lem}
  \begin{proof}
      The Item (2) is a consequence of Item (1) together with \cite[Lem. 1.9]{MT}. So it is enough to prove Item (1). By Zariski descent, we may assume $M/tM$ is a finite free $B_1$-module of rank $r$ and are reduced to the case to show $M$ is a finite free $B_m$-module of rank $r$. Pick a $B_1$-basis $e_1,\dots,e_r$ of $M/tM$ and let $\widetilde e_1,\dots,\widetilde e_r$ be any lifting of them in $M$. By induction on $m$, one can conclude by showing $\widetilde e_1,\dots,\widetilde e_r$ form a $B_m$-basis of $M$.
  \end{proof}

 \begin{lem}\label{lem:kedlaya-Liu} 
     Let $Y = \Spa(U,U^+)$ be an affinoid perfectoid space over $K$. For   $* \in\{{\rm an}, \et, \proet, {\rm v}\}$, let $Y_*$ be the associated site (that is, analytic, \'etale, pro-\'etale and $v$ site). Then the rule $P\mapsto P\otimes_{\bbdrplusm(U,U^+)}\bbdrplusm$ induces an equivalence of categories:
     \[ \vect(\bbdrplusm(U,U^+)) \simeq \vect(Y_*, \bbdrplusm) \]
      A quasi-inverse is given by   taking global sections on $Y$.
 \end{lem}
 \begin{proof}  
When  $m = 1$, this is  exactly \cite[Th. 3.5.8]{KL2}; for the general case, we use induction on $m$. Denote  $B_m:=\bbdrplusm(U,U^+)$ for short.
     For any finite projective $B_m$-module $P$, we have a short exact sequence
     \[0\to P/t^{m-1}\xrightarrow{\times t} P\to P/t\to 0\]
     and thus a short exact sequence of presheaves
     \[0\to P/t^{m-1}\otimes_{B_{m-1}}\bB_{\dR,m-1}^+\xrightarrow{\times t} P\otimes_{B_m}\bbdrplusm\to P/t\otimes_{B_1}\bB_{\dR,1}^+\to 0.\]
     By inductive hypothesis, we see that both $P/t^{m-1}\otimes_{B_{m-1}}\bB_{\dR,m-1}^+$ and $P/t\otimes_{B_{1}}\bB_{\dR,1}^+$ are sheaves on $Y_*$, and thus so is $P\otimes_{B_m}\bbdrplusm$. Let $Y^{\prime} = \Spa(S,S^+)\to Y$ be a covering in $Y_*$ such that $P/t\otimes_RS$ is finite free over $S$. As $Y$ is also affinoid perfectoid, we have $P\otimes_{B_m}\bbdrplusm(S,S^+)$ is finite free over $\bbdrplusm(S,S^+)$. So we conclude $P\otimes_{B_m}\bbdrplusm$ is indeed locally finite free. 

     The full faithfulness of the functor $P\mapsto P\otimes_{B_m}\bbdrplusm$ follows as 
     \[\rH^i(Y_*,\bbdrplusm) = \left\{
       \begin{array}{rcl}
           B_m, & i = 0 \\
           0, & i\geq 1.
       \end{array}
     \right.\]
     Indeed, when $* = \rm v$, the above follows from \cite[Cor. 3.5.6]{KL2} together with the standard d\'evissage argument. For general $*$, let $\nu:Y_{\rm v}\to Y_*$ be the morphism of sites. As $Y_*$ admits a basis of affinoid perfectoids, the $\rR\nu_*\bbdrplusm$ is discrete and $\nu_*\bbdrplusm$ is exactly the corresponding $\bbdrplusm$ on $Y_*$. Therefore, we have
     \[\rH^i(Y_{\proet},\bbdrplusm) = \rH^i(Y_*,\bbdrplusm)\]
     as desired.

     More generally, by \cite[Cor. 3.5.6 and Th. 3.5.8]{KL2}, for any locally finite free $\bB_{\dR,1}^+$-module $\bM$ on $Y_{\rm v}$, we have
     \[\rH^i(Y_{*},\bM) = \left\{
       \begin{array}{rcl}
           \bM(Y), & i = 0 \\
           0, & i\geq 1
       \end{array}
     \right.\]
     for $* = \rm v$. By induction on $m$, the above always holds true for any locally finite free $\bB_{\dR,m}^+$-module on $Y_{\rm v}$. As argued in the constant sheaf case, we know the above also holds true for any $*\in \{{\rm an}, \et, \proet, {\rm v}\}$. In particular, for any locally finite free $\bB_{\dR,m}^+$-module $\bM$ on $Y_*$, we have
     a short exact sequence 
     \[0\to \bM/t^{m-1}(Y)\xrightarrow{\times t}\bM(Y)\to \bM/t(Y)\to 0\]
     and thus a short exact sequence 
     \[0\to \bM/t^{m-1}(Y)\otimes_{B_{m-1}}\bB_{\dR,m-1}^+\xrightarrow{\times t}\bM(Y)\otimes_{B_m}\bbdrplusm\to \bM/t(Y)\otimes_{B_1}\bB_{\dR,1}^+\to 0.\]
     We claim the canonical morphism
     \[\bM(Y)\otimes_{B_m}\bbdrplusm\to\bM\]
     is an isomorphism. Indeed, by induction on $m$, we are reduced to the case for $m=1$, which is proved by \cite[Th. 3.5.8]{KL2}. 
     Finally, we show that $\bM(Y)$ is a finite projective $B_m$-module. Recall when $m=1$, this again follows from \cite[Th. 3.5.8]{KL2}. Now, Lemma \ref{lem:projectivity} applies.

 \end{proof}

 Thanks to the following lemma, one can always regard $\bbdrplusm$-local systems on $X_{\proet}$ as $\bbdrplusm$-crystals on $X_{{\rm aff},\perf,{\rm v}}$.

 \begin{lem}\label{cor:kedlaya-liu}
     Let $X$ be a smooth rigid analytic space over $K$ (that is not necessarily a perfectoid). For $*\in\{\proet,{\rm v}\}$, let $X_{{\rm aff},\perf,*}$ be the full subcategory of $X_*$ consisting of affinoid perfectoids.
     \begin{enumerate}
      \item There are  equivalences of categories
      \[\vect(X_v, \bbdrplusm) \simeq \vect(X_\proet, \bbdrplusm) \simeq \vect(X_{{\rm aff},\perf,*}, \bbdrplusm).\]

         \item Let $\mu:X_{\rm v}\to X_{\proet}$ denote  the canonical morphism of sites. Let $\bm \in \vect(X_v, \bbdrplusm)$. Then  \[\rR^{\geq 1}\mu_*\bM = 0.\]
         As a consequence, if we still denote $\mu_*\bM$ by $\bM$, then there is a quasi-isomorphism
         \[\rR\Gamma(X_v,\bM)\simeq \rR\Gamma(X_{\proet},\bM).\] 
     \end{enumerate} 
 \end{lem}
 \begin{proof}
 Note that $X_{{\rm aff},\perf,*}$ form a base for $X_*$, thus everything follows from Lemma \ref{lem:kedlaya-Liu}.
 \end{proof}

\subsection{Local description as representations}

\begin{notation} \label{nota: small affinoid}  
Let $X = \Spa(R_K,R)$ be an affinoid space admitting a toric chart
  \[\Box: X\to \Spa(K\za T_1^{\pm 1},\dots, T_{d}^{\pm 1}\ya,\calO_K\za T_1^{\pm 1},\dots,T_d^{\pm 1}\ya).\]
    Namely, $\Box$ is a standard \'etale morphism as in \cite{LZ17}: that is,  a composition of rational localisations and finite \'etale morphisms. (Here, the notation $R$ is ``consistent" with previous sections where $R$ is always an integral ring.)
  Let $X_C:=\Spa(R_C,R_C^+)$ be the base-change of $X$ along $K\to C$. 
  Let $X_{\infty} = \Spa(\widehat R_{\infty},\widehat R_{\infty}^+)$ be the base-change of $X$ along
  \[\Spa(C\za T_1^{\pm 1/p^{\infty}},\dots, T_{d}^{\pm 1/p^{\infty}}\ya,\calO_C\za T_1^{\pm 1/p^{\infty}},\dots,T_d^{\pm 1/p^{\infty}}\ya)\to \Spa(K\za T_1^{\pm 1},\dots, T_{d}^{\pm 1}\ya,\calO_K\za T_1^{\pm 1},\dots,T_d^{\pm 1}\ya);\]
  here, we tacitly fix a system of $p$-power roots  $T_i^{1/p^n}$ of $T_i$.  
  We obtain the Galois pro-\'etale coverings $X_{\infty}\to X_C\to X$, which induces   an exact sequence of Galois groups
  \[0\to \Gamma_{\geo}\to\Gamma\to G_K\to 1.\]
  Moreover, we have that
  \[\Gamma_{\geo} = \Zp\gamma_1\oplus\cdots\oplus\Zp\gamma_d\]
  where $\gamma_i$ is determined by that for any $1\leq j\leq d$ and any $n\geq 0$, 
  \[\gamma_i(T_j^{1/p^n}) = \left\{\begin{array}{rcl}
     \zeta_{p^n}T_j^{1/p^n},  & i = j \\
      T_j^{1/p^n}, & i\neq j.
  \end{array}\right.\]
  We also have
  \[\Gamma\cong \Gamma_{\geo}\rtimes G_K\]
  such that $g\gamma_ig^{-1} = \gamma_i^{\chi(g)}$ for   $g\in G_K$ and   $1\leq i\leq d$.
\end{notation}

 \begin{notation}\label{nota:Gamma-decompositon}
 Let $[T_j^{\flat}] \in \bbdrplus(\widehat{R}_\infty)$ denotes Teichim\"uller lifting of $T_j = (T_j,T_j^{1/p},\dots)\in \widehat R_{\infty}^{\flat}$. For any $1\leq i,j\leq d$ and any $n\geq 0$, we have
 \[\gamma_i([T_j^{\flat}]^{1/p^n}) = [\epsilon^{\frac{1}{p^n}}]^{\delta_{ij}}[T_j^{\flat}]^{1/p^n}.\]
 Write
 \[ R_{\bdrplusm}:=R_K\hatotimes_K \bdrplusm.\]
 There is a $\bdrplusm$-linear embedding 
 \[ R_{\bdrplusm} \to \bbdrplusm(\widehat{R}_\infty),\]
sending $T_i$ to $[T_i^{\flat}]$ such that the image is $\Gamma$-stable.  
Note this $\Gamma$-action is not trivial on $R$, in contrast to the \emph{Galois} action in Notation \ref{nota: small affinoid}, \emph{unless $m=1$}; to distinguish the two situations, denote image of above embedding as  $\bfb_m$.
The embedding ``extends" to a $\Gamma$-equivariant (topological) decomposition
 \[\bbdrplusm(X_{\infty}) = \widehat \bigoplus_{\underline \beta\in(\bN[1/p]\cap[0,1))^d} \bfb_m \cdot [\underline T^{\flat}]^{\underline \beta}.\]
 \end{notation}

 \begin{lem}\label{lem:dR-local system as Gamma-representation}
   Let $X = \Spa(R_K,R)$ be the small affinoid in Notation \ref{nota: small affinoid}.
      The evaluation functor $\mathbb W \mapsto W=\bw(X_{\infty})$ at $X_{\infty}$ induces an equivalence 
      \[\Vect(X_\proet,\bbdrplusm)\simeq \Rep_{\Gamma}(\bbdrplus(\widehat{R}_\infty)).\]
     Moreover, for the pair $\bw, W$ above, there exists a quasi-isomorphism
      \[\rR\Gamma(X_{\proet},\bw) \simeq \rR\Gamma(\Gamma,W).\]
  \end{lem}
  \begin{proof} 
      Due to Corollary \ref{cor:kedlaya-liu}, the result follows from the same argument in the proof of \cite[Lem. 3.8]{MW22}, which treated $m=1$ case.
  \end{proof}

  \subsection{Local unipotent decompletion}
 
\begin{defn} Recall we use $\bfb_m$ to denote the image of $R_\bdrplusm$ under embedding in Notation \ref{nota:Gamma-decompositon}; and recall when $m=1$, $\bfb_1=R_C$ in a ``canonical" way.
\begin{enumerate}
\item Let $\Rep^{\rm uni}_{\Gamma}(R_C)=\Rep^{\rm uni}_{\Gamma}(\bfb_1)$ denote the category  of (linear) continuous representation of $\Gamma$ on finite projective $R_C$-modules such that $\Gamma_{\geo}$ acts unipotently.

\item Let $2 \leq m \leq +\infty$. 
Denote by $\Rep^{\rm 
 uni}_{\Gamma}(\bfb_m)$ the full subcategory of $\Rep_{\Gamma}(\bfb_m)$ consisting of $V$ such that $\Gamma_{\geo}$ acts on $V/tV$ unipotently; that is,  
 \[ V/tV \in \Rep^{\rm uni}_{\Gamma}(R_C).\]

\end{enumerate} 
\end{defn}

 \begin{prop}\label{prop:decompletion} 
     The functor
      \[ V \mapsto W= V\otimes_{\bfb_m}\bbdrplusm(\widehat{R}_\infty)\]
      induces an equivalence of categories
      \[\Rep^{\rm uni}_{\Gamma}(\bfb_m) \simeq \Rep_{\Gamma}(\bbdrplusm(\widehat{R}_\infty)).\]
      For the pair $V, W$ above, we have  a $G_K$-equivariant quasi-isomorphism
      \[\rR\Gamma(\Gamma_{\geo},V)\simeq \rR\Gamma(\Gamma_{\geo},W).\]
      As a consequence, we have a quasi-isomorphism
      \[\rR\Gamma(\Gamma,V)\simeq \rR\Gamma(\Gamma,W).\]
  \end{prop}
  \begin{proof}
When $m=1$, everything is  established in \cite[Prop. 2.16 (and Rem. 2.20)]{MW22}; that is, we have an equivalence of categories 
\[ \Rep^{\rm uni}_{\Gamma}(R_C) \xrightarrow{\simeq}\Rep_{\Gamma}(\widehat R_{\infty})\]
together with cohomology comparisons.

By Lemma \ref{lem:projectivity} (2), it suffices to treat the  $m<\infty$ cases. The functor induced by $V \mapsto W$ preserves tensor products and duals; thus to see the functor is fully faithful, it reduces to prove the the natural morphism
      \[\rR\Gamma(\Gamma,V)\to \rR\Gamma(\Gamma, W)\]
      is a quasi-isomorphism; by standard d\'evissage, this reduces to the known $m=1$ case.

It remains to use an induction argument to prove the functor is essentially surjective.
  Let $W \in \Rep_{\Gamma}(\bbdrplusm(\widehat{R}_\infty))$. Then we have a short exact sequence
      \[0\to tW\to W\to W/t\to 0.\]
      Then by the $m=1$ case, there exists an $V_1\in \Rep^{\rm uni}_{\Gamma}(R_C)$ such that $V_1\otimes_{R_C}\widehat R_{\infty}\cong W/t$ satisfying 
      \[\rR\Gamma(\Gamma,V_1)\cong \rR\Gamma(\Gamma,W/t).\]
      In particular, one can regard $V_1$ as a sub-$R_C$-module of $W/t$. Denote by $\wt{V}$ the pre-image of $V_1$ under the projection $W\to W/t$, and then we get a $\Gamma$-equivariant short exact sequence
      \[0\to tW \to \wt{V} \to V_1\to 0.\]
      By inductive hypothesis, there exists an $V_2\in \Rep^{\rm uni}_{\Gamma}(\bfb_{m-1})$ such that $tW \cong V_2\otimes_{\bfb}\bbdrplusm(\widehat{R}_\infty)$ and that
      \[\rR\Gamma(\Gamma, V_2)\cong \rR\Gamma(\Gamma,tW).\]
      As $\bfb_m$ is a direct summand of $\bbdrplus(\widehat{R}_\infty)$, we have a $\Gamma$-equivariant decomposition 
      \[tW\cong V_2\oplus V_2^{\prime}.\] 
      Put $V = \wt{V}/V_2^{\prime}$, and then we get a $\Gamma$-equivariant short exact sequence
      \[0\to V_2\to V\to V_1\to 0.\]
      By construction, the following diagram is commutative
      \[\xymatrix@C=0.5cm{
      0\ar[r]& V_2\otimes_{\bfb}\bbdrplus(\widehat{R}_\infty)\ar[rr] \ar[d]&& V\otimes_{\bfb}\bbdrplus(\widehat{R}_\infty)\ar[rr]\ar[d]&& V_1\otimes_{\bfb}\bbdrplus(\widehat{R}_\infty)\ar[r]\ar[d]& 0\\
      0\ar[r]& tW\ar[rr] &&W\ar[rr]&& W/t\ar[r]& 0.
      }\]
      As both the left and the right vertical arrows are isomorphisms, so is the middle one. Finally, as 
      \[t^{n-1}W/t^{n}W\cong W/tW\otimes_{\Zp}\Zp(n-1),\]
      we have for any $1\leq n\leq m$
      \[t^{n-1}V/t^nV\cong V_1\otimes_{\Zp}\Zp(n-1).\]
      By Lemma \ref{lem:projectivity} (1), $V$ is finite projective over $\bfb_m$ as desired, yielding that $V\in \Rep^{\rm uni}_{\Gamma}(\bfb_m)$ as desired.
     \end{proof}

 \newpage
\section{Global Riemann--Hilbert correspondence for $\bbdrplus$-local systems} \label{sec: global RH}

In \cite{LZ17}, Liu--Zhu constructs a (geometric) $p$-adic  Riemann--Hilbert functor $\mathcal{RH}$ on   $\bbdrplusm$-local systems coming from   (\'etale) \emph{$\qp$-local systems}. This section is devoted to generalizing their functor to an equivalence of categories, by working on the whole category of $\bbdrplusm$-local systems. The main theorem of this section is  Theorem \ref{thm:RH correspondence}, which says that for $X$  a smooth rigid space over $K$, we have 
 \[\Vect(\xproet,\bbdrplusm) 
  \simeq \MIC_\gk(X_{C,\et},\calO_X\widehat \otimes_K\bdrplusm).\]

\subsection{$\gk$-equivariant connections}\label{ssec:G_K-equivariant connection}


Consider the ringed space  $\calX_{m}:=(X_C, \o_{X_C, \bdrplus, m})$ in Notation \ref{nota: ringed space}. By \cite[Prop. 3.3 and Cor. 3.4]{LZ17}, this ringed space admits a good theory of vector bundles (cf. \cite[Def. 3.5]{LZ17}). In particular, any vector bundle $\calf$ on $X$ induces a vector bundle $\calf\widehat \otimes_K\bdrplus/t^m$ on $\calX_{m}$ via the ``base change'' along $K\to \bdrplus/t^m$ such that the functor $\calf\mapsto \calf\widehat \otimes_K\bdrplus/t^m$ is exact.    
  In particular,   for any $i\geq 1$, denote
  \[\Omega^i_{\calX_{m}}: = \Omega^i_{X/K}\widehat \otimes_K\bdrplus/t^m.\]
Abuse notation, and denote the composite
\[ \mathrm{d}: \calO_{\calX_{m}}\to \Omega^1_{\calX_{m}} \to \Omega^1_{\calX_{m}}\{-1\}\]
where the first map is  $\bdrplusm$-base change of the usual derivation on $\calO_X$, and the right most object is twist by $(t)^{-1}$ similar as in Notation \ref{nota: BK twist rel pris}.
As $t$ is a generator of $\zp(1)$, we have 
\[ \Omega^1_{\calX_{m}}\{-1\} =\Omega^1_{\calX_{m}}(-1). \]


  \begin{defn}\label{def:t-connection} 
  \begin{enumerate}
      \item     A flat (or integrable) \emph{$t$-connection} on $\calX_{m}$ is a vector bundle $\calE$ equipped with a $\bdrplus$-linear morphism 
      \[\nabla_{\calE}:\calE\to \calE\otimes_{\calO_{\calX_m}}\Omega^1_{\calX_{m}}(-1)\]
      satisfying Leibniz rule with respect to $\mathrm{d}$ such that $\nabla_{\calE}\circ\nabla_{\calE} = 0$. 
 Denote the category of flat connections by  $\MIC(\calX_{m})$. For a flat connection,  define the  de Rham complex  as:
      \[\rD\rR(\calE,\nabla_{\calE}): \quad \calE\xrightarrow{\nabla_{\calE}}\calE\otimes_{\calO_{\calX_m}}\Omega^1_{\calX_{m}}(-1) \xrightarrow{\nabla_{\calE}}\calE\otimes_{\calO_{\calX,m}}\Omega^2_{\calX_{m}}(-2) \to \cdots.\]

       \item A flat $t$-connection is called $\gk$-equivariant, if $\calE$ is a $\gk$-equivariant bundle and that $\nabla_{\calE}$ is $\gk$-equivariant; denote the category of such objects by   $\MIC_\gk(\calX_{m})=\MIC_\gk(X_{C,\et},\calO_X\widehat \otimes_K\bdrplusm)$. 
  \end{enumerate}   
  \end{defn}

 \begin{construction}[Local $t$-connection] \label{cons:local t-connection}
Let $X = \Spa(R_K,R)$ be an affinoid with a fixed toric chart as in Notation \ref{nota: small affinoid} and keep the notations there.  
We have 
 \[\rd: R_\bdrplusm \to \Omega^1_{R_\bdrplusm}(-1)  \cong \bigoplus_{i=1}^dR_\bdrplusm\cdot\frac{\dlog T_i}{t}.\]
Thus a $t$-connection $\cm$ on $X  $ is a finite projective $R_\bdrplusm$-module together with a connection operator
\[\nabla_\cm = \sum_{i=1}^d\nabla_i\otimes\frac{\dlog T_i}{t}: \cm\to \cm\otimes_{R_\bdrplusm}\Omega^1_{R_\bdrplusm}(-1)\]
where each $\nabla_i: \cm \to \cm$ satisfies $t$-Leibniz rule (Def \ref{def: b-conn}) with respect to $\frac{d}{d\log T_i}$ (explaining the terminology ``$t$-connection").   With this explicit form, $(\cm,\nabla_\cm)$ is flat if and only if $[\nabla_i,\nabla_j] = 0$ for all $i,j$. It is $G_K$-equivariant if and only if for any $g\in G_K$ and $i$, we have
      \[g\circ\nabla_i\circ g^{-1} = \chi(g)\nabla_i.\]      
       Note that for any $G_K$-equivariant $t$-connection $(\cm,\nabla_\cm)$, its underlying $t$-connection $\nabla_\cm$ is \emph{automatically} $t$-adically topologically nilpotent: indeed, it suffices to check it  mod $t$, then $t$-connection becomes a Higgs field, which is topologically nilpotent by \cite[Rem. 3.2]{MW22}.
\end{construction}

\begin{rem}[Relation with Liu--Zhu's filtered connection]
\label{rem: fil conn}
  A benefit of Definition \ref{def:t-connection} is that it works   for any $1\leq m \leq \infty$. When $m=\infty$ and hence relevant rings are integral domains and thus we can \emph{invert $t$}, Definition \ref{def:t-connection} is equivalent to \cite[Def. 3.6]{LZ17}.
  Indeed, as in \cite[Def. 3.5]{LZ17}, one can also form a ringed space $\calx=(X_C, \o_X\hatotimes \bdr)$. By   \cite[Def. 3.6]{LZ17}, a filtered connection on $\calx$ is a filtered vector bundle $\fil^\bullet \calf$ together with a flat connection $\nabla: \calf \to \calf \otimes_{\o_\calx} \Omega^1_\calx$  satisfying Griffiths transversality; that is, for any $n\in\bZ$,
 \[\nabla(\Fil^n\calF)\subset \Fil^{n-1}\calF\otimes_{\calO_{X}}\Omega^1_{X} = \Fil^{n-1}\calF\otimes_{\calO_{\calX^+}}\Omega^1_{\calX^+}.\]
 Put $\calE = \Fil^0\calF$ and then for any $n\in\bZ$, 
 \[\Fil^n\calF = t^n\calE.\]
 Via the identification $\Omega^1_{\calX^+} = t\Omega^1_{\calX^+}(-1)$, we see that $(\calF,\nabla)$ satisfies Griffiths transversality if and only if
 \[\nabla(\calE)\subset t^{-1}\calE\otimes_{\calO_{\calX^+}}t\Omega^1_{\calX^+}(-1) = \calE\otimes_{\calO_{\calX^+}}\Omega^1_{\calX^+}(-1).\]
 In other words, $(\calE:=\Fil^0\calF,\nabla_{\calE}:=\nabla\big|_{\calE})$ defines a $t$-connection on $\calX_{\infty} = \calX^+$ per Definition \ref{def:t-connection}.
Conversely, given a $t$-connection $\cale$ on $\calX_{\infty}=\calx^+$, let $\calf=\cale[1/t]$, and let $\fil^\bullet \calf =t^\bullet \cale$. Define
\[ \nabla_\calf: \calf \xrightarrow{\nabla_\cale \otimes 1} \calf\otimes_{\o_{\calx^+}} \Omega^1_{\calx^+}(-1) \into \calf\otimes_{\o_{\calx}} \Omega^1_{\calx}.\]
Then this is a  filtered connection per \cite[Def. 3.6]{LZ17}.
  These constructions induce  an equivalence between the two relevant categories. See \cite[Rem. 3.2]{LZ17} for relevant discussion. 
  
\end{rem}

\subsection{$\BBdRp$-local systems and $t$-connections}
 
To construct $p$-adic Riemann-Hilbert correspondence, we start by reviewing $\calO\BBdR$.

\begin{construction}\label{construction:global period sheaf}
  Let $\calO\BBdR$ be the de Rham period sheaf\footnote{The period sheaf $\calO\BBdR$ in \cite{DLLZ} is indeed the completion of the de Rham period sheaf in \cite{LZ17}, also denoted by $\calO\BBdR$, with respect to the filtration. All arguments in \cite{LZ17} still hold true when replace the period sheaf in \emph{loc.cit.} by the new one in \cite{DLLZ}. See \cite[Rem. 2.2.11]{DLLZ} for a bit more discussion.} on $X_{\proet}$ introduced in \cite[Def. 2.2.10]{DLLZ}. 
  It carries an exhaustive, complete decreasing filtration $\Fil^{\bullet}$ and a $\BBdR$-linear flat connection   
  \[\nabla_{\calO\BBdR}:\calO\BBdR\to \calO\BBdR\otimes_{\nu^{-1}(\oxet)}\Omega^1_\xet  \] 
  satisfying the Griffiths transversality: for $r\in\bZ$, we have 
  \[\nabla_{\calO\BBdR}(\Fil^r\calO\BBdR)\subset \Fil^{r-1}\calO\BBdR\otimes_{\nu^{-1}(\oxet)}\Omega^1_\xet.\]
Note 
\[\calO\BBdR\otimes_{\nu^{-1}(\oxet)}\Omega^1_\xet = \calO\BBdR\otimes_{\nu^{-1}\calO_{\calX}}\Omega^1_{\calX} \] 
where now we can twist by $-1$ on the right hand side. Abuse notations and define the following  composite:
\[ \rd:\calO\BBdR\xrightarrow{\nabla_{\calO\BBdR}} \calO\BBdR\otimes_{\nu^{-1}\calO_{\calX}}\Omega^1_{\calX} \to \calO\BBdR\otimes_{\nu^{-1}\calO_{\calX}}\Omega^1_{\calX}(-1).\]
Then Griffiths transversality of $\nabla_{\calO\BBdR}$ translates into 
  \[\rd(\Fil^r\calO\BBdR)\subset \Fil^{r}\calO\BBdR\otimes_{\nu^{-1}\calO_{\calX}}\Omega^1_{\calX}(-1).\]
  For any $-\infty\leq a<b\leq +\infty$,   define 
  \[\calO\bB_{\dR}^{[a,b]}:=\Fil^a\calO\BBdR/\Fil^b\calO\BBdR.\]
  It is easy to see that $\rd$ induces a $\BBdRp/t^{b-a}$-linear flat $t$-connection
  \[\rd:\calO\bB_{\dR}^{[a,b]}\to\calO\bB_{\dR}^{[a,b]} \otimes_{\nu^{-1}\calO_{\calX}}\Omega^1_{\calX}(-1).\] 
  In particular, if we put
  \[\calO\bC := \Gr^0\calO\BBdR = \calO\bB_{\dR}^{[0,1]},\]
  then we get the Higgs field
  \[\Theta:\calO\bC\to \calO\bC\otimes\Omega^1_X(-1)\]
  as in \cite[\S 2.1]{LZ17}.
\end{construction}

 We   have the following local description of $(\calO\bB_{\dR}^{[0,\infty]},\rd)$.

\begin{construction}[Local section of $\calO\bB_{\dR}^{[0,\infty]}$]\label{lem:local description of OBdR}
      Let $X$ be an affinoid space admitting a toric chart and $X_{\infty}\to X$ be the pro-\'etale $\Gamma$-torsor defined in Notation \ref{nota: small affinoid}; also use Notation \ref{nota:Gamma-decompositon}. 
      Define  
      \[\bB_{\dR}^+\{\underline W\}:=\BBdRp\{W_1,\dots,W_d\} = \{\sum_{\underline n\in\bN^d}a_{\underline n}\underline W^{\underline n}\in \BBdRp[[W_1,\dots,W_d]]\mid a_{\underline n} \to 0, \text{ $t$-adically, as } |\underline n|\to +\infty\}.\]
      As a special case of \cite[Cor. 2.3.17]{DLLZ}, the morphism of $\bB_{\dR}^+$-algebras
      \[\bB_{\dR}^+\{\underline W\}_{|_{X_{\infty}}}\to (\calO\bB_{\dR}^{[0,\infty]})_{|_{X_{\infty}}}\]
      sending each $W_i$ to $\frac{1}{t}\log(\frac{[T_i^{\flat}]}{T_i})$ is an isomorphism. Via this isomorphism and the identification
      \[\Omega^1_{\calX}(-1)\cong \bigoplus_{i=1}^d\calO_{\calX}\cdot\frac{\dlog T_i}{t},\]
      the $\rd$ is given by
      \[\rd = \sum_{i=1}^d-\frac{\partial}{\partial W_i}\otimes\frac{\dlog T_i}{t}.\] 
  \end{construction}

  The key property $(\calO\bB_{\dR}^{[a,b]},\rd)$ is the following Poincar\'e Lemma:
  \begin{lem}[{cf. \cite[Cor. 2.4.2(3)]{DLLZ}}]\label{lem:poincare lemma}
      For any $-\infty\leq a<b\leq +\infty$, the following complex is exact:
      \[0\to\BBdRp/t^{b-a}\xrightarrow{\times t^a}\calO\bB_{\dR}^{[a,b]}\xrightarrow{\rd}\calO\bB_{\dR}^{[a,b]}\otimes\Omega^1_{\calX}(-1)\xrightarrow{\rd}\calO\bB_{\dR}^{[a,b]}\otimes\Omega^2_{\calX}(-2)\to \cdots.\]
      In particular, when $[a,b]=[0,+\infty]$, we have the following exact sequence:
      \[0\to\BBdRp\to\calO\bB_{\dR}^{[0,\infty]}\xrightarrow{\rd}\calO\bB_{\dR}^{[0,\infty]}\otimes\Omega^1_{\calX}(-1)\xrightarrow{\rd}\calO\bB_{\dR}^{[0,\infty]}\otimes\Omega^2_{\calX}(-2)\to \cdots.\]
  \end{lem}

\begin{thm}\label{thm:RH correspondence}
    Let $\nu:X_{\proet}\to X_{C,\et}$ be the natural morphism of sites.  
    \begin{enumerate}
        \item  Let $\bW$ be a $\bbdrplusm$-local system  of rank $r$. We have 
        \begin{equation} \label{eq bdr higher coho vanish}
           \rR^i\nu_*(\bW\otimes_{\BBdRp}\calO\bB_{\dR}^{[0,\infty]}) = 0, \forall i\geq 1 
        \end{equation}
          and 
        \begin{equation}\label{eq cm is vb}
            (\calM,\nabla_{\calM}):=\nu_*(\bW\otimes_{\BBdRp}\calO\bB_{\dR}^{[0,\infty]},\id_{\bW}\otimes\rd)
        \end{equation}
                gives rise to a $G_K$-equivariant flat $t$-connection of rank $r$ on $\calX_{m}$.
        
        \item Let $\bw$ and $(\calM,\nabla_{\calM})$ be as in Item (1). Then there exists a $G_K$-equivariant quasi-isomorphism
        \[\rR\nu_*(\bW)\simeq\rD\rR(\calM,\nabla_{\calM}).\]
        As a consequence, we have 
        \[\rR\Gamma(X_{\proet},\bW)\simeq \rR\Gamma(G_K,\rR\Gamma(\calX,\rD\rR(\calM,\nabla_{\calM}))).\]

        \item For any $G_K$-equivariant flat $t$-connection $(\calM,\nabla_{\calM})$ of rank $r$ on $\calX_{m}$, define
        \[\rd_{\calM}:=\nabla_{\calM}\otimes\id+\id\otimes\Theta:\calM\otimes\calO\bB_{\dR}^{[0,\infty]}\to \calM\otimes\calO\bB_{\dR}^{[0,\infty]}\otimes\Omega^1_X(-1).\]
        Then we have $\rd_{\calM}\wedge\rd_{\calM} = 0$ and moreover, the induced de Rham complex $\rD\rR(\calM\otimes\calO\bB_{\dR}^{[0,\infty]},\rd_{\calM})$ is acyclic in degree $\geq 1$. Furthermore, let  
        $$\bW:=(\calM\otimes\calO\bB_{\dR}^{[0,\infty]})^{\rd_{\calM} = 0},$$
        then it is a $\bbdrplusm$-local system of rank $r$ on $X_{\proet}$.

        \item  The functors in Items (1) and (3) induce  an equivalence of categories
        \[\Vect(\xproet,\bbdrplusm)\simeq \MIC^{G_K}_t(\calX_{m}),\]

    \end{enumerate}
\end{thm}
  
\begin{proof} 
When $m=1$, this is \cite[Th. 3.3]{MW22}. The general case follows by standard d\'evissage argument (e.g, as done in \cite{LZ17} when  $\bw$ comes from a $\qp$-local system. As a small complement to \cite{LZ17}, we should show, e.g. $\cm$ constructed in Item (1) is a \emph{vector bundle}; this is a local statement, and still reduces to $m=1$ case using Lem \ref{lem:projectivity}. 
\end{proof}
 
 \subsection{Local Riemann--Hilbert correspondence: explicit formulae} \label{subsec local RH}

Let $X = \Spa(R_K,R)$ be an affinoid space admitting a toric chart as in Notation \ref{nota: small affinoid}. We have equivalences of categories:
\begin{equation} \label{eq: 2equiv}
\Rep^{\rm uni}_{\Gamma}(\bfb_m) \simeq \Rep_{\Gamma}(\bbdrplusm(\widehat{R}_\infty))   \simeq \MIC_{G_K}(R_{\bdrplusm})  
\end{equation}
where the first equivalence is Prop \ref{prop:decompletion} and the second is local version of Thm \ref{thm:RH correspondence}. The goal of this subsection is to give \emph{explicit} formulae explicating these equivalences, which will be used in \S \ref{sec: m to cm}.

 \begin{construction}\label{construction:local period sheaf}
      Define
      \[\bfb_m\{\underline W\}:= \{\sum_{\underline n\in\bN^d}a_{\underline n}\underline W^{\underline n}\in \bfb_m[[W_1,\dots,W_d]]\mid a_{\underline n} \to 0, \text{ $t$-adically, as } |\underline n|\to +\infty\}.\]
Then
\[ \bfb_m\{\underline W\} \subset \bB_{\dR}^+(X_\infty)\{\underline W\}=(\calO\bB_{\dR}^{[0,\infty]})_{|_{X_{\infty}}} \]
and is $\Gamma$-stable. Indeed, one can check  for  $1\leq i,j\leq d$ and  $g\in G_K$, we have
      \[\gamma_j(W_i) = W_i+\delta_{ij} \text{ and }g(W_i) = \chi(g)^{-1}W_i.\]
In addition, the explicit map $\rd$ in Cons \ref{lem:local description of OBdR} is ``stable" on it, in the sense it induces
\[\rd:  \bfb_m\{\underline W\}  \to  \bfb_m\{\underline W\} \otimes \Omega^1_{\calX}(-1).\] 
Furthermore, one can (easily) check the Poincar\'e Lemma for $\calO\bB_{\dR}^{[0,\infty]}$ restricts to a  Poincar\'e Lemma for $ \bfb_m\{\underline W\}$, in the sense that we have an  exact sequence
      \[0\to\bfb_m\to \bfb_m\{\underline W\} \xrightarrow{\rd}\bfb_m\{\underline W\}\otimes_{\wt R_m}\Omega^1_{\wt R_m}(-1)\to \cdots.\] 
  \end{construction}

\begin{prop} \label{prop: local RH formula}
Let $V \mapsto W \mapsto \cm$ be a group of corresponding objects in the equivalences in \eqref{eq: 2equiv}.
\begin{enumerate}
\item We have
\[ \cm=(V \otimes_{\bfb_m} \bfb_m\{\underline W\})^{\Gamma_\geo=1} \]
and $\nabla_\cm$ is induced by $\mathrm{id}_V \otimes \rd_{\bfb_m\{\underline W\}}$.

\item The projection map 
\[  \bfb_m\{\underline W\} \to \bfb_m \]
sending each $W_i$ to $0$ induces an isomorphism
\[ \cm=(V \otimes_{\bfb_m} \bfb_m\{\underline W\})^{\Gamma_\geo=1} \simeq V. \]

\item The isomorphism $\cm \simeq V$ in above item is compatible with connections and $\gk$-actions. More precisely,  we have
\begin{enumerate}
\item for each $i$, we have a commutative diagram 
\[
\begin{tikzcd}
\cm \arrow[d, "\nabla_i"] \arrow[r, "\simeq"] & V \arrow[d, "\log \gamma_i"] \\
\cm \arrow[r, "\simeq"]                       & V                           
\end{tikzcd}
\]
here $\log \gamma_i$ is well-defined since $\gamma_i$ acts on $V$ unipotently.
\item for any $g\in \gk$, we have
\[
\begin{tikzcd}
\cm \arrow[d, "g"] \arrow[r, "\simeq"] & V \arrow[d, "g"] \\
\cm \arrow[r, "\simeq"]                       & V                           
\end{tikzcd}
\] 
\end{enumerate} 
\end{enumerate}
\end{prop}
\begin{proof}
One can always only consider $m <\infty$ case. 
By Thm \ref{thm:RH correspondence}, we have
\[ \cm= (W \otimes_{\bbdrplusm(\widehat{R}_\infty)} \bbdrplusm(\widehat{R}_\infty)\{\underline W\})^{\Gamma_\geo=1}\]
As $t$ is nilpotent in $\bfB_m$, the algebra 
    \[\bfB_{m}\{\underline W\} = \bfB_{m}[\underline W] = \bfB_{m}[W_1,\dots,W_d]\]
    is actually a polynomial algebra over $\bfB_{m}$.   By the same argument as in the paragraph after \cite[Lem. 6.17]{Sch13}, the natural morphism
    \[\rR\Gamma(\Gamma_{\geo},V\otimes_{\bfB_m}\bfB_m\{\underline W\})\to \rR\Gamma(\Gamma_{\geo},W \otimes_{\bB_{\dR}^+(X_\infty)} \bB_{\dR}^+(X_\infty)\{\underline W\})\]
    is a quasi-isomorphism, which in particular concludes Item (1). 

Item (2). Consider the   map
\[ \cm=(V \otimes_{\bfb_m} \bfb_m\{\underline W\})^{\Gamma_\geo=1} \to  V \]
where the second map is induced by projection. 
Define a map
\[ s: V \to V \otimes_{\bfb_m} \bfb_m\{\underline W\}\]
where 
\[s(v)=(\prod_{i=1}^d \gamma_i^{-W_i})v \in V \otimes_{\bfb_m} \bfb_m\{\underline W\}  \]
where the right hand side stands for
   \[\begin{split}          & \left(\sum_{n_1,\dots,n_d\geq 0}\binom{W_1}{n_1}\cdots\binom{W_d}{n_d}(\gamma_1^{-1}-1)^{n_1}\cdots(\gamma_d^{-1}-1)^{n_d}\right)(v)\\
          =&\left(\sum_{n_1,\dots,n_d\geq 0}(-1)^{n_1+\cdots+n_d}\frac{\prod_{k_1=0}^{n_1-1}(W_1-k_1)}{n_1!}\cdots\frac{\prod_{k_d=0}^{n_d-1}(W_d-k_d)}{n_d!}(\gamma_1^{-1}-1)^{n_1}\cdots(\gamma_d^{-1}-1)^{n_d}\right)(v).\end{split}\]
This is a finite summation as $\Gamma_\geo$-action is unipotent. We claim $s(v)$ is fixed by each $\gamma_i$; it suffices to use the following formal verification:
\[ \gamma_i(s(v)) = (\prod_{i=1}^d \gamma_i^{-W_i-1})(\gamma_i(v)) =s(v)\]
which uses the fact \[\gamma_j(W_i) = W_i+\delta_{ij}.\]
Thus indeed $s$ defines a section map
\[s: \cm \to V\]
which is clearly   a bijection.

Item (3). To check compatibility of connections, it suffices to check  
\[ \nabla_i(s(v)) = s(\log \gamma_i(v)). \]
This   can be checked formally using the fact that $\nabla_i$ is induced by $-\frac{d}{d W_i}$; indeed  it suffices to consider the $i$-th component, then we check
\[ \nabla_i(\gamma_i^{-W_i})v)  
=\nabla_i(e^{-\log   \gamma_i \cdot W_i})v)  
= \log \gamma_i \cdot (e^{-\log \gamma_i \cdot W_i})v) 
=(\gamma_i^{-W_i}) (\log \gamma_i (v)) \]
To check compatibility with $\gk$-actions, it suffices to note in the expression $s(v)=(\prod_{i=1}^d \gamma_i^{-W_i})v $, each $\gamma_i^{-W_i}=e^{-\log \gamma_i \cdot W_i}$ is fixed by $\gk$. Indeed, one checks
\[ g\cdot \log \gamma_i \cdot W_i =(g\cdot \log \gamma_i g^{-1})\cdot (g W_i) =\chi(g) \cdot \log \gamma_i \cdot \chi(g)^{-1}\cdot W_i=\log \gamma_i \cdot W_i.\]
\end{proof}

 \newpage
 \addtocontents{toc}{\ghblue{Relate prismatic crystals with   RH}}

 \section{Global $\mathbbl{\Delta}_\dR^+$-crystals on the perfect site and $\bB_{\dR}^+$-local systems} \label{sec: crystal perf pris} 
 
 Let  $\frakX$ be a semi-stable formal scheme over $\calO_K$ with   rigid analytic generic fibre $X$.  Denote by $(\frakX)_{\Prism,\ast}^{\perf}$ the prismatic site of perfect prisms (resp. perfect log-prisms) for $\ast = \emptyset$ (resp. $\ast = \log$).  The main theorem of this section is the following.
 
 \begin{thm}\label{Thm-de Rham realization}
   Let $*\in \{\emptyset,\log\}$.
   There exists a natural equivalence between categories 
   \[\rL_{\ast}:\Vect((\frakX)_{\Prism,\ast}^{\perf},\Prism_{\dR,m}^+)\xrightarrow{\simeq} \Vect(\xproet,\bB_{\dR,m}^+)\]
   which preserves ranks, tensor products and duals. For any $\bM\in \Vect((\frakX)_{\Prism}^{\perf},\Prism_{\dR,m}^+)$ with associated $\bB_{\dR,m}^+$-local system $\rL(\bM)$ on $X_{\proet}$, we have a canonical quasi-isomorphism
   \[\rR\Gamma((\frakX)_{\Prism,\ast}^{\perf},\bM)\simeq \rR\Gamma(X_{\proet},\rL(\bM)).\]
 \end{thm}
 According to the lemma below, we only need to work with $(\frakX)_{\Prism}^{\perf}$. 
 \begin{lem}\label{lem:the same perfect site and crystals}
     The forgetful functor $(\frakX)_{\Prism,\log}\to (\frakX)_{\Prism}$ sending each $(A,I,M)$ to $(A,I)$ induces an equivalence of categories
     \[\Vect((\frakX)_{\Prism}^{\perf},\Prism_{\dR,m}^+)\xrightarrow{\simeq}\Vect((\frakX)_{\Prism,\log}^{\perf},\Prism_{\dR,m}^+)\]
     which is compatible with cohomology in the following sense: for   $\bM\in \Vect((\frakX)_{\Prism}^{\perf},\Prism_{\dR,m}^+)$, there exists a quasi-isomorphism
 \[\rR\Gamma((\frakX)_{\Prism}^{\perf},\bM)\simeq \rR\Gamma((\frakX)_{\Prism}^{\perf,\log},\bM).\]
 \end{lem}
 \begin{proof}
    This follows from Corollary \ref{cor:the same perfect site} immediately.
 \end{proof}

 Therefore, in order to conclude Theorem \ref{Thm-de Rham realization}, it is enough to construct an equivalence of categories
 \[\rL:\Vect((\frakX)_{\Prism}^{\perf},\Prism_{\dR,m}^+)\to\Vect(X_{{\rm aff},\perf,v},\bB_{\dR,m}^+)\]
 which is compatible with cohomologies. To do so, we shall proceed as in \cite[\S5]{MW22}.

 We first deal with a special case where $\frakX = \Spf(R)$ is small semi-stable in the sense of Definition \ref{Convention-Rel-Log}.

 \begin{construction}\label{constr:Gamma group log case}
     Suppose $\frakX = \Spf(R)$ is small semi-stable with the chart 
     \[\Box: A_K^+:=\calO_K\za T_0,\dots,T_r,T_{r+1}^{\pm 1},\dots,T_d^{\pm 1}\ya/(T_0\cdots T_r-\pi)\to R\]
     and the generic fiber $X = \Spa(R_K,R)$. For any $n\geq 0$, we put 
     \[A_{C,n}^+:=\calO_K\za T_0^{1/p^n},\dots,T_r^{1/p^n},T_{r+1}^{\pm 1/p^n},\dots,T_d^{\pm 1/p^n}\ya/(T_0^{1/p^n}\cdots T_r^{1/p^n}-\pi^{1/p^n})\]
     and $A_C^+:=A_{C,0}^+$. Then $\Spa(A_{C,n}:=A_{C,n}^+[1/p],A_{C,n}^+)\to \Spa(A_K:=A_K^+[1/p],A_K^+)$ is a pro-\'etale Galois covering for any $n\geq 1$. Denote by $\Spa(\widehat A_{\infty},\widehat A_{\infty}^+)$ the affinoid perfectoid space corresponding to $\lim_n\Spa(A_{C,n},A_{C,n}^+)$. Then we have 
     \[\widehat A_{\infty}^+ = \calO_C\za T_0^{\frac{1}{p^{\infty}}},\dots,T_r^{\frac{1}{p^{\infty}}},T_{r+1}^{\pm \frac{1}{p^{\infty}}},\dots,T_d^{\pm \frac{1}{p^{\infty}}}\ya/((T_0\cdots T_r)^{\frac{1}{p^{\infty}}}-\pi^{\frac{1}{p^{\infty}}})\]
     where $((T_0\cdots T_r)^{\frac{1}{p^{\infty}}}-\pi^{\frac{1}{p^{\infty}}})$ denotes the closed ideal of $\widehat A_{\infty}^+ = (\colim_nA_{C,n}^+)^{\wedge}_p$, the $p$-adic completion of $\colim_nA_{C,n}^+$. It is generated by $\{T_0^{1/p^n}\cdots T_r^{1/p^n}-\pi^{1/p^n}\}_{n\geq 0}$. 
     
     Denote by $X_C = \Spa(R_C,R_C^+)$ and $X_{\infty} = \Spa(\widehat R_{\infty},\widehat R_{\infty}^+)$ the base-change of $X$ along 
     \[\Spa(A_C,A_C^+)\to \Spa(A_K,A_K^+) \text{ and }\Spa(\widehat A_{\infty},\widehat A_{\infty}^+)\]
     respectively. Then $X_{\infty}\to X$ is a pro-\'etale Galois covering with Galois group $\Gamma = \Gamma_{\geo}\rtimes G_K$ with
     \[\Gamma_{\geo} = \{\delta_0^{n_0}\cdots\delta_d^{n_d}\mid n_i\in\Zp \text{ such that }n_0+\cdots+n_r = 0\},\]
     where for any $0\leq i,j\leq d$, any $n\geq 0$ and any $g\in G_K$, we have 
     \[\delta_i(T_j^{1/p^n}) = \zeta_{p^n}^{\delta_{ij}}T_j^{1/p^n} \text{ and } g\delta_ig^{-1} = \delta_i^{\chi(g)}.\]
     Put $\gamma_i = \delta_0^{-1}\delta_i$ for $1\leq i\leq r$ and $\gamma_i = \delta_i$ for $r+1\leq i\leq d$. Then we have 
     \[\Gamma_{\geo}\cong \Zp\gamma_1\oplus\cdots\oplus\Zp\gamma_d.\]
     Clearly, $\Gamma_{\geo}$ is the Galois group for the pro-\'etale torsor $X_{\infty}\to X_C$. Let
     \[T_i^{\flat}:=(T_i,T_i^{1/p},T_i^{1/p^2},\dots)\in \widehat R_{\infty}^{\flat,+}.\]
     Then we have $[T_i^{\flat}]\in \AAinf(X_{\infty})$ such that $\prod_{i=0}^r[T_i^{\flat}] = [\pi^{\flat}]$. Then the morphism $\frakS^{\Box}\to \AAinf(X_{\infty})$ carrying each $T_i$ to $[T_i^{\flat}]$ (cf. Construction \ref{constr: cover of final obj rel log pris}) is compatible with $\delta$-structures and thus defines a morphism 
     \[\frakS(R)\to \AAinf(X_{\infty})\]
     compatible with $\delta$-structure as well (cf. \cite[Lem. 2.18]{BS22}). The log-structure $M$ on $\AAinf(X_{\infty})$ coming from that on $\frakS(R)$ makes $(\AAinf(X_{\infty}),(\xi),M)$ a well-defined perfect log-prism in $(R)_{\Prism,\log}$. 
 \end{construction}
 \begin{lem}\label{lem:Fontaine prism is covering}
     Suppose that $R$ is small semi-stable over $\calO_K$.
     The perfect log-prism $(\AAinf(X_{\infty}),(\xi),M)$ is a cover  of the final object of $\Sh((R)_{\Prism,\log}^{\perf})$.
 \end{lem}
 \begin{proof}
     By Corollary \ref{cor:the same perfect site}, it suffices to show the perfect prism $(\AAinf(X_{\infty}),(\xi))$ is a covering of the final object of $\Sh((R)_{\Prism}^{\perf})$. But this follows from the same argument in the proof of \cite[Lem. 4.9(2)]{MW22}.
 \end{proof}

 Let $(\AAinf(X_{\infty})^{\bullet},(\xi))$ be the \v Cech nerve associated to the covering $(\AAinf(X_{\infty}),(\xi))$ above. Then there exists a unique morphism of cosimplicial prisms
 \[(\AAinf(X_{\infty})^{\bullet},(\xi))\to (\rC(\Gamma^{\bullet},\AAinf(X_{\infty})),(\xi))\]
 in $(\frakX)_{\Prism}^{\perf}$, where $\rC(\Gamma^{\bullet},\AAinf(X_{\infty}))$ denotes the ring of continuous functors from $\Gamma^{\bullet}$ into $\AAinf(X_{\infty})$.

 \begin{lem}\label{lem:perfect cech nerve is ring of functions}
     The morphism $(\AAinf(X_{\infty})^{\bullet},(\xi))\to (\rC(\Gamma^{\bullet},\AAinf(X_{\infty})),(\xi))$ induces an isomorphism 
     \[\Prism_{\dR,m}^+((\AAinf(X_{\infty})^{\bullet},(\xi)))\xrightarrow{\cong}\Prism_{\dR,m}^+((\rC(\Gamma^{\bullet},\AAinf(X_{\infty})),(\xi))) = \rC(\Gamma^{\bullet},\bB_{\dR,m}^+(X_{\infty})).\]
 \end{lem}
 \begin{proof}
     Since we already have a morphism $\Prism_{\dR,m}^+((\AAinf(X_{\infty})^{\bullet},(\xi)))\to\rC(\Gamma^{\bullet},\bB_{\dR,m}^+(X_{\infty}))$, it is enough to show it is an isomorphism. As both sides are $\xi$-complete, we can use derived Nakayama Lemma to reduce to the case for $m=1$. Then the result follows from the same argument in the proof of \cite[Lem. 4.11(3)]{MW22}.
 \end{proof}
  Using this, one can prove Theorem \ref{Thm-de Rham realization} in the following special (but typical) case. The general case actually follows from the same but slightly more tedious argument.
  \begin{lem}\label{lem:dR realization local case}
      Suppose $\frakX = \Spf(R)$ is small semi-stable over $\calO_K$.
      There are equivalences of categories
      \[\Vect((\frakX)_{\Prism}^{\perf},\Prism_{\dR,m}^+) \simeq \Rep_{\Gamma}(\bbdrplusm(\widehat R_{\infty})) \simeq \Vect(X_v,\bbdrplusm) \]
      which is compatible with cohomologies.
  \end{lem}
  \begin{proof}
    By Lemma \ref{lem:dR-local system as Gamma-representation}, it suffices to construct an equivalence of categories
    \[\Vect((\frakX)_{\Prism}^{\perf},\Prism_{\dR,m}^+) \simeq \Rep_{\Gamma}(\bbdrplusm(\widehat R_{\infty}))\]
    which is compatible with cohomologies. By \cite[Prop. 2.7]{BS23} together with Lemma \ref{lem:perfect cech nerve is ring of functions}, we have the equivalence of categories 
    \[\Vect((\frakX)_{\Prism}^{\perf},\Prism_{\dR,m}^+)\simeq \Strat(\Prism_{\dR,m}^+((\AAinf(X_{\infty})^{\bullet},(\xi)))) = \Strat(\rC(\Gamma^{\bullet},\bB_{\dR,m}^+(X_{\infty}))).\]
    So one can conclude form the equivalence of categories
    \[\Strat(\rC(\Gamma^{\bullet},\bB_{\dR,m}^+(X_{\infty})))\simeq\Rep_{\Gamma}(\bbdrplusm(\widehat R_{\infty}))\]
    by Galois descent.
    
    For any $\bM\in \Vect((\frakX)_{\Prism}^{\perf},\Prism_{\dR,m}^+)$, by Lemma \ref{Lem-CechAlex-Rel-I} together with the \v Cech-to-derived spectral sequence, we have a quasi-isomorphism
    \[\rR\Gamma((\frakX)_{\Prism}^{\perf},\bM)\simeq \bM((\AAinf(X_{\infty})^{\bullet},(\xi)))\]
    where $\bM((\AAinf(X_{\infty})^{\bullet},(\xi)))$ stands for the total complex induced by the evaluation of $\bM$ at $(\AAinf(X_{\infty})^{\bullet},(\xi))$. As $\bM$ is a crystal, we have a quasi-isomorphism
    \[\bM((\AAinf(X_{\infty})^{\bullet},(\xi)))\simeq \rC(\Gamma^{\bullet},\bM(\AAinf(X_{\infty}),(\xi))).\]
    One can conclude by noting that the $\Gamma$-representation of $\bM$ is exactly $\bM(\AAinf(X_{\infty}),(\xi))$ while $\rR\Gamma(\Gamma,\bM(\AAinf(X_{\infty}),(\xi)))$ is represented by $\rC(\Gamma^{\bullet},\bM(\AAinf(X_{\infty}),(\xi)))$.
  \end{proof}

 \begin{construction}\label{construction:the dR-realization functor L}
   \begin{enumerate}
       \item[(1)] By \cite[Th. 3.10]{BS22}, one can assign to each $U = \Spa(S,S^+)\in X_{{\rm aff},\perf,v}$ a unique perfect prism
     \[\calA_U:=(\AAinf(U),\Ker(\theta:\AAinf(U)\to S^+)).\]
     Noticing that that 
     \[\bB_{\dR,m}^+(U) = \Prism_{\dR,m}^+(\calA_U),\]
     one can define a functor
     \[\rL:\Vect((\frakX)_{\Prism}^{\perf},\Prism_{\dR,m}^+)\to\Vect(X_{{\rm aff},\perf,v},\bB_{\dR,m}^+)\]
     such that for any $\bM\in\Vect((\frakX)_{\Prism}^{\perf},\Prism_{\dR,m}^+)$,
     \[\rL(\bM)(U):=\bM(\calA_U).\]

       \item[(2)] For any non-crystalline perfect prism $\calA = (A,I)\in (\frakX)^{\perf}_{\Prism}$ such that $I\neq (p)$, we can assign to it an affinoid perfectoid space $U_{\calA}\in X_{{\rm aff},\perf,v}$ by letting
       \[U_{\calA} = \Spa((A/I)[\frac{1}{p}],(A/I)[\frac{1}{p}]^+)\]
       where $(A/I)[\frac{1}{p}]^+$ denotes the $p$-adic completion of the integral closure of $A/I$ in $(A/I)[\frac{1}{p}]$. Noticing that 
       \[\bB_{\dR,m}^+(U_{\calA}) = \Prism_{\dR,m}^+(\calA),\]
       we can define a functor
       \[\rL^{-1}:\Vect(X_{{\rm aff},\perf,v},\bB_{\dR,m}^+)\to\Vect((\frakX)_{\Prism}^{\perf},\Prism_{\dR,m}^+)\]
       such that for any non-crystalline perfect prism $\calA$ as above,
       \[\rL^{-1}(\bW)(\calA):=\bW(U_{\calA})\]
       while for any crystalline perfect prism $\calA$, 
       \[\rL^{-1}(\bW)(\calA) = 0.\]
   \end{enumerate}
 \end{construction}
 For example, when $\frakX = \Spf(R)$ is small semi-stable and keep notations as above, we have 
 \[\calA_{X_{\infty}} = (\AAinf(X_{\infty}),(\xi)) \text{ and }U_{(\AAinf(X_{\infty}),(\xi))} = X_{\infty}\]
 and then the equivalence in Lemma \ref{lem:dR realization local case} is exactly induced by $\rL$ and $\rL^{-1}$ above.
 \begin{proof}[\textbf{Proof of Theorem \ref{Thm-de Rham realization}}]
     We proceed as in \cite[\S 5]{MW22}. The desired equivalence of categories follows by checking that $\rL$ and $\rL^{-1}$ in Construction \ref{construction:the dR-realization functor L} are the quasi-inverses of each other. We have to prove the desired cohomological comparison.

     Let us fix an \'etale covering $\{\frakX_i = \Spf(R_i)\to \frakX\}_{i\in I}$ be small semi-stables and denote by $X_{i,\infty}\to X_i$ the corresponding pro-\'etale torsors, where $X_i$ denotes the generic fiber of $\frakX$. Put $\calA_i:=(\AAinf(X_{i,\infty}),(\xi))$. We claim that $\{\calA_i\}_{i\in I}$ forms a covering of the final object of $\Sh((\frakX)_{\Prism}^{\perf})$. 
     
     Indeed, for any perfect prism $(A,I)\in (\frakX)_{\Prism}^{\perf}$ with $\overline A:=A/I$, the base-change
     \[\Spf(\overline A)\times_{\frakX}\frakX_i\]
     is affine by the separatedness of $\frakX$, which will be referred as $\Spf(\overline A_i)$. Clearly, $\Spf(\overline A_i)\to\Spf(\overline A)$ is \'etale and thus by \cite[Cor. 2.16]{CS24}, the $\overline A_i$ is perfectoid. By \cite[Th. 3.10]{BS22}, there exists a unique perfect prism $(A_i,IA_i)\in (\frakX)_{\Prism}^{\perf}$ over $(A,I)$ such that $\overline A_i = A_i/IA_i$. The quasi-compactness of $\Spf(\overline A)$ tells us there is a finite subset $J\subset I$ such that $\{\Spf(\overline A_j)\to\Spf(\overline A)\}_{j\in J}$ forms a covering of $\Spf(\overline A)$, yielding that $(A,I)\to (\prod_{j\in J}A_j,I\cdot\prod_{j\in J}A_j)$ is a covering of $(A,I)$. Note that for any $j\in J$, $(A_j,IA_j)\in (\frakX_j)_{\Prism}^{\perf}$. By Lemma \ref{lem:Fontaine prism is covering}, there exists a covering $(A_j,IA_j)\to (B_j,IB_j)$ together with a morphism $\calA_j\to (B_j,IB_j)$, yielding a covering 
     \[(A,I)\to (\prod_{j\in J}B_j,I\cdot\prod_{j\in J}B_j)\]
     of $(A,I)$ together with a morphism $\prod_{j\in J}\calA_j\to (\prod_{j\in J}B_j,I\cdot\prod_{j\in J}B_j)$ as desired.

     For any finite subset $J=\{j_1,\dots,j_r\}\subset I$, denote by $X_{J,\infty}$ the fiber product
     \[X_{j_1,\infty}\times_XX_{j_2,\infty}\times_X\cdots\times_XX_{j_r,\infty}.\]
     The same argument for the proof \cite[Lem. 5.6]{MW22} shows that the fiber product
     \[\calA_J:=\calA_{j_1}\times\cdots\times\calA_{j_r}\]
     exists in $(\frakX)_{\Prism}^{\perf}$ such that $U_{\calA_J} = X_{J,\infty}$ (cf. Construction \ref{construction:the dR-realization functor L}(2)).

     Fix an $\bM\in \Vect((\frakX)_{\Prism}^{\perf},\Prism_{\dR,m}^+)$.
     By Lemma \ref{Lem-CechAlex-Rel-I} together with the \v Cech-to-derived spectral sequence, we see that $\rR\Gamma((\frakX)_{\Prism}^{\perf},\bM)$ is computed by the \v Cech complex
     \[\check \rC(\{\calA_i\}_{i\in I},\bM)\simeq \rR\Gamma((\frakX)_{\Prism}^{\perf},\bM).\]
     By \cite[Cor. 3.5.6 and Th. 3.5.8]{KL2} together with a standard d\'evissage argument, for any affinoid perfectoid $U\in X_{v}$, we have 
     \[\rH^i(U,\rL(\bM)) = 0\]
     for any $i\geq 1$. Using \v Cech-to-derived spectral sequence again, we see that $\rR\Gamma(X_v,\rL(\bM))$ is computed by the \v Cech complex
     \[\check \rC(\{X_{i,\infty}\to X\}_{i\in I},\bM)\simeq \rR\Gamma(X_v,\rL(\bM)).\]
     As for any finite $J\subset I$, we have $U_{\calA_J} = X_{J,\infty}$, by Construction \ref{construction:the dR-realization functor L}, the two \v Cech complexes
     \[\check \rC(\{\calA_i\}_{i\in I},\bM) = \check \rC(\{X_{i,\infty}\to X\}_{i\in I},\bM)\]
     are actually same, yielding the desired quasi-isomorphism
     \[\rR\Gamma(X_v,\rL(\bM))\simeq \rR\Gamma((\frakX)_{\Prism}^{\perf},\bM).\]
     Now one can conclude by Lemmas \ref{lem:the same perfect site and crystals} and \ref{cor:kedlaya-liu}.
 \end{proof}

 \newpage  
 \section{Local absolute $\mathbbl{\Delta}_\dR^+$-crystals and $\gk$-equivariant connections}
 \label{sec: m to cm}


This section serves as an \emph{interlude}: we summarize some functors we have constructed so far, explicate their relations in Prop \ref{prop:explicit description of Res}, and point out some remaining questions in Rem \ref{rem: donotknowyet}.

  Let $\ast \in \{\emptyset, \log \}$. Let $R$ be a small smooth (resp. semi-stable) $\ok$-algebra as in Notation \ref{nota: bk prism abs pris}. Let $\fkx=\spf(R)$ with $X$ the generic fiber. For simplicity, write  
\[ \mic_\en^a(R_\gsdrm)=\MIC_\en^a(X_\et, \o_{X, \gs_\dR^+, m}),\]
\[ \MIC_{G_K}(R_{\bdrplusm})  =\MIC_{\gk}(X_{C, \et}, \o_{X_C, \bdrplus, m}).\]
  We now have diagram
\begin{equation} \label{diag many fun local}
\begin{tikzcd}
\MIC^a_{\en}(R_\gsdrm) \arrow[d, dashed] &  & {\vect(R_{\ast,\pris}, \pris_{\mathrm{dR}}^+) } \arrow[ll, "\simeq"'] \arrow[d]                  &  &                                                                    &                                          \\
\MIC_{G_K}(R_{\bdrplusm})                &  & {\vect(R^\perf_{\ast,\pris}, \pris_{\mathrm{dR}}^+) } \arrow[rr, "\simeq"] \arrow[ll, "\simeq"'] &  & \Rep_{\Gamma}(\bbdrplusm(\widehat R_{\infty})) \arrow[r, "\simeq"] & \Rep^{\rm   uni}_{\Gamma}(\bfb_m).
\end{tikzcd}
\end{equation}
Here the top row is discussed in \S \ref{sec: loc abs pris}; the equivalences in bottom row are discussed in  \S \ref{sec: bdr loc sys} and \S \ref{sec: crystal perf pris}. The right vertical arrow is induced by restriction to the subsite; it induces the left dotted arrow.

\begin{remark} \label{rem: donotknowyet}
Before we proceed to explicate the relation of objects in these categories, let us point out what we do \emph{not} know so far.
\begin{enumerate}
\item We do not know if the   vertical arrows are fully faithful;
\item We do not know if the equivalence in top row can be \emph{glued globally}, since its construction relies on local charts and in particular choice of relative Breuil--Kisin prisms.
\end{enumerate}
 Both will be solved using an \emph{analytic} arithmetic Sen theory in the following sections.
\end{remark}
 
\begin{notation} \label{nota:group of loc obj}
Consider a group of corresponding objects from above diagram \eqref{diag many fun local}
\[
\begin{tikzcd}
M     & \bm \arrow[l, mapsto] \arrow[d, mapsto]       &             &   \\
\calm & \bm^\perf \arrow[l, mapsto] \arrow[r, mapsto] & W \arrow[r, mapsto] & V.
\end{tikzcd}
\] 
For $(M,\nabla_M,\phi)$, use the basis $\frac{\dlog T_i}{E}$ for $\Omega^1_{R_\gsdrm}\{-1\}$, and write
\[\nabla_M = \sum_{i=1}^d \nabla_{M,i, E} \otimes\frac{\dlog T_i}{E}. \]
For $(\cm, \nabla_\cm, \gk)$, use the basis $\frac{\dlog T_i}{t}$ for $ \Omega^1_{R_\bdrplusm}(-1)$, and write
\[\nabla_\cm = \sum_{i=1}^d \nabla_{\cm,i, t} \otimes\frac{\dlog T_i}{t}. \]
\end{notation}

\begin{prop} \label{prop:explicit description of Res}
 Use Notation \ref{nota:group of loc obj}.
\begin{enumerate}
\item We have 
 \[W = M\otimes_{R_\gsdr}\bdrplus(\widehat R_{\infty})\]
with semi-linear $\Gamma$-action  given by the following: for $\gamma_1^{n_1}\cdots\gamma_d^{n_d}g\in \Gamma\cong \Gamma_{\geo}\rtimes G_K$ and  $x\in M$, we have     
\[  \big(\gamma_1^{n_1}\cdots\gamma_d^{n_d}g\big)(x) = \exp(\frac{t}{E }\sum_{i=1}^d n_i\nabla_{M,i,E})(\frac{E([\epsilon]^{c(g)}u)}{E(u)})^{\frac{\phi}{a}}(x)\]
or simply
\[ \underline{\gamma}^{\underline n}g(x) = \exp(\frac{t}{E} (\underline n \cdot  \underline \nabla) (\frac{g(E)}{E})^{\frac{\phi}{a}}(x). \]


 \item As a $\Gamma$-stable  subspace of $W$, we have
 \[ V= M\otimes_{R_\gsdr} R_\bdrplus \]
 with $\Gamma$-action described in previous item.

\item \label{item loggamma} We have 
\[ \cm \simeq M\otimes_{R_\gsdrm}  R_{\bdrplusm} =M\hatotimes_\gsdr \bdrplus  \] 
such that:
\begin{enumerate}
\item  the $\gk$-action on $\cm$ is determined such that for $g\in \gk, x \in M$ 
\[g(x)   =(\frac{g(E)}{E})^{\frac{\phi}{a}}(x).  \]
\item the connection $\nabla_\cm$ is $\bdrplus$-linear extension of $\nabla_M$; more precisely, for each $i$ and $x\in M$
\[\nabla_{\cm, i, t}(x) =\frac{t}{E}\cdot \nabla_{M, i,E}(x). \]
\end{enumerate} 
\end{enumerate}
\end{prop}
\begin{proof}
Item (1). The morphism of prisms
\[ (\gs(R), (E), \ast) \to (\mathbb{A}_\inf(X_\infty), (E), \ast)\]
induces the isomorphism
 \[W = M\otimes_{\frakS(R)_{\dR}^+}\bdrplus(\widehat R_{\infty}).\]
The stratification $(M,\varepsilon) \in \Strat(\frakS(R)_{\dR}^{+,\bullet})$ corresponding to $(M,\nabla,\phi)$ is given by (\ref{Equ-Stratification-Abs-I}). That is, for any $x\in M$, we have
    \begin{equation} \label{eq: local strat}
 \varepsilon(x)   = (1+E\underline Y_1)^{E^{-1}\underline \nabla}(1+aX_1)^{a^{-1}\phi}(x).   
\end{equation}               
The induced stratification corresponding to $W$ hence $\bm^\perf$ in $\Strat(\rC(\Gamma^{\bullet},\bbdrplusm(\widehat R_{\infty})))$ is via the natural morphism
    \[\frakS(R)^{+,\bullet}_{\dR} \to \rC(\Gamma^{\bullet},\bbdrplusm(\widehat R_{\infty})).\]
Since we have
  \[X_1(\gamma_1^{n_1}\cdots\gamma_d^{n_d}g) = \frac{E-g(E)}{-aE},\]
 \[Y_{i,1}(\gamma_1^{n_1}\cdots\gamma_d^{n_d}g) = \frac{[T_i^{\flat}]-\gamma_i^{n_i}([T_i^{\flat}])}{-E([\pi^{\flat}])[T_i^{\flat}]} = \frac{1-[\epsilon]^{n_i}}{-E([\pi^{\flat}])}.\]
We can plug into \eqref{eq: local strat} to conclude.

Item (2), it suffices to check the action of $\Gamma_\geo$ on $M$ is unipotent, which is clear from the formula in Item (1), since $\nabla_i$ is nilpotent.

Item (3).  This is direct consequence of the explicit formula in Prop \ref{prop: local RH formula}(3). Indeed, for $\gk$-action,  Prop \ref{prop: local RH formula}(3) says that $\gk$-action on $\cm$  is the same as the one on $V$, hence is defined by the formula in Item (1). For connections, Prop \ref{prop: local RH formula}(3) says that using  the basis $\frac{\dlog T_i}{t}$ for $ \Omega^1_{R_\bdrplusm}(-1)$, $\nabla_{\cm, i,t}$ is the same as $\log \gamma_i$ on $V$; computing $\log \gamma_i$ explicitly using formula in Item (1), it is exactly $\frac{t}{E}\cdot \nabla_{M, i,E}(x)$.
\end{proof}

\newpage
\addtocontents{toc}{\ghblue{Arithmetic Sen theory}}
     
\section{Infinite dimensional Sen theory: cyclotomic and Kummer}  \label{sec: infinite Sen}
In this section, we first review  locally analytic (cyclotomic) Sen theory for \emph{infinite dimensional} $C$-representations as developed by \cite{RC22}; this theory also generalizes to $\bdrplusm$-representations. 
Using techniques from \cite{GMWHT}, we can construct a similar theory over the Kummer tower. 
The main result is Thm \ref{thm: Arith Sen Kummer}, which gives rise to a functor
\[ D_{\sen, \kinfty}:  \rep^\rella_\gk(\bdrplusm) \to \mic^\wedge(\kinfty[[E]]/E^m) \]
sending $W$ to $(D_{\Sen, \kinfty}(W), \phi_\kinfty)$; cf. Def \ref{def: a small arith} for the category on RHS.
This infinite dimensional generalization has two benefits: we can use it to treat a family version of  representations; but more crucially, we can use it to treat a   natural ``period ring" in \S \ref{sec: analytic sen kummer}, which is infinite dimensional but is relatively locally analytic.
In this section, $m<\infty$ and all representations are Banach spaces.

\subsection{Arithmetic Sen theory over cyclotomic tower}

\begin{defn}
\label{Def: rel lav rep}  
(Special case of \cite[Def. 2.3.4]{RC22}).
\begin{enumerate}
\item Let $W$ be an orthonormal Banach $C$-semilinear  representation of $\gk$. Say it is \emph{relative locally analytic}  if  there exists an orthonormal basis $\{v_i\}_{i\in I}$ (over $C$) generating an $\o_C$-lattice $W^0$ such that there is an open subgroup $G_L \subset G_K$ stabilizing $W^0$ and  $\epsilon >0$ such that the action of $G_L$ on $\{v_i/p^{\epsilon}\}_i$ is  trivial.
\footnote{Caution:  we are not saying the action on $W^0/p^{\epsilon}W^0$ is 	``trivial", a notion that would require definition.}
  We say that $\{v_i\}_{i\in I}$ is a relative locally analytic basis of $W$.  
Denote the category of such objects by $\Rep_\gk^{\rella}(C)$.

\item Say  an orthonormal  Banach $\bdrplusm$-semilinear representation $W$ of $\gk$ is relatively locally analytic if the $C$-semilinear representations $t^iW/t^{i+1}W$ are relatively locally analytic for each $i \geq 0$. Denote the category of such objects by $\rep^\rella_\gk(\bdrplusm)$.
\end{enumerate}

 \end{defn}

\begin{remark}
The notion of ``relative locally analytic representations"  in \cite{RC22} are defined in broader contexts. For a relative locally analytic $\bdrplusm$-representation above, its invariant $W^{G_\kpinfty}$ is a locally analytic representation of $\Gamma_K$  by \cite[Lem 2.1.4]{RC22} (which treats $m=1$ case).
\end{remark}

\begin{defn} (See \S \ref{subsec lav nota} for relevant notations on locally analytic vectors.) 
Let $W$ be a Banach $\bdrplusm$-semilinear representation of $\gk$. 
For each $n \geq 0$, define 
\begin{equation} 
D_{\Sen, \kmun}(W): =(W^{G_\kpinfty})^{\gamma_n\dan}.
\end{equation} 
Also define 
\begin{equation}\label{senlav}
D_{\Sen, \kpinfty}(W): =(W^{G_\kpinfty})^{\gammak\dla} =\cup_{n \geq 0} D_{\Sen, \kmun}(W).
\end{equation} 
Both spaces admit $\nabla_\gamma$-actions.
\end{defn}

\begin{theorem}[\cite{RC22}] 
\label{thm: Arith Sen cyclo}
Let $W\in \rep^\rella_\gk(\bdrplusm)$. There exists some $n_0 \geq 0$ (depending on $W$) such that if $n >n_0$, then
\[ 
D_{\Sen, \kmun}(W) \in \rep_{\gk}^\rla(\kmun[[t]]/t^m);
\]
and the natural map
\[ D_{\Sen, \kmun}(W)\hatotimes_{\kmun[[t]]} \bdrplusm \to W\] 
is an isomorphism. Let $n >n_0$, we furthermore have: 
\begin{enumerate}
\item For any $s >n$ 
\[D_{\Sen, \kmun}(W)\otimes_\kmun \kmus \simeq D_{\Sen, \kmus}(W);\]
as a consequence
\[   D_{\Sen, \kmun}(W)\otimes_\kmun \kpinfty  \simeq D_{\Sen, \kpinfty}(W). \]
\item We have cohomology comparisons
\begin{align*}
\rg(G_\kmun, W) & \simeq \rg_{\mathrm{sm}}(\Gamma_\kmun,  \rg(\nabla_\gamma,  D_{\Sen, \kmun}(W))), \\
\rg(\gk, W)   & \simeq  \rg_{\mathrm{sm}}(\Gamma_K,  \rg(\nabla_\gamma,  D_{\Sen, \kpinfty}(W))),   \\
\rg(\gk, W)\otimes_K \kpinfty  & \simeq     \rg(\nabla_\gamma,  D_{\Sen, \kpinfty}(W)).
\end{align*} 
\end{enumerate}
\end{theorem}
\begin{proof}
By standard d\'evissage, it suffices to prove the $m=1$ case. We sketch the argument from \cite{RC22}.
Let $\chi^{\mathrm{add}}: \gk \to \mathbb{Z}_p^\times \xrightarrow{\log} \zp$ be the cyclotomic character composed with logarithm.
The triple 
$$(C, \gk, \chi^{\mathrm{add}})$$
 is a `` (strongly decomposable) \emph{Sen theory}" \`a la \cite[Def. 2.2.6]{RC22}. 

According to  \cite[Thm. 2.4.3]{RC22}, we can first take $G_{K'}$ (notation ``$\Pi'$" there) an open normal subgroup of $\gk$, such that there is a module
\[ D_{\sen, K'(\mu_n)}(W) \in \rep^\rla_{\Gamma_{K'}}({K'(\mu_n)})   \]
 which is the  ``$A_{H',n}$-module  $S_{H',n}(V)$"  in \cite{RC22}, such that
 \[ D_{\sen, K'(\mu_n)}(W)\hatotimes_{K'(\mu_n)} C \simeq W.\]
One obtains the module  $D_{\sen, K(\mu_n)}(W)$ by Galois descent along ${K'(\mu_n)}/{K(\mu_n)}$. 
In addition, by enlarging $n$ if necessary, \cite[Lem. 2.1.5]{RC22} guarantees that the $\gamma_n$-action on $D_{\sen, K(\mu_n)}(W)$ is in fact \emph{analytic.} 
 
 According to \cite[Cor. 2.5.1]{RC22}, we have
 \[ \rg(G_{K'}, W) \simeq \rg_{\mathrm{sm}}(\Gamma_{K'}, \rg(\nabla_\gamma, D_{\sen, K'(\mu_n)}(W)))\]
 This implies Item (2)  by Hochschild--Serre spectral sequence (and Galois descent).
 \end{proof}

\subsection{Arithmetic Sen theory over  Kummer tower}
In this subsection, we generalize the (arithmetic) Sen theory over the Kummer tower (for \emph{finite dimensional} $C$-representations) in \cite[\S 7.3]{GMWHT} to the case of (infinite dimensional) relative locally analytic representations. The idea remains the same, i.e., we want to use the formula:
\begin{equation*}
 D_{\Sen, \kinfty}(W)= (W^{G_L})^{\tau\dla, \gamma=1}.
 \end{equation*}
 The subtle issue here is that we need to take care of topological issues (particularly completed tensor products).
 
 \begin{lemma} \label{lem: fkt analytic}
     \begin{enumerate}
       \item Let $x \in \bdrpluslm$ and suppose $tx \in  (\bdrpluslm)^{\hat{G}\dan}$, then $x \in  (\bdrpluslm)^{\hat{G}\dan}$.

    \item We have $(\bdrpluslm)^{\gamma=1, \tau_n\dan}=K(\pi_n)[[E]]/E^m$; as a special case $(\bdrpluslm)^{\gamma=1, \tau\dan}=\gsdrm$.

    \item Regard $\fkt^{\pm 1}$ as an element in $\bdrplus$ hence $\bdrpluslm$. We have 
$\fkt^{-1} \in (\bdrpluslm)^{\hat{G}\dan}$.
     \end{enumerate}
 \end{lemma}
 \begin{proof}
    Item (1) follows same argument as in \cite[Lem 9.13(1)]{GMWdR}.
Item (2) follows from similar argument as in \cite[Prop 9.14]{GMWdR}, which treats the $\tau\dla$-case. Indeed, it suffices to replace the fact ``$(\hatl)^{\gamma=1,\tau\dla}=\kinfty$" used there by the fact $(\hatl)^{\gamma=1,\tau_n\dan}=K(\pi_n)$. To see the later fact, note we already know these vectors are smooth vectors, hence $\tau_n\dan$ means $\tau_n=1$.
For Item (3), note it strengthens the fact $\fkt^{-1} \in (\bdrplusl)^{\hat{G}\dla}$ proved in  \cite[Lem 9.11]{GMWdR}. To prove analyticity: note Item (2) implies $\lambda$ is analytic; thus Item (1) implies  $\fkt^{-1}=p\lambda/t$ is analytic.
 \end{proof}

\begin{defn} Let $W$ be a Banach $\bdrplusm$-semi-linear representation of $\gk$. 
Define
\[  D_{\Sen, \kpin}(W) := (W_{\hat{L}})^{ \gamma=1, \tau_n\dan}, \]
\[D_{\Sen, \kinfty}(W) := (W_{\hat{L}})^{ \gamma=1, \tau\dla} =\cup_{n \geq 0} D_{\Sen, \kpin}(W) \]
which is a module over $K(\pi_n)[[E]]/E^m$ resp. $\gsdrm$.
Recall in \cite[\S 10]{GMWdR}, we defined an operator 
\[ N_\nabla: =\frac{1}{p\fkt} \nabla_\tau. \]
It acts on locally analytic representations over $\bdrpluslm$, hence also sends $D_{\Sen, \kpin}(W)$ resp. $D_{\Sen, \kinfty}(W)$ to a bigger space of analytic vectors resp. locally analytic vectors.  But the fact that $N_\nabla$ commutes with $\gal(L/\kinfty)$ quickly implies that $N_\nabla$ stabilizes the above spaces. 
Note in particular, since $\mathfrak{t}^{-1}$ is \emph{analytic} (not just locally analytic) by Lem \ref{lem: fkt analytic}, $N_\nabla$ stabilizes the space of \emph{analytic} vectors $ D_{\Sen, \kpin}(W)$. Define yet the normalization
\[ \phi_\kinfty:=\frac{E}{ u\lambda E' }\cdot N_\nabla  =\frac{E}{utE'}\cdot \nabla_\tau. \]
Then we have
\[ \phi_\kinfty:  D_{\Sen, \kpin}(W) \to  D_{\Sen, \kpin}(W),\]
\[ \phi_\kinfty:  D_{\Sen, \kinfty}(W) \to  D_{\Sen, \kinfty}(W)\] 
Note on the ring $K(\pi_n)[[E]]/E^m$ resp. $\gsdrm$, we have 
\[ \phi_\kinfty(E)=E.\]
\end{defn}

\begin{theorem}
\label{thm: Arith Sen Kummer}
Let $W\in \rep^\rella_\gk(\bdrplusm)$. There exists some $n_0 \geq 0$ (depending on $W$) such that if $n >n_0$, then 
$D_{\Sen, \kpin}(W)$ is a Banach $\kpin[[E]]/E^m$-space, such that the natural map
\[ D_{\Sen, \kpin}(W)\hatotimes_{\kpin[[E]]/E^m} \bdrplusm \to W\] 
is an isomorphism. Let $n >n_0$, we furthermore have: 
\begin{enumerate}
\item For any $s >n$ 
\[D_{\Sen, \kpin}(W)\otimes_\kpin \kpis \simeq D_{\Sen, \kpis}(W);\]
as a consequence
\[   D_{\Sen, \kpin}(W)\otimes_\kpin \kinfty  \simeq D_{\Sen, \kinfty}(W)).\]
\item We have cohomology comparisons
\begin{align*}
\rg(\gk, W)\otimes_K \kpin & \simeq \rg(\phi_\kinfty,  D_{\Sen, \kpin}(W)), \\
 \rg(\gk, W)\otimes_K \kinfty  & \simeq \rg(\phi_\kinfty,  D_{\Sen, \kinfty}(W)) .
\end{align*}

\item Consider  the case $m=1$ (i.e., $C$-representations). Under the  identification
\[D_{\Sen, \kmun}(W)\hatotimes_{\kmun} C  \simeq W \simeq D_{\Sen, \kpin}(W)\hatotimes_{\kpin} C, \]
the two $C$-linearized operators $\phi_\kpinfty$ and $\phi_\kinfty$ are the \emph{same.}
\end{enumerate}
\end{theorem}
\begin{proof}
By d\'evissage, again it suffices to treat the $m=1$ case.
By Thm. \ref{thm: Arith Sen cyclo}, we have
\[   D_{\Sen, \kmun}(W)\otimes_\kmun \kpinfty  \simeq D_{\Sen, \kpinfty}(W) \]
where $\gamma_n$-action on $D_{\Sen, \kmun}(W)$ is \emph{analytic}.
Thus
\[Y= (W_{\hat L})^{\gamma_n\dan, \tau_n\dan} = D_{\Sen, \kmun}(W)\hatotimes_\kmun (\hat L)^{\gamma_n\dan, \tau_n\dan}. \]
We claim that by enlarging $n$ if necessary, the space $Y$ is ``trivial" under $\nabla_\gamma$-action, in the sense that the natural map
\[ Y^{\nabla_\gamma=0} \hatotimes_{K(\mu_n, \pi_n)}  (\hat L)^{\gamma_n\dan, \tau_n\dan}  \to Y \]
is an isomorphism.
The idea is now similar to Step 1 in proof of \cite[Thm. 7.12]{GMWHT}. Take an ON basis $v_i$ of $D_{\Sen, \kmun}(W)$ with $i \in I$, and let
\[ D_\gamma=\Mat(\partial_\gamma) \in \GL_I(\kmun) \]
(note: it is important to use \emph{analytic} actions here so the matrix is defined over $\kmun$). We need to construct some 
\[ H \in \GL_I( (\hat L)^{\gamma_n\dan, \tau_n\dan}) \]
such that
\[ \partial_\gamma(H) +D_\gamma H=0 \]
exactly the same formula as in \cite[Thm. 7.12]{GMWHT}:
\[ H=\sum_{k \geq 0} (-1)^k \Mat(\partial_\gamma^k) \frac{(\beta-\beta_n)^k }{k!}\]
still works. 
With $Y^{\nabla_\gamma=0}$ constructed, it suffices to use Galois descent to obtain
\[  D_{\Sen, \kpin}(W): =  (Y^{\nabla_\gamma=0})^{\gal(K(\mu_n, \pi_n)/\kpin)}.\]
This concludes proof of Item (1).

The proof of Item (2) is exactly the same as \cite[Thm. 7.17]{GMWHT}, via Thm. \ref{thm: Arith Sen cyclo} and Galois descent. 

 For Item (3), it follows from the same argument as \cite[Thm. 7.13]{GMWHT}, except that in the proof, one should use the identification of \emph{analytic} vectors:
\[D_{\Sen, \kmun}(W)\hatotimes_{\kmun} (\hat L)^{\gamma_n\dan, \tau_n\dan}     \simeq D_{\Sen, \kpin}(W)\hatotimes_{\kpin} (\hat L)^{\gamma_n\dan, \tau_n\dan}  \]
(the displayed equation in proof of  \cite[Thm. 7.13]{GMWHT} cannot hold because of topological issues).
\end{proof}

 As a summary, we now have two functors
\[
\begin{tikzcd}
                                                                                                & {\mic_t^\wedge(\kpinfty[[t]]/t^m)} \\
\rep^\rella_\gk(\bdrplusm) \arrow[ru, "{D_{\Sen, \kpinfty}}"] \arrow[rd, "{D_{\Sen, \kinfty}}"] &                                    \\
                                                                                                & {\mic_E^\wedge(\kinfty[[E]]/E^m)} 
\end{tikzcd}
\]

\begin{defn} \label{def:point rel la small}
Say $W \in \rep^\rella_\gk(\bdrplusm)$ is $a$-small if $(D_{\Sen, \kinfty}(W), \phi_\kinfty)$ is $a$-small as in Def \ref{def: a small arith}.
\end{defn}

One can also use the $t$-connection $D_{\Sen, \kpinfty}(W)$ to define $a$-smallness. This would coincide with Def \ref{def:point rel la small}, as in $m=1$ case, the two operators are the same  after linearization by Thm \ref{thm: Arith Sen Kummer}(3).

 \newpage 
 
\section{Decompleted Riemann--Hilbert correspondence} \label{sec: decomp RH}

In this section,  we construct decompleted Riemann--Hilbert correspondence using arithmetic Sen theory from previous section.
The main result is Prop \ref{prop: from t-conn to enhanced Kummer-Sen}, where in particular we construct the functor
\[  \MIC_\gk(X_{C, \et}, \o_{X_C, \bdrplus, m})  \to  \MIC_\en(X_{\et}, \o_{X,  \kinfty[[E]], m}).\]
 We then use this decompleted correspondence to define the notion of $a$-small $\bbdrplusm$-local systems on the pro-\'etale site.

\begin{prop} \label{prop: global kummer sen}
Let $\calm \in \VB_\gk(X_{C, \et}, \o_{X_C, \bdrplus, m})$.
For an affinoid open $\spa(A, A^+) =U \in X$, define 
\[ D_{\sen, \kpinfty}(U):= ((\calm(U_C))^{G_\kpinfty})^{\gamma \dla }, \]
\[ D_{\sen, \kinfty}(U):= ((\calm(U_C))^{G_L})^{\gamma=1, \tau\dla }. \]
These rules define two objects
\[ D_{\sen, \kpinfty}(\cm) \in \VB(X_{\et}, \o_{X,  \kpinfty[[t]], m}), \]
\[ D_{\sen, \kinfty}(\cm) \in \VB(X_{\et}, \o_{X,  \kinfty[[E]], m}) \]
 such that both their pullbacks to $X_C$ are $\cm$; the vector bundle $ D_{\sen, \kpinfty}(\cm)$ resp.  $D_{\sen, \kinfty}(\cm)$ admits $\phi_\kpinfty$ resp. $\phi_\kinfty$-action. 
 In addition
  \[ \rg(\gk, \rg(X_C, \cm))\otimes_K \kpinfty =\rg(\phi_\kpinfty, \rg(X, D_{\sen, \kpinfty}(\cm))),\]
 \[ \rg(\gk, \rg(X_C, \cm))\otimes_K \kinfty =\rg(\phi_\kinfty, \rg(X, D_{\sen, \kinfty}(\cm))).\]
\end{prop}
\begin{proof}
The fact $D_{\sen, \kpinfty}(\cm) \in \VB(X_{\et}, \o_{X,  \kpinfty[[t]], m})$ is proved in \cite[Prop 3.2, Prop 3.4]{Pet23}; but the cohomology comparison is not recorded there.

We first prove $D_{\sen, \kinfty}$ defines a vector bundle. Firstly, it is a sheaf: it suffices to verify this on an  affinoid $U$ with a finite covering $U=\cup_{i \in I} U_i$; then the sheaf condition quickly follows from the fact that taking $\gamma$-invariants and taking $\tau\dla$-vectors are left exact functors.
One further needs to prove $D_{\sen, \kinfty}(\cm)$ is a vector bundle on the ringed space $(X_\et,\o_{X, K_\infty[[E]], m})$; that is, one needs to find a cover of $X$ to trivialize $M$. This follows from similar argument as the next to last paragraph of \cite[Prop. 3.2]{Pet23}.
Indeed,  \cite[Lem. 3.1(1)]{Pet23} implies there exists some $n \gg 0$ and an open cover of $X_{K(\pi_n)}$ whose base change to $X_{\hatkinfty}$  trivializes $\calm$; thus $D_{\sen, \kinfty}$ defines a vector bundle on $(X_{K(\pi_n)}, \o_{X, K_\infty[[E]], m})$ hence also  $(X_{K(\pi_n,\mu_n)}, \o_{X, K_\infty[[E]], m})$. 
This descends to  a vector bundle on  $(X_\et,\o_{X, K_\infty[[E]], m})$ by \cite[Lem. 3.1(2)]{Pet23}.

Finally, we prove the   cohomology comparisons; we only treat the case with $D_{\sen, \kinfty}(\cm)$ since the other case is similar. 
Cover $X$ by finite many  affinoids $X=\cup U_i$ which trivializes   $D_{\sen, \kinfty}(\cm)$; then   $\rg(X_C, \cm)$ is computed by a finite \v{C}ech complex $\mathcal{C}^\bullet$ where each term is   a relative locally analytic representation. Thus Thm \ref{thm: Arith Sen Kummer} implies
\[ \rg(\gk, \mathcal{C}^\bullet)\otimes_K \kinfty \simeq \rg(\phi_\kinfty, D_{\sen, \kinfty}(\mathcal{C}^\bullet)) \]
as the statement is true for each term in $\mathcal{C}^\bullet$; now note RHS is precisely a \v{C}ech complex for $\rg(\phi_\kinfty, \rg(X, D_{\sen, \kinfty}(\cm)))$. 
\end{proof}

\begin{construction} \label{cons: gk conn to enhanced conn}
    Let $\cm \in \MIC_\gk(X_{C, \et}, \o_{X_C, \bdrplus, m})$ as in Def \ref{def:t-connection}. Prop \ref{prop: global kummer sen} shows that
    \[ D_{\sen, \kinfty}(\cm) \in \VB(X_{\et}, \o_{X,  \kinfty[[E]], m}). \]
Since the connection     
      \[\nabla_{\cm}:\cm\to \cm\otimes_{\o_{\cx}}\Omega^1_{\cx}\{-1\}\]
      is $\gk$-equivariant (and continuous), it respects taking locally analytic vectors; thus it induces a connection operator 
\[\nabla: D_{\sen, \kinfty}(\cm) \to D_{\sen, \kinfty}(\cm)\otimes_{\o_{X, \kinfty[[E]]}} \Omega_{X,\kinfty[[E]]}\{-1\}. \]
Note there is also $\phi_\kinfty$-operator on $D_\kinfty=D_{\sen, \kinfty}(\cm)$. we claim  the following diagram is commutative:
  \[
    \begin{tikzcd}
D_\kinfty \arrow[d, "\phi_{\kinfty,D}"] \arrow[r, "\nabla"] & D_\kinfty\otimes_{\o_{X, \kinfty[[E]]}} \Omega_{X,\kinfty[[E]]}\{-1\} \arrow[d, "  \phi_{\kinfty, D}\otimes 1 +1\otimes \phi_{\kinfty, (E)^{-1}}\otimes1"] \\
D_\kinfty \arrow[r, "\nabla"]                   & D_\kinfty\otimes_{\o_{X, \kinfty[[E]]}} \Omega_{X,\kinfty[[E]]}\{-1\}.                             
\end{tikzcd}
\]
To check commutativity, it suffices to work locally as in set up of Notation \ref{nota: small affinoid}.
Recall the operator
 \[\nabla_{\cm}:\cm\to \cm\otimes_{\o_{\cx}}\Omega^1_{\cx}\{-1\}\]
can be written
 \[\nabla_M = \sum_{i=1}^d\nabla_i\otimes\frac{\dlog T_i}{t}\]
where  $\nabla_i$ is precisely $\log \gamma_i$ by Proposition \ref{prop: local RH formula}, which commutes with $\tau$ since $\gamma_i$ commutes with $\tau$ (see Notation \ref{nota: small affinoid}).
But once decompleted using Kummer tower as
\[\nabla: D_\kinfty \to D_\kinfty \otimes_{\o_{X, \kinfty[[E]]}} \Omega_{X,\kinfty[[E]]}\{-1\}, \]
one needs to use the Breuil--Kisin twist coordinate, and thus
 \[\nabla_{D_\kinfty} = \sum_{i=1}^d \frac{E}{t} \cdot \nabla_i \otimes  \frac{\dlog T_i}{E}.\]
We are thus reduced to prove
\[ [\phi_\kinfty, \frac{E}{t} \nabla_i] =\frac{E}{t} \nabla_i \]
which follows from the fact that $\phi_\kinfty$ commutes with $\nabla_i$ (since both $\tau$ and $\fkt$ commutes with $\gamma_i$) and the fact $\phi_\kinfty(\frac{E}{t})=\frac{E}{t}$.

\end{construction}

 \begin{prop} \label{prop: from t-conn to enhanced Kummer-Sen}
The above construction induces functors (extending those in Prop \ref{prop: global kummer sen})
 \[
 \begin{tikzcd}
                                                                    & {\MIC_\en(X_{\et}, \o_{X,  \kpinfty[[t]], m})} \\
{\MIC_\gk(X_{C, \et}, \o_{X_C, \bdrplus, m})} \arrow[ru] \arrow[rd] &                                                \\
                                                                    & {\MIC_\en(X_{\et}, \o_{X,  \kinfty[[E]], m})}. 
\end{tikzcd}
\]
In addition, we have cohomology comparisons:
\[ \rg(X_C, \rg(\gk, \DR(\calm))) \otimes_K \kinfty  \simeq \rg(X, \rg(\phi,\DR(D))). \]
 \end{prop}
 \begin{proof}

  To prove the cohomology comparison, take a finite cover that trivializes $D$, then each term in the \v{C}ech complex is a relative locally analytic representation (note $\Omega_X^1$ has trivial $\gk$-action); then  proceed as in Prop \ref{prop: global kummer sen}.

 \end{proof}

\subsection{$a$-small $\bbdrplus$-local systems} \label{sec: a small bdrplus loc sys}
 
In this subsection, we define the notion of  $a$-small $\bbdrplus$-local systems on the pro-\'etale site.

\begin{defn} \label{defn: a small proet}
    We have
\[ \vect(\xproet, \bbdrplusm) \simeq \MIC_{\gk}(X_{C, \et}, \o_{X_C, \bdrplus, m}) \to \MIC_\en(X_{\et}, \o_{X,  \kinfty[[E]], m}). \]
For corresponding objects $\bw \mapsto \cm \mapsto D$, say $\bw$ is $a$-small if $\phi_D$ is $a$-small.
\end{defn}

\begin{lemma} The following are equivalent.
\begin{enumerate}
    \item $\bw$ is $a$-small;
    \item $\lim_{n \to \infty} a^n\Pi_{i=0}^n (  \phi_\kinfty - i) \to 0 \text{ on }  D_{\sen, \kinfty}(W) $;
        \item $\lim_{n \to \infty} a^n \Pi_{i=0}^n (  \phi_\kpinfty - i) \to 0 \text{ on }  D_{\sen, \kpinfty}(W) $;
    \item $\bw_{\bar x}$ is $a$-small for each classical point $x\in X$. 
\end{enumerate} 
\end{lemma}
\begin{proof}
    It suffices to consider $m=1$ case. $(1) \Leftrightarrow (2)$ by definition. $(2) \Leftrightarrow (3)$ by Thm \ref{thm: Arith Sen Kummer}(3). $(3) \Leftrightarrow (4)$ by the fact that the global $\phi_\kpinfty$ specializes to the (cyclotomic) Sen operator on each classical point.
\end{proof}

\subsection{The case $m=\infty$}
In this subsection, we record a \emph{constancy} (of Sen weights) result for $\bbdrplus$-local systems that is essentially due to Shimizu \cite{Shi18}. It means that when $m=\infty$, $a$-smallness is easier to check; other than this, the results here do not play any further role.

 \begin{prop}
Let $X/K$ be a geometrically connected smooth rigid analytic variety. Let $\bw$ be an object in $\vect(\xproet, \bbdrplus)$ of rank $d$. Then there exists $a_1, \cdots, a_d \in \barK$ such that 
the Sen weights of $\bbw_{\bar x}/t\bbw_{\bar x}$ are $a_1, \cdots, a_d$ for any classical point $x \in X$. 
 \end{prop}
 \begin{proof}
  This is proved in \cite[Thm 1.1]{Shi18} if $\bw$ comes from a $\qp$-local system $\bbl$ in the sense that $\bw=\bbl \otimes_\qp \bbdrplus$: the reason being $p$-adic Riemann--Hilbert functor is constructed only for $\bbl$ in \cite{LZ17}.
But now   we also constructed $p$-adic Riemann--Hilbert functor for general $\bw$, thus Shimizu's strategy now works verbatim. We note that two other main ingredients for proof of \cite[Thm 1.1]{Shi18} are the decompleted Riemann--Hilbert functor  \cite[Def 2.28]{Shi18} (which now we also have for general $\bw$, even globally), and the theory of formal connections \cite[\S 4.2]{Shi18} which   works verbatim for us. 
 \end{proof}

 \begin{rem}
     To use the theory of formal connections (\cite[\S 4.2]{Shi18} mentioned in proof above), one needs to invert $t$ in the process: cf. \cite[Thm 4.5]{Shi18} where the kernel of $(tD_0-\alpha)$ is computed inside a $R((t))$-module. Thus in particular, Shimizu's proof cannot work for a $\bw \in  \vect(\xproet, \bbdrplusm)$ with $1\leq  m<\infty$. (It seems likely to be false already when $m=1$, although we do not know of any example).
 \end{rem}
 
 \begin{cor}
 Let $\bw \in \vect(\xproet, \bbdrplus)$, then it is $a$-small if and only it is so on one classical point on each connected component.
 \end{cor}

\newpage
\section{Analytic Sen theory over   Kummer tower}  \label{sec: analytic sen kummer}

In \S \ref{sec: infinite Sen}, we constructed Sen theory for relative locally analytic $\bdrplusm$-representations. In this section, we show that when the representation is $a$-small in the sense of Def \ref{defn: a small proet}, then it admits a refined \emph{analytic} (not just locally analytic) Sen theory over the Kummer tower.
The main result is Prop \ref{prop: a small implies analytic}, where we construct a functor
\[ D_{\tau\dan}: \rep^{\rella, a}_\gk(\bdrplusm) \to \mic^{\wedge, a}(K[[E]]/E^m) \]
which refines $D_{\sen, \kinfty}$ in \S \ref{sec: infinite Sen} (cf. Def \ref{def: a small arith} for category on RHS).
 It turns out to be a key tool  for \emph{analytic} study of $\pris_\dR^+$-crystals.
 
 \subsection{Space of functions}
 \begin{construction} \label{cons: Space of functions}
We recall relation between analytic functions and analytic vectors.
Let $G$ be a $p$-adic Lie group acting continuously on a $\qp$-Banach representation $W$.
\begin{enumerate}
\item  Let $C^\cont(G, W)$ be the space of continuous functions; it admits a $G$-action such that for $F\in C^\cont(G, W)$  and $g\in G$, we have
\[ (gF)(h) =gF(g^{-1}h). \]
Under this action, a fixed point is an orbit function  $ F_m(h)=hm$  for some $m \in W$, and  $F \mapsto F(1)$ induces a bijection:
\[ (C^\cont(G, W))^G \simeq W. \]
Restricted to the subspace $C^\la(G, W)$ of locally analytic functions, we have
\[ (C^\la(G, W))^G \simeq W^\la. \]

\item Suppose furthermore there is a homeomorphism (of $p$-adic manifolds) $c: \bbz_p^d \to G$. Then one can define the $G$-stable subspace of analytic functions:
\[ C^\an(G, W) \subset C^\cont(G, W).\]
Then $F \mapsto F(1)$ induces a bijection:
\[ (C^\an(G, W))^G \simeq W^{G\dan}.\]
\end{enumerate} 
  \end{construction}
 
\begin{construction} \label{cons tau ana  fun}
Let $W$ be a $\qp$-Banach representation of $\hatg$. Let
\[ C^{\gamma=1, \tau\dan}(\hat{G}, W) \subset C^\cont(\hat{G}, W)\]
be the subspace of   functions such that:
\begin{enumerate} 
\item $F(h\gamma)=F(h), \forall h \in \hat{G}, \gamma \in \gal(L/\kinfty)$, and
\item $F|_{\langle{\tau} \rangle }:=F|_{\gal(L/\kpinfty)}$ is analytic (not just locally analytic).
\end{enumerate}
One checks that $C^{\gamma=1, \tau\dan}(\hat{G}, W)$ is stable under $\hat{G}$-action; indeed, let $g\in \hat{G}, F \in C^{\gamma=1, \tau\dan}(\hat{G}, W)$, then
\begin{itemize}
\item for the first condition, one checks that
\[ (gF)(h\gamma)=gF(g^{-1}h\gamma)=gF(g^{-1}h)= (gF)(h).\]
\item For the second condition, we want to check \[ F'(\tau^{x}) :=F(g^{-1} \tau^{x})\]
is an analytic function.
Suppose $g=\gamma \tau^a$ for some $\gamma\in \gal(L/\kinfty)$, then
\[ g^{-1} \tau^{x} =\tau^{-a-x\chi(\gamma)}\gamma^{-1}\]
thus $F'(\tau^{x}) =F(\tau^{-a-x\chi(\gamma)})$, and  is still analytic.
\end{itemize}

\end{construction}

\begin{remark}
Note that $\hat{G}$ (as a $p$-adic manifold) is not necessarily isomorphic to $\mathbb{Z}_p^2$ (since $\hat{G}$ could contain torsion), thus one might not be able to define the space $C^\an(\hat{G}, W)$ (which would then contain $C^{\gamma=1, \tau\dan}(\hat{G}, W)$).
\end{remark}

\begin{lemma}
Let $W$ be a $\qp$-Banach representation of $\hat{G}$, then the map $F \mapsto F(1)$ induces a bijection
\[ (  C^{\gamma=1, \tau\dan}(\hat{G}, W))^{\hat{G}}\simeq W^{\gamma=1, \tau\dan}. \]
\end{lemma}
\begin{proof}
A fixed point on the left hand side is an orbit function satisfying the two conditions in Cons \ref{cons tau ana  fun}. 
\end{proof}

\subsection{nearly de Rham period rings}\label{subsec: ndr period ring}

\begin{construction}
Recall some notations from \cite[\S 7]{GMWdR}.
Recall $(\ainf, (\xi))$ is a cover of the final object of $\mathrm{Sh}(\okprisperf)$  equivalently $\mathrm{Sh}((\ok)_{\pris, \log}^\perf)$. Let $(\bfa_\inf^\bullet, (\xi))$ be the corresponding \v{C}ech nerve. In \cite[Lem. 7.4]{GMWdR}, we constructed a morphism of cosimplicial prisms
 \begin{equation} \label{eqn: ainfone to fun}
  (\bfa_{\inf}^{\bullet},(\xi))\to (C(G_K^{\bullet},\Ainf),(\xi)) 
\end{equation}
which then induces an isomorphism of cosimplicial rings
 \begin{equation} \label{eqn: ainfone to fun dr} \bfB_{\dR,m}^{\bullet,+}\simeq  C(G_K^{\bullet},\bfB_{\dR,m}^+).\end{equation}
 Consider $\bullet=1$ case in \eqref{eqn: ainfone to fun}. The $\gk \otimes 1$ action on $\ainf \otimes \ainf$ induces a $\gk$-action on $\bfa_\inf^1$ hence on $\bfB_{\dR,m}^{1,+}$; then \eqref{eqn: ainfone to fun dr} for $\bullet=1$ is $\gk$-equivariant where RHS has $\gk$-action as in Construction \ref{cons: Space of functions}. As important examples:
 \begin{itemize}
     \item  For $x\in \ainf \xrightarrow{p_0}   \bfa_{\inf}^1$, it is sent to the orbit function $F_x(g)=gx$;
     \item  For $y \in \ainf \xrightarrow{p_1}   \bfa_{\inf}^1$, it is sent to the constant function $\delta_y(g)=y$.
 \end{itemize}
\end{construction}
 
 \begin{notation}[$\ast$-nearly de Rham period ring] \label{nota: ndr ring}
 We introduce certain ``nearly de Rham" (and ``nearly Hodge--Tate") period rings; these are inspired by constructions of \cite{AHLB1}, cf. Rem \ref{rem: relation ahlb} for more comments.
 \begin{enumerate}
     \item  Let $\aastninf$ be the coproduct of  $(\ainf, (\xi), \ast)$ and  $(\gs, (E), \ast)$  on the site $\okprisast$ (with $\ainf$ as the first component). 
Define
\[ \bastndrm=\pris_{\dR,m}^+(\aastninf).\]
When $m=1$, denote
\[ \bastnht=\pris_{\HT}^+(\aastninf).\]

 \item Similar to the computation of $(\gs^1_\ast, (E), \ast)$ (cf. \cite[\S 2]{GMWHT} or more generally \S \ref{sec: loc abs pris}), one  has
\[ \aastninf=(\gs\hatotimes_{W(k)} \ainf) \{ X_1\}_\delta^\wedge. \] 
Similar to the computation of $\gs^1_\ast/E$ (cf. \cite[\S 2]{GMWHT}), we have
\[ \aastninf/E \simeq \o_C\{X_1\}_\pd. \]
Thus, we have
\[ \bastndrm=\bdrplusm\{X_1\}_\pd. \]

 \item  The morphism $\bfa_{\ast-\mathrm{ninf}} \to \bfa_{\inf}^1$  induces the map
\[ \bastndrm \to \bfb^{1, +}_{\dR, m} \simeq   C(G_K,\bfB_{\dR,m}^+).\]
The map is injective, as can be checked modulo $E$. As isomorphism \eqref{eqn: ainfone to fun dr} is compatible with cosimplicial structure, we have a commutative diagram where all vertical arrows are (induced by) degeneracy maps:
\[
\begin{tikzcd}
\bastndrm \arrow[d, "\text{modulo $X_1$}"'] \arrow[r, hook] & {\bfb^{1, +}_{\dR, m}} \arrow[d, "\sigma_0"] \arrow[r] & {C(G_K,\bfB_{\dR,m}^+)} \arrow[d, "F\mapsto F(1)"] \\
\bdrplusm \arrow[r, "="]                                    & \bdrplusm \arrow[r, "="]                               & \bdrplusm                                         
\end{tikzcd}
\]  
 \end{enumerate}
  \end{notation}

\begin{notation} We introduce an element convenient for computations.
\begin{enumerate}
    \item  Consider the element in \cite[1.1.1]{Kis06}
\[ \ell_u:=\log(\frac{u}{\pi})=\log(1+\frac{u-\pi}{\pi}) =\sum_{i \geq 1} \frac{(-1)^{i-1}}{i}(\frac{u-\pi}{\pi})^i \]
which is a uniformizer of $\gs^+_\dR$. 
That is
\[\gsdr=K[[E]]=K[[\ell_u]].\]
The benefit of this uniformizer is that its $\gk$-action is easy to describe, as we have
\[ g(\ell_u)=\ell_u +c(g)t \]
where $c(g)$ is the cocycle such that $g(u)=u[\varepsilon]^{c(g)}$.

\item Thus we have
\[ \bastndrm =\bdrplusm \{ X_1 \}_\pd=\bdrplusm \{ Z_1 \}_\pd \]
where 
\[ Z_1= \frac{\ell_{u_0} -\ell_{u_1}}{\ell_{u_1}} \]
(note for computational convenience, we use $\ell_{u_1}$ as denominator). Again, its $\gk$-action is easy to describe, namely,
\begin{equation} \label{eq: g on zone}
    g(Z_1)=Z_1+c(g)\frac{t}{\ell_{u_1}}.
\end{equation} 
\end{enumerate} 
\end{notation}
 
\begin{lemma} \label{lem str of bastndrm}
\begin{enumerate}
\item $\bastndrm$ is relatively locally analytic $\bdrplusm$-representation of $\gk$  as in Def  \ref{Def: rel lav rep}.
In addition, it is $a$-small as in Def \ref{def:point rel la small}.
\item $ ((\bastndrm)^{G_L})^{\gamma=1,\tau\dla}=\kinfty[[E]]/E^m \{ X_1 \}_\pd=\kinfty[[E]]/E^m \{ Z_1 \}_\pd$.
\item $ ((\bastndrm)^{G_L})^{\gamma=1,\tau\dan}=\gsdrm  \{ X_1 \}_\pd=\gsdrm \{ Z_1 \}_\pd$.
\item We have
$ \gsdrm[0] \simeq \rg(\gk, \bastndrm)  $.
\end{enumerate} 
\end{lemma}
\begin{proof}
(1). It suffices to prove $m=1$ case, that is we need to prove
 \[ \bastnht=C\{X_1\}_\pd \]
is relatively locally analytic per Def. \ref{Def: rel lav rep}: one simply take the $\gk$-stable $\o_C$ lattice $\o_C\{X_1\}_\pd$ to verify the condition there. (Note one could not take the lattice  $\o_C\{Z_1\}_\pd$ as it is not $\gk$-stable). In addition, the $a$-smallness is already verified in Lem \ref{coho of Sone}.

(2) and (3). With Lem \ref{lem: fkt analytic} in mind, it reduces to proving $Z_1$ is a $(\gamma=1, \tau\dan)$-element. But this follows from  \eqref{eq: g on zone}.

(4). It suffices to prove after tensoring $\kinfty$.
  We have
\[ \rg(\gk, \bastndrm)\otimes_K \kinfty \simeq \rg(\phi_\kinfty, \kinfty[[E]]/E^m \{ X_1 \}_\pd) \simeq \kinfty[[E]]/E^m \]
where the first comparison follows from Thm \ref{thm: Arith Sen Kummer}, and the second cohomology computation follows from Lem \ref{coho of Sone}.
\end{proof}

\begin{construction} \label{cons d tau w sub}
Let $W\in \rep^\rella_\gk(\bdrplusm)$. 
Consider $ \bastndrm  \into  C(G_K,\bfB_{\dR,m}^+)$ in Notation \ref{nota: ndr ring}, take $\gal(\barK/L)$-invariant leads to
\[  \bastndrlm  \into  C(\hat{G},\bfB_{\dR,L,m}^+ ).\]
By Lem \ref{lem str of bastndrm}, this factors through 
\[  \bastndrlm  \into  C^{\gamma=1, \tau\dan}(\hat{G},\bfB_{\dR,L,m}^+). \]
Thus, it induces an injection
\[ (W\hatotimes_\bdrplus \bastndrm)^\gk \into D_{\tau\dan}(W)\]
as well as
\begin{equation} \label{eq: ndrinjsen}
(W\hatotimes_\bdrplus \bastndrm)^\gk \otimes_K \kinfty \into D_{\tau\dan}(W)\otimes_K \kinfty \into  D_{\sen, \kinfty}(W).
\end{equation}  
\end{construction}

\begin{prop} \label{prop: a small implies analytic}
Use notations in Construction \ref{cons d tau w sub}. Suppose furthermore that  $W$ is $a$-small. We have
\begin{enumerate}
\item  $  (W\hatotimes_\bdrplus \bastndrm)^\gk = D_{\tau\dan}(W)$.
 
\item The natural map $ D_{\tau\dan}(W)\otimes_K \kinfty \to D_{\sen, \kinfty}(W)$ is an isomorphism.
 
\item $ \rg(\gk, W) \simeq \rg(\phi_\kinfty, D_{\tau\dan}(W))$.
\end{enumerate} 
\end{prop}
\begin{proof}
The tensor product representation $ W\hatotimes_\bdrplus \bastndrm$ is relatively locally analytic  and $a$-small (by Lem \ref{lem str of bastndrm}(1)). Thus 
\[ \rg(\gk, W\hatotimes_\bdrplus \bastndrm)\otimes_K \kinfty \simeq
\rg(\phi_\kinfty, D_{\sen, \kinfty}(W) \otimes \kinfty[[E]]/E^m\{X_1\}_\pd) \simeq D_{\sen, \kinfty}(W)[0] 
 \]
 where the first comparison follows from Thm \ref{thm: Arith Sen Kummer}, and the second cohomology computation follows from Lem \ref{lem a small solution}(3).
Thus all inclusions in \eqref{eq: ndrinjsen} are equality. This implies Items (1) and (2). Item (3) follows from Item (2).
\end{proof}

\begin{remark}[Relation with \cite{AHLB1}] \label{rem: relation ahlb} The constructions in \S\ref{subsec: ndr period ring} are strongly influenced by the work \cite{AHLB1}; we briefly  discuss the relations here. 
\begin{enumerate}
    \item When $\ast=\emptyset$ (i.e., the prismatic case) and $m=1$ (i.e., the Hodge--Tate crystal case), the ring $\bfb_\nht$ in Notation \ref{nota: ndr ring} is exactly  ``$B_\en$" in \cite[\S 3.2]{AHLB1} (the subscript ``en" there stands for ``enhanced", cf. \cite[Rem 3.10]{AHLB1}).
The important fact that $\rg(\gk, B_\en) \simeq K[0]$ is proved in \cite[Thm 3.12]{AHLB1}.  A key observation there says that there is  an embedding, cf. \cite[Cor 3.9]{AHLB1},
\[ B_\en \into B_{\sen} \]
which is $G_{K'}$-equivariant where $K'/K$ is a finite extension, and where $B_\sen$ is the period ring (for Sen theory) introduced by Colmez \cite{Colmez94Sen}.  Indeed, then the proof of \cite[Thm 3.12]{AHLB1} could be reduced to $B_\sen$-case, or use similar explicit computations as in $B_\sen$-case.


\item Compared with the proof of  \cite[Thm 3.12]{AHLB1}, our proof of Lem \ref{lem str of bastndrm} is more conceptual (using infinite dimensional Sen theory, particularly over the Kummer tower) and simpler (it reduces to Lie algebra cohomology computation which is much easier).  

\item More crucially, Prop \ref{prop: a small implies analytic} will be of central importance in following applications. Indeed, although the cohomology comparison Lem \ref{lem str of bastndrm}(3) can be deduced from that of \cite{AHLB1} using d\'evissage, the \emph{analytic} cohomology comparison in Prop \ref{prop: a small implies analytic}  does \emph{not} follow from work of \cite{AHLB1}. As we discussed in Rem \ref{rem: new Sen}: in  \cite{AHLB1}, the Sen operator can be \emph{linearly} extended over $C$; but this is not possible for $\bdrplusm$-representations when $m\geq 2$. We have to use (locally) analytic Sen theory   developed here to \emph{descend} representation to the \emph{arithmetic} level, over which  there is actual  \emph{differential equation}.
\end{enumerate} 
\end{remark}

\newpage 
\section{Global enhanced connections as analytic Sen modules} \label{sec: M to W}
 
 The main result of this section is Thm \ref{thm enhanced conn gk conn equiv}, where we  relate (global) small enhanced connections with small $\gk$-equivariant connections. As we shall see in Prop \ref{prop: ana Sen as inverse no conn}, we can recover enhanced connections as \emph{analytic Sen modules} inside $\gk$-equivariant connections. In particular, the proof depends crucially on analytic Sen theory constructed in the previous section.

 \begin{theorem} \label{thm enhanced conn gk conn equiv}
Let $X$ be a smooth quasi-compact  rigid-analytic variety. Pull-back along the morphism of ringed spaces
$$(X_{C, \et}, \o_{X_C, \bdrplus, m})  \to (X_\et,\o_{X, \gs^+_{\dR}, m})$$
induces an equivalence of categories
\begin{equation} \label{eq: functor enhanced to gk equiv}
\MIC_{\en}^{a}(X_\et,\o_{X, \gs^+_{\dR}, m})  
\simeq 
 \MIC_{\gk}^{a}(X_{C, \et}, \o_{X_C, \bdrplus, m}) . 
\end{equation}  
For   corresponding objects $M ,\cm$, we have
\[ \rg(X,  \rg(\phi,\DR(M)) ) \simeq \rg(\gk, \rg(X_C, \DR(\cm))).\] 
 \end{theorem}
\begin{proof}
We sketch the road-map here, all details are discussed in this section. 
The functor is constructed in Cons \ref{def:The functor F}. Cohomology comparison is proved in Prop \ref{prop: construct global morphism no conn}, which formally implies full faithfulness of the functor. Finally, a quasi-inverse functor is constructed in Prop \ref{prop: ana Sen as inverse no conn}.
\end{proof}

\begin{construction} \label{def:The functor F}  
Let $(M, \nabla_M, \phi) \in \MIC_{\en}^{a}(X_\et,\o_{X, \gs^+_{\dR}, m}) $ as   in Def \ref{defn: abs enhanced conn}.  
Define
\[ \calm:= M\widehat \otimes_{\frakS_{\dR}^+} \bdrplus. \]
\begin{itemize}
\item For $g\in G_K$ and a local section $x$ of $M$, set
          \[ g(x) =(\frac{g(E)}{E})  ^{\phi}(x)\]  
          where the RHS is a local section of $\cm$ defined by
  \begin{equation} \label{eq sec 15 g action}
  (\frac{g(E)}{E})^{\phi}(x) :=\sum_{i\geq 0} (\frac{E-g(E)}{-aE})^{[i]}\prod_{r=0}^{i-1}(a\phi-ar)(x),
  \end{equation}
  which converges since $\phi$ is $a$-small.   
  Lem \ref{lem expansion identity} quickly implies this rule induces a well-defined $\bdrplus$-semi-linear $\gk$-action on $\cm$.
  
  \item            Further define   a $t$-connection on $\calm$ by base change. More precisely, locally write $\nabla_M = \sum_{i=1}^d\nabla_{M,i}\otimes\frac{\dlog T_i}{E}$ as in  Def \ref{defn: abs enhanced conn}, then locally we have $\nabla_\cm =\sum_{i=1}^d\nabla_{\cm,i}\otimes\frac{\dlog T_i}{E}$ where $\nabla_{\cm,i}=\nabla_{M,i}\otimes 1$ on $\cm=M\otimes_\gsdr \bdrplus$. This is a $t$-connection as in Def \ref{def:t-connection}, since $E$ and $t$ are both uniformizers of $\bdrplus$. 
  
  \item We claim   $\nabla_\cm$ is a $G_K$-equivariant connection as in Def \ref{def:t-connection}. Indeed, it reduces to proving for each $i$,
$ g \circ \frac{\nabla_i}{E} = \frac{\nabla_i}{E} \circ g$,  which reduces to 
\[(\frac{g(E)}{E})^{\phi-1}\circ\nabla_i=\nabla_i\circ(\frac{g(E)}{E})^{ \phi }.\]
Via the expression \eqref{eq sec 15 g action}, it suffices to prove that for any polynomial $f(\phi)$ with variable $\phi$, we have
\[ f(\phi -1) \circ \nabla_i =  \nabla_i  \circ f(\phi  ). \]
When $f(\phi)=\phi$, this is the explicit relation in Notation \ref{nota: enhanced conn local}; the general case for $f(\phi)=\phi^n$ follows by easy induction.

   
\end{itemize}
In conclusion, the above construction induces a functor
 \[  \MIC_{\en}^{a}(X_\et,\o_{X, \gs^+_{\dR}, m})  \to \MIC_t^{G_K}(X_{C, \et}, \o_{X_C, \bdrplus, m}).\]
 \end{construction}
  
  \begin{remark}
  Note in the local case, the formulae in Cons \ref{def:The functor F}  are  exactly the ones in Prop \ref{prop:explicit description of Res}\eqref{item loggamma}.
  \end{remark}

\begin{prop} \label{prop M to D}
 Let $M  \mapsto \cm$ be two corresponding objects related by the functor in Construction \ref{def:The functor F}.  Let $D_{\sen, \kinfty}(\cm) \in \MIC_\en(X_{\et}, \o_{X,  \kinfty[[E]], m})$ be the object constructed via Prop \ref{prop: from t-conn to enhanced Kummer-Sen}. 
    We have  
    \[ D_{\sen, \kinfty}(\cm)=M\otimes_\gsdr \kinfty[[E]] \]
   such that $\nabla_D=\nabla_M\otimes 1$, $\phi_D=\phi_M\otimes 1+1 \otimes \phi_{\kinfty[[E]]}$.  
  As a consequence
    \[\rg(\phi, \DR(M)) \otimes_K \kinfty \simeq \rg(\phi, \DR(D)).\] 
\end{prop} 
\begin{proof}
 It suffices to prove local sections of $M$ are $(\gamma=1,\tau\dla)$-vectors. In fact, they are even  $(\gamma=1,\tau\dan)$-vectors, which follows from inspection of the formula \eqref{eq sec 15 g action}.
\end{proof}

 \begin{prop} \label{prop: construct global morphism no conn}
Let $M  \mapsto \cm$ be two corresponding objects from Construction \ref{def:The functor F}. 
There is a morphism of complex of sheaves on $X$
  \[ \rg(\phi,M) \to \rg(\gk, \cm),\]
which induces quasi-isomorphism
\[ \rg(X,  \rg(\phi,\DR(M)) ) \simeq \rg(\gk, \rg(X_C, \DR(\cm))).\] 
 \end{prop}
\begin{proof}
The push-forward of $\cm$ along $X_C \to X$ is precisely $ \rg(\gk, \cm)$. The desired morphism $ \rg(\phi,M) \to \rg(\gk, \cm)$ is
$$
\begin{tikzcd}
M \arrow[rr, "\phi"] \arrow[d, hook]  &  & M \arrow[d, "{y\mapsto (g\mapsto F(\phi,g)(y))}"] &  &                                    &  &         \\
\cm \arrow[rr, "x\mapsto (g-1)(x)"] &  & {\rC(G_K,\cm)} \arrow[rr]                       &  & {\rC(G_K^2,\cm)} \arrow[rr] &  & \cdots.
\end{tikzcd}
$$
where the right vertical arrow is the map
\[F(\phi,g) = ((\frac{g(E)}{E})^{\phi}-1)/\phi:  M \to C(\gk, \cm).\]
The above diagram, after incorporating connections, induces 
$$
\begin{tikzcd}
\DR(M) \arrow[rr, "\phi"] \arrow[d, hook]  &  & \DR(M) \arrow[d, "{y\mapsto (g\mapsto F(\phi,g)(y))}"] &  &                                    &  &         \\
\DR(\calm) \arrow[rr, "x\mapsto (g-1)(x)"] &  & {\rC(G_K,\DR(\calm))} \arrow[rr]                       &  & {\rC(G_K^2,\DR(\calm))} \arrow[rr] &  & \cdots.
\end{tikzcd}
$$
To see the induced morphism 
\[ \rg(X,  \rg(\phi,\DR(M)) ) \to \rg(\gk, \rg(X_C, \DR(\cm)))\] 
is a quasi-isomorphism, one can base change along $K \to \kinfty$, then use Prop \ref{prop M to D} and Prop \ref{prop: from t-conn to enhanced Kummer-Sen} to conclude.   
\end{proof}

\begin{prop} \label{prop: ana Sen as inverse no conn}
Let $\cm \in  \mic_{\gk}^{a}(X_{C, \et}, \o_{X_C, \bdrplus, m})$.
For an affinoid open $\spa(A, A^+) =U \in X$, define an $A \otimes_K \gs^+_{\dR,m}$-module
\[ M(U):= ((\calm(U_C))^{G_L})^{\gamma=1, \tau\dan }. \]
This rule defines an object 
\[ M \in \mic^{a}_\en(X_\et,\o_{X, \gs^+_{\dR}, m})   \]
 such that its pullback to $(X_{C, \et}, \o_{X_C, \bdrplus, m})$ is $\calm$.
\end{prop} 
\begin{proof}
One only needs to check that $M$ is a vector bundle.
Let $U$ be an affinoid such that $W=\cm(U_C)$ is a finite free module over $A\hatotimes \bdrplusm$ with a $\gk$-action. It is relatively locally analytic as in Def \ref{Def: rel lav rep}.
Indeed, it suffices to check the reduction $W/E$ is relatively locally analytic: to construct an ON basis satisfying condition in Def \ref{Def: rel lav rep}, simply use any ON basis of $A$, and use an $A\hatotimes C$-basis of $W$. Prop \ref{prop: a small implies analytic} then implies
\[ M(U)\otimes_K \kinfty \simeq \cm(U_C); \]
and thus $M(U)$ is   finite free over $A\hatotimes \gsdrm$. 
 \end{proof}

\newpage 
\addtocontents{toc}{\ghblue{Final theorem}}
\section{Global absolute $\mathbbl{\Delta}_\dR^+$-crystals}

We finally complete the proof of Theorem \ref{thm: intro main thm} which we copy here.  

\begin{theorem} \label{thm: final main thm} 
 Let $\mathfrak X$ be a quasi-compact smooth (resp. semi-stable) formal scheme over $\ok$, let  $\ast =\emptyset$ (resp. $\ast= \log$) accordingly; let $X$ be its rigid generic fiber.
Fix a compatible system $\pi_n$ as in Notation \ref{notafields}; these choices lead to a commutative diagram of tensor functors:
\[
\begin{tikzcd}
{ \MIC_{\en}^{a}(X_\et,\o_{X, \gs^+_{\dR}, m})} \arrow[d, hook] &  & {\vect(\fkxastpris, \pris_{\dR,m}^+) } \arrow[ll, "\simeq"'] \arrow[rr, "\simeq"] \arrow[d, hook] &  & {\vect^{a}(\xproet, \bbdrplusm)} \arrow[d, hook] \\
{\MIC_{\gk}(X_{C, \et}, \o_{X_C, \bdrplus, m})}                                     &  & {\vect(\fkxastprisperf, \pris_{\dR,m}^+) } \arrow[rr, "\simeq"] \arrow[ll, "\simeq"']        &  & { \vect(\xproet, \bbdrplusm)}.                         
\end{tikzcd}
\]
Let $\bm \in \vect(\fkxastpris, \pris_{\dR,m}^+)$, and denote the corresponding objects  by
\[
\begin{tikzcd}
M \arrow[d,mapsto]  &  & \bm \arrow[ll,mapsto] \arrow[rr,mapsto] \arrow[d,mapsto] &  & \mathbb W \arrow[d,mapsto] \\
{\calm} &  & \bm^\perf \arrow[ll,mapsto] \arrow[rr,mapsto]     &  & \mathbb W          
\end{tikzcd}
\]
then we have  $K$-linear quasi-isomorphisms of cohomology theories
\[ 
\begin{tikzcd}
{\rg(\phi, \DR(M))} \arrow[d, "\simeq"] &  & {\rg(\fkxastpris, \bm)} \arrow[ll, "\simeq"'] \arrow[rr, "\simeq"] \arrow[d, "\simeq"] &  & {\rg(\xproet, \bW)} \arrow[d, "="] \\
{\rg(\gk, \mathrm{DR}(\calm))}           &  & {\rg(\fkxastprisperf, \bm^\perf)} \arrow[ll, "\simeq"'] \arrow[rr, "\simeq"]           &  & {\rg(\xproet, \bW)}          .     
\end{tikzcd} \]
Here,   $\rg(\gk, \DR(\calm)$ be the $\gk$-cohomology of the de Rham complex associated to $(\calm, \nabla)$; $\rg(\phi, \DR(M))$ is the ``$\phi$-enhanced" de Rham complex  of $(M, \nabla)$. 
\end{theorem}
\begin{proof} 
At this point, we have a diagram
\[
\begin{tikzcd}
{ \MIC_{\en}^{a}(X_\et,\o_{X, \gs^+_{\dR}, m})} \arrow[d, hook] &  & {\vect(\fkxastpris, \pris_{\dR,m}^+) } \arrow[d] \arrow[lld, dashed] \arrow[rrd, dashed] &  & {\vect^{a}(\xproet, \bbdrplusm)} \arrow[d, hook] \\
{\MIC_{\gk}(X_{C, \et}, \o_{X_C, \bdrplus, m})}                         &  & {\vect(\fkxastprisperf, \pris_{\dR,m}^+) } \arrow[rr, "\simeq"] \arrow[ll, "\simeq"']    &  & { \vect(\xproet, \bbdrplusm)}                   
\end{tikzcd}
\]
where $\simeq$ resp. $\hookrightarrow$ denotes \emph{known} equivalences resp. full faithfulness. 
Here the bottom equivalences are proved in \S \ref{sec: crystal perf pris} and \S \ref{sec: global RH}. The left vertical full faithful functor is constructed in \S \ref{sec: M to W}; the right vertical arrow is obvious inclusion.  
However, the central vertical functor, which is naturally induced by restriction, is \emph{not}  yet known to be fully faithful. Nonetheless, the two dotted arrows are constructed via composition. In addition, we know the left resp. right dotted arrow lands in the subcategory $\MIC^a_{\gk}(X_{C, \et}, \o_{X_C, \bdrplus, m})$
resp. $\vect^{a}(\xproet, \bbdrplusm)$ since this can be checked locally. That is, when $\fkx=\spf R$ where $R$ is small in the sense of Notation \ref{nota: bk prism abs pris}, we know the pro-\'etale realization of a crystal is  $a$-small following explicit formula in Prop \ref{prop:explicit description of Res} and Prop \ref{prop M to D}. 
Using  Thm \ref{thm enhanced conn gk conn equiv}, we can lift the dotted arrows, and
finally have the \emph{complete} diagram as stated in the theorem:
\[
\begin{tikzcd}
{ \MIC_{\en}^{a}(X_\et,\o_{X, \gs^+_{\dR}, m})} \arrow[d, hook] &  & {\vect(\fkxastpris, \pris_{\dR,m}^+) } \arrow[d] \arrow[ll, dashed] \arrow[rr, dashed] &  & {\vect^{a}(\xproet, \bbdrplusm)} \arrow[d, hook] \\
{\MIC_{\gk}(X_{C, \et}, \o_{X_C, \bdrplus, m})}                         &  & {\vect(\fkxastprisperf, \pris_{\dR,m}^+) } \arrow[rr, "\simeq"] \arrow[ll, "\simeq"']  &  & { \vect(\xproet, \bbdrplusm)}.                   
\end{tikzcd}
\]
As we already we have Thm \ref{thm enhanced conn gk conn equiv}, it only remains   to prove:
\begin{enumerate}
    \item The functor (constructed via composition) $\vect(\fkxastpris, \pris_{\dR,m}^+) \to  \MIC_{\en}^{a}(X_\et,\o_{X, \gs^+_{\dR}, m})$ is an equivalence;
    \item and the morphism
$\rg(\fkxastpris, \bm)  \to \rg(\fkxastprisperf, \bm^\perf) \simeq \rg(X,\rg(\phi, \DR(M))) $
---where the first map is induced by restriction and the second map is constructed via Prop \ref{prop: construct global morphism no conn},--- is a quasi-isomorphism.
\end{enumerate}  
Both statements can now be checked Zariski locally since all  categories in the diagram satisfy Zariski descent. 
When $\fkx=\spf R$ where $R$ is small in the sense of Notation \ref{nota: bk prism abs pris}, the equivalence of Item (1) is proved in \S \ref{sec: loc abs pris}, and cohomology comparison in Item (2) is proved in \S \ref{sec: coho abs local}.  
\end{proof}

\begin{remark}
All the categories in Theorem \ref{thm: final main thm}  are intrinsically defined (independent of choices of $\pi$ etc.); but the left vertical arrow (cf.~ Construction \ref{def:The functor F}) depends on choice of $E$ (namely, choices of $\pi$) and its image inside $\ainf$ (namely, choices of $\pi_n$ for all $n$).
\end{remark}

\newpage

\section{Appendix: Log structures on perfect log prisms} \label{sec log perf prism}

The main result in this section is Prop \ref{prop:the same perfect site}, which proves that for $R$ a small semi-stable $\ok$-algebra, the forgetful functor induces  an isomorphism of sites $$(R)^{\perf}_{\Prism,\log}\simeq (R)^{\perf}_{\Prism};$$
namely, for a perfect prism on $R$, there exists one and only one extension to a perfect log-prism.

\begin{notation}
     Let $\frakX = \Spf(R)$ be a small semi-stable formal scheme over $\calO_K$; that is, there exists a ($p$-completely) \'etale morphism (called a chart on $\frakX$ or $R$)
  \[\Box:\calO_K\za T_0,\dots,T_r,T_{r+1}^{\pm 1},\dots,T_d^{\pm 1}\ya/(T_0\cdots T_r - \pi)\to R.\]
  Let $X = \Spa(R[\frac{1}{p}],R)$ be the generic fiber of $\frakX$ and then $\calO_X^{\times}\cap\calO_{\frakX}\to \calO_{\frakX}$ defines a log-structure on $\frakX$. Denote by $M_R \to R$ the induced log-structure on $R$ and then it is associated to the pre-log-structure
  \[P_r:=\bigoplus_{i=0}^r\bN\cdot e_i\xrightarrow{e_i\mapsto T_i,~\forall~i}R.\]
  For example, when $R = \calO_K$, we have $M_{\calO_K} = K^{\times}\cap \calO_K = \calO_K\setminus\{0\}$, which is associated to the pre-log-structure $\bN\cdot e\xrightarrow{e\mapsto \pi}\calO_K$.
\end{notation}

  \begin{lem}\label{lem:existence of log-structure}
      For any log-prism $(A,I,M)\in (R)_{\log,\Prism}$, there are elements $t_i$'s in $A$ lifting the images of $T_i$ via the natural morphism $R\to A/I$ such that the log-structure $M$ is associated to the pre-log structure \[P_r=\bigoplus_{i=0}^r\bN\cdot e_i\xrightarrow{e_i\mapsto t_i,~\forall~i}A. \]
  \end{lem}
  \begin{proof}
      For $R = \calO_K$, the result is exactly \cite[Lem. 5.0.10]{DL23}. The general case follows from a similar argument as follows: 
      As the log-structures on $A/I$ associated to the pre-log-structures $(M\xrightarrow{\alpha} A\to A/I)$ and $(P_r\to R\to A/I)$ coincide, there are $m_i$'s in $M$ such that $\alpha(m_i) = \overline u_iT_i$ in $A/I$ for some $\overline u_i\in (A/I)^{\times}$. As $A$ is $I$-complete, one can find units $u_i\in A^{\times}$ lifting $\overline u_i$ for all $i$. Replacing $m_i$'s by $u_i^{-1}m_i$'s, we may assume $\alpha(m_i) = T_i$ modulo $I$ for all $i$. Put $t_i:=\alpha(m_i)$ and then we obtain a pre-log-structure $P_r\xrightarrow{e_i\mapsto t_i,~\forall~i}A$ whose associated log-structure $P_r^a\to A$ factors as $P_r^a\to M\to A$ and induces the log-structure $P_r\to R$ modulo $I$. This forces that $P_r^a = M$ (cf. \cite[Th. 8.36]{Ols05}).
  \end{proof}

  \begin{lem}\label{lem:rigidity of log-structrue}
      Let $(A,I,M)\in (R)_{\Prism,\log}$, let $\overline A:=A/I$,  and let  $\overline B$ be an $\overline A$-algebra. Then the rule
      \[(B,J,M^{\prime}) \mapsto (B,J)\] 
      induces   an equivalence  of categories
      \[\{\text{$(B,J,M^{\prime})\in (R)_{\Prism,\log}$ over $(A,I,M)$ with $B/J = \overline B$}\} \simeq \{\text{$(B,J)\in (R)_{\Prism}$ over $(A,I)$ with $B/J = \overline B$\}},\]
      A quasi-inverse functor sends   $(B,J)$ to the log-prism $(B,J,M^{\prime})$ such that the log-structure $M^{\prime}$ is induced by the pre-log-structure $(M\to A\to B)$.
  \end{lem}
  \begin{proof}
      It is clear that for any $(B,J)\in (R)_{\Prism}$ over $(A,I)$, the log-structure $M^{\prime} \to B$ corresponding to the composite $(M\to A\to B)$ makes $(B,J,M^{\prime})$ a well-defined object in $(R)_{\Prism,\log}$ over $(A,I,M)$. So it is enough to show that for any $(B,J,M^{\prime})\in(R)_{\Prism,\log}$ over $(A,I,M)$, the log-structure $M^{\prime}$ must be induced from that on $A$ as above. Equip $\overline B$ with the log-structure (denoted by $N\to B$) induced from that on $\overline A$. By \cite[Lem. 8.22 and Lem. 8.26]{Ols05}, we have a quasi-isomorphism of $p$-complete log-cotangent complexes
      \[\widehat \rL_{(N\to \overline B)/(M\to \overline A)}\simeq\widehat \rL_{\overline B/\overline A}\]
      yielding an isomorphism of categories
      \[\{\text{liftings $(M\to A)\to (N^{\prime}\to B)$ of $(M\to \overline A)\to (N\to \overline B)$}\} \xrightarrow{\cong} \{\text{liftings $A\to B$ of $\overline A\to \overline B$}\}.\]
      Then the result follows.
  \end{proof}
  \begin{cor}\label{cor:rigidity of log-structrue}
      Let $(A,I,M)\to (B,IB,N)$ be a covering in $(R)_{\Prism,\log}$, which induces a covering $(A,I)\to (B,IB)$ in $(R)_{\Prism}$. Let $(B^{\bullet},IB^{\bullet})$ be the \v Cech nerve associated to the covering $(A,I)\to (B,I)$ and $N^{\bullet}\to B^{\bullet}$ be the log-structure associated to $M\to A\to B^{\bullet}$. Then $(B^{\bullet},IB^{\bullet},N^{\bullet})$ is the \v Cech nerve associated to the covering $(A,I,M)\to (B,IB,N)$. 
  \end{cor}
  \begin{proof}
     It follows from Lemma \ref{lem:rigidity of log-structrue} by checking the universal property for defining \v Cech nerve directly.
  \end{proof}
  
  Now, we focus on the perfect log-prisms in $(R)_{\Prism,\log}$.

  \begin{lem}\label{lem:factorization}
   Let $(A,I)$ be a perfect prism. Then for any $x\in A$ such that $\varphi(x) = x^p(1+py)$ for some $y\in A$, the $x$ factors as 
   \[x = [a]\prod_{i=1}^{\infty}(1+p\varphi^{-i}(y))^{p^{i-1}},\]
   where $a$ is the reduction of $x$ modulo $p$.
 \end{lem}
   Note that $[a]\prod_{i=1}^{\infty}(1+p\varphi^{-i}(y))^{p^{i-1}}$ is well-defined as $A$ is classically $p$-complete.
 \begin{proof}
   Put $y_n = \varphi^{-n}(y)$ for any $n\geq 0$. It suffices to show that for any $n\geq 1$,
   \[x \equiv [a]\prod_{i=1}^{n}(1+py_i)^{p^{i-1}}\mod p^{n+1} .\]

   We shall prover this by induction on $n$.
   Write $x = [a]+px_1$ for a unique $x_1\in A$. Then the assumption on $x$ can be restated as follows:
   \[[a]^p+p\varphi(x_1) = ([a]+px_1)^p(1+p\varphi(y_1)).\]
   So we get that 
   \[\varphi(x_1)\equiv \varphi([a]y_1) \mod p\]
   and a fortiori that 
   \[x_1 \equiv [a]y_1 \mod p.\] Therefore, we can write $x = [a](1+py_1)+p^2x_2$ for some $x_2\in A$ (which is unique as $A$ is $p$-torsion free).
   
   Suppose that $x = [a]\prod_{i=1}^n(1+py_i)^{p^{i-1}}+p^{n+1}x_{n+1}$ for some $n\geq 1$ and some $x_{n+1}\in A$. Then the assumption on $x$ says that
   \[\begin{split}
   [a]^p\prod_{i=0}^{n-1}(1+py_i)^{p^i}+p^{n+1}\varphi(x_{n+1}) &= ([a]\prod_{i=1}^n(1+py_i)^{p^{i-1}}+p^{n+1}x_{n+1})^p(1+py)\\
   & \equiv [a]^p\prod_{i=1}^n(1+py_i)^{p^i}(1+py) \mod p^{n+2}.
   \end{split}\]
   As a consequence, we have
   \[p^{n+1}x_{n+1} \equiv [a]\prod_{i=1}^n(1+py_i)^{p^{i-1}}((1+py_{n+1})^{p^n}-1)\mod p^{n+2}.\]
   In particular, we get $x \equiv [a]\prod_{i=1}^{n+1}(1+py_i)^{p^{i-1}}\mod p^{n+2}$ as desired.
 \end{proof}

 As \cite[Th. 3.10]{BS22}, we want to classify the perfect log-prisms in $(R)_{\Prism,\log}$.

 \begin{lem}\label{lem:p-power}
   Let $S$ be a $p$-torsion free perfectoid ring. Then for any $x\in S[\frac{1}{p}]$ such that $x^p\in S$, we have $x\in S$.
 \end{lem}
 \begin{proof}
   This is well-known and we provide a proof here for the convenience of readers.
   By \cite[Lem. 3.9]{BMS1}, there exists a $\varpi\in S^{\flat}$ such that $\varpi^{\sharp}$ is a unit multiple of $p$. Define $\rho:=\varphi^{-1}(\varpi)^{\sharp}$. Then \cite[Lem. 3.10]{BMS1} implies that the Frobenius induces an isomorphism 
   \[S/\rho S\xrightarrow{\cong}S/\rho^p S.\]
   Now given an $x\in S[\frac{1}{p}]$ such that $x^p\in S$, there exists an integer $N$ such that $\rho^Nx\in S$ and let $l$ be the smallest one satisfying the above property. It is enough to show that $l\leq 0$. Otherwise, assume that $l\geq 1$. Then we see that $\rho^lx$ is not zero in $S/\rho S$, which implies that $\rho^{pl}x^p$ does not vanish modulo $\rho^pS$. This contradicts that $l\geq 1$ as well as $x^p\in S$.
 \end{proof}
 \begin{lem}\label{lem:tilt}
     Let $S$ be a perfectoid ring with tilt $S^{\flat}$. Let $a_1,a_2\in S^{\flat}$ and $x\in S$. If $a_1^{\sharp}=a_2^{\sharp}(1+px)$ and $a_1^{\sharp}\mid p$, then there exists a unit $c\in R^{\flat}$ such that $a_1=a_2c$ and $c^{\sharp} = 1+px$.\footnote{We only proved the lemma for $S$ being absolutely integral closed in an early draft. This stronger version is due to some discussions with Heng Du.}
 \end{lem}
 \begin{proof}
   Let $\varpi\in S^{\flat}$ such that $\varpi^{\sharp}$ is a unit multiple of $p$. 
   
   We first assume $S$ is $p$-torsion free and hence $S^{\flat}$ is $\varpi$-torsion free. In this case, $c:=\frac{a_1}{a_2}$ is well-defined in $S^{\flat}[\frac{1}{\varpi}]$ with $c^{\sharp} = (1+px)$. By Lemma \ref{lem:p-power}, for any $n\geq 0$, $(\varphi^{-n}(c))^{\sharp}\in S$. So $c\in S^{\flat}$. Combining this with that $c^{\sharp}\in S^{\times}$, we see that $c\in (S^{\flat})^{\times}$.
   
   In general, by \cite[Lec IV, Prop. 3.2]{Bhattcolumbia}, if we put $Q = S/S[\sqrt{pS}]$, $\overline S=S/\sqrt{pS}$ and $\overline Q=Q/\sqrt{pQ}$, then there is a fibre square of perfectoid rings:
    \begin{equation*}
        \xymatrix@C=0.5cm{
         S\ar[rr]\ar[d]&&Q\ar[d]\\
         \overline S\ar[rr]&&\overline Q.}
    \end{equation*}
   By what we have proved, there exists a $c\in (Q^{\flat})^{\times}$ which is represented by the compatible sequence $\{c_n\}_{n\geq 0}$ such that $a_1=a_2c$ and $c^{\sharp}=c_0 = (1+px)$ in $Q^{\flat}$. Since $c_0$ coincides with $1$ in $\overline Q$, so are $c_n$'s. In particular, denote by $\tilde c_n$ the element $(1,c_n)\in\overline S\times_{\overline Q}Q\cong S$, then $\tilde c_0=1+px$ and $\tilde c_{n+1}^p=\tilde c_n$ for all $n$. Put $\tilde c\in S^{\flat}$ which is represented by $\{\tilde c_n\}$. Then $u=\tilde c$ is desired.
 \end{proof}

 \begin{lem}\label{lem:uniqueness of log-structure}
   For any perfect log-prisms $(A,I,M_1)$ and $(A,I,M_2)$ in $(R)_{\Prism,\log}$ with the same underlying prism, we always have 
   $(A,I,M_1)=(A,I,M_2)$.
 \end{lem}
 \begin{proof}
   By Lemma \ref{lem:existence of log-structure}, for any $0\leq i\leq r$, there are liftings $t_{i1},t_{i2}\in A$ of $T_i$ such that the log structures $M_1\to A$ and $M_2\to A$ are associated to pre-log-structures $(P_r=\bigoplus_{i=0}^r\bN\cdot e_i\xrightarrow{e_i\mapsto t_{i1},~\forall~i}A)$ and $(P_r=\bigoplus_{i=0}^r\bN\cdot e_i\xrightarrow{e_i\mapsto t_{i2},~\forall~i}A)$ for $0\leq i\leq r$, respectively.
   To conclude, it suffices to show that for any $0\leq i\leq r$, there is a unit $u_i\in A^{\times}$ such that $t_{i1}=t_{i2}u_i$. 
   
   By Lemma \ref{lem:factorization}, there are $a_{i1},a_{i2}\in (A/I)^{\flat}$ and $x_{i1},x_{i2}\in A$ such that $t_{i1}=[a_{i1}](1+px_{i1})$ and $t_{i2}=[a_{i2}](1+px_{i2})$ for any $0\leq i\leq r$.
   Let $\theta:A\to A/I$ be the natural surjection. Then for all $i$, we have 
   \[a_{i1}^{\sharp}(1+p\theta(x_{i1})) = T_i = a_{i2}^{\sharp}(1+p\theta(x_{i2})).\]
   As a consequence, for any $0\leq i\leq r$, there exists a $y_i\in A/I$ such that $a_{i1}^{\sharp} = a_{i2}^{\sharp}(1+py_i)$. By Lemma \ref{lem:tilt}, we deduce that $a_{i1}$ and $a_{i2}$ differ from a unit in $(A/I)^{\flat}$. So for any $0\leq i\leq r$, there exists $u_i\in A^{\times}$ such that $t_{i1}=t_{i2}u_i$ as desired.
 \end{proof}

 \begin{prop}\label{prop:the same perfect site}
     The forgetful functor $(A,I,M)\mapsto (A,I)$ induces an isomorphism of sites $(R)^{\perf}_{\Prism,\log}\to (R)^{\perf}_{\Prism}$.
 \end{prop}
 \begin{proof}
     By Lemma \ref{lem:uniqueness of log-structure}, it suffices to show for any $(A,I)\in (R)^{\perf}_{\Prism}$, there exists a log-structure $M\to A$ making $(A,I,M)$ a well-defined log-prism in $(R)_{\Prism,\log}$. Recall the log-structure on $\overline A = A/I$ is induced from that on $R$ and thus associated to the composite $P_r\to R\to \overline A$. Fix a $\varpi \in \overline{A}^{\flat}$ such that $\varpi^{\sharp}$ is a unit-multiple of $p$ in $\overline A$, and denote by $\rho:=\varphi^{-1}(\varpi)^{\sharp}$. By \cite[Lem. 3.9]{BMS1}, for any $0\leq i\leq r$, one can find 
     \[\underline t_i:=(t_{i0},t_{i1},\dots)\in \varprojlim_{x\mapsto x^p}\overline A =  \overline{A}^{\flat}\]
     such that $t_{i0} = \underline{t_i}^{\sharp}\equiv T_i\mod \rho p\overline A$. Noting that $\prod_{i=0}^rT_i = \pi$ divides $p$, one can find $y_i\in \overline A$ such that 
     \[t_{i0} = T_i(1+\rho y_i).\]
     Noting that $1+\rho y_i\in \overline{A}^{\times}$, one can conclude by checking the pre-log-structure $(P_r = \bigoplus_{i=0}^r\bN\cdot e_i\xrightarrow{e_i\mapsto [\underline t_i],~\forall~i}A)$ on $A$ induces a log-structure $M\to A$ as expected.
 \end{proof}
 \begin{cor}\label{cor:the same perfect site}
     Let $\frakX$ be a semi-stable formal scheme (not necessarily small affine) over $\calO_K$ with generic fiber (and the log-structure $\calM_{\frakX} = \calO_X^{\times}\cap\calO_{\frakX}\to\calO_{\frakX}$). Then the forgetful functor $(\frakX)_{\Prism,\log}^{\perf}\to (\frakX)_{\Prism}^{\perf}$ sending each $(A,I,M)$ to $(A,I)$ induces an equivalence of topoi 
     \[\Sh((\frakX)_{\Prism}^{\perf})\simeq \Sh((\frakX)_{\Prism,\log}^{\perf}).\]
 \end{cor}
 \begin{proof}
     This follows from Proposition \ref{prop:the same perfect site} as \'etale locally $\frakX$ is   affine semi-stable.
 \end{proof}

\newpage
  
  \bibliographystyle{alpha}

\begin{thebibliography}{AHLB24}

\bibitem[AHLB22]{AHLB1}
Johannes Ansch\"{u}tz, Ben Heuer, and Arthur-C\'{e}sar Le~Bras.
\newblock $v$-vector bundles on $p$-adic fields and {S}en theory via the
  {H}odge--{T}ate stack.
\newblock {\em preprint}, 2022.

\bibitem[AHLB23]{AHLB2}
Johannes Ansch\"{u}tz, Ben Heuer, and Arthur-C\'{e}sar Le~Bras.
\newblock {H}odge--{T}ate stacks and non-abelian $p$-adic {H}odge theory of
  $v$-perfect complexes on rigid spaces.
\newblock {\em preprint}, 2023.

\bibitem[AHLB24]{AHLB3}
Johannes Ansch\"{u}tz, Ben Heuer, and Arthur-C\'{e}sar Le~Bras.
\newblock {T}he small $p$-adic {S}impson correspondence in terms of moduli
  spaces.
\newblock {\em to appear, Math. Res. Lett.}, 2024.

\bibitem[BC16]{BC16}
Laurent Berger and Pierre Colmez.
\newblock {Th{\'e}orie de {S}en et vecteurs localement analytiques}.
\newblock {\em Ann. Sci. {\'E}c. Norm. Sup{\'e}r. (4)}, 49(4):947--970, 2016.

\bibitem[Bha18]{Bhattcolumbia}
Bhargav Bhatt.
\newblock Prismatic cohomology.
\newblock {\em {E}ilenberg {L}ectures at {C}olumbia {U}niversity,
  \url{http://www-personal.umich.edu/~bhattb/teaching/prismatic-columbia/}},
  2018.

\bibitem[BL22a]{BL-a}
Bhargav Bhatt and Jacob Lurie.
\newblock Absolute prismatic cohomology.
\newblock {\em preprint}, 2022.

\bibitem[BL22b]{BL-b}
Bhargav Bhatt and Jacob Lurie.
\newblock {T}he prismatization of $p$-adic formal schemes.
\newblock {\em preprint}, 2022.

\bibitem[BMS18]{BMS1}
Bhargav Bhatt, Matthew Morrow, and Peter Scholze.
\newblock Integral {$p$}-adic {H}odge theory.
\newblock {\em Publ. Math. Inst. Hautes \'{E}tudes Sci.}, 128:219--397, 2018.

\bibitem[BS22]{BS22}
Bhargav Bhatt and Peter Scholze.
\newblock Prisms and prismatic cohomology.
\newblock {\em Ann. of Math. (2)}, 196(3):1135--1275, 2022.

\bibitem[BS23]{BS23}
Bhargav Bhatt and Peter Scholze.
\newblock Prismatic {$F$}-crystals and crystalline {G}alois representations.
\newblock {\em Camb. J. Math.}, 11(2):507--562, 2023.

\bibitem[Col94]{Colmez94Sen}
Pierre Colmez.
\newblock Sur un r\'esultat de {S}hankar {S}en.
\newblock {\em C. R. Acad. Sci. Paris S\'er. I Math.}, 318(11):983--985, 1994.

\bibitem[CS24]{CS24}
Kestutis Cesnavicius and Peter Scholze.
\newblock Purity for flat cohomology.
\newblock {\em Ann. of Math. (2)}, 199(1):51--180, 2024.

\bibitem[DL23]{DL23}
Heng Du and Tong Liu.
\newblock A prismatic approach to $(\varphi, \hat{G})$-modules and
  {$F$}-crystals.
\newblock {\em to appear, J. Eur. Math. Soc.}, 2023.

\bibitem[DLLZ23]{DLLZ}
Hansheng Diao, Kai-Wen Lan, Ruochuan Liu, and Xinwen Zhu.
\newblock Logarithmic {R}iemann-{H}ilbert correspondences for rigid varieties.
\newblock {\em J. Amer. Math. Soc.}, 36(2):483--562, 2023.

\bibitem[Fon04]{Fon04}
Jean-Marc Fontaine.
\newblock Arithm\'{e}tique des repr\'{e}sentations galoisiennes {$p$}-adiques.
\newblock Number 295, pages xi, 1--115. 2004.
\newblock Cohomologies $p$-adiques et applications arithm\'{e}tiques. III.

\bibitem[GMW]{GMWdR}
Hui Gao, Yu~Min, and Yupeng Wang.
\newblock Prismatic crystals over de {R}ham period sheaf.
\newblock {\em preprint}.

\bibitem[GMW23]{GMWHT}
Hui Gao, Yu~Min, and Yupeng Wang.
\newblock {H}odge--{T}ate prismatic crystals and {S}en theory.
\newblock {\em arxiv:2311.07024}, 2023.

\bibitem[Kis06]{Kis06}
Mark Kisin.
\newblock {Crystalline representations and {$F$}-crystals}.
\newblock In {\em {Algebraic geometry and number theory}}, volume 253 of {\em
  {Progr. Math.}}, pages 459--496. Birkh{\"a}user Boston, Boston, MA, 2006.

\bibitem[KL]{KL2}
Kiran Kedlaya and Ruochuan Liu.
\newblock {Relative $p$-adic {H}odge theory, {II}: imperfect period rings}.
\newblock {\em preprint}.

\bibitem[Kos21]{Kos21}
Teruhisa Koshikawa.
\newblock Logarithmic prismatic cohomology {I}.
\newblock {\em preprint}, 2021.

\bibitem[Liu]{Liustack}
Zeyu Liu.
\newblock On the prismatization of $\mathcal{O}_k$ beyond the {H}odge--{T}ate
  locus.
\newblock {\em preprint}.

\bibitem[Liu08]{Liu08}
Tong Liu.
\newblock {On lattices in semi-stable representations: a proof of a conjecture
  of {B}reuil}.
\newblock {\em Compos. Math.}, 144(1):61--88, 2008.

\bibitem[Liu23]{Liu23}
Zeyu Liu.
\newblock De {R}ham prismatic crystals over {$\calO_K$}.
\newblock {\em Math. Z.}, 305(4):53, 2023.

\bibitem[LZ17]{LZ17}
Ruochuan Liu and Xinwen Zhu.
\newblock Rigidity and a {R}iemann-{H}ilbert correspondence for {$p$}-adic
  local systems.
\newblock {\em Invent. Math.}, 207(1):291--343, 2017.

\bibitem[Mat22]{mathew2022faithfully}
Akhil Mathew.
\newblock Faithfully flat descent of almost perfect complexes in rigid
  geometry.
\newblock {\em Journal of Pure and Applied Algebra}, 226(5):106938, 2022.

\bibitem[MT]{MT}
Matthew Morrow and Takeshi Tsuji.
\newblock Generalised representations as $q$-connections in integral $p$-adic
  {H}odge theory.
\newblock {\em preprint}.

\bibitem[MW]{MW22}
Yu~Min and Yupeng Wang.
\newblock $p$-adic {S}impson correpondence via prismatic crystals.
\newblock {\em to appear, JEMS}.

\bibitem[Ols05]{Ols05}
Martin~C. Olsson.
\newblock The logarithmic cotangent complex.
\newblock {\em Mathematische Annalen}, 333:859--931, 2005.

\bibitem[Pet23]{Pet23}
Alexander Petrov.
\newblock Geometrically irreducible {$p$}-adic local systems are de {R}ham up
  to a twist.
\newblock {\em Duke Math. J.}, 172(5):963--994, 2023.

\bibitem[RC22]{RC22}
Juan~Esteban Rodr\'{\i}guez~Camargo.
\newblock {G}eometric {S}en theory over rigid analytic spaces.
\newblock {\em preprint}, 2022.

\bibitem[Sch13]{Sch13}
Peter Scholze.
\newblock {$p$}-adic {H}odge theory for rigid-analytic varieties.
\newblock {\em Forum Math. Pi}, 1:e1, 77, 2013.

\bibitem[Sen81]{Sen81}
Shankar Sen.
\newblock Continuous cohomology and {$p$}-adic {G}alois representations.
\newblock {\em Invent. Math.}, 62(1):89--116, 1980/81.

\bibitem[Shi18]{Shi18}
Koji Shimizu.
\newblock Constancy of generalized {H}odge-{T}ate weights of a local system.
\newblock {\em Compos. Math.}, 154(12):2606--2642, 2018.

\bibitem[Tia23]{Tia23}
Yichao Tian.
\newblock Finiteness and duality for the cohomology of prismatic crystals.
\newblock {\em J. Reine Angew. Math.}, 800:217--257, 2023.

\bibitem[Wan22]{Wangxiyuan}
Xiyuan Wang.
\newblock Weight elimination in two dimensions when {$p = 2$}.
\newblock {\em Math. Res. Lett.}, 29(3):887--901, 2022.

\end{thebibliography}

\end{document}